\documentclass[12pt]{article}
\usepackage{amsmath,amssymb,amsfonts,amsthm}
\usepackage[all]{xy}

\newcounter{item}[section]
\newcounter{kirshr}
\newcounter{kirsha}
\newcounter{kirshb}
\newenvironment{enumroman}{\setcounter{kirshr}{1}
\begin{list}{(\roman{kirshr})}{\usecounter{kirshr}} }{\end{list}}
\newenvironment{enumarab}{\setcounter{kirshb}{1}
\begin{list}{(\arabic{kirshb})}{\usecounter{kirshb}} }{\end{list}}
\newenvironment{athm}[1]{\vskip3mm\par\noindent
{\bf #1 }. \slshape }
{\upshape\par\vskip10pt minus3pt}
\newtheorem{theorem}{Theorem}[section]

\newtheorem{lemma}[theorem]{Lemma}
\newtheorem{corollary}[theorem]{Corollary}
\newenvironment{demo}[1]{\noindent{\bf #1.}\upshape\mdseries}
{\nopagebreak{\hfill\rule{2mm}{2mm}\nopagebreak}\par\normalfont}
\theoremstyle{definition}

\newtheorem{definition}[theorem]{Definition}

\def\R{\mathbb{R}}

\def\C{{\mathfrak{C}}}
\def\Fm{{\mathfrak{Fm}}}

\def\At{{\bf At}}
\def\Nr{{\mathfrak{Nr}}}
\def\Fr{{\mathfrak{Fr}}}
\def\Sg{{\mathfrak{Sg}}}
\def\Fm{{\mathfrak{Fm}}}
\def\A{{\mathfrak{A}}}
\def\B{{\mathfrak{B}}}
\def\C{{\mathfrak{C}}}
\def\D{{\mathfrak{D}}}
\def\M{{\mathfrak{M}}}
\def\N{{\mathfrak{N}}}
\def\Sn{{\mathfrak{Sn}}}
\def\CA{{\bf CA}}
\def\QA{{\bf QA}}

\def\QEA{{\bf QEA}}

\def\Df{{\bf Df}}

\def\Lf{{\bf Lf}}

\def\PA{{\bf PA}}
\def\PEA{{\bf PEA}}

\def\K{{\bf K}}
\def\K{{\bf K}}

\def\RCA{{\bf RCA}}

\def\Rd{{\ Rd}}
\def\(R)RA{{\bf (R)RA}}
\def\RA{{\bf RA}}

\def\R{\mathbb{R}}

\def\Sc{{\bf Sc}}

\def\tr{{\sf tr}}
\def\c #1{{\cal #1}}
 \def\CA{{\sf CA}}
\def\B{{\sf B}}
\def\G{{\sf G}}
\def\w{{\sf w}}
\def\y{{\sf y}}
\def\g{{\sf g}}

\def\r{{\sf r}}
\def\K{{\sf K}}
\def\tp{{\sf tp}}

 \def\Cm{{\mathfrak{Cm}}}
\def\Nr{{\mathfrak{Nr}}}
\def\SNr{{\bf S}{\mathfrak{Nr}}}

\def\restr #1{{\restriction_{#1}}}
\def\cyl#1{{\sf c}_{#1}}
\def\diag#1#2{{\sf d}_{#1#2}}
\def\sub#1#2{{\sf s}^{#1}_{#2}}

\def\Ra{{\mathfrak{Ra}}}
\def\Ca{{\mathfrak{Ca}}}
\def\set#1{\{#1\} }
\def\Ra{{\mathfrak{Ra}}}
\def\Nr{{\mathfrak{Nr}}}
\def\Tm{{\mathfrak{Tm}}}
\def\A{{\mathfrak{A}}}
\def\B{{\mathfrak{B}}}
\def\C{{\mathfrak{C}}}
\def\D{{\mathfrak{D}}}
\def\E{{\mathfrak{E}}}
\def\A{{\mathfrak{A}}}
\def\B{{\mathfrak{B}}}
\def\C{{\mathfrak{C}}}
\def\D{{\mathfrak{D}}}
\def\E{{\mathfrak{E}}}

\def\GG{{\mathfrak{GG}}}

\def\L{{\mathfrak{L}}}
\def\Rd{{\mathfrak{Rd}}}

\def\At{{\mathfrak{At}}}
\def\L{{\mathfrak{L}}}

\def\CA{{\bf CA}}
\def\RA{{\bf RA}}

\def\RCA{{\bf RCA}}

\def\G{{\bf G}}

\def\F{{\mathfrak{F}}}
\def\At{{\sf{At}}}
\def\N{\mathbb{N}}
\def\R{\mathfrak{R}}

\def\Cs{{\bf Cs}}

\def\L{{\mathfrak L}}

\def\sub#1#2{{\sf s}^{#1}_{#2}}
\def\cyl#1{{\sf c}_{#1}}
\def\diag#1#2{{\sf d}_{#1#2}}

\def\c #1{{\cal #1}}

\def\pa{$\forall$}
\def\pe{$\exists$}

\def\ef{Ehren\-feucht--Fra\"\i ss\'e}

\def\nodes{{\sf nodes}}

\def\restr #1{{\restriction_{#1}}}

\def\Ra{{\mathfrak{Ra}}}
\def\Nr{{\mathfrak{Nr}}}
\def\Z{{\cal Z}}
\def\CA{{\bf CA}}
\def\RCA{{\bf RCA}}
\def\c#1{{\mathcal #1}}

\def\set#1{ \{#1\}}

\def\Ca{{\mathfrak Ca}}

\def\pe{$\exists$}
\def\pa{$\forall$}
\def\Cm{{\mathfrak Cm}}
\def\Sg{{\mathfrak Sg}}
\def\P{{\mathfrak P}}
\def\Rl{{\mathfrak Rl}}
\def\N{{\cal N}}
\def\d{Dedekind-MacNeille}
\def\ls { L\"owenheim--Skolem}
\def\At{{\sf At}}
\def\Ig{{\sf Ig}}

\def\rng{{\sf rng}}
\def\dom{{\sf dom}}

\def\w{{\sf w}}
\def\g{{\sf g}}
\def\y{{\sf y}}
\def\r{{\sf r}}
\def\QEA{{\sf QEA}}
\def\RCA{\sf RCA}
\def\tp{{\sf tp}}
\def\cyl#1{{\sf c}_{#1}}
\def\sub#1#2{{\sf s}^{#1}_{#2}}
\def\diag#1#2{{\sf d}_{#1#2}}

\def\ws{winning strategy}
\def\ef{Ehren\-feucht--Fra\"\i ss\'e}

\def\Mat{{\sf Mat}}

\def\y{{\sf y}}
\def\g{{\sf g}}
\def\r{{\sf r}}
\def\w{{\sf w}}
\def\Uf{{\sf Uf}}
\def\QPEA{{\sf QPEA}}

\def\y{{\sf y}}
\def\g{{\sf g}}
\def\r{{\sf r}}
\def\w{{\sf w}}
\def\RA{{\sf RA}}
\def\CA{{\sf CA}}

\def\CA{{\sf CA}}
\def\Ra{\mathfrak{Ra}}
\def\Lf{{\sf Lf}}
\def\RCA{\sf RCA}

\def\G{\sf G}
\def\EF{\sf EF}
\def\PEA{\sf PEA}

\def\N{\mathbb N}
\def\Z{\mathbb Z}
\def\G{{\bold G}}
\def\Sc{{\sf Sc}}
\def\Df{{\sf Df}}
\def\PA{{\sf PA}}
\def\Id{{\sf Id}}
\def\A{{\cal{A}}}
\def\B{{\mathfrak{B}}}
\def\C{{\mathfrak{C}}}
\def\D{{\mathfrak{D}}}
\def\P{{\mathfrak{P}}}

\def\Ra{{\mathfrak{Ra}}}
\def\Nr{{\mathfrak{Nr}}}
\def\F{{\mathfrak{F}}}
\def\CA{{\bf CA}}
\def\RCA{{\bf RCA}}
\def\c#1{{\mathcal #1}}

\def\Ca{{\mathfrak Ca}}

\def\pe{$\exists$}
\def\pa{$\forall$}
\def\Cm{{\mathfrak Cm}}
\def\Sg{{\mathfrak Sg}}

\def\ls { L\"owenheim--Skolem}
\def\At{{\sf At}}

\def\rng{{\sf rng}}
\def\dom{{\sf dom}}

\def\dim{{\sf dim}}

\def\ls { L\"owenheim--Skolem}

\def\w{{\sf w}}
\def\g{{\sf g}}
\def\y{{\sf y}}
\def\r{{\sf r}}

\def\cyl#1{{\sf c}_{#1}}
\def\sub#1#2{{\sf s}_{[{#1}/ {#2}]}}
\def\diag#1#2{{\sf d}_{#1#2}}

\def\swap#1#2{{\sf s}_{[#1, #2]}}
\def\ef{Ehren\-feucht--Fra\"\i ss\'e}

\def\y{{\sf y}}
\def\g{{\sf g}}
\def\r{{\sf r}}

\def\w{{\sf w}}

\def\K{{\sf K}}
\def\nodes{{\sf nodes}}
\def\edges{{\sf edges}}
\def\PTA{{\sf PTA}}
\def\TA{{\sf TA}}
\def\G{{\bold G}}
\def\Sc{{\sf Sc}}
\def\Df{{\sf Df}}
\def\PA{{\sf PA}}
\def\Id{{\sf Id}}
\def\QEA{{\sf QEA}}
\def\s{{\sf s}}

\def\QPEA{{\sf QPEA}}
\def\CA{{\sf CA}}
\def\K{{\sf K}}
\def\QA{{\sf QA}}
\def\RCA{{\sf RCA}}

\def\A{{\mathfrak{A}}}
\def\Cs{{\sf Cs}}
\def\t{{\sf t}}
\def\si{{i_0, i_1, \cdots, i_k}}
\def\sj{{j_0, j_1, \cdots, j_k}}
\def\tr{{[i_0|j_0]|[i_1|j_1]|\cdots |[i_k|j_k]}}
\def\tt{{(\t^{j_0}_{i_0}\cdots \t^{j_k}_{i_k})^{(\A)}}}
\def\ttr{{(\t^{j_0}_{i_0}\cdots \t^{j_k}_{i_k})^{(\mathfrak{Rd}_{pt}\A)}}}

\def\PTA{{\sf PTA}}

\title{Atom-canonicity,
relativized representations and omitting types in clique guarded semantics and guarded logics}
\author{Tarek Sayed Ahmed \\
Department of Mathematics, Faculty of Science,\\
Cairo University, Giza, Egypt.
  }
\begin{document}
\maketitle

\begin{abstract} We show that several varieties of cylindric-like algebras whose members 
have a neat  embedding property, meaning that they embed into neat reducts of algebras 
in higher dimensions, possibly only finite, are not atom-canonical.
Algebras in such varieties posses relativized representations, and we thereby
obtain many negative results on the omitting types theorem for finite variable fragments of first order 
logic even when we consider  only clique guarded semantics.

As a sample we show that there is a consistent atomic countable theory $T$ using $n$ variables, $n$ finite  $n>2$, 
such that the non-principal
type consisting of co-atoms cannot be omitted even in {\it $n+4$ square} models, where cylindrifiers can find witnesses 
only on $<n+4$ cliques (to be defined). Such models may not also reflect commutativity of quantifiers.
We also show that for every $k\geq 1$, there is
a consistent atomic  such  theory
using only $n$ variables, $n$ finite $n>2$, and a type $\Gamma$ realized in every {\it $n+k+1$ smooth} model,  but $\Gamma$ 
cannot be  isolated by formulas using $n+k$ variables. 

Conversely, we show that when we have quantifier elimination, then countably 
many non-isolated types can be omitted in standard models, 
and if the types happen to be maximal, there could be uncountably many of them, omitted in standard 
models. We obtain positive omitting types theorems 
for finite variable weaker (syntactically) 
fragments of $n$ variable first order logic, when semantics are altered; more precisely, 
when models are relativized differently. This can be viewed as an omitting types theorem for guarded fragments 
of first order logic.

We show that several classes of algebras having a neat embedding
property are not elementary, 
and that the class of strongly representable atom structure of  many cylindic like algebras is not
elementary, too. The latter result is known but we propose a different proof.

We prove several completeness theorems relating the syntactical notion of algebras having a (complete)
neat embedding property,  to the semantical one of having (complete) relativized representations, which are the cylindric counterpart 
of results obtained for relation algebras by Hirsch, Hodkinson and Maddux.

As a sample, we show that for finite $n>2$, the class $S_c\Nr_n\CA_{n+k},$ $k\geq 3$, which coincides with the class of algebras having
complete $n+k$ smooth representations, is not elementary. Here $S_c$ is the operation of forming complete subalgebras. 
We also show, that for $n$ as above, and $k\geq 1$, then the class of algebras having $n+k$ smooth (square) representations 
is a variety,  that the variety of algebras having  $n+k+1$ smooth 
(square) relativized representations is not
finitely axiomatizable over that of having $n+k$ smooth (square) representations, and that for $2<m<n$ the class of algebras having $n$ 
smooth or $n$ flat representations is not finitely axiomatizable over the class
having $n$ square representations.

On neat embeddings for infinite dimension, we show that for any 
infinite ordinal $\alpha$, finite $k\geq 1,$ 
and $r\in \omega$, there exists an atomic algebra 
$\A^r\in  S\Nr_{\alpha}\CA_{\alpha+k+1}\sim S\Nr_{\alpha}\CA_{\alpha+k}$
such that $\Pi_r \A_r/F\in \sf RCA_{\alpha}$ 
for any non-principl ultrafilter $F$ on $\omega$.
This is strictly stronger than an analogous  result recently proved by Robin Hirsch and the author \cite{t} addressing 
other cylindric like algebras,  too, when restricted to $\CA$; 
for here the ultraproduct is {\it representable} which is not the case with the result 
in {\it op.cit}.

Several connections between the notions of neat embeddings, complete representability, and 
strong representability are investigated in some depth and elaborated upon from the view point of the common 
thread  of hitherto established 
representability results.
For example, we show that for any $n>0$, if $\A\in S_c\Nr_n\CA_{\omega}$, 
then $\A$ is strongly representable, in the sense that its \d\ completion is representable; furthermore it has a complete
$\omega$ relativized representation.  We give an example of a necessarily 
uncountable algebra $\A\in \Nr_n\CA_{\omega}$ 
that has no  complete representation. 
\end{abstract}

\section{Introduction}

We follow the notation of \cite{1} which is in conformity with that of \cite{tarski}.
Assume that we have a class of Boolean algebras with operators for which we have a semantical notion of representability
(like Boolean set algebras or cylindric set algebras).
A weakly representable atom structure is an atom structure such that at least one atomic algebra based on it is representable.
It is strongly representable  if all atomic algebras having this atom structure are representable.
The former is equivalent to that the term algebra, that is, the algebra generated by the atoms,
in the complex algebra is representable, while the latter is equivalent  to that the complex algebra is representable.

Could an atom structure be possibly weakly representable but {\it not} strongly representable?
Ian Hodkinson \cite{Hodkinson}, showed that this can indeed happen for both cylindric  algebras of finite dimension $\geq 3$, and relation algebras,
in the context of showing that the class of representable algebras, in both cases,  is not closed under \d\ completions.
In fact, he showed that this can be witnessed on an atomic algebras, so that
the variety of representable relation algebras algebras and cylindric algebras of finite dimension $>2$
are not atom-canonical. (The complex algebra of an atom structure is
the completion of the term algebra.)
This construction is somewhat complicated using a rainbow atom structure. It has the striking consequence
that there are two atomic algebras sharing the same
atom structure, one is representable the other is not.

This construction was simplified and streamlined,
by many authors, including the author \cite{weak}, but Hodkinson's  construction,
as we indicate below, has the supreme advantage that it has a huge potential to prove analogous  
theorems on \d\ completions, and atom-canonicity
for several varieties of algebras
including properly the variety of representable cylindric-like algebras, whose members have a neat embedding property, 
such as polyadic algebras with and without equality and Pinter's substitution algebras.
In fact, in such cases atomic {\it representable countable algebras} 
will be constructed so that their \d\ completions are outside such varieties.

More precisely, we show, 
that for $n>2$ finite, there is is a polyadic equality representable atomic algebra $\A$ such that
$\Rd_{sc}\Cm\At\A\notin \SNr_n\Sc_{n+4}$ inferring that the varieties
$S\Nr_n\K_{n+k}$, for $n>2$ finite for any $k\geq 4,$ 
and for any $\K\in \{\Sc, \QA, \CA, \PEA\}$ are
not closed under \d\ completions.
Such results, as illustrated below 
will have non-trivial (to say the least)  repercussions on omtting types for finite variable fragments.

Our first result shows the existence of weakly representable atom structures that are not strongly representable
to many cylindric-like algebras of relations; the proof
uses three different constructions based on three different graphs, and in the process highlighting 
an explicit situation when rainbows and Monk-like algebras do the same thing.
(There is an unproven meta-theorem to the effect that what can be done with Monk like algebras can be done with rainbows but not the 
other way round, witness theorems \ref{completerepresentation}, \ref{squarerepresentation}.)

The first is due to Hodkinson,
the second was given in \cite{weak} and the third in \cite{k}. The three constructions
presented herein model-theoretically, in the spirit of Hodkinson's rainbow construction,
gives a polyadic atomic representable equality algebra of finite dimension $n>2$
such that  the diagonal free reduct
of its completion is not
representable.

Now that we have two distinct classes,
namely, the class of weakly atom structures and that of strongly atom structures; the most pressing need is to try to
classify them.
Venema proved (in a more general setting) that the former is elementary,
while Hirsch and Hodkinson show that the latter is {\it not} elementary. Their proof is amazing
depending on an ultraproduct of Erdos probabilistic graphs, witness theorem \ref{el}.

We know that there is a sequence of strongly representable atom structures whose ultraproduct is {\it only } weakly representable,
it is {\it not } strongly representable. This gives that the class $\K=\{\A\in \CA_n: \text { $\A$ is atomic and }
\Cm\At\A\in \sf RCA_n\}$ is not elementary, as well.

Here we extend Hirsch and Hodkinson's result to many cylindric-like algebras, answering a question of Hodkinson's for $\PA$ and $\PEA$.
The proof is based on the algebras constructed in \cite{hirsh} by noting that these algebras can be endowed with
polyadic operations (in an obvious way) and that they are generated by elements
whose dimension sets $<n$. The latter implies that an algebra is representable if and only if
its diagonal free reduct
is representable (One side is trivial, the other does not hold in general).

Lately, it has become fashionable in algebraic logic to
study representations of abstract algebras that has a complete representation \cite{Sayed} for an extensive overview.
A representation of $\A$ is roughly an injective homomorphism from $f:\A\to \wp(V)$
where $V$ is a set of $n$-ary sequences; $n$ is the dimension of $\A$, and the operations on $\wp(V)$
are concrete  and set theoretically
defined, like the Boolean intersection and cylindrifiers or projections.
A complete representation is one that preserves  arbitrary disjuncts carrying
them to set theoretic unions.
If $f:\A\to \wp(V)$ is such a representation, then $\A$ is necessarily
atomic and $\bigcup_{x\in \At\A}f(x)=V$.

Let us focus on cylindric algebras for some time to come.
It is known that there are countable atomic $\RCA_n$s when $n>2$,
that have no complete representations;
in fact, the class of completely representable $\CA_n$s when $n>2$, is not even elementary \cite[corollary 3.7.1]{HHbook2}.

Such a phenomena is also closely
related to the algebraic notion of {\it atom-canonicity}, as indicated, which is an important persistence property in modal logic
and to the metalogical property of  omitting types in finite variable fragments of first order logic
\cite[theorems 3.1.1-2, p.211, theorems 3.2.8, 9, 10]{Sayed}.
Recall that a variety $V$ of Boolean algebras with operators is atom-canonical,
if whenever $\A\in V$, and $\A$ is atomic, then the complex algebra of its atom structure, $\Cm\At\A$ for short, is also
in $V$.

If $\At$ is a weakly representable but
not strongly  representable, then
$\Cm\At$ is not representable; this gives that $\RCA_n$ for $n>2$ $n$ finite, is {\it not} atom-canonical.
Also $\Cm\At\A$  is the \d\ completion of $\A$, and so obviously $\RCA_n$ is not closed under \d\ completions.

On the other hand, $\A$
cannot be completely  representable for, it can be shown without much ado,  that
a complete representation of $\A$ induces a representation  of $\Cm\At\A$ \cite[definition 3.5.1, and p.74]{HHbook2}.

Finally, if $\A$ is countable, atomic and has no complete representation
then the set of co-atoms (a co-atom is the complement of an atom), viewed in the corresponding Tarski-Lindenbaum algebra,
$\Fm_T$, as a set of formulas, is a non principal-type that cannot be omitted in any model of $T$; here $T$ is consistent
if $|A|>1$. This last connection  was first established by the  author leading up to \cite{ANT} and more, see e.g \cite{HHbook2}.

The reference \cite{Sayed} contains an extensive discussion of such notions.

In this paper we also introduce a plathora of deep concepts that are $\CA$ analogues of equally 
deep concepts in relations algebras, introduced by 
Hirsch, Hodkinson and Maddux.

We define the cylindric analogues of $\sf RA_n$ (relation algebras that embed in complete and atomic 
relation algebras having $n$ dimensional basis)
and several notions of $n$ 
{\it relativized} representability for a given $\CA_m$ where $n>m$. Such local representations
include $n$ square, $n$ flat and $n$ smooth, notions introduced by Hirsch and Hodkinson for relation algebras. 
(Exact definitions will be provided below). 

We show that such algebras have $n$ square representations,
and that this class, that turns out a canonical conjugated Sahlqvist
variety, can be characterized by a simple $n$ pebbled game atomic $\omega$ rounded game, 
where \pa\ has only one move, namely, a cylindrifier move. These algebras also have $n$ dimensional basis
that are like Maddux's relational basis but formulated for cylindric algebras.
Such basis can be treated as a saturated set of mosaics that gives the relativized representation.

When we impose more conditions on the game making it more difficult for \pe\  to handle, permitting 
\pa\ to play amalgamation moves, too, a \ws\ for \pe\ 
in the resulting $n$ pebbled $\omega$ rounded new game, 
is equivalent to that the algebra has $n$ flat representations; this is a better approximation to genuine representations
for in such representations {\it commutativity of cylindrifiers} are witnessed only on $m$ cliques; where $m<n $ is the dimension.
An $m$ clique is, roughly, a hypergraph having $n$ nodes coming from the base of the representation 
all of whose $m$ hyperedges are labelled by the top element.

The notion $n$ smooth is yet another apparently closer to genuine representations on $m$ cliques.
(It is stronger in case of complete representations not ordinary ones).
Such representations can be characterized by `hyperbasis games'. A hyperbasis consists of 
hypernetworks - which are generalization of Maddux's $n$ dimensional matrices, allowing labelled hyperedges -
whose hyperedges have labels coming from a possibly infinite set; these can also be treated as a saturated set of mosaics 
that can be patched to give $n$ smooth representations.
Syntactically all such classes have a neat embedding property; they embed into the $m$ neat reduct of
$n$ dimensional algebras.

The idea here is that the properties of genuine representations manifest itself 
only locally, that is on $m$- cliques
so we have a notion of clique guarded  semantics, where witnesses of cylindrifiers are only found if 
we zoom in adequately
on the representation by a `movable window'. We may also require that cylindrifier commute on this localized 
level; this is the case with
$n$ flat and $n$ smooth representations.

We also relate algebras having such different neat embedding properties, embedding in a special way in the $m$  neat reduct
of an $n$ dimensionl algebra - possibly satisfying weaker axioms -  
to algebras having (complete) different types of 
$n$  relativized representations (that preseve arbitrary conjuncts) 
from the viewpoint of finite axiomatizability and first order definability. 

We will discover that  basically
varieties whose $m$ dimensional algebras, with $m$ finite $>2$, that have $n>m$ relativized representations 
are not finitely axiomatizable, while those that have $n$ complete relativized representations, when $n\geq m+3$, 
are not elementary. In other words, non-finite axiomatizability results on equational
axiomatizations of the variety of representable
algebras and the non-first order definability of the class of completely repesentable
ones do not survive severely localized relativizations. 

Rainbows and Monk-like algebras will be used. 
In the presence of amalgamation moves, Monk-like algebras are more handy, and
in the absence of them,  rainbows offer solace and they seem to be the only choice. 

The rainbow construction in algebraic logic is invented by Hirsch and Hodkinson \cite{HHbook}.
This ingenious construction reduces finding subtle differences between seemingly
related notions or concepts using a very simple \ef\ forth pebble game between two players \pa\ and \pe\ on two very simple structures.
From those structures a relation or cylindric algebra can be constructed and a \ws\ for either player lifts to a \ws\ on the atom structure
of the algebra, though the number of pebbles and rounds increases in the algebra.

In the case of relation algebras, the atoms are coloured, so that the games are played on colours. For cylindric algebras,
matters are a little bit more complicated
because games are played on so-called coloured graphs, which are models of the rainbow signature coded in an
$L_{\omega_1,\omega}$ theory. The atom structure consists of finite
coloured graphs rather than colours. 

Nevertheless, the essence of the two construction is very similar, because in the cylindric
algebra constructed from $A$ and $B$, the relation algebra atom structure based also on $A$ and $B$
is coded in the cylindric atom structure, but the latter has additional shades of yellow
that are used to label $n-1$ hyperedges coding the cylindric information.
The colours used can slightly vary because of the presence of the shades of yellow, for example in \cite{HH}
a black was used for relation algebras, there were no blacks in the cylindric case.
The moves using a black by \pe\ in the relation algebra game 
were handled by the shades of yellow in the corresponding cylindric `graph game'.

The strategy for \pe\ in a rainbow game is try while, if it doesn't work try black, and finally if it doesn't work
try red. In the latter case she is kind of cornered, so it is the hardest part in the game.
She never uses green.

In the cylindric algebra case, the most difficult part
for \pe\ is to label the edge between apexes of two cones (a cone is a special coloured graph) having
a common base, when she is also forced a red.

In both cases  the choice of a red, when she is forced one,
is the most difficult part for \pa\ and if she succeeds in every round then she wins.

Indeed,  it is always the case that \pa\ wins on a red clique,
using his greens to force \pa\  play an inconsistent triple of reds.
In case of cylindric algebra \pa\ bombards \pe\ with cones having green tints, 
based on the same base.

For Monk-like algebras, this number (the number of pebbles) is a Ramsey large uncontrollable number;
though it can be sometimes controlled, in the availability  of amalgamation moves, like in the case
of the proof used to solve the famous neat embedding problem, initiated by Monk and formulated
as  problem 2.12 in \cite{tarski}, witness theorem \ref{thm:cmnr}.
So as stated earlier, but not in so many words, 
there is this `semi rational, but gounded, feeling in the air' among algebraic logicians that what can be done
with a Monk-like algebra can also be done by a rainbow algebra, but there is no meta-theorem yet, at least to 
the best of our our knowledge, 
to this effect; only a lot of evidence supporting this partially rational feeling and no evidence for
its contrary.

Rundown of our main  results.

Let $n>2$ be finite.
\begin{enumarab}

\item Showing that for any class $\K$ between $\Df_n$ and $\PEA_n,$
there is a weakly representable $\K$ atom structure that is not strongly representable.
In fact, we show that there is a relation algebra $\R$, such that the set of all $n$ basic matrices ${\sf Mat}_n\At\R$
forms a cylindric bases,
and $\Tm{\sf Mat_n}\At\R,$ the term $n$ dimensional cylindric algebra algebra based on it is representable,
but  $\Rd_{df}\Cm{\sf Mat_n}\At\R$, the diagonal free reduct of the complex algebra of 
$\Mat_n\At\R$, is not representable.

This gives the same result for $\R$
cf. theorem \ref{hodkinson}. This result is known, but here we give one
model theoretic construction that addresses simultaneously Monk-like and rainbow algebras.

\item Introducing, for each $n>m>2$, $m$ finite, 
a new class of $m$ dimensional cylindric algebras denoted by ${\sf CB}_{m,n}$.
This class is a strict  approximation to the class $\RCA_m$, when $n$ is finite, and it the $\CA$ analogue of relational
algebras embedding into complete and atomic relation algebras having $n$  dimensional
relational bases introduced by Maddux, cf. definition \ref{basis}. 
In particular, $\bigcap_{k}\sf CB_{m,m+k}=\RCA_m$.

We show that ${\sf CB}_{m,n}$ is a canonical variety for all $n>m$, cf. theorem \ref{can2}, 
that is not atom-canonical, hence not closed under \d\ completions, nor Sahlqvist 
axiomatizable, when $n\geq m+4$ and $m$ finite $>3$,  witness theorem \ref{blowupandblur}.

\item We show that for $\K\in\{\PEA, \CA\},$ the class $S\Nr_n\K_{n+k}$ is not atom-canonical for $k\geq 4$, theorems \ref{can}, \ref{smooth}
\ref{blowupandblur}, extending
a result in \cite{can}. We give two entirely different proofs. This answers a question in \cite[problem 12, p.627]{HHbook}
which  was raised again in \cite[problem 1, p.131]{Sayedneat}. This result gives an entirely new proof of a recent result of the author's and
it also widens the range of algebras. The second proof, uses a rainbow construction conjuncted with a blow up and blur argument in the sense of 
\cite{ANT}  and it works for any class between $\Sc$ and $\PEA$. The special case of only cylindric algebras 
is proved in \cite{can} and reviewed in \cite{recent}.

\item Showing that the class of strongly representable atom structures  for any $\K$ between $\Df_n$ and $\PEA_n$ is not elementary,
theorem \ref{el}. In particular, it follows that for $\K\in \{\PA, \PEA\},$ 
the class $\{\A\in \K: \text {$\A$ is atomic }: \Cm\At\A\in \sf RK_n\}$ is not elementary,
witness corollary \ref{Hodkinson}. This result is due to Hirsch and Hodkinson. We propose a different proof that gives
the result for relation and cylindric algebras in one go.
Here $\sf RK_n$ denotes the class of representable $\K_n$s. 
Our result for $\sf PA$ and $\sf PEA$, answers a question of Hodkinson's \cite[p.284]{AU}, though admittedly he proposed the line of 
the argument which we use.

\item We show that  the omitting types theorem fails for finite variable fragments of first order logic, as long as we have more than $2$
variables,  even if we consider clique guarded
semantics, theorem  \ref{OTT}. This is stronger, in a way,  than the result reported in \cite{Sayed} and proved in detail in \cite{ANT}
using a blow up and blur construction.
The latter result is also approached, and it is first strengthened conditionally, cf. theorem \ref{blurs}.
The condition is the existence of certain finite relation
algebras. Such relation algebras are provided later. They are Monk-like algebras constructed by Hirsch and Hodkinson.

When we exclude witnesses isolating types realized in every relativized model, 
using only $n+k$ variables, $n>2$ and 
$k\geq 1$,  we use Monk-like 
algebras having $n+k$ dimensional hyperbasis, but not $n+k+1$ dimensional 
hyperbasis. Similar but not identical results are obtained in \cite{recent}.

Such a result is a variation on a theme concerning results on 
omitting types proved recently by the author \cite{recent}, but in the last reference the sufficient condition provided
remained conditional.

\item We show that the omitting types theorem fails for finite first order definable 
extensions of $n$ variable fragments of first order logic for  finite $n>2$ $(L_n)$.
We approach the problem of omitting possibly uncountably many types  $< {}^{\omega}2$ 
for $L_n$ theories ($n>2$ finite),   delineating the edges of an independence proof, giving a proof in $ZFC$ of
a statement that on first sight would be suspected of being independent.
Proofs of this kind is often extremely subtle and surprising;  a 
very similar statement is proved here to be independent 
and it is hard to detect the underlying difference, witness theorems, \ref{Shelah1}, \ref{Shelah2}, \ref{independence}.

A challenging oddity here is that the consistency of the independent statement appeals blatantly 
to the Baire category theorem, its independence has affinity to Martin's axiom restricted to countable partially ordered set,
but the proved (in $ZFC$) similar statement 
has no obvious topological counterpart, at least to us. The proof we adopt is a restriction of a proof of Shelah's
on omitting uncountably many types, namely $< 2^{\aleph_0}$ 
in countable first order theories, but restricted to $L_n$.

In contrast to clique guarded semantics, we obtain positive omitting types theorems, when we weaken commutativity of quantifiers in 
our countable $L_n$ theories; consequently our models are relativized differently. This can be viewed as 
an omitting types theorem for certain guarded finite variable fragments
of first order logic.  

\item We prove several completeness theorems relating the syntactical notion of algebras having a (complete)
neat embedding property,  to the semantical one of having (complete) relativized representations, and investigate 
model theoretic problems about them. 

As a sample, we show:

\begin{itemize}

\item the class $S_c\Nr_n\CA_{n+k},$ for finite $n>2$ and $k\geq 3$, which coincides with the class of algebras having
complete $n+k$ smooth representations, is not elementary. We give two different proofs. 
Here $S_c$ is the operation of forming complete subalgebras.

\item the class of $\CA_n$s having $n+k$ square representations (for any $k$)
is a variety, and that or finite $n>2$,  the variety of algebras having  $n+k+1$ square
relativized representations is not
finitely axiomatizable over that of having $n+k$ square representations, for any finite
$k\geq 1$. 

\end{itemize}

For the  first item we use both a rainbow construction, theorem \ref{nofinite} 
and a Monk-like construction, theorem \ref{completerepresentation}.

In the second acse, we use rainbow constructions, theorem \ref{squarerepresentation}
which work best when the games involved are basic
offering \pa\ only {\it one move}, namely,  a cylindrifier move. 
No amalgamation moves are involved in the game. 
This result is the exact $\CA$ analogue of the result of Hirsch and Hodkinson that ${\sf RA}_{n+1}$ is not finitely
axiomatizable over
${\sf RA}_n$ for $n\geq 5$, but now 
presented semantically.

In case {\it we do have} amalgamation moves, their presence renders the use 
of Monk-like algebras highly probable for the problem at hand, witness theorems, \ref{thm:cmnr} and \ref{2.12}.

\item Showing that for a finite algebra in $\CA_3$ and $\PEA_3$,
it is undecidable whether $\A\in S\Nr_3\CA_{3+k}$ or not, for $k\geq 3$, theorem \ref{decidability}.
It follows that there are finite algebras $S\Nr_3\CA_{3+k}$s, with {\it no finite} relativized $3+k$ smooth 
representations, though they necessarily have infinite $3+k$ relativized representations, by our 
`$3+k$ completeness
theorem.' 

The underlying idea here is that  even over a finite algebra 
there are infinitely many $3+k$  dimensional hypernetworks ($k\geq 3$), 
hyperlabels are not restricted to come from a 
finite set. This, however, 
does not occur in the cases of $n$ square representations, 
because here, in the corresponding basis, we do not have a hyperlabels. 
In particular, such varieties have the finite algebra on finite 
base property, witness theorem \ref{smoothcompleteness2}.

In contrast, we show that, in {\it all cases}, when
we restrict taking subneat reducts to finite $n$ dimensional algebras, then for any $2\leq m<n$,
every finite algebra has a finite relativized $n$ 
representation, witness theorem \ref{smoothcompleteness1}. 

\item We show that for $2<m<n$, the sub-variety of $\CA_m$s 
having $n$ smooth representations is not finitely axiomatizable over 
the class 
that having $n$ square representations, by showing that $S\Nr_m{\sf CA}_n$ is not finitely 
axiomatizable over $\sf CB_{m,n}$, witness theorem \ref{squarerepresentation}. 
An analogous result holds for the class of $m$ dimensional algebras 
having $n$ flat representations; such a variety
is not finitely axiomatizable over $\sf CB_{m,n}$. 

We learn from such results that commutativity of cylindrifiers 
adds a lot. On the other hand,
$n$ smoothness does not  add to $n$ flatness in case of ordinary representations but it adds to it
in the case of complete ones. 
An algebra having a complete $n$ flat representation may not have a complete one. 

Both notions add to squareness; making a legitimate projection to relation 
algebras, in so much Hirsch and Hodkinson's  hyperbasis add to Maddux's 
relational basis. This distinction appears blatantly in decision procedures, not finite axiomatizability.
For example though the equational theory  of the class ${\sf CB}_{m,n}$ for $2<m<n<\omega$ 
is not finitely axiomatizable, it is decidable; in fact its universal theory is decidable, 
witness theorem \ref{smoothcompleteness2}.

In contrast, we will show below the equational theory of $\CA_3$s that 
have  relativized $n$ smooth representations, when $n\geq 6$, 
is undecidable, and does not have the finite 
algebra finite base property. But if we require 
that dilations are finite, then the equational theory of the resulting class {\it has} the finite algebra finite base property, as pointed
out above, but it {\it remains} 
undecidable, witness theorem \ref{undecidability}.

\item We show that for infinite $\alpha$, $r\in \omega$ and $k\geq 1$, 
there is an algebra  $\B^r\in S\Nr_{\alpha}\CA_{\alpha+k+1}\sim S\Nr_{\alpha}\CA_{\alpha+k}$ 
such that $\Pi_r \B_r\in \RCA_{\alpha}.$ In particular, the former class is not finitely axiomatizable over the latter
and $\RCA_{\alpha}$ is not finitely axiomatizable over $S\Nr_{\alpha}\CA_{\alpha+k}.$ 
This lifts results obtained by Hirsch and Hodkinson for finite dimensions, 
and generalizes a recent of Robin Hirsch and the author \cite{t}, but only,  for cylindric and quasi-polyadic equality
algebras of infinite dimension. 

The reference \cite{t} proves a weaker result but with the bonus that it works for many diagonal free reducts
of cylindric and quasi-polyadic equality algebras. It is highly likely that the stronger result established here works for the 
diagonal free algebras addressed in \cite{t}, 
but further research is needed.  Such a result implies a strong incompleteness theorem for the so-called finitary logics of infinitary relations.
We will compare both results, discussing possibilities. 

The striking difference between the two constructions in the finite dimensional case,
is that the algebras constructed to cover also the diagonal free case are {\it finite}; this is not the case
with the original proof of Hirsch and Hodkinson for $\CA$s.

Another very significant distinction  is that in \cite{t}  {\it three} dimensional 
hypernetworks (and not the usual two dimensional case) are dealt with.
The parameter defining the algebras between two succesive dimensions, call it $r$ (which varies over $\omega$) is not just a number, it has a 
linear order.  This order induces a third dimension. 

In this case the ultraproduct of the constructed 
$m$ dimensional algebras, where $m<n$, over $r$,
is only in $\Nr_m\K_{n+1}$ $(\K$ between $\Sc$ and $\PEA$),  and it is not clear that it is representable like in the $\CA$ case, 
when we have diagonal elements.
In other words, the ultraproduct could make it only to one extra dimension and {\it not}  to $\omega$ many, 
witness theorem \ref{thm:cmnr}.

When $\A\subseteq \Nr_{\alpha}\B$, and $\B\in \CA_{\alpha+k}$, call $\B$ is a {\it $k$ dilation of $\A$.}
If moreover $\A$ generates $\B$ (using all operations of $\B$), then $\B$ is called a 
{\it minimal $k$  dilation of $\A$.} Addressing {\it definability issues} in such logics, 
we show there could be two non-isomorphic algebras in $\RCA_{\alpha}$ 
that generate the same $\omega$ dilation, and dually, there could be an algebra $\A\in \RCA_{\alpha}$
that generate two non-isomorphic $\omega$ minimal dilations
if we restrict our attention to  isomorphisms that fixes this algebra pointwise. This is proved in a categorial framework in \cite{conference}.
\end{enumarab}

There are other new results that we cannot formulate at this early stage, 
witness theorems \ref{main}, \ref{2.12}, \ref{maintheorem}, either for undefined notation, or/and for their length.
The last theorem intersects one of the main theorem in \cite{recent}, 
but it also contains manny new connections.

Very occasionally known results or proofs are given, when we find it useful for stressing certain concepts, or/and
clarifying some ideas.  On the other hand, proofs of slightly new results may be sketched if they are not very relevant to the mainstream.
In all such occasions this will not affect
the flow of the paper, in fact it will add to it; furthermore, every time we do this, it will be 
explicity said,  referring to
the relevant reference, see e.g theorem \ref{decidability}. Sketched proofs will always 
give more than a glimpse of the proof.
This paper also intersects with our recent  submission \cite{recent}; this will be explicitly stated at each point of intersection.

Proofs existing elsewhere are repeated to give a rounded  picture of certain techniques or concepts, by
drawing analogies or stressing differences. This occurs for example in the case of 
blow up and blur constructions of which several instances are sprinkled 
throughout the paper, cf. theorem \ref{blurs}, \ref{blowupandblur}, \ref{decidability}.

We have also included theorems, \ref{11}, \ref{blowupandblur}, 
\ref{main} item (4), \ref{j}, which have been recently submitted for publication 
(though theorem \ref{blowupandblur} proves a stronger result than the result in \cite{can}, 
but the idea is essentially the same), together with their proofs, 
since the techniques used in {\it the proofs} of
such theorems, not only their statement, are complementary and/or enriching 
to several  proof techniques of  the new results obtained  here.

On the other hand, the main new results that are, in particular,  also not addressed in \cite{recent}, \cite{can} and in  \cite{games} 
are theorems:

\ref{hodkinson}, \ref{can2}, \ref{blowupandblur}, \ref{blurs}, \ref{Shelah1}, \ref{Shelah2}, \ref{firstorderdefinable}, \ref{independence} 
\ref{el}, \ref{Hodkinson}, \ref{decidability}, \ref{undecidability}, \ref{smoothcompleteness1}, 
\ref{nofinite}, \ref{smoothcompleteness2}, \ref{completerepresentation}, \ref{squarerepresentation}, \ref{rel}, \ref{longer}, \ref{main}, 
\ref{infinitedistance}. Theorem \ref{maintheorem} intersects with a few 
results proved in \cite{recent} but its bulk is new.

Due to its length the paper is divided into two parts:
After the preliminaries (given for all parts), we have:

\subsection*{Layout of part one}

Part one has four sections:

\begin{enumarab}

\item In the first section we give a model theoretic proof of the existence of several weakly representable 
algebras that are not strongly representable, inferring that the class of 
representable algebras for many cylindric-like algebras of finite dimension  $>2$ is not atom canonical,
theorem  \ref{hodkinson}; this result applies to  any class between $\Df$ and $\PEA$. 

We introduce a new class of cylindric algebras of dimension $m$ $>2$ for each $n>m$, denoted below
by ${\sf CB}_{m,n}$ approximating the class of representable algebras of the same dimension. In fact, such classes 
converge to $\RCA_m$ in the sense 
that $\bigcap_{n\in \omega}{\sf CB}_{m,n}=\RCA_m$.

We show that ${\sf CB}_{m,n}$ is a conjugated Sahlqvist canonical variety for each $n>m.$
This class is the cylindric- algebra analogue of 
$\sf RA_n$, the class of relation algebras that embed into complete and atomic relation algebras having $n$ 
dimensional relational bases; introduced by Maddux and studied extensively 
by Hirsch and Hodkinson \cite{HHbook2}, and others. 

Cylindric algebras belonging 
to this class, namely, to ${\sf CB}_{m,n}$  can be characterized by an easy
atomic game $\omega$ rounded games
played on atomic networks having only $n$ nodes and a \ws\ for \pe\ entails
his response in every round of the game to only
one kind of move, namely, a cylindrifier move by \pa\ . 

They also have relativized representations, that are only {\it $n$ square}, 
and a {\it weak} neat embedding property, meaning that they neatly embed into 
$n$ dimensional cylindric-like algebras, but  the latter may fail commutativity of cylindrifiers.
Such properties are summarized in theorem \ref{can2}.

\item In section 2 we strengthen the result in theorem \ref{hodkinson} showing that several subvarieties of the class
of representable algebras are also not atom-canonical. For the latter result,
we give two entirely different proofs, a semantical model
theoretic one, theorems  \ref{step}, \ref{smooth}, \ref{can} and another one based on blowing up and blurring
a finite rainbow cylindric algebra, theorems \ref{blowupandblur}, \ref{blurs}.
(The latter slightly generalizes a result recently proved by the author \cite{recent}; however, the semantical proof is new.)
We also show the new result that the variety ${\sf CB}_{m,n}$ is not atom - canonical when $n\geq m+4$ and $m$ is finite
$>2$, hence is not closed under \d\ completions, too.

\item In section 3, we give an example of a rainbow construction witnessing that several classes
of algebras having a neat embedding property, and containing the class of completely representable 
algebras are not
elementary \ref{neat}. This is a substantial strengthening of the result of Hirsch and Hodkinson in \cite{HH}.

We apply this purely algebraic result
together with our blow up and blur constructions (applied to both Monk-like  algebras and rainbow
algebras) to show that the omitting types
theorem for finite variable fragments fails even if we consider clique guarded semantics \ref{OTT}.
Several other results on omitting types for finite variable fragments of first order logic are presented, both positive and negative.
\item In section 4, we propose a new proof to the fact that the class of strongly representable
atom structures for several cylindric-like algebras of relations 
is not elementary, using, like Hirsch and Hodkinson  Erdos' probabilistic  graphs, theorem \ref{new}.

The novelty here, formulated for cylindric algebras though it equally applies to all other cases,
is that the required relation and cylindric algebras are constructed in one go.
More precisely, the cylindric algebra's atom structure 
is the set of basic matrices 
of an atomic relation algeba; both 
atom structures are based on
the same graph; both are strongly representable
if and only if  the chromatic number of the graph is infinite.

\end{enumarab}

\subsection*{Layout of part two}

Part two, has three sections:

\begin{enumarab}

\item  In section 5, we show that the question as to 
whether a finite algebra is in various subvarieties of representable algebras of dimension $3$ is undecidable,
theorem \ref{decidability}. Our result uses the known result due to Hirsch and Hodkinson, namely, that it is undecidable to
tell whether a finite relation algebra has  an embedding property into a cylindric algebra of dimension $>4$.

Using this result, we show that their equational theory of such subvarieties is undecidable, too.
We prove a stronger result that addresses higher dimensions, using the known result of N\'emeti's that finitely
generated many cylindric-like free
algebras of dimension $>3$ are not atomic \ref{pairing}. (The result is true also for dimension $3$, but fortunately
we do need this much,  because lowering the dimension from $4$ to $3$, is really a long story, witness 
\cite[Chapter 1]{HHbook2}).

\item In section 6, we approach the two issues 
of first order definability and finite axiomatizability for classes
of algebras having relaivized (complete) representations.
We prove a completeness theorem relating the syntactic notion of neat embeddings to the semantical one of
relativized representation.

All our results are negative as far as these two issues are concerned when the dimension of algebras aproached are 
finite $>2$.
In this respect, we show that the class of algebras having $n$ complete
relativized representations is not elementary,
while the variety of algebras having $n+1$ square relativized representations (${\sf CB}_{m,n+1},$)
is not finitely axiomatizable over the class of algebras having $n$ relativized square representations (${\sf CB}_{m,n})$,
when the dimension $m$ is finite $>2$, witness 
theorems \ref{nofinite}, \ref{completerepresentation} and  \ref{squarerepresentation}.

\item In the final section, notions involving neat reducts and neat embeddings
are approached rather extensively. One of the new results in this section is 
that for $2<m<n<\omega$, the class $S\Nr_m\sf CA_n$ is not finitely axiomatizable 
over  $\sf CB_{m,n}.$ (In case $n=\omega$ both coincide the class $\RCA_m$).

Several new notions that apply to what we call {\it neat} atom structures are  introduced
characterized by games, and approached from the view point of 
first order definability.
The main theorems in this last section, 
are theorems,  \ref{main}, \ref{2.12} and \ref{maintheorem}. In the last theorem a plathora of results on neat 
embeddings in connection to notion like complete representability and strong representabilty are given. 

In fact, this last theorem
has (9) items, establishing inter-connections between seemingly unrelated notions, 
and, conversely, giving delicate counterexamples distinguishing between
seemingly related ones, with similarities and differences, in all cases, illuminating the problem at hand.
In the last item we show that any $\A\in S_c\Nr_n\CA_{\omega}$, has a representable \d\ completion, but if uncountable it may fail to be
completely representable.

\end{enumarab}

\section{Preliminaries}

The notation used is in conformity with  \cite{1} but the following list may help.
For cardinals $m, n$ we write $^mn$ for the set of maps from $m$ to $n$.
If  $U$ is an ultrafilter over $\wp(I)$ and if $\A_i$ is some structure (for $i\in I$)
we write either  $\Pi_{i\in I}\A_i/U$ or $\Pi_{i/U}\A_i$ for the ultraproduct of the $\A_i$ over $U$.
Fix some ordinal $n\geq 2$.
For $i, j<n$ the replacement $[i/j]$ is the map that is like the identity on $n$, except that $i$ is mapped to $j$ and the transposition
$[i, j]$ is the like the identity on $n$, except that $i$ is swapped with $j$.     A map $\tau:n\rightarrow n$ is  finitary if
the set $\set{i<n:\tau(i)\neq i}$ is finite, so  if $n$ is finite then all maps $n\rightarrow n$ are finitary.
It is known, and indeed not hard to show, that any finitary permutation is a product of transpositions
and any finitary non-injective map is a product of replacements.

The standard reference for all the classes of algebras mentioned previously  is  \cite{tarski}.
Each class in  $\set{\Df_n, \Sc_n, \CA_n, \PA_n, \PEA_n, \sf QA_n, \sf QEA_n}$ consists of Boolean algebras with extra operators,
as shown in figure~\ref{fig:classes}, where $\diag i j$ is a nullary operator (constant), $\cyl i, \s_\tau,  \sub i j$ and $\swap i j$
are unary operators, for $i, j<n,\; \tau:n\rightarrow n$.

For finite $n$, polyadic algebras are the same as quasi-polyadic algebra and for the infinite
dimensional case we restrict our attention to quasi-polyadic algebras in $\sf QA_n, \QEA_n$.
Each class is defined by a finite set of equation schema. Existing in a somewhat scattered form in the literature, equations defining
$\sf Sc_n, \sf QA_n$ and $\QEA_n$ are given in the appendix of \cite{t} and also in \cite{AGMNS}.
For $\CA_n$ we follow the standard axiomatization given in definition 1.1.1 in \cite{tarski}.
For any operator $o$ of any of these signatures, we write $\dim(o)\; (\subseteq n)$
for the set of  dimension ordinals used by $o$, e.g. $\dim(\cyl i)=\set i, \; \dim (\sub i j)=\set{i, j}$.  An algebra $\A$ in $\QEA_n$
has operators that can define any operator of $\sf QA_n, \CA_n,\;\Sc_n$ and $\sf Df_n$. Thus we may obtain the
reducts $\Rd_{\sf K}(\A)$ for $\K\in\set{\QEA_n, \sf QA_n, \CA_n, \Sc_n, \sf Df_n}$ and it turns out that the reduct always
satisfies the equations defining the relevant class so $\Rd_{\sf K}\A\in \K$.
Similarly from any algebra $\A$ in any of the classes $\sf QEA_n, \sf QA_n, \sf CA_n, \sf Sc_n$
we may obtain the reduct $\Rd_\Sc\A\in\Sc_n$ \cite{AGMNS}.

\bigskip
\begin{figure}
\[\begin{array}{l|l}
\mbox{class}&\mbox{extra operators}\\
\hline
\Df_n&\cyl i:i<n\\
\Sc_n& \cyl i, \s_{[i/j]} :i, j<n\\
\CA_n&\cyl i, \diag i j: i, j<n\\
\PA_n&\cyl i, \s_\tau: i<n,\; \tau\in\;^nn\\
\PEA_n&\cyl i, \diag i j,  \s_\tau: i, j<n,\;  \tau\in\;^nn\\
\QA_n&  \cyl i, \s_{[i/j]}, \s_{[i, j]} :i, j<n  \\
\QEA_n&\cyl i, \diag i j, \s_{[i/j]}, \s_{[i, j]}: i, j<n
\end{array}\]
\caption{Non-Boolean operators for the classes\label{fig:classes}}
\end{figure}
Let $\K\in\set{\QEA, \QA, \CA, \Sc, \Df}$, let $\A\in \K_n$ and let $2\leq m\leq n$ (possibly infinite ordinals).
The \emph{reduct to $m$ dimensions} $\Rd_m(\K_n)\in\K_m$ is obtained from $\A$ by discarding all operators with indices $m\leq i<n$.
The \emph{neat reduct to $m$ dimensions}  is the algebra  $\Nr_m\A\in \K_m$ with universe $\set{a\in\c A: m\leq i<n\rightarrow \cyl i a = a}$
where  all the operators are induced from $\A$ (see \cite[definition~2.6.28]{tarski} for the $\CA$ case).
More generally, for $\Gamma\subseteq n$ we write $\Nr_\Gamma\A$ for the algebra whose universe
is $\set{a\in\A: i \in n\setminus\Gamma\rightarrow \cyl i a=a}$ with 
all the operators $o$ of $\A$ where $\dim(o)\subseteq\Gamma$.
Let $\A\in\K_m, \; \B\in \K_n$.  An injective homomorphism $f:\A\rightarrow \B$ is a \emph{neat embedding} if the range of $f$
is a subalgebra of $\Nr_m\B$.
The notions of neat reducts and  neat embeddings have proved useful in
analyzing the number of variables needed in proofs,
as well as for proving representability results,
via the so-called neat embedding theorems \cite{tarski,  basim, basim2, Sayedneat}.

Let $m\leq n$ be ordinals and let $\rho:m\rightarrow n$ be an injection.
For any $n$-dimensional algebra $\B$ (substitution, cylindric or quasi-polyadic algebra with or without equality)
we define an $m$-dimensional algebra $\Rd^\rho\B$, with the same universe and Boolean structure as
$\B$, where the $(ij)$th diagonal of $\Rd^\rho\B$ is $\diag {\rho(i)}{\rho(j)}\in\B$
(if diagonals are included in the signature of the algebra), the $i$th cylindrifier is $\cyl{\rho(i)}$, the $i$ for $j$
replacement operator is the operator $\s^{\rho(i)}_{\rho(j)}$ of $\A$, the $ij$ transposition operator is $\s_{\rho(i)\rho(j)}$
(if included in the signature), for $i, j<m$.  It is easy to check, for $\K\in\set{\Df,\Sc, \CA, \QA, \QEA}$,
that if $\B\in\K_n$ then $\Rd^\rho\B\in\K_m$.    

For games we basically follow Hirsch Hodkinson's article \cite{HHbook2}, but not always.
We write $G_k(\A)$ for $G^k(\A)$, or simply $G_k$, for the $k$ rounded atomic game played on atomic networks of a cylindric algebra $\A$.
We will have occasion to deal with several other games that are variants of such games.
For example we may limit the number of pebbles used or adopt stronger games (harder for \pe\ to win) by allowing
\pa\ to make certain amalgamation
moves.

For rainbow cylindric algebras, we deviate from the notation therein,  we find it more convenient
to write $\CA_{\G, \Gamma}$ - treating the greens $\sf G$, as a parameter that can vary -
for the rainbow complex algebra $R(\Gamma)$ \cite[definition 3.6.9]{HHbook2}.
The latter is defined to be $\Cm(\rho(\K))$ where $\K$ is a class of models in the rainbow signature satisfying the
$L_{\omega_1,\omega}$ rainbow theory, and $\rho(\K)$ is the rainbow atom structure as defined in \cite[definitinion 3.6.3]{HHbook2}.

We may also add polyadic operations, obtaining polyadic equality rainbow algebras,
denoted by $\PEA_{\G, \Gamma}.$
It is obvious how to add the polyadic operations; using the notation in \cite[definition 3.6.3]{HHbook2}
$$R_{s_{[ij]}}=\{([f], [g]): f\circ [i,j]=g\}.$$

But in all cases considered our view is the conventional (more restrictive) one adopted in \cite{HH}, \cite{Hodkinson};
we view these models as coloured graphs,
that is complete graphs labelled by the rainbow colours.
As usual,  we interpret `an edge coloured by a green say', to mean that the pair defining
the edge satisfies the corresponding green binary relation.
Our atom structures will consist of surjections from $n$ (the dimension) to finite coloured graphs,
roughly finite coloured graphs.

${\sf QEA}$ coincide with $\PEA$ for finite dimensions,
but in the infinite dimensional case, the difference between these classes is substantial,
for example the latter always has uncountably many operations, which is not the case with the former in case of countable
dimensions.

Furthermore, $\sf QEA$ can be formulated in the context of a system of varieties
definable by a schema in the infinite dimensional case; which is not the case with $\PEA$.
Finally, $\sf QA$ denotes quasipolyadic algebras, and in comparison to $\PA$
the same can be said. To avoid conflicting notation, we write $\sf QEA$ and $\sf QA$
only for infinite dimensions, meaning that we are only dealing with finite cylindrifiers and finite substitutions.

By a graph we will mean a pair $\Gamma=(G, E)$,
where $G\not=\phi$ and $E\subseteq G\times G$ is a reflexive and
symmetric binary relation on $G$. We will often use the same
notation for $\Gamma$ and for its set of nodes ($G$ above). A pair
$(x, y)\in E$ will be called an edge of $\Gamma$. See \cite{graph}
for basic information (and a lot more) about graphs.
For neat reducts we follow \cite{Sayedneat} and for omitting types we follow
\cite{Sayed}.

$L_n$ or $L^n$ will denote first order logic restricted to the first $n$ variables and $L_{\infty, \omega}^n$
will denote $L_{\infty, \omega}$ restricted to the first $n$ variables. (In the latter case infinitary conjunctions, hence also infinitary
disjunctions, are allowed).

We follow Von Neuman's convention for finite ordinals, $n=\{0,\ldots, n-1\}$.

\subsection{Some basic duality}

The action of the non-Boolean operators in a completely additive
atomic Boolean algebra with operators is determined by their behavior over the atoms, and
this in turn is encoded by the atom structure of the algebra.

\begin{definition}(\textbf{Atom Structure})
Let $\A=\langle A, +, -, 0, 1, \Omega_{i}:i\in I\rangle$ be
an atomic Boolean algebra with operators $\Omega_{i}:i\in I$. Let
the rank of $\Omega_{i}$ be $\rho_{i}$. The \textit{atom structure}
$\At\A$ of $\A$ is a relational structure
$$\langle \At\A, R_{\Omega_{i}}:i\in I\rangle$$
where $\At\A$ is the set of atoms of $\A$ as
before, and $R_{\Omega_{i}}$ is a $(\rho(i)+1)$-ary relation over
$\At\A$ defined by
$$R_{\Omega_{i}}(a_{0},
\cdots, a_{\rho(i)})\Longleftrightarrow\Omega_{i}(a_{1}, \cdots,
a_{\rho(i)})\geq a_{0}.$$
\end{definition}
Similar 'dual' structure arise in other ways, too. For any not
necessarily atomic $BAO$ $\A$ as above, its
\textit{ultrafilter frame} is the
structure
$$\A_{+}=\langle {\sf Uf}(\A),R_{\Omega_{i}}:i\in I\rangle,$$
 where ${\sf Uf}(\A)$
is the set of all ultrafilters of (the Boolean reduct of)
$\A$, and for $\mu_{0}, \cdots, \mu_{\rho(i)}\in
{\sf Uf}(\A)$, we put $R_{\Omega_{i}}(\mu_{0}, \cdots,
\mu_{\rho(i)})$ iff $\{\Omega(a_{1}, \cdots,
a_{\rho(i)}):a_{j}\in\mu_{j}$ for
$0<j\leq\rho(i)\}\subseteq\mu_{0}$.
\begin{definition}(\textbf{Complex algebra})
Conversely, if we are given an arbitrary structure
$\mathcal{S}=\langle S, r_{i}:i\in I\rangle$ where $r_{i}$ is a
$(\rho(i)+1)$-ary relation over $S$, we can define its
\textit{complex
algebra}
$$\Cm(\mathcal{S})=\langle \wp(S),
\cup, \setminus, \phi, S, \Omega_{i}\rangle_{i\in
I},$$
where $\wp(S)$ is the power set of $S$, and
$\Omega_{i}$ is the $\rho(i)$-ary operator defined
by$$\Omega_{i}(X_{1}, \cdots, X_{\rho(i)})=\{s\in
S:\exists s_{1}\in X_{1}\cdots\exists s_{\rho(i)}\in X_{\rho(i)},
r_{i}(s, s_{1}, \cdots, s_{\rho(i)})\},$$ for each
$X_{1}, \cdots, X_{\rho(i)}\in\wp(S)$.
\end{definition}
It is easy to check that, up to isomorphism,
$\At(\Cm(\mathcal{S}))\cong\mathcal{S}$ always, and
$\A\subseteq\Cm(\At\A)$ for any
completely additive atomic Boolean algebra with operators $\A$. If $\A$ is
finite then of course
$\A\cong\Cm(\At\A)$.\\

A variety $V$ of Boolean algebras with operators
is atom-canonical if whenever $\A\in V$, and $\A$ is atomic, then $\Cm\At\A\in V$. 
If $V$ is completely additive, then $\Cm\At\A$ is the \d\ completion of $\A$. 
$\A$ is always dense in its \d\ completion and so the \d\ completion of $\A$ is atomic iff $\A$ is atomic.

Not all varieties we will encounter are completely additive; the varieties $\Sc_n$ and $\PA_n$ for $n>1$,
are not.
The canonical extention of $\A$, namely, $\Cm(\Uf(\A))$ will be denoted by $\A^+$ or $\A^{\sigma}$. 
The notation $\A^{\sigma}$ will be used only twice, when part of the proof used happens to overlap with an existing proof in the literature 
using this notation,
as is the case in item (4) of theorem \ref{decidability} and
in theorem \ref{hhchapterbook}. 

This algebra is always 
complete and atomic. $\A$ embeds into $\A^+$ under the usual 'Stone representability' function.
$\A^+$ is the \d\ completion of $\A$ if and only if $\A$ is finite.

We will not need this level of abstraction. More concretely:
\begin{enumerate}
\item{Atom structure of diagonal free-type algebra is $\mathcal{S}=\langle S, R_{{\sf c}_{i}}:i<n\rangle$, where the $R_{{\sf c}_{i}}$
is binary relation on $S$.}
\item{Atom structure of cylindric-type algebra is $\mathcal{S}=\langle S, R_{{\sf c}_{i}}, R_{{\sf d}_{ij}}:i, j<n\rangle$, where the
$R_{{\sf d}_{ij}}$, $R_{{\sf c}_{i}}$ are unary and binary relations on $S$. The
reduct $\mathfrak{Rd}_{df}\mathcal{S}=\langle S,
R_{{\sf c}_{i}}:i<n\rangle$ is an atom structure of diagonal free-type.}
\item{Atom structure of substitution-type algebra is $\mathcal{S}=\langle S, R_{{\sf c}_{i}}, R_{{\sf s}^{i}_{j}}:i, j<n\rangle$, where the
$R_{{\sf d}_{ij}}$, $R_{{\sf s}^{i}_{j}}$ are unary and binary relations on $S$,
respectively. The reduct $\mathfrak{Rd}_{df}\mathcal{S}=\langle S,
R_{{\sf c}_{i}}:i<n\rangle$ is an atom structure of diagonal free-type.}
\item{Atom structure of 
quasi polyadic-type algebra is $\mathcal{S}=\langle S, R_{{\sf c}_{i}}, R_{{\sf s}^{i}_{j}}$, $R_{{\sf s}_{ij}}:i, j<n\rangle$, where the
$R_{{\sf c}_{i}}$, $R_{{\sf s}^{i}_{j}}$ and $R_{{\sf s}_{ij}}$ are binary relations
on $S$. The reducts $\mathfrak{Rd}_{df}\mathcal{S}=\langle S,
R_{{\sf c}_{i}}:i<n\rangle$ and $\mathfrak{Rd}_{Sc}\mathcal{S}=\langle S,
R_{{\sf c}_{i}}, R_{{\sf s}^{i}_{j}}:i, j<n\rangle$ are atom structures of
diagonal free and substitution types, respectively.}

\item{ The atom structure of quasi polyadic equality-type algebra
is 
$$\mathcal{S}=\langle S, R_{{\sf c}_{i}}, R_{{\sf d}_{ij}}, R_{{\sf s}^{i}_{j}}, R_{{\sf s}_{ij}}:i, j<n\rangle,$$ 
where the $R_{{\sf d}_{ij}}$
is unary relation on $S$, and $R_{{\sf c}_{i}}$, $R_{{\sf s}^{i}_{j}}$ and
$R_{{\sf s}_{ij}}$ are binary relations on $S$.
\begin{enumerate}\item{The
reduct $\mathfrak{Rd}_{df}\mathcal{S}=\langle S, R_{{\sf c}_{i}}:i\in
I\rangle$ is an atom structure of diagonal free-type.}
\item{The
reduct $\mathfrak{Rd}_{ca}\mathcal{S}=\langle S, R_{{\sf c}_{i}},
R_{{\sf d}_{ij}}:i, j\in I\rangle$ is an atom structure of
cylindric-type.}
\item{The
reduct $\mathfrak{Rd}_{Sc}\mathcal{S}=\langle S, R_{{\sf c}_{i}},
R_{{\sf s}^{i}_{j}}:i, j\in I\rangle$ is an atom structure of
substitution-type.}
\item{The
reduct $\mathfrak{Rd}_{qa}\mathcal{S}=\langle S, R_{{\sf c}_{i}},
R_{{\sf s}^{i}_{j}}, R_{{\sf s}_{ij}}:i, j\in I\rangle$ is an atom structure of
quasi polyadic-type.}
\end{enumerate}}
\end{enumerate}
\begin{definition}
An algebra is said to be representable if and only if it is isomorphic to a subalgebra of a direct product of set algebras of the same type.
\end{definition}

\begin{definition}\label{strong rep}Let $\mathcal{S}$ be an $n$-dimensional algebra atom
structure. $\mathcal{S}$ is \textit{strongly representable} if every
atomic $n$-dimensional algebra $\A$ with
$\At\A=\mathcal{S}$ is representable.
\end{definition}
Note that for any $n$-dimensional algebra $\A$ and atom
structure $\mathcal{S}$, if $\At\A=\mathcal{S}$ then
$\A$ embeds into $\Cm\mathcal{S}$, and hence
$\mathcal{S}$ is strongly representable iff
$\Cm\mathcal{S}$ is representable.

\section*{Part I}

\section{Weakly representable atom structures}

\subsection{ Weakly representable atom structures that are not strongly representable}

We use fairly simple model theoretic arguments, to prove the non atom-canonicity of several classes consisting 
of algebras embedable into neat reducts of algebras having higher dimensions.
We start wih the case when the `higher dimensions' are infinite. So here, by the Henkin neat embedding theorem,
we are dealing with representable algebras.
This proof unifies Rainbow and Monk like constructions used in \cite{Hodkinson}, \cite{weak}, and \cite{k}. 
Such Monk-like algebras will be abstracted later on, presented as atom structures of models of a first order theory
as presented in \cite{HHbook2}; however, the construction will be  different.

We need some standard model theoretic preparations.
\begin{theorem}
Let $\Theta$ be an $n$-back-and-forth system
of partial isomorphism on a structure $A$, let $\bar{a}, \bar{b} \in {}^{n}A$,
and suppose that $ \theta = ( \bar{a} \mapsto \bar{b})$ is a map in
$\Theta$. Then $ A \models \phi(\bar{a})$ iff $ A \models
\phi(\bar{b})$, for any formula $\phi$ of $L^n_{\infty \omega}$.
\end{theorem}
\begin{proof} By induction on the structure of $\phi$.
\end{proof}
Suppose that $W \subseteq {}^{n}A$ is a given non-empty set. We can
relativize quantifiers to $W$, giving a new semantics $\models_W$
for $L^n_{\infty \omega}$, which has been intensively studied in
recent times. If $\bar{a} \in W$:
\begin{itemize}
\item for atomic $\phi$, $A\models_W \phi(\bar{a})$
iff $A \models \phi(\bar{a})$

\item the Boolean clauses are as expected

\item for $ i < n, A \models_W \exists x_i \phi(\bar{a})$ iff $A \models_W
\phi(\bar{a}')$ for some $ \bar{a}' \in W$ with $\bar{a}' \equiv_i
\bar{a}$.
\end{itemize}

\begin{theorem} If $W$ is $L^n_{\infty \omega}$ definable, $\Theta$ is an
 $n$-\textit{back-and-forth} system
of partial isomorphisms on $A$, $\bar{a}, \bar{b} \in W$, and $
\bar{a} \mapsto \bar{b} \in \Theta$, then $ A \models \phi(\bar{a})$
iff $ A \models \phi(\bar{b})$ for any formula $\phi$ of
$L^n_{\infty \omega}$.
\end{theorem}
\begin{proof} Assume that $W$ is definable by the $L^n_{\infty \omega}$
formula $\psi$, so that $W = \{ \bar{a} \in {}^{n}A:A\models \psi(a)\}$. We may
relativize the quantifiers of $L^n_{\infty \omega}$-formulas to
$\psi$. For each $L^n_{\infty
\omega}$-formula $\phi$ we obtain a relativized one, $\phi^\psi$, by
induction, the main clause in the definition being:
\begin{itemize}
\item $( \exists x_i \phi)^\psi = \exists x_i ( \psi \wedge
\phi^\psi)$.
\end{itemize}
 Then clearly, $ A \models_W \phi(\bar{a})$ iff $ A \models
 \phi^\psi(\bar{a})$, for all $ \bar{a} \in W$.
\end{proof}

The following theorem unifies and generalizes the main theorems in \cite{Hodkinson} and in \cite{weak}.
It shows that sometime Monk-like algebras and rainbow algebras do the same thing.
We shall see that each of the constructions has its assets and liabilities.
In the rainbow case the construction can be refined by truncating the greens and reds to be finite, to give sharper results
as shown in theorem, \ref{can}, \ref{blowupandblur}.

While Monk's algebras in one go gives the required result for both relation and cylindric like algebras, 
and it can be generalized to show that the class of strongly representable atom structures for
both relation and cylindric algebras is not elementary reproving a profound result
of Hirsch and Hodkinson, using Erdos probabilistic graphs, or {\it anti Monk} ultraproducts \ref{el}.

First a piece of notation. For an atomic  relation algebra $\R$, ${\sf Mat_n}\At\R$
denotes the set of basic matrices on $\R$. We say that ${\sf Mat_n}\At\R$ is a polyadic basis
if it is a symmetric cylindric bases
(closed under substitutions). We often denote ${\sf Mat_n}\At\R$
simply by ${\sf Mat_n}$ when $\R$ is clear from context.

The idea of the proof, which is quite technical,  is summarized in the following:

\begin{enumarab}

\item  We construct a labelled hypergraph $M$ that can be viewed as an $n$ homogeneous  model of a certain theory
(in the rainbow case it is an $L_{\omega_1, \omega}$ theory, in the Monk case it is first order.)
This model gets its labels from a fixed in advance graph $\G$; that also determines the signature of $M$.
By $n$ homogeneous we mean that every partial isomorphism of $M$ of size $\leq n$ can be extended
to an automorphism of $M$.

\item We have finitely many shades of red, outside the signature; these are basically non principal ultrafilters, but
they can be used as labels.

\item We build a relativized set algebra based on $M$, by discarding all assignments whose edges are labelled
by shades of red getting $W\subseteq {}^nM$.

\item $W$ is definable in $^nM$ be an $L_{\infty, \omega}$ formula
hence the semantics with respect to  $W$ coincides with classical Tarskian semantics (when assignments are in $^nM$).
This is proved using certain $n$ back and forth systems.

\item The set algebra based on $W$ (consisting of sets of sequences (without shades of reds labelling edges)
satisfying formulas in $L^n$ in the given signature)
will be an atomic algebra such that its completion is not representable.
The completion will consist of interpretations of $L_{\infty,\omega}^n$ formulas; though represented over $W$,
it will not be, and cannot be, representable in the classical  sense.

\end{enumarab}

\begin{theorem}\label{hodkinson}

\begin{enumarab}
\item There exists a polyadic equality atom structure $\At$ of dimension $n$,
such that $\Tm\At$ is representable as a $\sf PEA_n$, but not strongly representable.
In fact, $\Rd_t\Cm\At$ is not representable for any signature $t$ between that of $\sf Df$ and $\sf PEA$.

\item Furthermore,
there exists an atomic  relation algebra $\R$ that that the set of all basic matrices
forms an $n$ dimensional polyadic equality basis
such that $\Tm{\sf Mat}_n\At\R$ is  representable as a polyadic equality algebra, while
$\Rd_{df}\Cm{\sf Mat_n}\At\R$ is not (as a diagonal free cylindric algebra of dimension $n$).
In particular, $\At\R$ is weakly but not strongly representable.
\end{enumarab}
\end{theorem}

\begin{proof}
\begin{enumarab}

\item {\sf Constructing an $n$ homogeneous model}

$L^+$ is the rainbow signature consisting of the binary
relation symbols $\g_i :i<n-1 , \g_0^i: i< |\sf G|$, $\w, \w_i: i <n-2, \r_{jk}^i:  i<\omega, j<k<|\sf R|$
and the $(n-1)$ ary-relation symbols
$\y_S: S\subseteq {\sf G}$
together with a shade of red $\rho$ which is outside the rainbow signature,
but it is a binary relation, in the sense that it can label edges in coloured
graphs. Here we take, like Hodkinson \cite{Hodkinson}, $\sf G=\sf R=\omega$.
graphs. In the following theorem we shall see that by varying these parameters,
namely when $|\sf G|=n+2$ and $|\sf R|=n+1$ we get sharper results.

Let $\GG$ be the class of all coloured graph in this rainbow signature.
Let $T_r$ denote the rainbow $L_{\omega_1,\omega}$ theory \cite[definition 3.6.9]{HHbook2}.
Let $\G$ by a countable disjoint union of cliques each of size $n(n-1)/2$ or $\N$ with edge relation defined by $(i,j)\in E$ iff $0<|i-j|<N$.
Let $L^+$ be the signature consisting of the binary
relation symbols $(a, i)$, for each $a \in \G \cup \{ \rho \}$ and
$i < n$. Let $T_m$ denote the following (Monk) theory:

$M\models T_m$ iff
for all $a,b\in M$, there is a unique $p\in \G\cup \{\rho\}\times n$, such that
$(a,b)\in p$ and if  $M\models (a,i)(x,y)\land (b,j)(y,z)\land (c,l)(x,z)$, then $| \{ i, j, l \}> 1 $, or
$ a, b, c \in \G$ and $\{ a, b, c\} $ has at least one edge
of $\G$, or exactly one of $a, b, c$ -- say, $a$ -- is $\rho$, and $bc$ is
an edge of $\G$, or two or more of $a, b, c$ are $\rho$.
We denote the class of models of $T_m$ which can also be viewed as coloured graphs with labels coming from
$\G\cup {\rho}\times n$ also by $\GG$.

We deliberately
use this notation to emphasize the fact
that the proof for all three cases considered is essentially the same; the
difference is that our our  three term algebras, to be constructed  are set algebras with domain a subset of $^nM$
(where $M$ is an $n$ homogeneous  model that embeds all coloured graphs  to be defined shortly)
are just based on a different graph. So such a notation simplifies matters considerably.

Now in all cases,  there is a countable $n$ homogeneous coloured  graph  (model) $M\in \GG$ of both theories with the following
property:\\
$\bullet$ If $\triangle \subseteq \triangle' \in \GG$, $|\triangle'|
\leq n$, and $\theta : \triangle \rightarrow M$ is an embedding,
then $\theta$ extends to an embedding $\theta' : \triangle'
\rightarrow M$.

We proceed as follows. We use a simple game.
Two players, $\forall$ and $\exists$, play a game to build a
labelled graph $M$. They play by choosing a chain $\Gamma_0
\subseteq \Gamma_1 \subseteq\ldots $ of finite graphs in $\GG$; the
union of
the chain will be the graph $M.$
There are $\omega$ rounds. In each round, $\forall$ and $\exists$ do
the following. Let $ \Gamma \in \GG$ be the graph constructed up to
this point in the game. $\forall$ chooses $\triangle \in \GG$ of
size $< n$, and an embedding $\theta : \triangle \rightarrow
\Gamma$. He then chooses an extension $ \triangle \subseteq
\triangle^+ \in \GG$, where $| \triangle^+ \backslash \triangle |
\leq 1$. These choices, $ (\triangle, \theta, \triangle^+)$,
constitute his move. $\exists$ must respond with an extension $
\Gamma \subseteq \Gamma^+ \in \GG$ such that $\theta $ extends to an
embedding $\theta^+ : \triangle^+ \rightarrow \Gamma^+$. Her
response ends the round.
The starting graph $\Gamma_0 \in \GG$ is arbitrary but we will take
it to be the empty graph in $\GG$.
We claim that $\exists$ never gets stuck -- she can always find a suitable
extension $\Gamma^+ \in \GG$.  Let $\Gamma \in \GG$ be the graph built at some stage, and let
$\forall$ choose the graphs $ \triangle \subseteq \triangle^+ \in
\GG$ and the embedding $\theta : \triangle \rightarrow \Gamma$.
Thus, his move is $ (\triangle, \theta, \triangle^+)$.
We now describe $\exists$'s response. If $\Gamma$ is empty, she may
simply plays $\triangle^+$, and if $\triangle = \triangle^+$, she
plays $\Gamma$. Otherwise, let $ F = \rng(\theta) \subseteq \Gamma$.
(So $|F| < n$.) Since $\triangle$ and $\Gamma \upharpoonright F$ are
isomorphic labelled graphs (via $\theta$), and $\GG$ is closed under
isomorphism, we may assume with no loss of generality that $\forall$
actually played $ ( \Gamma \upharpoonright F, Id_F, \triangle^+)$,
where $\Gamma \upharpoonright F \subseteq \triangle^+ \in \GG$,
$\triangle^+ \backslash F = \{\delta\}$, and $\delta \notin \Gamma$.
We may view $\forall$'s move as building a labelled graph $ \Gamma^*
\supseteq \Gamma$, whose nodes are those of $\Gamma$ together with
$\delta$, and whose edges are the edges of $\Gamma$ together with
edges from $\delta$ to every node of $F$. The labelled graph
structure on $\Gamma^*$ is given by\\
$\bullet$ $\Gamma$ is an induced subgraph of $\Gamma^*$ (i.e., $
\Gamma \subseteq \Gamma^*$)\\
$\bullet$ $\Gamma^* \upharpoonright ( F \cup \{\delta\} ) =
\triangle^+$.
Now $ \exists$ must extend $ \Gamma^*$ to a complete
graph on the same node and complete the colouring yielding a graph
$ \Gamma^+ \in \GG$. Thus, she has to define the colour $
\Gamma^+(\beta, \delta)$ for all nodes $ \beta \in \Gamma \backslash
F$, in such a way as to meet the required conditions.  For rainbow case \pe\ plays as follows:
\begin{enumarab}

\item If there is no $f\in F$,
such that $\Gamma^*(\beta, f), \Gamma^*(\delta ,f)$ are coloured $\g_0^t$ and $\g_0^u$
for some $t,u$, then \pe\ defined $\Gamma^+(\beta, \delta)$ to  be $\w_0$.

\item Otherwise, if for some $i$ with $0<i<n-1$, there is no $f\in F$
such that $\Gamma^*(\beta,f), \Gamma^*(\delta, f)$ are both coloured $\g_i$, then \pe\
defines the colour $\Gamma^+(\beta,\delta)$ to
to be $\w_i$ say the least such

\item Otherwise $\delta$ and $\beta$ are both the apexes on $F$ in $\Gamma^*$ that induce
the same linear ordering on (there are nor green edges in $F$ because
$\Delta^+\in \GG$, so it has no green triangles).
Now \pe\ has no choice but to pick a red. The colour she chooses is $\rho.$

\end{enumarab}
This defines the colour of edges. Now for hyperedges,
for  each tuple of distinct elements
$\bar{a}=(a_0,\ldots, a_{n-2})\in {}^{n-1}(\Gamma^+)$
such that $\bar{a}\notin {}^{n-1}\Gamma\cup {}^{n-1}\Delta$ and with no edge $(a_i, a)$
coloured greed in  $\Gamma^+$, \pe\ colours $\bar{a}$ by $\y_{S}$
where
$S=\{i <\omega: \text { there is a $i$ cone with base  } \bar{a}\}$.
Notice that $|S|\leq F$. This strategy works \cite[lemma 2.7]{Hodkinson}.

For the Monk case \pe\ plays as follows:

Now $ \exists$ must extend $ \Gamma^*$ to a complete
graph on the same node and complete the colouring yielding a graph
$ \Gamma^+ \in \GG$. Thus, she has to define the colour $
\Gamma^+(\beta, \delta)$ for all nodes $ \beta \in \Gamma \backslash
F$, in such a way as to meet the conditions of definition 1. She
does this as follows. The set of colours of the labels in $ \{
\triangle^+(\delta, \phi) : \phi \in F \} $ has cardinality at most
$ n - 1 $. Let  $ i < n$ be a "colour" not in this set. $ \exists$
labels $(\delta, \beta) $ by $(\rho, i)$ for every $ \beta \in
\Gamma \backslash F$. This completes the definition of $ \Gamma^+$.
\\
It remains to check that this strategy works--that the conditions
from the definition of $\GG$ are met. It is
sufficient to note that

\begin{itemize}
 \item if $\phi \in F$ and $ \beta \in \Gamma \backslash F$, then
the labels in $ \Gamma^+$ on the edges of the triangle $(\beta,
\delta, \phi)$ are not all of the same colour ( by choice of $i$ )

\item if $ \beta, \gamma \in \Gamma \backslash F$, then two the
labels in $ \Gamma^+$ on the edges of the triangle $( \beta, \gamma,
\delta )$ are $( \rho, i)$.\\

\end{itemize}
This strategy also works \cite{weak}.
Now there are only countably many
finite graphs in $ \GG$ up to isomorphism, and each of the graphs
built during the game is finite. Hence $\forall$ may arrange to play
every possible $(\triangle, \theta, \triangle^+)$ (up to
isomorphism) at some round in the game. Suppose he does this, and
let $M$ be the union of the graphs played in the game. We check that
$M$ is as required. Certainly, $M \in \GG$, since $\GG$ is clearly
closed under unions of chains. Also, let $\triangle \subseteq
\triangle' \in \GG$ with $|\triangle'| \leq n$, and $ \theta :
\triangle \rightarrow M$ be an embedding. We prove that $\theta$
extends to $\triangle'$, by induction on $d = | \triangle'
\backslash \triangle|.$ If this is $0$, there is nothing to prove.
Assume the result for smaller $d$. Choose  $ a \in \triangle'
\backslash \triangle $ and let $ \triangle^+ = \triangle'
\upharpoonright ( \triangle \cup \{ a \} ) \in \GG$. As, $
|\triangle| < n$, at some round in the game, at which the graph
built so far was $\Gamma$, say, $\forall$ would have played
$(\triangle, \theta, \triangle^+)$ (or some isomorphic triple).
Hence, if $\exists$ constructed $ \Gamma^+$ in that round, there is
an embedding $ \theta^+ : \triangle^+ \rightarrow \Gamma^+ $
extending $\theta$. As $ \Gamma^+ \subseteq M,  \theta^+$ is also an
embedding: $ \triangle^+ \rightarrow M$. Since $ |\triangle'
\backslash \triangle^+| < d, \theta^+$ extends inductively to an
embedding $\theta' : \triangle' \rightarrow M $, as required.

\item {\sf Relativization, back and forth systems ensuring that relativized semantics coincide with the classical semantics}

For the rainbow algebra, let $$W_r = \{ \bar{a} \in {}^n M : M \models ( \bigwedge_{i < j < n,
l < n} \neg \rho(x_i, x_j))(\bar{a}) \},$$
and for the Monk like algebra, $W_m$
is defined exactly the same way by discarding assignments whose edges are coloured by one of the $n$ reds
$(\rho, i)$, $i<n$.  In the former case we are discarding assignments who have a $\rho$ labelled edge,
and in the latter we are discarding assignments
that involve edges coloured by any of the $n$ reds $(\rho, i)$, $i<n$.
We denote both by $W$ relying on context.

The $n$-homogeneity built into
$M$, in all three cases by its construction implies that the set of all partial
isomorphisms of $M$ of cardinality at most $n$ forms an
$n$-back-and-forth system. But we can even go
further. We formulate our statement for the Monk algebra based on $\G$ whose underlying set
is $\N$. (For the other two cases the reader is referred to \cite{Hodkinson} and \cite{weak}).

A definition: Let $\chi$ be a permutation of the set $\omega \cup \{ \rho\}$. Let
$ \Gamma, \triangle \in \GG$ have the same size, and let $ \theta :
\Gamma \rightarrow \triangle$ be a bijection. We say that $\theta$
is a $\chi$-\textit{isomorphism} from $\Gamma$ to $\triangle$ if for
each distinct $ x, y \in \Gamma$,
\begin{itemize}
\item If $\Gamma ( x, y) = (a, j)$ with $a\in \N$, then there exist unique $l\in \N$ and $r$ with $0\leq r<N$ such that
$a=Nl+r$.
\begin{equation*}
\triangle( \theta(x),\theta(y)) =
\begin{cases}
( N\chi(i)+r, j), & \hbox{if $\chi(i) \neq \rho$} \\
(\rho, j),  & \hbox{otherwise.} \end{cases}
\end{equation*}
\end{itemize}

\begin{itemize}
\item If $\Gamma ( x, y) = ( \rho, j)$, then
\begin{equation*}
\triangle( \theta(x),\theta(y)) \in
\begin{cases}
\{( N\chi(\rho)+s, j): 0\leq s < N \}, & \hbox{if $\chi(\rho) \neq \rho$} \\
\{(\rho, j)\},  & \hbox{otherwise.} \end{cases}
\end{equation*}
\end{itemize}

We now have for any permutation $\chi$ of $\omega \cup \{\rho\}$, $\Theta^\chi$
is the set of partial one-to-one maps from $M$ to $M$ of size at
most $n$ that are $\chi$-isomorphisms on their domains. We write
$\Theta$ for $\Theta^{Id_{\omega \cup \{\rho\}}}$.

We claim that for any any permutation $\chi$ of $\omega \cup \{\rho\}$, $\Theta^\chi$
is an $n$-back-and-forth system on $M$.

Clearly, $\Theta^\chi$ is closed under restrictions. We check the
``forth" property. Let $\theta \in \Theta^\chi$ have size $t < n$.
Enumerate $\dom(\theta)$, $\rng(\theta).$ respectively as $ \{ a_0,
\ldots, a_{t-1} \}$, $ \{ b_0,\ldots, b_{t-1} \}$, with $\theta(a_i)
= b_i$ for $i < t$. Let $a_t \in M$ be arbitrary, let $b_t \notin M$
be a new element, and define a complete labelled graph $\triangle
\supseteq M \upharpoonright \{ b_0,\ldots, b_{t-1} \}$ with nodes
$\{ b_0,\ldots, b_{t} \}$ as follows.\\

Choose distinct "nodes"$ e_s < N$ for each $s < t$, such that no
$(e_s, j)$ labels any edge in $M \upharpoonright \{ b_0,\dots,
b_{t-1} \}$. This is possible because $N \geq n(n-1)/2$, which
bounds the number of edges in $\triangle$. We can now define the
colour of edges $(b_s, b_t)$ of $\triangle$ for $s = 0,\ldots, t-1$.

\begin{itemize}
\item If $M ( a_s, a_t) = ( Ni+r, j)$, for some $i\in \N$ and $0\leq r<N$, then
\begin{equation*}
\triangle( b_s, b_t) =
\begin{cases}
( N\chi(i)+r, j), & \hbox{if $\chi(i) \neq \rho$} \\
\{(\rho, j)\},  & \hbox{otherwise.} \end{cases}
\end{equation*}
\end{itemize}

\begin{itemize}
\item If $M ( a_s, a_t) = ( \rho, j)$, then assuming that $e_s=Ni+r,$ $i\in \N$ and $0\leq r<N$,
\begin{equation*}
\triangle( b_s, b_t) =
\begin{cases}
( N\chi(\rho)+r, j), & \hbox{if $\chi(\rho) \neq \rho$} \\
\{(\rho, j)\},  & \hbox{otherwise.} \end{cases}
\end{equation*}
\end{itemize}

This completes the definition of $\triangle$. It is easy to check
that $\triangle \in \GG$. Hence, there is a graph embedding $ \phi : \triangle \rightarrow M$
extending the map $ Id_{\{ b_0,\ldots, b_{t-1} \}}$. Note that
$\phi(b_t) \notin \rng(\theta)$. So the map $\theta^+ = \theta \cup
\{(a_t, \phi(b_t))\}$ is injective, and it is easily seen to be a
$\chi$-isomorphism in $\Theta^\chi$ and defined on $a_t$.
The converse,``back" property is similarly proved ( or by symmetry,
using the fact that the inverse of maps in $\Theta$ are
$\chi^{-1}$-isomorphisms).

We now  derive a connection between classical and
relativized semantics in $M$, over the set $W$:\\

Recall that $W$ is simply the set of tuples $\bar{a}$ in ${}^nM$ such that the
edges between the elements of $\bar{a}$ don't have a label involving
$\rho$ (these are $(\rho, i):  i<n$). Their labels are all of the form $(Ni+r, j)$. We can replace $\rho$-labels by suitable $(a, j)$-labels
within an $n$-back-and-forth system. Thus, we may arrange that the
system maps a tuple $\bar{b} \in {}^n M \backslash W$ to a tuple
$\bar{c} \in W$ and this will preserve any formula
containing no relation symbols $(a, j)$ that are ``moved" by the
system.

Indeed, we show that the
classical and $W$-relativized semantics agree.
$M \models_W \varphi(\bar{a})$ iff $M \models \varphi(\bar{a})$, for
all $\bar{a} \in W$ and all $L^n$-formulas $\varphi$.

The proof is by induction on $\varphi$. If $\varphi$ is atomic, the
result is clear; and the Boolean cases are simple.
Let $i < n$ and consider $\exists x_i \varphi$. If $M \models_W
\exists x_i \varphi(\bar{a})$, then there is $\bar{b} \in W$ with
$\bar{b} =_i \bar{a}$ and $M \models_W \varphi(\bar{b})$.
Inductively, $M \models \varphi(\bar{b})$, so clearly, $M \models_W
\exists x_i \varphi(\bar{a})$.
For the (more interesting) converse, suppose that $M \models_W
\exists x_i \varphi(\bar{a})$. Then there is $ \bar{b} \in {}^n M$
with $\bar{b} =_i \bar{a}$ and $M \models \varphi(\bar{b})$. Take
$L_{\varphi, \bar{b}}$ to be any finite subsignature of $L$
containing all the symbols from $L$ that occur in $\varphi$ or as a
label in $M \upharpoonright \rng(\bar{b})$. (Here we use the fact
that $\varphi$ is first-order. The result may fail for infinitary
formulas with infinite signature.) Choose a permutation $\chi$ of
$\omega \cup \{\rho\}$ fixing any $i'$ such that some $(Ni'+r, j)$
occurs in $L_{\varphi, \bar{b}}$ for some $r<N$, and moving $\rho$.
Let $\theta = Id_{\{a_m : m \neq i\}}$. Take any distinct $l, m \in
n \setminus \{i\}$. If $M(a_l, a_m) = (Ni'+r, j)$, then $M( b_l,
b_m) = (Ni'+r, j)$ because $ \bar{a} = _i \bar{b}$, so $(Ni'+r, j)
\in L_{\varphi, \bar{b}}$ by definition of $L_{\varphi, \bar{b}}$.
So, $\chi(i') = i'$ by definition of $\chi$. Also, $M(a_l, a_m) \neq
( \rho, j)$(any $j$) because $\bar{a} \in W$. It now follows that
$\theta$ is a $\chi$-isomorphism on its domain, so that $ \theta \in
\Theta^\chi$.
Extend $\theta $ to $\theta' \in \Theta^\chi$ defined on $b_i$,
using the ``forth" property of $ \Theta^\chi$. Let $
\bar{c} = \theta'(\bar{b})$. Now by choice of of $\chi$, no labels
on edges of the subgraph of $M$ with domain $\rng(\bar{c})$ involve
$\rho$. Hence, $\bar{c} \in W$.
Moreover, each map in $ \Theta^\chi$ is evidently a partial
isomorphism of the reduct of $M$ to the signature $L_{\varphi,
\bar{b}}$. Now $\varphi$ is an $L_{\varphi, \bar{b}}$-formula.
We have $M \models \varphi(\bar{a})$ iff $M \models \varphi(\bar{c})$.
So $M \models \varphi(\bar{c})$. Inductively, $M \models_W
\varphi(\bar{c})$. Since $ \bar{c} =_i \bar{a}$, we have $M
\models_W \exists x_i \varphi(\bar{a})$ by definition of the
relativized semantics. This completes the induction.

\item {\sf The atoms, the required atomic algebra}

Now let $L$ be $L^+$ without the red relation symbols.
The logics $L_n$ and $L^n_{\infty \omega}$ are taken in this
signature.

For an $L^n_{\infty \omega}$-formula $\varphi $,  define
$\varphi^W$ to be the set $\{ \bar{a} \in W : M \models_W \varphi
(\bar{a}) \}$.

Then the set algebra $\A$ (actually  in all three cases) is taken to be the relativized set algebra with domain
$$\{\varphi^W : \varphi \,\ \textrm {a first-order} \;\ L_n-
\textrm{formula} \}$$  and unit $W$, endowed with the algebraic
operations ${\sf d}_{ij}, {\sf c}_i, $ etc., in the standard way, and of course formulas are taken in the suitable
signature.
The completion of $\A$ is the algebra $\C$ with universe $\{\phi^W: \phi\in L_{\infty,\omega}^n\}$
with operations defined as for $\A$, namely, usual cylindrifiers and diagonal elements,
reflecting existential quantifiers, polyadic operations and equality.
Indeed, $\A$ is a representable (countable) atomic polyadic algebra of dimension $n$

Let $\cal S$ be the polyadic set algebra with domain  $\wp ({}^{n} M )$ and
unit $ {}^{n} M $. Then the map $h : \A
\longrightarrow S$ given by $h:\varphi ^W \longmapsto \{ \bar{a}\in
{}^{n} M: M \models \varphi (\bar{a})\}$ can be checked to be well -
defined and one-one. It clearly respects the polyadic operations, also because relativized semantics and classical semantics coincide on $L_n$
formulas in the given signature, this is a representation of $\A.$

\item {\sf The atoms}

A formula $ \alpha$  of  $L_n$ is said to be $\sf MCA$
('maximal conjunction of atomic formulas') if (i) $M \models \exists
x_0, \ldots, x_{n-1} \alpha $ and (ii) $\alpha$ is of the form
$$\bigwedge_{i \neq j < n} \alpha_{ij}(x_i, x_j),$$
where for each $i,j,\alpha_{ij}$ is either $x_i=x_i$ or $R(x_i,x_j)$
for some binary relation symbol $R$ of $L$.

Let $\varphi$ be any $L^n_{\infty\omega}$-formula, and $\alpha$ any
$MCA$-formula. If $\varphi^W \cap \alpha^W \neq \emptyset $, then
$\alpha^W \subseteq \varphi^W $.
Indeed, take $\bar{a} \in  \varphi^W \cap \alpha^W$. Let $\bar{a} \in
\alpha^W$ be arbitrary. Clearly, the map $( \bar{a} \mapsto
\bar{b})$ is in $\Theta$. Also, $W$ is
$L^n_{\infty\omega}$-definable in $M$, since we have
$$ W = \{
\bar{a} \in {}^n M : M \models (\bigwedge_{i < j< n} (x_i = x_j \vee
\bigvee_{R \in L} R(x_i, x_j)))(\bar{a})\}.$$
We have $M \models_W \varphi (\bar{a})$
 iff $M \models_W \varphi (\bar{b})$. Since $M \models_W \varphi (\bar{a})$, we have
$M \models_W \varphi (\bar{b})$. Since $\bar{b} $ was arbitrary, we
see that $\alpha^W \subseteq \varphi^W$.
Let $$F = \{ \alpha^W : \alpha \,\ \textrm{an $\sf MCA$},
L^n-\textrm{formula}\} \subseteq \A.$$
Evidently, $W = \bigcup F$. We claim that
$\A$ is an atomic algebra, with $F$ as its set of atoms.
First, we show that any non-empty element $\varphi^W$ of $\A$ contains an
element of $F$. Take $\bar{a} \in W$ with $M \models_W \varphi
(\bar{a})$. Since $\bar{a} \in W$, there is an $\sf MCA$-formula $\alpha$
such that $M \models_W \alpha(\bar{a})$. Then $\alpha^W
\subseteq \varphi^W $. By definition, if $\alpha$ is an $\sf MCA$ formula
then $ \alpha^W$ is non-empty. If $ \varphi$ is
an $L^n$-formula and $\emptyset \neq \varphi^W \subseteq \alpha^W $,
then $\varphi^W = \alpha^W$. It follows that each $\alpha^W$ (for
$MCA$ $\alpha$) is an atom of $\A$.

Now for the rainbow signature, a formula $ \alpha$  of  $MCA$ formulas of $L_n$
are adapted to the rainbow signature.
$\alpha$ is such if  $M \models \exists
x_0, \ldots, x_{n-1} \alpha $ and (ii) $\alpha$ is of the form
$$\bigwedge_{i \neq j < n} \alpha_{ij}(x_i, x_j) \land \bigwedge_{\eta:(n-1)\to n, \eta \text { injective }}\mu_{\eta}(x_0,\ldots, x_{n-1}),$$
where for each $i,j,\alpha_{ij}$ is  $x_i=x_i$ or $R(x_i,x_j)$
for some binary relation symbol $R$ of the rainbow signature, and for each injection $\eta: n-1\to n$, $\mu_{\eta}$
is $y_S(x_{\eta(0)},\ldots, x_{\eta(n-2)})$ for some $y_S\in L$, if for all
distinct $i,j<n$, $\alpha_{\eta(i), \eta(j)}$
is not equality or green.

In this case there is a one to one correspondence
between coloured graphs whose edges do not involve the red shade and the $\sf MCA$ formulas, and both are the the atom
of $\A$, and for that matter its completion.
(Here any $s$ satisfying an $\sf MCA$ formula defines a coloured graph and for any other $s'$ satisfying the same formula,  the graph determined by
$s'$ is isomorphic to that determined by $s$; and thats why precisely
$\sf MCA$ formulas are atoms, namely, surjections from $n$ to coloured graphs; they are literally indivisible.

\item {\sf The complex algebra}

Define $\C$ to be the complex algebra over $\At\A$, the atom structure of $\A$.
Then $\C$ is the completion of $\A$. The domain of $\C$ is $\wp(\At\A)$. The diagonal ${\sf d}_{ij}$ is interpreted
as the set of all $S\in \At\A$ with $a_i=a_j$ for some $\bar{a}\in S$.
The cylindrification ${\sf c}_i$ is interpreted by ${\sf c}_iX=\{S\in \At\A: S\subseteq c_i^{\A}(S')\text { for some } S'\in X\}$, for $X\subseteq \At\A$.
Finally ${\sf p}_{ij}X=\{S\in \At\A: S\subseteq {\sf p}_{ij}^{\A}(S')\text { for some } S'\in X\}.$
Let $\cal D$ be the relativized set algebra with domain $\{\phi^W: \phi\text { an $L_{\infty\omega}^n$ formula }\}$,  unit $W$
and operations defined like those of $\A$.
Now we have
${\C}\cong \D$, via the map $X\mapsto \bigcup X$.

The map is clearly injective. It is surjective, since
$$\phi^W=\bigcup\{\alpha^W: \alpha \text { an $MCA$-formula}, \alpha^W\subseteq \phi^W\}$$
for any $L_{\infty\omega}^n$ formula $\phi$.
Preservation of the Boolean operations and diagonals is clear.
We check cylindrifications. We require that for any $X\subseteq \At\A$,
we have
$\bigcup {\sf c}_i^{\C}X={\sf c}_i^{\cal D}(\bigcup X)$ that is
$$\bigcup \{S\in \At\A: S\subseteq {\sf c}_i^{\A}S'\text { for some $S'\in X$ }\}=$$
$$\{\bar{a}\in W: \bar{a}\equiv_i \bar{a'} \text { for some } \bar{a'}\in \bigcup X\}.$$
Let $\bar{a}\in S\subseteq {\sf c}_iS'$, where $S'\in X$. So there is $\bar{a'}\equiv_i \bar{a}$ with $\bar{a'}\in S'$, and so $\bar{a'}\in \bigcup X$.

Conversely, let $\bar{a}\in W$ with $\bar{a}\equiv_i \bar{a'}$ for some $\bar{a'}\in \bigcup X$.
Let $S\in \At\A$, $S'\in X$ with $\bar{a}\in S$ and $\bar{a'}\in S'$.
Choose $MCA$ formulas $\alpha$ and $\alpha'$ with $S=\alpha^W$ and $S'=\alpha'^{W}$.
then $\bar{a}\in \alpha^{W}\cap (\exists x_i\alpha')^W$ so $\alpha^W\subseteq (\exists x_i\alpha')^W$, or $S\subseteq c_i^{\A}(S')$.
The required now follows. We leave the checking of substitutions to the reader.

The non representability of the complex algebra in the case of the rainbow algebra is in \cite{Hodkinson},
witness also the argument in the coming theorem \ref{smooth}.

\item{\sf Using the Monk algebras}

To prove item (2) we use Monk's algebras. Any of the two will do just as well:

We show that their atom structure consists of the $n$
basic matrices of a relation algebra.
In more detail, we define a relation algebra atom structure $\alpha(\G)$ of the form
$(\{1'\}\cup (\G\times n), R_{1'}, \breve{R}, R_;)$.
The only identity atom is $1'$. All atoms are self converse,
so $\breve{R}=\{(a, a): a \text { an atom }\}.$
The colour of an atom $(a,i)\in \G\times n$ is $i$. The identity $1'$ has no colour. A triple $(a,b,c)$
of atoms in $\alpha(\G)$ is consistent if
$R;(a,b,c)$ holds. Then the consistent triples are $(a,b,c)$ where

\begin{itemize}

\item one of $a,b,c$ is $1'$ and the other two are equal, or

\item none of $a,b,c$ is $1'$ and they do not all have the same colour, or

\item $a=(a', i), b=(b', i)$ and $c=(c', i)$ for some $i<n$ and
$a',b',c'\in \G$, and there exists at least one graph edge
of $G$ in $\{a', b', c'\}$.

\end{itemize}

$\alpha(\G)$ can be checked to be a relation atom structure. It is exactly the same as that used by Hirsch and Hodkinson in
\cite[definition 14.10]{HHbook}, except
that we use $n$ colours, instead of just $3$, so that it a Monk algebra not a rainbow one. However, some monochromatic triangles
are allowed namely the 'dependent' ones.
This allows the relation algebra to have an $n$ dimensional cylindric basis
and, in fact, the atom structure of of $\A$ is isomorphic (as a cylindric algebra
atom structure) to the atom structure $\Mat_n$ of all $n$-dimensional basic
matrices over the relation algebra atom structure $\alpha(\G)$.

Indeed, for each  $m  \in {\Mat}_n, \,\ \textrm{let} \,\ \alpha_m
= \bigwedge_{i,j<n}  \alpha_{ij}. $ Here $ \alpha_{ij}$ is $x_i =
x_j$ if $ m_{ij} = 1$' and $R(x_i, x_j)$ otherwise, where $R =
m_{ij} \in L$. Then the map $(m \mapsto
\alpha^W_m)_{m \in {\cal M}_n}$ is a well - defined isomorphism of
$n$-dimensional cylindric algebra atom structures.

We show $\Cm\alpha(\G)$ is not in $\sf RRA$.
Hence the full complex cylindric algebra over the set of $n$ by $n$ basic matrices
- which is isomorphic to $\cal C$ is not in $\RCA_n$ for we have a relation algebra
embedding of $\Cm\alpha(\G)$ onto $\Ra\Cm\M_n$.

Assume for contradiction that $\Cm\alpha(\G)\in \RCA_n$
For $Y\subseteq \N$ and $s<n$, set
$$[Y,s]=\{(l,s): l\in Y\}.$$
For $r\in \{0, \ldots, N-1\},$ $N\N+r$ denotes the set $\{Nq+r: q\in \N\}.$
Let $$J=\{1', [N\N+r, s]: r<N,  s<n\}.$$
Then $\sum J=1$ in $\Cm\alpha(\G).$
As $J$ is finite, we have for any $x,y\in X$ there is a $P\in J$ with
$(x,y)\in h(P)$.
Since $\Cm\alpha(\G)$ is infinite then $X$ is infinite.
By Ramsey's Theorem, there are distinct
$x_i\in X$ $(i<\omega)$, $J\subseteq \omega\times \omega$ infinite
and $P\in J$ such that $(x_i, x_j)\in h(P)$ for $(i, j)\in J$, $i\neq j$. Then $P\neq 1'$.
Also $(P;P)\cdot P\neq 0$. if $x,y, z\in M$,
$a,b,c\in \Cm\alpha(\G)$, $(x,y)\in h(a)$, $(y, z)\in h(b)$, and
$(x, z)\in h(c)$, then $(a;b)\cdot c\neq 0$.
A non -zero element $a$ of $\Cm\alpha(\G)$ is monochromatic, if $a\leq 1'$,
or $a\leq [\N,s]$ for some $s<n$.
Now  $P$ is monochromatic, it follows from the definition of $\alpha$ that
$(P;P)\cdot P=0$.

This contradiction shows that
$\Cm\alpha(\G)$ is not representable, and so,
 $\Cm\Mat_n\alpha(G)$ is not representable, too.

\item {\sf Summarizing}

We have in all cases a labelled graph  defined $M$ as a model of a first order theory in the Monk case
and of $L_{\omega_1, \omega}$ in the rainbow case.
We have also relativized $^nM$ to  $W\subseteq {}^nM$ by deleting assignments
whose edges involve reds, and we defined an algebra containing the term algebra, namely,
$\A$ as the atomic relativized set algebra
with unit $W$. In the case of  Monk's algebra we defined the relation algebra $\sf R$, such that
$\At\A\cong {\sf Mat}_n\At\R$, and in all cases the complex algebra $\C$
of the atom structure is isomorphic to the set algebra
$\{\phi^M: \phi\in L_{\infty, \omega}^n\}$.
Finally, we the $\sf Df$ reduct of this algebra, namely, $\C=\Cm\At\A$, cannot be representable for else this induces a complete representation
of $\Rd_{df}\A$, hence a complete representation of $\A$,
which in turn induces a representation of $\C$, but we have shown that the latter cannot be representable.

To see the last part,  we assume that $\A$ is simple (the general case follows easily).
The argument is essentially due to Johnson \cite[theorem 5.4.26]{tarski}, so we will be sketchy.
Assume that $h:\Rd_{df}\A\to \B$ is a complete representation  where
$B=\wp(W)$ and $W=\prod U_i$. We will build a complete representation of $\A$.
We can assume that $U_i=U_j$ for each $i,j<n$ and if $\bar{a}\in W$ and
$a_i=a_j$ then $\bar{a}\in h({\sf d}_{ij})$. Let $U$ be the disjoint union of the $U_i$
Define $\sim_{ij} =\{(a_i, a_j): \bar{a}\in h({\sf d}_{ij}).\}$

Then $\sim_{ij}$ is an equivalence relation on $U$, and $\sim _{ij} =\sim _{01}$. Let
$d=\prod_{i<j<n} {\sf d}_{ij}$.
Let $E=\{d\in A: h(d)\text {  is the union of $\sim_{01}$ classes}\}$
Then $E$ is the domain of a complete subalgebra $\E$ of $\A.$
Hence $\E=\A$, and if  $V=U/ \sim_{01}$ and
$g(c)=\{(a_0/{\sim_{01}}, \ldots,  a_{n-1}/\sim_{01}: \bar{a}\in h(d)\}$ for $d\in D$.
it can be easily checked that $g$ is as required, namely, it is 
a complete representation of the polyadic equality algebra
$\A$.

\end{enumarab}
\end{proof}
\begin{corollary}Let $\sf L$ be any signature between $\sf Df$ and $\sf PEA$.
\begin{enumarab}
\item There is an $\sf L$ atom structure that is weakly representable, but not strongly representable.
In particular, there is an $\sf L$ atom structure, that carries two (atomic) algebras, one is representable, the other is not; in fact its $\Df$
reduct is not representable.
\item The class of $\sf L$ representable algebras is not atom-canonical, it is not closed under
\d\ completions and cannot be axiomatized by Sahlqvist equations.
\item The same  holds for the class of representable
relation algebras
\end{enumarab}

\end{corollary}
\begin{proof} A special case of theorem \ref{sah} below, 
using the existence of weakly representable atom structures that are not strongly representable
just proved in theorem \ref{hodkinson}.
\end{proof}
In contrast, we have:

\begin{theorem} Let $\R$ and $M$ be as above. If $\R$
is finitely generated and ${\sf Th}M$ is $\omega$ categorial, or $M$ is ultrahomogeneous, 
or admits elimination of quantifiers, 
then $\Tm{\Mat_n\At\R}$ is strongly representable.
\end{theorem}
\begin{proof}
Assume that $M$ admits elimination of quantifiers, and that $\R$ is finitely generated.
Then $M$ is a countable model in binary relational finite signature, that has quantifier elimination. Since the relativized semantics for $L^n$ formulas
coincide with the classical one, we have that $\A\in \Nr_n\CA_{\omega}$. (Remember that semantics  is perturbed
at formulas in $L_{\infty, \omega}^n$.) Now we show that $\A$ is completely representable, then we build from this complete representation
a representation of the complex algebra.
Let $X=\At\A$. Let $\B\in {\sf Lf}_{\omega}$ be such that $\A=\Nr_n\B$ and $A$ generates $\B$, hence $\sum X=1$ in $\B$.
Let $Y$ be the set of co-atoms, then $Y\subseteq \Nr_n\B$ is a non-principal type. By the usual Orey Henkin omitting
types theorem for first order logic, we get a model, that is a representation that omits $Y$.
That is for every non-zero $b\in \B$, we get a set algebra $\C$ having countable base, and
$f:\B\to \C$ such that $f(a)\neq 0$ and $\bigcap_{y\in Y} f(y)=\emptyset$.
Define $g(x)=\{s\in {}^nU: s^+\in f(x)\},$ where $s^+=s$ on $n$ and fixes all other points in $\omega\sim n$,
then $g$ gives a complete simple representation such that $g(a)\neq 0$.
Taking disjoint unions over the non zero elements in $\A$ as is common practise in algebraic logic, we
get the desired complete representation, call it $f$.

Next we represent the complex algebra $\C$. Recall that $\C$ is the
completion of $\A$. Let $c\in C$, let $M$ be the base of the complete representation of $\A$,
and let $c^*=\{x\in \A: x\leq c\}$. Define $g(c)=\bigcup_{x\in c^*} f(x)$, then clearly $g$ defines
a representation of $\C$ into $\wp(^nM),$ and we are done.
\end{proof}
Later on, we shall have occasion to 
encounter another situation where quantifier elimination renders positive 
results on omitting types theorems for finite variable fragments of first order logic, 
witness theorem \ref{Shelah1}.

\subsection{Relativization}

Here we shall deal with different notions of relativization. 
One can find well motivated appropriate notions of relativized semantics by first locating them while giving up classical semantical prejudices. 
It is hard to give a precise mathematical underpinning to such intuitions. What really counts at the end of the day
is a completeness theorem stating a natural fit between chosen intuitive concrete-enough, but perhaps not {\it excessively} concrete, 
semantics and well behaved axiomatizations. 

The move of altering semantics has radical phiosophical repercussions, taking us away from the conventional 
Tarskian semantics captured by Fregean-Godel-like axiomatization; the latter completeness proof is effective 
but highly undecidable; in modal logic 
and finite variable fragments of first order logic, which have a modal formalism, this is 
highly undesirable. Here we try to circumvent such negative results. 

We will discover that the culprit is `commutativity of cylindrifiers' even when thoroughly 
localized. 
If we dispense with full fledged commutativity of cylindrifiers, then we get fragments of first order logic that have nice modal behaviour, similar
to the guarded, loosely guarded and packed fragments of first order logic (but using only finitely many variables).

We start off with a few deep algebraic results, from which we obtain several negative results on omitting types.
We need to prepare some more.

We recall the definition of network.
But first a piece of notation.
Since we are dealing with polyadic algebras with equality,  a clause to handle
substitutions is added. 

But first we recall a piece of notation. Let $\delta$ be a map. Then $\delta[i\to d]$ is defined as follows. $\delta[i\to d](x)=\delta(x)$
if $x\neq i$ and $\delta[i\to d](i)=d$. We write $\delta_i^j$ for $\delta[i\to \delta_j]$.

\begin{definition}\label{net}
From now on let $2< n<\omega.$ Let $\C$ be an atomic ${\sf PEA}_{n}$.
An \emph{atomic  network} over $\C$ is a map
$$N: {}^{n}\Delta\to \At\C$$
such that the following hold for each $i,j<n$, $\delta\in {}^{n}\Delta$
and $d\in \Delta$:
\begin{itemize}
\item $N(\delta^i_j)\leq {\sf d}_{ij}$
\item $N(\delta[i\to d])\leq {\sf c}_iN(\delta)$
\item $N(\bar{x}\circ [i,j])= {\sf s}_{[i,j]}N(\bar{x})$ for all $i,j<n$.

\end{itemize}
\end{definition}

Note that we have added a clause to reflect substitutions. Networks can be defined for all 
cylindric-like algebras considered herein, the obvious way, for example for $\CA$'s
we just discard the last item, cf. \cite[definition 3.3.1]{HHbook2};
for $\PA$s we discard the first involving diagonal elements (because we do not have them in the signature)
and replace it by a clause similar to the last but
dealing instead with substitutions corresponding to replacements.

Recall that we dealt with relativized algebras in our previous proof.
For an algebra $\A$ and a logic $L$, that is an extension of $n$
variable first order logic,
$L(A)$ denotes the signature (taken in this logic) obtained by adding an $m$
relation symbol for every element of $\A$.

The notion of relativized representations is indeed important
in algebraic logic. $M$ is a relativized representation of an abstract algebra $\A$, if there exists $V\subseteq {}^mM$, and an injective
homomorphism $f:\A\to \wp(V)$, and $M=\bigcup_{s\in V}\rng(s).$ We write $M\models 1(\bar{s})$, if $\bar{s}\in V$.

In the next definition $m$ denotes the dimension.

\begin{definition}\label{rel}
\begin{enumarab}
\item  Let $M$ be a relativized representation of a $\CA_m$.
An $m$ clique in $M$, or simply a clique in $M$, is a subset $C$ of $M$ such that $M\models 1(\bar{s})$ for all $\bar{s}\in {}^mC$, that is
it is it can be viewed as a hypergraph such that evey $m$ tuple is labelled by the top element.
For $n>m$, let $C^{n}(M)=\{\bar{a}\in {}^nM: \text { $\rng(\bar{a})$ is a clique in $M$}\}.$
\item Let $\A\in \CA_m$, and $M$ be a relativized representation of $\A$. $M$ is said to be $n$ square, $n>m$,
if whenever $\bar{s}\in C^n(M)$, $a\in A$, and $M\models {\sf c}_ia(\bar{s})$,
then there is a $t\in C^n(M)$ with $\bar{t}\equiv _i \bar{s}$,
and $M\models a(\bar{t})$.
\item  $M$ is infinitary $n$ flat if  for all $\phi\in L(A)_{\infty, \omega}^n$, for all $\bar{a}\in C^n(M)$, for all $i,j<n$, we have
$$M\models \exists x_i\exists x_j\phi\longleftarrow \exists x_j\exists x_i\phi(\bar{a}).$$
$M$ is just $n$ flat, if the above holds for first order formulas using $n$ variables,
namely, for $\phi\in L(A)^n$.
\item $M$ is said to be $n$ smooth if it is $n$ square, and there is are equivalence relations $E^t$, for $t=1,\ldots, n$ on $C^{t}(M)$
such that  $\Theta=\{(\bar{x}\to \bar{y}): (\bar{x}, \bar{y})\in \bigcup_{1\leq t\leq n} E^t\}$
(here   $\bar{x}\to \bar{y}$ is the map $x_i\mapsto y_i$)
is an $n$ back and for system of partial isomorphisms on $M$.
\end{enumarab}
\end{definition}

Note that all of the above are a local form of representation where roughly cylindrifiers have witnesses only
on $<m$ cliques.
Closely related to relativized representations are the notion of hyperbasis. This is only formulated for relation algebras \cite[12.2.2]{HHbook}.
We extend it to the polyadic equality case, and all its reducts down to $\Sc$s.
\begin{definition}
Let $\A\in \sf PEA_j$. Let $j+1\leq m\leq n\leq k<\omega$, and let $\Lambda$ be a non-empty set.
An $n$ wide $m$ dimensional $\Lambda$ hypernetwork over $\A$ is a map
$N:{}^{\leq n}m\to \Lambda\cup \At\A$ such that
$N(\bar{x})\in \At\A$ if $|\bar{x}|=j$ and $N(\bar{x})\in \Lambda$ if $|\bar{x}|\neq j$,
(so that $j$ edges are labelled by atoms, and other edges with labels from $\Lambda$,
with the following properties for each $i,k<j$, $\delta\in {}^{\leq n}m$, $\bar{z}\in ^{\leq n-2}m$
and $d\in m$:
\begin{itemize}
\item $N(\delta^i_j)\leq {\sf d}_{ij}$
\item $N(x,x, \bar{z})\leq d_{01}$
\item $N(\delta[i\to d])\leq {\sf c}_iN(\delta)$
\item $N(\delta\circ [i,j])={\sf s}_{[i,j]}N(\delta).$
\item If $\bar{x}, \bar{y}\in {}^{\leq n}m$, $|\bar{x}|=|\bar{y}|$ and $\exists \bar{z}$, such that
$N(x_i,y_i,\bar{z})\leq {\sf d}_{01}$,  for all $i<|\bar{x}|$, then $N(\bar{x})=N(\bar{y})$
\item when $n=m$, then $N$ is called an $n$ dimensional $\Lambda$ hypernetwork.
\end{itemize}
\end{definition}

Then $n$ wide $m$ dimensional {\it hyperbasis} can be defined like the relation algebra case \cite{HHbook}, where the amalgamation property is
the most important (it corresponds to commutativity of cylindrifiers).

For an $n$ wide $m$ dimensional $\Lambda$ hypernetworks, and $\bar{x}\in {}^{<\omega}m)$, we define
$N\equiv_x M$
if $N(\bar{y})=M(\bar{y})$ for all
$\bar{y}\in {}^{\leq n}(m\sim \rng(\bar{x}))$.

In more detail:

\begin{definition}\label{usualbasis} $j$ remains to be the dimension of $\A$. 
The set of all $n$ wide $m$ dimensional hypernetworks will be denoted by $H_m^n(\A,\Lambda)$.
An $n$ wide $m$ dimensional $\Lambda$, with $j+1\leq m\leq n$
hyperbasis for $\A$ is a set $H\subseteq H_m^n(\A,\lambda)$ with the following properties:
\begin{itemize}
\item For all $a\in \At\A$, there is an $N\in H$ such that $N(0,1,\ldots, j-1)=a$
\item For all $N\in H$ all $\bar{x}\in {}^n\nodes(N)$, for all $i<j$ for all $a\in\At\A$ such that
$N(\bar{x})\leq {\sf c_i}a$, there exists $\bar{y}\equiv_i \bar{x}$ such that $N(\bar{y})=a$
\item For all $M,N\in H$ and $x,y<n$, with $M\equiv_{xy}N$, there is $L\in H$ such that
$M\equiv_xL\equiv_yN$. Here $M\equiv_SN$, means that $M$ and $N$ agree off of $S$.

\item Assume that $j+1\leq m\leq n\leq k$.
For a $k$ wide $n$ dimensional hypernetwork $N$, we let $N|_m^k$ the restriction of the map $N$ to $^{\leq k}m$.
So we restrict the first $m$ nodes but keep all hyperlabels on them.
For $H\subseteq H_n^k(\A,\Lambda)$ we let $H|_k^m=\{N|_m^k: N\in H\}$.

\item When $n=m$, $H_n(\A,\Lambda)$ is called an $n$ dimensional hyperbases.
\end{itemize}
We say that $H$ is symmetric, if whenever $N\in H$ and
$\sigma:m\to m$, then $N\circ\sigma\in H$, the latter, which we denote simply by $N\sigma$ is defined by
$N\sigma(\bar{x})=N(\sigma(x_0),\ldots ,\sigma(x_{n)})$, for $\bar{x}\in {}^n\nodes(N)$.

\end{definition}

When $n=m$, then $H_n^n(\A, \Lambda)$ denoted simply by $H_n(\A, \Lambda)$ is called an $n$ dimensional hyperbasis.
If an addition, $|\Lambda|=1$, then $H_m^n(\A, \Lambda)$ is called an $n$ dimensional cylindric basis. (These were defined
by Maddux, and later generalized by Hirsch and Hodkinson to hyperbasis).

\subsection{Cylindric algebras possesing basis}

We introduce  a new definition:

\begin{definition}\label{basis} An $n$ dimensional basis is a set of $n$ dimensional networks that satisfy only
(i) and (ii) in the above definition \ref{usualbasis}.
\end{definition}

We define a game on atomic $\CA_m$s that characterizes algebras with $n$ dimensional basis.
Let $2<m<n$. Let $r\leq \omega$, the game $G_r^n(\A)$ is played over atomic $\A$ networks 
between \pa\ and \pe\ and has $r$ rounds, and $n$ pebbles, where $r,n\leq \omega.$

\begin{enumerate}
\item In round $0$ \pa\ plays an atom $a\in \At\A$. \pe\ must respond with a network $N_a$ with set of nodes $\bar{x}$ such that $N_a(\bar{x})=a$.
\item In any round $t$ with $0<t<r$, assume that the current netwrork is $N_{t-1}.$
Then \pa\ plays as follows. 
First if $|N_{t-1}|=n$ then he deletes a node $z\in N_{t-1}$ and defines $N_t$ to be the resulting network. 
Otherwise, let $N=N_{t-1}$; now \pa\ chooses $\bar{x}\in {}^nN$, $i<m$ and an atom $a\in \At\A$, 
such that $N(\bar{x})\leq {\sf c}_ia$. \pe\ must respond to this cylindrifier move 
by providing a network $N_t\supset N$, with $|N_t|\leq n$, having a node $z$
such
and $N_t(\bar{y})=a$, where $\bar{y}\equiv_i \bar{x}$ and $\bar{y}_i=z$.
\end{enumerate}
The above game is similar to but {\it is not } the usual $r$ rounded 
atomic games characterized by the Lyndon conditions, as $r$ gets larger, capturing representability of finite algebras
in the limit, and  ultimately defining the elementary closure of the class of completely representable algebras.
Here  the pebbles in play or the nodes are not allowed to exceed $n<\omega$.

The number of pebbles $n$ measures the squarenes, flatness or smoothness of the relativized representations; these are 
significantly distinct notions
in relativized semantics, though the distinction between the first two diffuses at the limit of genuine representations. 

Crudely, the classical 
case then becomes a limiting case, when $n$ goes to infinity.
However, this is not completely accurate because $\omega$ complete relativized representations 
may not coincide with classical ones on uncountable algebras.
Examples will be provided below, witness theorems \ref{longer}, and last item in theorem \ref{maintheorem}.

We will encounter similar games, where the number of pebbles are finite, but in some cases 
\pa\ will have the option to re-use them.

\begin{theorem}\label{gamebasis} 
Let $2<m<n$. Then following are equivalent for an atomic $\CA_m$
\begin{enumarab}
\item $\A$ has an $n$ dimensional bases.
\item \pe\ has a \ws\ in $G_{\omega}^n.$
\end{enumarab}
\end{theorem}
\begin{proof} The proof is similar to \cite[proposition 12.25]{HHbook}, so we will be sketchy.
If $H$ is an $n$ dimensional basis, then \pe\ can win $G_{\omega}^n$ by always 
playing a subnetwork of a network in $H$. In round $0$, when \pa\ plays 
the atom $a\in \A$, \pe\ chooses $N\in H$ with $N(0,1,\ldots, n-1)=a$ and plays $N\upharpoonright n$. 
In round $t>n$, if the curent network is $N_{t-1}\subseteq M\in H$, then no matter how \pa\ defines $N$, we have
$N\subseteq M$ and $|N|<n$, so there is $z<n$, with $z\notin N$. 
Assume that \pa\ picks $x_0,\ldots, x_{n-1}$ and $a\in 
\At\A$ such that $N(x_0,\ldots, x_{n-1})\leq {\sf c_i}a$. So 
$M(x_0,\ldots, x_{n-1})\leq {\sf c}_ia$, and hence, by properties of basis, 
there is $M'\in H$ with $M'\equiv _z M$ and  $M(x_0, z,\ldots, x_{n-1})=a$. \pe\ 
responds with the restriction of $M'$ to the nodes
$\nodes(N)\cup \{z\}$.

Conversely, if \pe\ has a \ws\ then 
for any atomic network $N$ with $|N|\leq n$, let $G_{\omega}^n(N, \A)$ be like $G_{\omega}^n\A$ except that in
round $0$, \pe\ plays $N$. Let $H$ be the set of all networks $M$ with nodes $n$ 
such that \pe\ has a \ws\ in $G_{\omega}^n(M,\A)$; then $H$ is as required.
\end{proof}

The next theorem is analogous to \cite[proposition 13.37]{HHbook}, except that  we deal with 
$m$ hyperedges instead of just
edges ($m$ is the dimension). Also the definition of a clique is modified (according to the above definition \ref{rel}),
to mean a
hypergraph with sets of nodes $\{x_0,\ldots, x_{n-1}\}$
all of whose $m$ hyperedges are labelled by the top element, in the sense that $M\models 1(x_{i_0},\ldots, x_{i_{m-1}})$, for any
$i:m\to n$. The technique used is an instance of the idea of treating a basis as a saturated set of mosaics
\cite{mosaics}. A similar argument is given in \cite{Maddux}.
The theorem tells us how to pass from basis to relativized representations, and variants of it will be used below.
For example we might not put any  restrictions on 
the length of hyperdeges, or we could deal only with usual cliques and edges. 

The converse (passing from relativized representations to basis) is much easier, as
we will show in a while.

Recall that $m$ denotes the dimension, it will be $>2$, and $n>m$ is finite.

\begin{lemma}\label{step} If  $\A\in \PEA_m$ has an $n$ dimensional basis, then $\A$ has an $n$ square relativized representation.
If it has  a hyperbasis
then $\A$ has a relativized $n$ smooth representation. This also holds for any $\K\in \{\Sc, \PA, \CA\}$ with
the obvious modifications.
\end{lemma}

\begin{proof}
We build a relativized representation $M$
in a step- by-step fashion; it will be a labelled complete hypergraph. We distinguish between $m$ hyperedges and
hyperedges of length $\neq m$, the labels are different, they are {\it not } atoms.
This $M$ is required to  satisfy  the following properties

\begin{enumarab}

\item Each $m$ hyperedge of of $M$ is labelled by an atom of $\A$

\item $M(\bar{x})\leq {\sf d}_{ij}$ iff $x_i=x_j$. (In this case, we say that $M$ is strict).

\item For any clique $\{x_0,\ldots, x_{n-1}\}\subseteq M$,
there is a unique $N\in H$ such that $(x_0,\ldots, x_{n-1})$ is labelled by $N$ and we write this as
$M(x_0,\ldots, x_{n-1})=N$

\item  If $x_0,\ldots, x_{n-1}\in M$ and $M(\bar{x})=N\in H$,
then for all $i_0,\ldots, i_{m-1}<n$, $(x_{i_0},\ldots x_{i_{m-1}})$
is a hyperedge and 
$$M(x_{i_0}, x_{i_1}, \ldots,  x_{i_{m-1}})=N(i_0,\ldots, i_{m-1})\in \At\A.$$

\item $M$ is symmetric; it closed under substitutions. That is, if $x_0,\ldots, x_{n-1}\in M$
are such that  $M(x_0,\ldots, x_{n-1})=N$,
and $\sigma:n\to n$ is any map, then  we have
$$M(x_{\sigma(0)}, \ldots, x_{\sigma(n-1)})=N\sigma.$$

\item If $\bar{x}$ is a clique, $k<n$ and $N\in H$, then $M(\bar{x})\equiv_k N$ if
and only if there is
$y\in M$ such that $M(x_0, \ldots, x_{k-1}, y, x_{k+1}, \ldots, x_{n-1})=N.$

\item For every $N\in H$, there are $x_0,\ldots,  x_{n-1}\in M$, $M(\bar{x})=N.$

\end{enumarab}

We build a chain of hypergraphs $M_t:t<\omega$ their limit (defined in a precise
sense) will as required.  Each $M_t$ will have hyperedges of length $m$ labelled by
atoms of $\A$, the labelling of other hyperedges will be done later.

To avoid confusion we call hyperedges that are to be labelled by atoms atom-hyperedges, these have length
exactly $m$, and those that have $\leq n$ length not equal to $m$ (recall that $3\leq m<n$)
that are labelled by elements in $n$ simply by hyperedges.

We proceed by a step by step manner, where the defects are treated one by one, and they are diffused at the limit obtaining the required
hypergraph, which also consists of
two parts, namely, atom-hyperedges and hyperedges.

This limiting hypergraph will be compete (all atom hypergraphs and hypergraphs will be labelled;
which might not be the case with $M_t$ for $t<\omega$.)
Every atom-hyperedge will indeed be labelled by an atom of $\A$ and hyperedges with length
$\neq  m$ will also be labelled, but  by indices from $n$.

We require inductively that $M_t$ also satisfies:

Any clique in $M_t$ is contained in $\rng(v)$ for some $N\in H$ and some embedding $v:N\to M_t$.
namely, $v:n\to \dom(M_t)$
and this embedding satisfies the following two conditions:

(a) if $(v(i_0), \ldots, v(i_{m-1}))$ is a an atom  hyperedge of $M_t$, then
$M_t(\bar{v})=N(i_0,\ldots, i_{n-1}).$

(b) Whenever $a\in {}^{\leq ^n}n$ with $|a|\neq m$, then $v(a)$ is a
hyperedge of $M_t$ and is labelled by $N(\bar{a})$.

(Note that an embedding might not be injective).
A network is strict if whenever  $N(\bar{x})\leq {\sf d}_{ij}$, then $x_i=x_j$.
For the base of the induction we construct  $M_0.$ We proceed as follows.
Let $N$ be a hypernetwork and let $S\subseteq n$. $N|S$ is maximal strict if it is strict
and  any for  $T\subseteq n$, such that  $T$ contains $S$, $N|T$ is not strict.

Let $M_0$ be the disjoint union
of all maximal strict labelled hypergraph
$N|S$ where $N\in H$ and $S\subseteq N$
and the atom-hyperedges and hyperedges are induced (the natural way) from $N$.
We have that for all $i<n$, there is a unique $s_i\in S$, and $\bar{z}$, such that
$N(i, s_i, \bar{z})\leq {\sf d}_{01}$,  then set  $v(i)=s_i$, for $i<n$

Now assume inductively that $M_t$ has been  defined for all $t<\omega$.

We now define $M_{t+1}$ such that for every quadruple
$(N, v, k, N')$ where $N, N'\in H$, $k<n$ and $M\equiv_k N'$ and $v:N\to M_t$
is an embedding, then  the restriction $v\upharpoonright n\sim \{k\}$ extends to an embedding
$v': N'\to M_{t+1}$

\begin{itemize}

\item For each such $(N, v,k, N')$ we add just one new node $\pi$,
and we add the following atom hyperedges $(\pi, v(i_1),\ldots, v(i_{m-2}))$ for each $i_1,\ldots,  i_{m-2}\in n\sim k$
that are pairwise distinct. Such new atom-hyperedges are labelled by $N'(k, i_0,\ldots, i_{m-2}) $. This is well defined.
We extend $ v$ by defining $v'(k)=\pi$.

\item We add a new hyperedge $v'(\bar{a})$ for every $\bar{a}\in {}^nn$ of length $\neq m$,
with $k\in \rng (\bar{a})$ labelled by $N'(\bar{a})$.

\end{itemize}
Then $M_{t+1}$ will be $M_t$ with its old labels, atom hyperedges,  labelled hyperedges
together with the new ones define as above.

It is easy to check that the inductive hypothesis  are preserved.

Now define a labelled hypergraph as follows:
$\nodes(M)=\bigcup \nodes(M_t)$,
for any $\bar{x}$ that is an atom-hyperedge; then it is a one in some $M_t$
and its label is defined in $M$ by $M_t(\bar{x}).$

The hyperedges are $n$ tuples $(x_0,\ldots, x_{n-1})$.
For each such tuple, we let $t<\omega$,
such that $\{x_0, \ldots, x_{n-1}\} \subseteq  M_t$,
and we set $M(x_0, \ldots, x_{n-1})$   to be the unique
$N\in H$ such that there is an embedding
$v:N\to M$ with  $\bar{x}\subseteq \rng(v).$
This can be easily checked to be well defined. We check existence and uniqueness. The latter is
clear from the definition of embedding. Note that there is an $N\in H$
and an embedding $v:N\to M_t$ with $(x_0, \ldots, x_{n-1})\subseteq \rng(v)$.
So take $\tau:n\to n$, such that $x_i=v(\tau(i))$ for each $i<n$.

As $H$ is symmetric, $N\tau\in H$, and clearly $v\circ \tau: N\tau\to M_t$
is also an embedding. But $x_i=v\circ \tau(u)$ for each $i<n$,
then let $M(x_0,\ldots, x_{n-1})=N\tau.$

Now we show that $M$ is an $n$ square relativized representation.
Let $L(A)$ be the signature obtained by adding an $n$ ary relation symbol for each element
of $\A$.
Define $M\models r(\bar{x})$ if $\bar{x}$ is an atom  hyperedge and $M(\bar{x})\leq r.$

Now let $\bar{x}\in C^n(M),$ $k<m$ and $M\models r(\bar{x})$. We require that there exists $y\in C^n(M)$
$y\equiv_k x$ and $M\models r(\bar{y})$. Take $i_0, \ldots, i_{m-1}<n$ different from $k$.
Now $M(\bar{x})=N$ for some $N\in H$.
By properties of hyperbasis, there is $P\in H$  with $P\equiv_k N$.
Hence $P(i_0, \ldots, i_{k-1}, k, i_{k+1}\ldots )\leq r$. But by properties of $M$, there  is
an $n$ tuple $\bar{y}\equiv k\bar{x}$ such that $M(y)=P$.
Then $\bar{y}\in C^n(M)$ is as required.

Finally, for $n$ smoothness we need to define equivalence relations
$E^m$ on $C^m(M)$, satisfying the definition, for $0\leq m\leq n$.
For each such $m$ and $\bar{a}\in C^m(M)$,
define  $a^*=(a_0, a_1,\ldots, a_{m-1}, a_0,\ldots, a_0)\in C^n(M)$, where we have $n$
copies of $a_0$ after $a_0, a_1, \ldots, a_{m-1}$. Now for $\bar{a}$, $\bar{b}\in C^n(M)$, set
$$E^m(\bar{a},\bar{b})\longleftrightarrow M(a^*)=M(b^*).$$
This can be easily seen to be as required.
\end{proof}

\begin{definition} Let $2<m<n$ and suppose that $m$ is finite.
We define $\A\in {\sf CB_n}$ if $\A$ is a subalgebra of a complete atomic $\B\in \CA_m$ 
that has an $n$ dimensional basis.
\end{definition}

Basis consists of networks that can be amalgamated 
in a step-by- step manner to give relativized representation, and they also code extra dimensions,
in the sense that they allow a {\it weak neat embedding theorem}. This means that the algebra possesing
such a basis neatly embeds into an algebra in higher dimensions, whose dimension is equal to the {\it degree} of relativization.
Call such an algebra an $n$ dilation, or simply, a dilation.

In case of {\it only basis}, the dilation thereby obtained can only be 
locally representable; and indeed,
it may fail commutativity of cylindrifiers because we do not 
require the amalgamation property of hypernetworks in basis.

But when we do, as is the case with {\it hyperbasis}, the dilation is a cylindric algebra, 
if the small algebra is, and the latter has a usual 
neat embedding property, embedding into an algebra having finitely many extra dimensions.
But it is still weak in the sense that the dilation has only {\it finitely many extra dimensions}, so by Monk's classical 
non-finite axiomatizability result, we know that such an algebra cannot posses a classical representation.

On the other hand, the dimension 
of the dilation  also determines how close,
or rather how far,  the representation is from a classical one. 

It is precisely the {\it degree} of {\it squareness}
in case we do not have commutativity of cylindrifiers in the dilation) 
or {\it smoothness} (in case we do). It measures how much we have to zoom in so that we mistake the 
$n$ relativized representation for  a genuine one. Our next theorem formally formulates 
some of these ideas and connections.

But first a definition.
Let $n>1$ be finite. Then $\sf D_n$ denotes the class of relativized set algebras whose units are arbitrary sets of sequences, 
such that, if $\A\in \sf D_n$ has unit $V$, and $s\in V$, then for any $i,j\in n$, $s\circ [i|j]\in V$. 
$\sf G_n$ are those relativized set algebras whose units are locally cube, 
that is if $\A\in \sf G_n$ with top element $V$ and  $s\in V$ implies $s\circ \tau\in V$ for any finite transformation $\tau$ on $n$.
It is known that $\sf D_n$ and $\sf G_n$ are varieties that are finitely axiomatizable, witness theorem \ref{11}.

\begin{theorem}\label{can2} Fix a finite  dimension $m>2,$ and let $n\in \omega$ with $n\geq m$.
\begin{enumarab}
\item ${\sf CB}_{m,n}$ is a canonical variety, and for $n\geq 6$ it is not atom-canonical
\item If $\A\in \CA_m$ then $\A\in {\sf CB}_{m,n}$ iff $\A^+$ has an $n$ dimensional basis.
\item If $\A\in \CA_m$, then $\A\in {\sf CB}_{m,n}$ if and only if $\A$ has 
an $n$ square relativized representation.
\item  ${\sf CB}_n=\CA_m\cap S\Nr_m{\sf D_n}$
\item $\bigcap_{k\in \omega, k\geq m} {\sf CB}_{m, k}=\sf RCA_m$
\end{enumarab}
\end{theorem}

\begin{proof}  
\begin{enumarab}
\item  The proof is similar to that of \cite[proposition 12.31]{HHbook2} attributed to 
Maddux, so we will be sketchy. Closure under subalgebras is obvious. Products is also easy; it goes as follows. Given 
$\A_i$, $i\in I$, for some set $I$,  let $\B_i$ be complete and atomic algebra 
having  an $n$ dimensional basis $\M_i$ such  $\A_i\subseteq \B_i$. 
We can assume, without loss, that the $\B_i$s are  pairwise disjont; accordingly 
let $\C$ be the complex algebra over the atom structure of disjoint unions of those of the $\B_i$. 
Then $\C$ is complete and atomic $\prod \A_i$ embeds in $\C$, and 
$\bigcup \M_i$ is the desired basis. 

Now we check closure under homomorphic images, the only part that is not direct. 
Let $\A\in {\sf CB}_{m,n}$ and let 
$h:\A\to \B$  be surjective
Let $K= ker(h)$  and let $z=-\sum^{\A^+}K$.
Let $\D$ be relativization of $\A^+$ to $z.$
Then $\D$ is complete and atomic.
Define  $g :\A^+\to \D$ by $a\mapsto a.z.$ 
Then $\B$ embeds in $\D$ via $b\mapsto g(a)$ for any $a\in h^{-1}[b]$. 
$\A^+$ has a basis $\M$, then $\{N\in \M: N(\bar{x})\in \D\}$ 
is a basis for $\D$.
(That it is not atom-canonical will be proved below).

\item One side is trivial. The other follows from theorem \ref{step}.

\item One side is easy. 
Let $M$ be the given representation.  For a network $N$ and $v: N\to M$, we say that 
$v$ embeds $N$ in $M$ if for all $r\in \At\A$, 
we have $N(\bar{x})=r$ iff $M\models r(v(\bar{x}))$. 

All such neworks that embed in $M$ are the desired basis. The other more involved part follows again from theorem \ref{step}.

\item Let $L(A)$ denotes the signature that contains
an $n$ ary predicate for every $a\in A$.
Let $M$ be an $n$ square -epresentation, which exists by theorem \ref{step}.
For $\phi\in L(A)_{\omega,\infty}^n$,
let $\phi^{M}=\{\bar{a}\in C^n(M):M\models_C \phi(\bar{a})\}$,
and let $\D$ be the algebra with universe $\{\phi^{M}: \phi\in L_{\infty,\omega}^n\}$ with usual
Boolean operations, cylindrifiers and diagonal elements, \cite[theorem 13.20]{HHbook}. The polyadic operations are defined
by swapping variables.
Define $\D_0$ be the algebra consisting of those $\phi^{M}$ where $\phi$ comes from $L_n$, that is it contains $n$ variables.
Assume that $M$ is $n$ square, then certainly $\D_0$ is a subalgebra of the $\sf Crs_n$ (the class
of algebras whose units are arbitrary sets of $n$ ary sequences)
with domain $\wp(C^n(M))$ so $\D_0\in {\sf Crs_n}$. The unit $C^n(M)$ of $\D_0$ is symmetric,
closed under substitutions, so
$\D_0\in \sf G_n$ (these are relativized set algebras whose units are locally cube, they
are closed under substitutions.)  
Now define the neat embedding by $\theta(r)=r(\bar{x})^{M}$.
Preservation of operations is straightforward.  We show that $\theta$ is injective.
Let $r\in A$ be non-zero. But $M$ is a relativized representation, so there $\bar{a}\in M$
with $r(\bar{a})$ hence $\bar{a}$ is a clique in $M$,
and so $M\models r(\bar{x})(\bar{a})$, and $\bar{a}\in \theta(r)$ proving the required.

Let $\A\subseteq \Nr_m\D_n$. For each $\bar{a}\in 1^{\D},$ define a labelled 
hypergraph $N_a$ with nodes $n$, and 
$N_{\bar{a}}(\bar{x})$ is the unique atom of $\A$ 
containing the tuple 
$(a_{x_0},\ldots, a_{x_{n-1}}, a_{x_0}\ldots\ldots a_{x_0}).$ 
This is well defined.  Let $\M$ be the symmetric closure
of $\{N_a: \bar{a}\in 1^M\}$.

\item Since a classical  represention is $\omega$ square, then $\bigcap_{k>m}{\sf CB_{k}}$
and $\RCA_m$ coincide on simple countable algebras. 
But each is a discriminator variety and so
they are equal.
\end{enumarab}
\end{proof}

For finite $m>2$ and $n>m$, 
$S\Nr_m\CA_n\subseteq \sf CB_{m,n}$ are  both 
canonical varieties that are approximations 
to $\RCA_n$;  but {\it strict} approximations (this is not trivial to show). 

For larger $n$ they get closer to $\RCA_m$, but they never actually get there. All the same $\RCA_m$ is their limit meaning that
$$\bigcap_{k\in \omega} S\Nr_m\CA_{m+k}=\bigcap {\sf CB}_{m, m+k}=\RCA_m.$$
A great portion in the paper will be devoted to study such varieties.

They share a lot of properties when the dimension $m$ is finite $>2$. Both have a neat embedding property, both are characterized 
by games, both have relativized representations and both are not finitely axiomatizable 
nor can  they be axiomatized by Sahlqvist equations for $n\geq m+3$. 
Finally, for each, the subclass whose members have complete $n$ relativized representations, both square and smooth,
is  not elementary.

But there are differences, and significant ones, particulary in terms of decision problems,
as the unfolding of the paper will show.  

Another striking difference, even more literal, is that for $2<m<n$, 
the variety $\sf S\Nr_m\CA_n$  is not even finitely axiomatizable over ${\sf CB}_{m,n}$.
The amalgamation moves in hyperbasis 
makes a big  difference. This has subtle manifestations, for example Monk-like algebras 
work to prove non finite axiomatizability of $S\Nr_n\CA_{n+k+1}$ over
$S\Nr_n\CA_{n+k}$ for $n>2$ and $k\geq 1$ finite. The analogous result for ${\sf CB}_n$ uses more delicate rainbow constructions, witness theorem
\ref{squarerepresentation}. 
On a more basic level, 
commutativity of cylindrifiers adds a lot on the level of the algebras.

However, such differences diffuse in the limit, when the algebras have genuine 
representations regaining  commutativity of cylindrifiers.

In any event differences and similarities are certainly 
illuminating for both. 

\section{Atom canonicity}

Here we show that for finite dimensions $n>2$, both the canonical varieties 
$S\Nr_n\CA_{n+k}$ and ${\sf CB}_{n, n+k}$, when $k\geq 4$,  are not atom-canonical, hence not closed under 
\d\ completions. Two proofs are given. The second is a blow up and blur construction. The idea, a little bit more, than in a nut shell:

We work with cylindric algebras, though the idea is much more universal as we will see. 
Assume that $\RCA_n\subseteq \K$, and $S\K=\K$.
Start with a finite algebra $\C$ outside $\K$. Blow up and blur $\C$, by splitting
each atom to infinitely many, getting a new atom structure $\At$. In this process a (finite) set of blurs are used.

They do not blur the complex algebra, in the sense that $\C$ is there on this global level. 
The algebra $\Cm\At$ will not be in $\K$
because $\C\notin \K$ and $\C$ embeds into $\Cm\At$. 
Here the completeness of the complex algebra will play a major role, 
because every element of $\C$,  is mapped, roughly, to {\it the join} of 
its splitted copies which exist in $\Cm\At$ because it is complete.

These precarious joins prohibiting membership in $\K$ {\it do not }exist in the term algebra, only finite-cofinite joins do, 
so that the blurs blur $\C$ on the 
this level; $\C$ does not embed in $\Tm\At.$

In fact, the the term algebra will  not only be in $\K$, but actually it will be in the possibly larger $\RCA_n$. 
This is where the blurs play their other role. Basically non-principal ultrafilters, the blurs
are used as colours to represent  $\Tm\At\A$.

In the process of representation we cannot use {\it only} principal ultrafilters, 
because $\Tm\At$ cannot be completely representable for this would give that $\Cm\At\A$ 
is representable. 

But the blurs will actually provide a {\it complete representation} of the {\it canonical extension} 
of $\Tm\At$, in symbols $\Tm\At^+$; the algebra whose underlying set consists of all ultrafilters of $\Tm\At\A$. The atoms of $\Tm\At$ 
are coded in the principal ones,  and the remaining non- principal ultrafilters, ot the blurs, 
will be finite, used as colours to completely represent $\Tm\At^+$, in the process representing $\Tm\At$. 

As mentioned, the blow up and blur construction will be used in the second proof, and will be encountered
again later on.  In fact, we will use this construction many times. We will blow up and 
blur a  finite 
rainbow  polyadic algebra, theorem \ref{blowupandblur},
and we will blow up and blur finite rainbow relation algebras \ref{decidability}. 

We will also blow up and blur 
a certain abstract 
finite relation algebra required to satisfy certain properties, 
proved later on to exist, theorems \ref{blurs}, \ref{main}.
These finite relation algebras have $n+k$ dimensional hyperbasis but not $n+k+1$
dimensional hyperbasis. They will be used to show that there are single types in countable $L_n$ 
theories that are realized in all $n+k+1$ relativized 
models, but they cannot be isolated by a formula using 
$n+k$ variables.

We will also have occasion to split each of the atoms of such relation algebras that have 
$n$ dimensional hyperbasis, each to uncountably many,  getting a cylindric algebra that does not have an $n$ 
relativized square representation,  but is elementary equivalent to one 
that has an $n$ smooth relativized representation, witness 
theorem \ref{completerepresentation}.

Our first proof of non-atom canonicity of the variety $\K=S\Nr_n\CA_{n+4}$ for $n\geq 3$ 
is semantical. 
It says that if a certain rainbow algebra  had an $n+4$ flat representation, then an impossibility 
will arise; {\it a fortiori} it does not have an $n+4$ smooth representation.

Our second proof splits up and blurs a finite rainbow polyadic 
algebra outside $\K$ as explained above.
Because term algebra constructed will be representable, and their \d\ completions are not in $\K$, we get the same result
for $S\Nr_n\CA_{n+k}$ for any $k\geq 4$. 
When $k=\omega$, then $\K$ is the variety of representable algebras. For the latter, the result is known.

We start off with a piece of notation.

\begin{definition}\label{subs}
Let $n$ be an ordinal. An $s$ word is a finite string of substitutions $({\sf s}_i^j)$,
a $c$ word is a finite string of cylindrifications $({\sf c}_k)$.
An $sc$ word is a finite string of substitutions and cylindrifications
Any $sc$ word $w$ induces a partial map $\hat{w}:n\to n$
by
\begin{itemize}

\item $\hat{\epsilon}=Id$

\item $\widehat{w_j^i}=\hat{w}\circ [i|j]$

\item $\widehat{w{\sf c}_i}= \hat{w}\upharpoonright(n\sim \{i\})$

\end{itemize}
\end{definition}

If $\bar a\in {}^{<n-1}n$, we write ${\sf s}_{\bar a}$, or more frequently
${\sf s}_{a_0\ldots a_{k-1}}$, where $k=|\bar a|$,
for an an arbitrary chosen $sc$ word $w$
such that $\hat{w}=\bar a.$
$w$  exists and does not
depend on $w$ by \cite[definition~5.23 ~lemma 13.29]{HHbook}.
We can, and will assume \cite[Lemma 13.29]{HHbook}
that $w=s{\sf c}_{n-1}{\sf c}_n.$
[In the notation of \cite[definition~5.23,~lemma~13.29]{HHbook},
$\widehat{s_{ijk}}$ for example is the function $n\to n$ taking $0$ to $i,$
$1$ to $j$ and $2$ to $k$, and fixing all $l\in n\setminus\set{i, j,k}$.]

The following easy lemma, saying basically that in certain tight dilations, any non-zero element
in the dilation  intersects a `permuted version' of an atom in the small algebra using spare dimensions. 
The algebra $\B$ is a tight dilation of
$\A$, if $\A$ is a complete subalgebra of $\Nr_n\B$, in symbols, $\A\subseteq_c \Nr_n\B$.
The lemma will  be used in the last item of our last theorem, namely, theorem \ref{maintheorem}.

\begin{lemma}\label{lem:atoms2}
Let $n<m$ and let $\A$ be an atomic $\Sc_n$,
$\A\subseteq_c\Nr_n\C$
for some $\C\in\Sc_m$.  For all $x\in\C\setminus\set0$ and all $i_0, \ldots i_{n-1} < m$ there is $a\in\At(\A)$ such that
${\sf s}_{i_0\ldots i_{n-1}}a\;.\; x\neq 0$.
\end{lemma}
\begin{proof}
We can assume, see definition  \ref{subs},
that ${\sf s}_{i_0,\ldots i_{n-1}}$ consists only of substitutions, since ${\sf c}_{m}\ldots {\sf c}_{m-1}\ldots
{\sf c}_nx=x$
for every $x\in \A$.We have ${\sf s}^i_j$ is a
completely additive operator (any $i, j$), hence ${\sf s}_{i_0,\ldots i_{\mu-1}}$
is too  (see definition~\ref{subs}).
So $\sum\set{{\sf s}_{i_0\ldots i_{n-1}}a:a\in\At(\A)}={\sf s}_{i_0\ldots i_{n-1}}
\sum\At(\A)={\sf s}_{i_0\ldots i_{n-1}}1=1$,
for any $i_0,\ldots i_{n-1}<n$.  Let $x\in\C\setminus\set0$.  It is impossible
that ${\sf s}_{i_0\ldots i_{n-1}}\;.\;x=0$ for all $a\in\At(\A)$ because this would
imply that $1-x$ was an upper bound for $\set{{\sf s}_{i_0\ldots i_{n-1}}a:
a\in\At(\A)}$, contradicting $\sum\set{{\sf s}_{i_0\ldots i_{n-1}}a :a\in\At(\A)}=1$.
\end{proof}

Our next theorem, using a rainbow argument, constructs countable atomic 
representable algebras whose \d\ completions, that is, 
the complex algebra of their atom structures, does not posses even an {\it $n+4$ flat} representation.

We prove a result  stronger than the result in \ref{hodkinson} in the following corollary.
It is obtained from Hodkinson's rainbow algebras, dealt with in \ref{hodkinson}
by  truncating the reds and greens to be finite, but the greens will be more,
to ensure a \ws\ for \pa\ . The number of rounds \pa\ needs to win together with number of pebbles in play,  determines 
the $n+4$ extra dimensions. 

Later, we will encounter other cylindric-like rainbow constructions generalizing rainbow constructions
used for relation algebras, witness theorems \ref{neat11},  \ref{completerepresentation}, \ref{j}. 
This lifting is not so hard; 
several instances are scattered in the literature, see \cite{HH} for example.

We can say that it can always be done, 
except that the connection between nodes and rounds needed for the games
played on coloured graphs \cite[lemma 30,]{HH}, to that between the number of pebbles and rounds used 
in the private \ef\ forth game 
played on the greens and reds, may not satisfy the same combinatorial connection established in the
relation algebra case by Hirsch and Hodkinson. 

In relation algebras the number of rounds increase by $1$, and the pebbles (nodes) 
increase by $2$, when passing from the private game to the rainbow game. 

In cylindric algebras it is safe to say that the number of nodes of graphs as well as that 
of the rounds of the game also increases, but 
a general cylindric rainbow theorem, seems to be missing. 

We will give, yet another two instances below, and sketrch some more, 
lifting  \ws\ s from the same private \ef forth games used in relation algebras   
to cylindric rainbow algebras, to give more than a glimpse 
that this task can always be implemented, witness theorems 
\ref{squarerepresentation}, \ref{completerepresentation}. 

The basic argument used in the next theorem, is that in the case of the existence of an $n+4$ flat 
representation, then an inconsistent triple of reds will be forced by an $n+4$ red clique 
(in the usual graph theoretic sense)
labelling edges between apexes of the same cone, 
with base labelled by the `yellow shade' denoted by $\y_{n+2}$.

\begin{theorem}\label{smooth} For every $n\geq 3$ there exists a
countable atomic $\sf RPEA_n$, such that the $\CA$ reduct of
its completion does not have an $n+4$ flat representation, in particular, it is not  representable. Furthermore, its
${\sf Df}$ reduct is not representable.
\end{theorem}
\begin{proof}
Here we closely follow \cite{Hodkinson}; but our reds and greens are finite,
so we obtain a stronger result. We take $|{\sf G}|=n+2$, and ${\sf R}=n+1$.
Let $L^+$ be the rainbow signature consisting of the binary
relation symbols $\g_i :i<n-1, \g_0^i:  i< n+2, \w, \w_i: i <n-2, \r_{jk}^i  (i<\omega, j<k<n+1)$
and the $(n-1)$ ary-relation symbols
$\y_S: S\subseteq n+2)$, together with a shade of red $\rho$ that is outside the rainbow
signature but is a binary relation in the sense that it can label edges
of coloured graphs. Let $\GG$ be the class of corresponding rainbow coloured graphs.
By the same methods as above, there is a countable model  $M\in \GG$ with the following
property:\\
$\bullet$ If $\triangle \subseteq \triangle' \in \GG$, $|\triangle'|
\leq n$, and $\theta : \triangle \rightarrow M$ is an embedding,
then $\theta$ extends to an embedding $\theta' : \triangle'
\rightarrow M$. Now let $W = \{ \bar{a} \in {}^n M : M \models ( \bigwedge_{i < j < n,
l < n} \neg \rho(x_i, x_j))(\bar{a}) \}$. Then $\A$ with universe $\{\phi^W: \phi\in L_n\}$ and operations defined the usual way,
is representable, and its
completion, the complex algebra over the  above rainbow atom structure, $\C$ has universe $\{\phi^W: \phi\in L_{\infty, \omega}^n\}$.

We show that $\C$ is as desired. Assume, for contradiction, that $g:\C\to \wp(V)$ induces a relativized $n+4$ flat representation $N$,
so that $V\subseteq {}^nN$ and $N=\bigcup_{s\in V}\rng(s)$.
Then $V\subseteq {}^nN$ and we can assume that
$g$ is injective because $\C$ is simple. First there are $b_0,\ldots, b_{n-1}\in N$ such $\bar{b}\in
h(\y_{n+2}(x_0,\ldots, x_{n-1}))^W$, cf \cite[lemma 5.7]{Hodkinson}.
(This tuple will be the base of finitely many cones, that will be used to force an inconsistent triple of reds.)
This is because $\y_{n+2}(\bar{x})^W\neq \emptyset$.  For any $t<n+4$, there is a $c_t\in N$, such
that $\bar{b}_t=(b_0,\ldots b_{n-2},\ldots c_t)$ lies in $h(\g_0^t(x_0, x_{n-1})^W$ and in $h(\g_i(x_i, x_{n-1})^W$ for each $i$ with
$1\leq i\leq n-2$, cf \cite[lemma 5.8]{Hodkinson}. Here $n$ flatness is used. The $c_t$'s are the apexes of the cones with base $\y_{n+2}$.

Take the formula
$$\phi_t=\y_{n+2}(x_0,\ldots ,x_{n-2})\to \exists x_{n-1}(\g_0^t(x_0, x_{n-1}))\land \bigwedge_{1\leq i\leq n-2}\g_i(x_i, x_{n-1})),$$
then $\phi_t^{W}=W$. Pick $c_t$ and $\bar{b_t}$ as above, and define for each $s<t<n+4,$ $\bar{c_{st}}$ to be
$(c_s, b_1,\ldots  b_{n-2}, c_t)\in {}^nN.$
Then $\bar{c}_{st}\notin h((x_0,\ldots  x_{n-1})^W$. Let $\mu$ be the formula
$$x_0=x_{n-1}\lor \w_0(x_0, x_{n-1})\lor \bigvee \g(x_0, x_{n-1}),$$
the latter formula is a first order formula consisting of the disjunction of  the (finitely many ) greens.
For $j<k<n+1$, let $R_{jk}$ be the $L_{\infty\omega}^n$-formula $\bigvee_{i<\omega}\r_{jk}^i(x_0, x_{n-1})$.
Then
$\bar{c}_{st}\notin h(\mu^W)$, now for each $s<t< n+4$, there are $j<k<n+1$ with $c_{st}\in h(R_{jk})^W.$
By the pigeon- hole principle, there are $s<t< n+4$ and $j<k<n+1$
with $\bar{c}_{0s}, \bar{c}_{0t}\in h(R_{jk}^W)$. 
We have also $\bar{c}_{st}\in h(R_{j',k'}^W)$
for some $j', k'$ then the sequence $(c_0, c_s, b_2,\ldots b_{n-2}, c_t)\in h(\chi^W)$
where
$$\chi=(\exists_1R_{jk})(\land (\exists x_{n-1}(x_{n-1}=x_1\land \exists x_1 R_{jk})\land (\exists x_0(x_0=x_1)\land \exists x_1R_{j'k})),$$
so $\chi^W\neq \emptyset$. Let $\bar{a}\in \chi ^W$. Then $M\models _W  R_{jk}(a_0,a_{n-1})\land R_{jk}(a_0,a_1)\land R_{j'k'}(a_1, a_{n-1})$.
Hence there are
$i$, $i'$ and $i''<\omega$ such that
$$M\models _W\r_{jk}^{i}(a_0,a_{n-1})\land \r_{jk}^{i'}(a_0,a_1)\land \r_{j'k'}^{i''}(a_1, a_{n-1}),$$
cf. \cite[lemma 5.12]{Hodkinson}.
But this triangle is inconsistent. 
For the last part, if its $\sf Df$ reduct is representable, then $\Rd_{df}\A$ will be completely representable,
hence $\A$ itself will be completely representable because it is generated by elements whose dimension set
$<n$, which is a contradiction.
\end{proof}

\begin{corollary}\label{can} We have $\C\notin S\Nr_n\CA_{n+4}$. In particular, for any $k\geq 4$,  the variety
$S\Nr_n\CA_{n+k}$ is not atom-canonical.
\end{corollary}
\begin{proof}
The first part.  Assume, for contradiction,  that $\C\in S\Nr_n\CA_{n+4}$; let $\C\subseteq \Nr_n\D$.
Then $\C^+\in S_c\Nr_n\D^+$, and $\D^+$ is of course atomic. We show that $\C^+$ has an $n+4$ dimensional hyperbasis,
from which we infer that it has an $n+4$ smooth representation, and so does $\C$ (in the obvious way)
which contradicts the previous
theorem.

First note that for every $n\leq l\leq m$, $\Nr_l\D^+$ is atomic.
Indeed, if $x$ is an atom in $\D^+$, and and $n\leq l<m$,
then ${\sf c}_{l}\ldots {\sf c}_{m-l+1}x$ is an atom in $\Nr_l\D^+$,
so if $c\neq 0$ in the latter, then there exists $0\neq a\in \At\D^+$,
such that $a\leq c$, and so ${\sf c}_{l}\ldots {\sf c}_{m-1+1}a\leq {\sf c}_{l}\ldots {\sf c}_{m-1+1}c=c$.

Let $\Lambda=\bigcup_{k<n+3}\At\Nr_k\D^+$, and let $\lambda\in \Lambda$.
In this proof we follow closely section 13.4 in \cite{HHbook}. The details are omitted because they are identical
to the corresponding ones in op.cit.
For each atom $x$ of $\D$, define $N_x$, easily checked to be an $m$ dimensional   $\Lambda$ hypernetwork, as follows.
Let $\bar{a}\in {}^{n+4}n+4$ Then if $|a|=n$,  $N_x(a)$ is the unique atom $r\in \At\D$ such that $x\leq {\sf s}_{\bar{a}}r$.
Here substitutions are defined as in \ref{subs}.
If $n\neq |\bar{a}| <n+3$, $N_x(\bar{a})$ the unique atom $r\in \Nr_{|a|}\D$ such that $x\leq s_{\bar{a}}r.$
$\Nr_{|a|}\D$ is easily checked to be atomic, so this is well defined.

Otherwise, put  $N_x(a)=\lambda$.
Then $N_x$ as an $n+4$ dimensional $\Lambda$ hyper-network, for each such chosen $x$ and $\{N_x: x\in \At\C\}$
is an $n+4$ dimensional $\Lambda$ hyperbasis.
Then viewing those as a saturated set of mosaics, one can can construct, like in the proof of theorem \ref{step} a complete
$n+4$ smooth representation of $M$ of $\C$.
But this contradicts theorem \ref{smooth} which excludes even the existence of 
$n+4$ flat representations.
\end{proof}

\subsection{A different view, blowing up and blurring a finite rainbow cylindric algebra}

In the following we denote the rainbow algebra $R(\Gamma)$ defined in \cite[definition 3.6.9]{HHbook2}
by $\CA_{\sf G, \sf R}$ where $\sf R=\Gamma$
is the graph of reds, which will be a complete irreflexive graph, and
$\sf G$ the indices greens with subscript $0$.
It will be denoted by $\PEA_{\sf G, \sf R}$ if we count in the polyadic operations.

The idea of a blow up and blur construction was outlined 
above for cylindric algebras, but the same idea can and indeed does work for relation algebras, too.
Concrete instances of this technique exists in the literature, see e.g.
Andr\'eka et all \cite[theorem 1.2]
{ANT}, approached in theorem \ref{blurs}.

This construction has affinity to the blow up and blur construction 
for relation algebras, witness  \cite[lemma 17.32, 17.34, 17.35, 17.36]{HHbook}, 
to be also approaced in theorem \ref{decidability} below. In fact, the essential idea is the same.

The major difference, in fact the {\it only} significant difference, between the two constructions,
is that in the former  a Maddux finite algebra is used, while in the latter case
a finite rainbow algebra is used.

In the former case, we can only infer that the complex algebra is not representable,
in the latter  case we can know and indeed we can prove more. 

The reason basically is  that non-representability
of the Maddux algebra on finite sets, depends on an uncontrollable big Ramsey number (that is a function in the dimension),
while for  rainbow algebras we can control {\it when the algebra stops to be representable} by \pa\ s moves.
\pa\ forces a win by using greens, it is precisely this number, that  determines the extra dimensions in which the complex
algebras `stop to be neatly embeddable'. 
More succintly,  it is precisely the point 
at which the greens  outfit the reds.

What can be done here is substitute a Maddux  finite relation algebra used by Andr\'eka and N\'emeti, by the rainbow algebra mentioned used
by Hirsch and Hodkinson and using the arguments of Andr\'eka and N\'emeti, to get a stronger result.
Or alternatively we can hope to lift Hirsch and  Hodkinsons
construction from relation algebras whose atoms are colours to cylindric algebras
whose atoms are coloured graphs.

But, while the algebras used in \cite{ANT} have $n$ dimensional cylindric basis,
the relation algebra obtained by Hirsch and Hodkinson  does not have an $n$ dimensional cylindric basis, except for $n=3$,
and using this it was proved that only for the lowest value of $n$ namely $n=3$, the class $S\Nr_3\CA_k$, $k\geq 6$
is not closed under completions \cite{tarek}.

A substitute of $n$ dimensional cylindric basis that serves in this context is surjections from $n$
to coloured graphs, more succintly atoms, 
used in rainbow constructions for cylindric algebras.

We have two relation algebras, a Maddux one, and a rainbow one, that do almost the same thing  at least
they provide weakly representable
atom structures that are not strongly representable. One renders a cylindric base, but witnesses 
non-atom canonicity only in infinite dimensions, the other does not have a cylindric base, 
but it has the supreme advantage that it witnesses non atom-canonicity in {\it finite many} extra dimensions, determined
by the number of pebbles in an atomic game. 

We want a finite  $n$ dimensional cylindric algebra,  preferably a polyadic equality one,
that also witnesses non atom-canoncity in extra dimensions using a hopefully simple atomic 
$n$ pebble  game, too. Rainbows offer solace here.

And indeed, we define a  rainbow {\it polyadic equality rainbow algebra} that fits the bill.
But  we first highlight the sketch given in the intoduction of how to construct 
rainbow polyadic equality atom structures formulating it as a definition.

\begin{definition} Let $\Gamma$ be a graph. Then we denote the rainbow algebra $R(\Gamma)$, endowed with substitutions by
by $\PEA_{{\sf G}, \Gamma}$ where $\sf G$ is the set of greens, and the accessibility relation
corresponding to the polyadic operations (of transpositions)  are defined (in the notation of \cite[definition 3.6.3]{HHbook2}),
by $R_{{\sf s}_{[ij]}}=\{([f], [g]):  f\circ [i,j]=g\}$.
\end{definition}
Games are played on coloured graphs corresponding to networks as defined in \cite{HH}. However, in the polyadic equality case
networks have to
be symmetric; they have to satisfy the additional condition that
${\sf s}_{[i,j]}N(\bar{x})=N(x\circ [i,j])$. This does not affect the translation to graphs
nor does it affect winning strategies.
The polyadic extra information is coded in the networks not the moves 
of the game. 

One can define rainbow atom structures for other algebras, too, like $\Sc$ and $\PA$, but we will 
not need to do this here. However, we will have occasion, in fact several ones, 
to take $\Sc$ or, for that matter,  $\PA$ reducts of a constructed 
rainbow polyadic equality algebra, $\PEA_{A,B}$ constructed on two structures.
This is different than starting off by defining an $\Sc$ say,  rainbow 
algebra, by defining onl74y the accessibility relations 
corresponding to cylindrifiers and substitutions corresponding to replacements. 
If we start with two structure $A$, $B$, then the complex $\Sc$ rainbow algebra
$\Sc_{A,B}$ will be, in general, 
smaller than $\Rd_{sc}\PEA_{A,B}$.

\begin{theorem}\label{blowupandblur}
\begin{enumarab}
\item  For any finite $n>2$, for any $\K$ between $\Sc$ and $\PEA,$
the variety $S\Nr_n\K_{n+4}$ is not atom-canonical. In more detail there exists  an atomic 
countable $\A\in \PEA_n$ such $\Tm\At\A\in \sf RPEA_n$,  but
$\Rd_{sc}\Cm\At\A\notin S\Nr_n\Sc_{n+4}$, and $\Rd_{df}\Cm\At\A$ is not in 
$\sf Rdf_n$ In particular, $S\Nr_n\K_{n+k}$ is not atom-canonical for any $k\geq 4$ and any 
class $\K$ between 
$\Sc_n$ and $\PEA_n$ 

\item Even more,  for $n>2$, the variety ${\sf CB_{n+4}}$ is not atom-canonical, 
and this results extends to all considered cylindric-like algebras dealt with in this paper
by the  obvious modifications of its 
definition.
\end{enumarab}
\end{theorem}
\begin{proof}
\begin{enumarab}
\item This is proved only for $\CA_n$ in \cite{can}. Here we generalize the result to any $\K$ between $\Sc$ and $\PEA$.
The idea of the proof is essentially the same. Let $\At$ be the rainbow atom structure in \cite{Hodkinson} except that we have $n+2$ greens and
$n+1$ reds, that is, the rainbow atom structure dealt with in \ref{smooth}.

The rainbow signature now consists of $\g_i: i<n-1$, $\g_0^i: i\in n+2$, $\r_{kl}^t: k,l\in n+1$, $t\in \omega$,
binary relations and $\y_S$ $S\subseteq n+2$,
and a shade of red $\rho$; the latter is outside the rainbow signature,
but it labels coloured graphs during the game, and in fact \pe\ can win the $\omega$ rounded game
and build the $n$ homogeneous model $M$ by using $\rho$ when
she is forced a red.

Then $\Tm\At$ is representable; in fact it is representable as a polyadic equality algebra;
this can be proved exactly as in \cite{Hodkinson}; by defining the polyadic operations of the set algebra $\A$
completely analogous to the one constructed in theorem \ref{hodkinson}
by swapping variables. Recall that $\At\A=\At$ and that $\Tm\At=\Tm\At\A\subseteq \A$.

The atoms of $\Tm\At\A$ are coloured graphs whose edges are not labelled by
the one shade of red  $\rho$; it can also be viewed as a set algebra based on $M$
by relativizing semantics discarding assignments whose edges are labelled
by $\rho$. A coloured graph (an atom) in $\PEA_{n+2, n+1}$
is one such that at least one of its edges is labelled red.

Now $\PEA_{n+2, n+1}$ embeds into $\Cm\At\A$,
by taking every red graph to the join of its copies, which exists because $\Cm\At\A$ is complete
(these joins do not exist in the (not complete) term algebra; only joins of finite or cofinitely many reds do, hence it serves non representability.)
A copy of a red graph is one that is isomorphic to this graph, modulo removing superscripts of reds.

Another way to express this is to take every coloured graph to the interpretation of an infinite disjunct of the $\sf MCA$ formulas
(as defined in \cite{Hodkinson}), and to be dealt with below; such formulas define coloured graphs whose edges 
are not labelled by the shade of red,
hence the atoms, corresponding
to its copies, in the relativized semantics; this defines an embedding,  because $\Cm\At\A$ is isomorphic to
the set algebra based on the same relativized semantics
using $L_{\infty,\omega}^n$ formulas in the rainbow signature.
(Notice that that the rainbow theory itself \cite[definition 3.6.9]{HHbook2} is only first order because we have only finitely
many greens).

Here again $M$ is the new  homogeneous model constructed
in the new rainbow signature, though the construction is the
same \cite{Hodkinson}.
But \pa\ can win a certain finite rounded game, namely $F^{n+2}$ to be defined in a while, theorem \ref{neat},
\cite{can} on $\Rd_{\Sc}\PEA_{n+1, n+2},$ hence it is
outside $S\Nr_n\Sc_{n+4}$ and so is $\Cm\At\A$,  because the former is embeddable in the latter
and $S\Nr_n\Sc_{n+4}$ is a variety; in particular, it is closed
under forming subalgebras.

\item Note that in the previous proofs, the \d\ completion of $\A$, namely
the complex algebra $\Cm\At\A$ can, in principal, be in $S\Nr_n{\sf D_{n+4}}$,
so that it could well posses an {\it $n+4$ square relativized} representation, 
which leaves the question whether ${\sf CB}_{n,n+4}$ is atom-canonical unsettled, so far.

Now we show that it does  not.

This part is entirely new. It is very easy to show that \pa\ can win the game $G_{\omega}^{n+4}$, without needing to re-use his 
pebbles, so that $\Rd_{ca}\PEA_{n+2, n+1}$ is actually not in 
${\sf CB}_{n,n+4}$ which is larger (strictly) than $S\Nr_n\CA_{n+4}$.
We are now done with this more difficult, and indeed stronger,  case as well,  
because as above,  $\PEA_{n+2, n+1}$ embeds into $\Cm\At\A$, and ${\sf CB}_{n,n+4}$ 
is a variety. 
The proof works for any $\K$ between $\Sc$ and $\PEA$ here.
\end{enumarab}
\end{proof}
For $2<m<n$, the existence of a basis of dimension $n$ for a $\CA_m$ captures 
$n$ square relativized representations while the existence of 
$n$ dimensional hyperbasis captures $n$ smooth relativized representations.

The definition of both basis and hyperbasis involves a set of labels. 
In the former case, it can be shown that such labels are superfluous, they carry no extra information.
In the second case the presence of labelled hyperedges makes a difference.
Such hyperlabels can be infinite, in which case there could be finite algebras wih no finite hyperbasis;
and these necessarily do not have smooh relativized representations, witness therem \ref{smoothcompleteness2}.
If we take the set of labels to be a one element set, 
then the following cannot help but to spring to mind.

One adds the amalgamation property to the definition of basis formulated without
labels, requiring that any overlapping networks
have an amalgam inside the basis.
We get what we call, following Maddux, cylindric basis. However, Maddux worked with relation algebras.

An $n$ pebbled $\omega$ 
rounded game can now be devised such that \pe\ has a \ws\ in this game 
on an atomic $\A$  iff $\A$ has an $n$ flat representation  by allowing amalgamation moves.

We leave the details to the reader, who might find \cite[def 12.11, 12.17, 12.26, the last definition 
for a more general game played on $n$ $m$ wide hyperbasis]{HHbook}
helpful.

\section{Omitting types} 

\subsection{ In clique guarded semantics}

We start by two new algebraic results that will be used 
to show that the omitting types theorems fails even if
we consider clique guarded semantics. The first is conditional; the condition will be fullfilled later on, see second
item of theorem \ref{main}.

In our next theorem we give only the idea of  proof referring to \cite{ANT} for the 
detailed argument. Later we will produce some of the details, cf theorem \ref{main}.
For the definition of $n+k$ complex blur the reader is referred to \cite[definition 3.1]{ANT};
this involves a set $J$ of blurs and a ternary relation
$E$.

\begin{theorem}\label{blurs} Let $k\geq 1$. Assume that there exists a finite relation algebra $\R\notin S\Ra\CA_{n+k+1}$
that has  $n+k$  complex blur  $(J, E)$. Let $\At$ be the infinite atom structure obtained by blowing up and blurring $\R$, in the sense of
\cite[p.72]{ANT}.
Then $\Mat_{n+k}\At$ is an $n+k$ dimensional cylindric basis. Furthermore there exists polyadic equality algebras $\C_n$ and
$\C_{n+k}$ such that  $\Tm\Mat_n\At\subseteq \C_n$ and $\Tm\Mat_{n+k}\At\subseteq \C_{n+k}$, $\C_n=\Nr_n\C_{n+k}$ and finally
$\Rd_{sc}\Cm\At\notin S\Nr_n\Sc_{n+k+1}$, and $\Rd_{df}\Cm\At$ is not representable.
\end{theorem}
\begin{proof} Exactly like the proof in \cite{ANT} by blowing up and blurring $\R$ instead of blowing up and blurring
the Maddux algebra $\M$ defined on p.84 \cite[theorem 1.2, theorem 3.2, lemma 4.3, lemma 5.1]{ANT}. 
Here we give an idea of the proof which is a blow up and blur construction, a very indicative term put forward by Andr\'eka and N\'emeti. 
For the technical details one is referred to the original paper \cite{ANT}, or to the sketch in \cite{Sayed}, and for some of the details, not all, to
theorem \ref{main}.

One starts with the finite relation algebra $\R$.
Then this algebra is blown up and blurred. It is blown up by splitting the atoms each to infinitely many.
It is blurred by using a finite set of blurs or colours $J$. This can be expressed by the product $\At=\omega\times \At \R\times J$,
which will define an infinite atom structure of a new
relation algebra. One can view such a product as a ternary matrix with $\omega$ rows, and for each fixed $n\in \omega$,  we have the rectangle
$\At \R\times J$.
Then two partitions are defined on $\At$, call them $P_1$ and $P_2$.
Composition is defined on this new infinite atom structure; it is induced by the composition in $\R$, and a ternary relation $E$
on $\omega$, that synchronizes which three rectangles sitting on the $i,j,k$ $E$ related rows compose like the original algebra $\R$.
This relation is definable in the first order structure $(\omega, <)$. 

The first partition $P_1$ is used to show that $\R$ embeds in the complex algebra of this new atom structure, so
the complex algebra cannot be in $S\Ra\CA_{n+k+1}.$

The second partition $P_2$ divides $\At$ into $\omega$ sided finitely many rectangles, each with base $W\in J$,
and the the term algebra over $\At$, are the sets that intersect co-finitely with every member of this partition.
On the level of the term algebra $\R$ is blurred, so that the embedding of the small algebra into
the complex algebra via taking infinite joins, do not exist in the term algebra for only finite and co-finite joins exist
in the term algebra.

The term algebra is representable using the finite number of blurs. These correspond to non-principal ultrafilters
in the Boolean reduct, which are necessary to
represent this term algebra, for the principal ultrafilter alone would give a complete representation,
hence a representation of the complex algebra and this is impossible.
Thereby, in particular,  an atom structure that is weakly representable but not strongly representable is obtained.

Because $(J, E)$ is a complex set of $n$ blurs, this atom structure has an $n$- dimensional cylindric basis, and so the $n$
basic matrices form an atom structure that is also only weakly representable.

The resulting $n$ dimensional cylindric term algebra obtained is subalgebra of a $k$ 
neat reduct that is not completely
representable, furthermore, its complex algebra is not in $S\Nr_n\CA_{n+k+1}$. 
To make the algebra  one generated one uses Maddux's combinatorial techniques,
and this entails using infinitely many ternary relations.
\end{proof}
We note that $k$ cannot be equal to $0$,
because Andr\'eka provided a Sahlqvist axiomatization of $S\Nr_n\CA_{n+1}$, for any finite $n$,
hence the latter is necessarily
atom-canonical. Below we shall give examples of such relations algebras. These algebras will be chosen to have
$n+k$ dimensional hyperbasis, but not $n+k+1$ dimensional ones.

We write $S_c\K$ for the class of complete subalgebras of algebras 
in $\K$. Note that $\K\subseteq S_c\K$.

\begin{theorem}\label{neat11} Let $n$ be finite $>2$. Any class $\K$, such that $\K$ contains
the class of completely representable algebras and is contained in $S_c\Nr_n\CA_{n+3}$ is not
elementary. In particular, the class of completely representable $\CA_n$s is not elementary. 
\end{theorem}
\begin{proof} Let $G_k$ be 
the usual atomic game played on networks with $k$ rounds \cite[definition 3.3.2]{HHbook2} and an unlimited number 
of pebbles, that is, there are no restriction on the number of pebbles. 
Let $\N^{-1}$ denote $\N$ with reverse order, let $f:\N\to \N^{-1}$ be the identity map, and denote $f(a)$ by 
$-a$, so that for $n,m \in \N$, we have $n<m$ iff $-m<-n$. 
We assume that $0$  belongs to $\N$.

Let $\A$ be the rainbow algebra $\CA_{\N^{-1},\N}$, then \pe\ has a \ws\ in $G_k$ for all finite $k\geq n$
hence it is elementary equivalent to a
countable completely representable algebra $\B$, cf. \cite[proposition 30]{HH}.
Strictly speaking the game in the last reference is played on $\CA_{\N,\N}$ if we allow that 
$0\in \N$ (that is on $\CA_{\omega,\omega}$) but the \ws\ here is the same.

Let $F^m$ be the $\omega$ rounded atomic game, $m>n$
except that \pa\ s moves are limited to $m$ pebbles, but he has the option to re-use them.

Then it can be shown that if $\A\in S_c\Nr_n\CA_m$, then \pe\ has a \ws\ in $F^m$; 
the proof of this can be easily destilled from the proof 
of the last item in theorem \ref{maintheorem}.

The game now is played
on coloured graphs \cite[4.3.3 p. 839, lemma 30]{HH}, but the forbidden triple
connecting two greens and one red is changed to
``$(\g^i_0, \g^j_0, \r_{kl})$ unless
${(i, k), (j, l)}$ is an order preserving partial function from
$\N^{-1}\to\N$.''

We show that \pa\ has a \ws\ in $F^{n+3}$, the argument used is the $\CA$ analogue of \cite[theorem 33, lemma 41]{r}.
The difference is that in the relation algebra case, the game is played on atomic networks, but now it is translated to playing on coloured graphs, 
\cite[lemma 30]{HH}.

In the initial round \pa\ plays a graph $\Gamma$ with nodes $0,1,\ldots n-1$ such that $\Gamma(i,j)=\w$ for $i<j<n-1$
and $\Gamma(i, n-1)=\g_i$
$(i=1, \ldots, n-2)$, $\Gamma(0,n-1)=\g_0^0$ and $\Gamma(0,1\ldots, n-2)=\y_{B}$.

In the following move \pa\ chooses the face $(0,\ldots n-2)$ and demands a node $n$
with $\Gamma_2(i,n)=\g_i$ $(i=1,\ldots, n-2)$, and $\Gamma_2(0,n)=\g_0^{-1}.$
\pe\ must choose a label for the edge $(n+1,n)$ of $\Gamma_2$. It must be a red atom $r_{mn}$. Since $-1<0$ we have $m<n$.
In the next move \pa\ plays the face $(0, \ldots, n-2)$ and demands a node $n+1$, with $\Gamma_3(i,n)=\g_i$ $(i=1,\ldots, n-2)$,
such that  $\Gamma_3(0,n+2)=\g_0^{-2}$.
Then $\Gamma_3(n+1,n)$ and  $\Gamma_3(n+1,n-1)$ both being red, the indices must match.
$\Gamma_3(n+1,n)=r_{ln}$ and $\Gamma_3(n+1, n-1)=r_{lm}$ with $l<m$.
In the next round \pa\ plays $(0,1,\ldots, n-2)$ and reuses the node $2$ such that $\Gamma_4(0,2)=\g_0^{-3}$.
This time we have $\Gamma_4(n,n-1)=\r_{jl}$ for some $j<l\in \N$.
Continuing in this manner leads to a decreasing sequence in $\N$.

Now that \pa\ has a \ws\ in $F^{n+3},$ it follows that $\A\notin S_c\Nr_n\CA_{n+3}$, but it is elementary equivalent
to a countable completely representable algebra. Indeed, using ultrapowers and an elementary chain argument,
we obtain $\B$ such  $\A\equiv \B$ \cite[lemma 44]{r},
and \pe\ has a \ws\ in $G_{\omega}$, so by \cite[theorem 3.3.3]{HHbook2}, $\B$ is completely representable.

So if $\K$ is as above, then $\A\notin \K$ but $\B$
is in $\K$, since $\A\equiv \B$, it readily follows that $\K$ is not
elementary.
\end{proof}

Algebraic logic proves useful when it has strong impact on first order logic.
In this connection we now formulate, and prove,  three (negative)
new omitting types theorems, that are consequences of the algebraic results formulated in theorems \ref{can}, \ref{blurs} and 
\ref{neat11}, given in this order.

Let $T$ be a countable consistent first order theory and let $\Gamma$ be a type that is realized in every model of $\Gamma$.
Then the usual Orey Henkin theorem
tells us that this type is necessarily principal, that is, it is isolated by a formula $\phi$.
We call such a $\phi$ an $m$ witness, if $\phi$ is built up of $m$ variables.
We say that $\phi$ is only a witness if is an $m$ witness for a theory containing
$m$ variables.

We denote the set of formulas in a given language by $\Fm$ and for a set of formula $\Sigma$ we write $\Fm_{\Sigma}$ for the Tarski-
Lindenbaum quotient (polyadic)
algebra.

Let $T$ be a consistent  $L_n$ theory, and $m>n$. A model $M$ of $T$ is $m$ smooth if it is the base of an $m$ smooth relativized
representation of the Tarski-Lindenbaum algebra
$\Fm_T$, that is there exists $V\subseteq {}^nM$, where $M$ is the smallest such set,
and an injective homomorphism $f:\Fm_T\to \wp(V)$.
If $T$ is countable then an $\omega$ smooth model is an ordinary model, and if it is consistent, then
it has an $m$ smooth model for every finite $m\geq n$.

For a (possibly relativized) representation $M$, a formula $\phi$,  and an assignment $s\in {}^{\alpha}M$, we write
$M\models \phi[s]$ if $s$ satisfies $\phi$ in $M$.
\begin{definition} Let $T$ be a given $L_n$ theory.
\begin{enumarab}
\item A formula $\phi$ is said to be complete in $T$ iff for every formula $\psi$ exactly one of
$$T\models \phi\to \psi, \\ T\models \phi\to \neg \psi$$
holds.
\item A formula $\theta$ is completable in $T$ iff there there is a complete formula $\phi$ with $T\models \phi\to \theta$.
\item $T$ is atomic iff if every formula consistent with $T$ is completable in $T.$

\item An $n$ smooth model $M$ of $T$ is atomic if for every $s\in V\subseteq {}^{n}M$,
there is a complete formula $\phi$ such that $M\models \phi[s].$
\end{enumarab}
\end{definition}
There is a prevalent misconception that {\it cylindric algebras} of dimension $n$  are suitable for dealing with $L_n$ 
(first order logic restricte to the first $n$
variables)
with its semantical and syntactical notions, whereas in fact, it is the class of polyadic algebras of dimension $n$, that constitute 
the ``real" algebraic counterpart 
of $L_n$. However, the theory of (representable) cylindric algebras $(\RCA_n)$ is far more developed than that 
of polyadic equality algebras.
This is simply not true, because $\CA_n$ corresponds to the so-called {\it restricted version}
of $L_n$ \cite[section 4.3-5.6]{tarski}; here variables occurring in formulas are only allowed to appear in
their natural order, which is to our mind,  an inappropriate and indeed strange  restriction. The motivation was that
in the case of first order logic {\it any} formula is equivalent
to a restricted one, but this is not the case when we truncate the available variables to $n>2$,
and this can indeed be proved.
What corresponds to $L_n$ is rather $\PEA_n$ (substitutions corresponding
to transpositions, considered  as unary connectives, allow arbitrary $L_n$ formulas to be defined from restricted ones).
But in all our needed theorems, fortunately,
we were able to  obtain $\PEA_n$s. Now we can apply our results
to obtain the following theorems on (unrestricted) usual $L_n$.

It turns out that when we seek atomic models for atomic theories
the situation is drastically different than usual first order logic.
We may not find one, even among the $n+3$
smooth models. In other words, atomic models may not be found even if we consider clique guarded semantics.
More precisely, we have:

\begin{theorem}\label{OTT}
\begin{enumarab}
\item There is a countable consistent atomic $L_n$ theory $T$
with no smooth $n+4$ atomic model

\item Assume the hypothesis of theorem \ref{blurs}.
Then there is a countable consistent theory $T$, a type realized in every $n+k+1$
smooth model, but there is no $n+k$ witness.

\item There is a countable consistent $L_n$ theory $T$ and a type $\Gamma$ such that
$\Gamma$ is realized in every smooth $n+3$ model, but does not have a witness.

\end{enumarab}
\end{theorem}
\begin{proof}
\begin{enumarab}
\item
Let $\A$ be as in corollary \ref{can}, and let $\Gamma'$ be the set of  atoms.
We can and will assume that $\A$ is simple. Recall that it has countably many atoms, so we can assume that it is countable
by replacing it by its term algebra if necessary.
Then $\A=\Fm_T$ for some countable consistent $L_n$ theory $T$ and because $\A$ is atomic (as an expansion of a Boolean algebra),
and so  $T$ is atomic according to the above definition.

Let $\Gamma=\{\phi: \phi_T\in \Gamma'\}$. Then we claim
that $\Gamma$ is realized in all $n+4$
models.  For this consider such  a model $\M$ of $T$. If  $\Gamma$ is not realized in $\M$,
then this gives an $n+4$ complete smooth representation of $\A=\Fm_T$,
which is impossible,  because $\Cm\At\A$ is not in $S\Nr_n\CA_{n+4}$.

Now  assume that  $M$ is an $n+4$ smooth atomic model of $T$.
Then $\Gamma$ is realized in $M$, and $M$ is atomic, so there must be a witness to $\Gamma$.

Suppose that $\phi$ is such a witness, so that $T\models \phi\to \Gamma$.
Then $\A$ is simple, and so we can assume
without loss of generality, that it is set algebra with a countable
base. 

Let $\M=(M,R)$  be the corresponding model to this set algebra in the sense of \cite[section 4.3]{tarski}.
Then $\M\models T$ and $\phi^{\M}\in \A$.
But $T\models \exists x\phi$, hence $\phi^{\M}\neq 0,$
from which it follows that  $\phi^{\M}$ must intersect an atom $\alpha\in \A$ (recall that the latter is atomic).
Let $\psi$ be the formula, such that $\psi^{\M}=\alpha$. Then it cannot
be the case that
that $T\models \phi\to \neg \psi$,
hence $\phi$ is not a  witness,
and we are done.

\item The same argument exactly using instead the statement in \ref{blurs}

\item No we use the algebra in theorem  \ref{neat}. Let $\A$ be the term algebra of $\PEA_{\Z,\N}$.
Then $\A$ is an  atomic countable representable algebra, such that its $\Sc$ reduct, is not
in $S_c\Nr_n\Sc_{n+3}$ (for the same reasons as in the proof of \ref{neat}, namely,  \pe\ still can win all finite rounded games,
while \pa\ can win the game $F^{n+3}$, because $\A$ and the term algebra have the same
atom structure).

Assume that $\A=\Fm_T$, and let $\Gamma$ be the set $\{\phi: \neg \phi_T \text { is an atom }\}$.
Then $\Gamma$ is a non-principal type, because $\A$ is atomic, but it has no $n+3$ flat representation
omitting $\Gamma$,  for such a representation would necessarily yield a complete $n+3$ relativized representation of $\A,$ which in turn
implies that it would be in
$S_c\Nr_n\PEA_{n+3},$ and we know that this is not the case.
\end{enumarab}
\end{proof}
We also have, using item (2) of theorem \ref{blowupandblur}, 
the stronger result that a type can be relized in all $n+4$ square models; these miss out on commutativity of quantifiers
and yet does not have a witness.

\begin{corollary}\label{sah} Assume the hypothesis
in \ref{blurs}; let $k\geq 1$ and $n$ be finite with $n>2$.
Then the following hold; in particular, when $k=3$ we know, by theorem \ref{smooth},  that the following indeed hold.

\begin{enumarab}

\item There exist two atomic
cylindric algebras of dimension $n$  with the same atom structure,
one  representable and the other is not in $S\Nr_n\CA_{n+k+1}$.

\item For $n\geq 3$ and $k\geq 3$, ${S}\Nr_n\CA_{n+k+1}$
is not closed under completions and is not atom-canonical.
In particular, $\RCA_n$ is not atom-canonical.

\item There exists an algebra in $S\Nr_n\CA_{n+k+1}$  with a dense representable
subalgebra.

\item For $m\geq 3$ and $k\geq 3$,  ${S}\Nr_n\CA_{n+k+1}$
is not Sahlqvist axiomatizable. In particular, $\RCA_n$ is not Sahlqvist axiomatizable.

\item There exists an atomic representable
$\CA_n$ with no $n+k+1$ smooth complete representation; in particular it has no complete
representation.

\item The omitting types theorem fails for clique guarded semantics, when size of cliques are $< n+k+1$.

\end{enumarab}

\end{corollary}\label{flat}

\begin{proof} We use the cylindric algebra, proved to exist conditionally, namely,
$\A={\C}_{n}$ in theorem \ref{blurs}.
The term algebra, which is contained in ${\C}_n$ also can be used.

\begin{enumarab}

\item $\A$ and $\Cm\At\A$ are such.

\item $\Cm\At\A$ is the \d\ completion of $\A$ (even in the $\PA$ and $\Sc$ cases, because $\A$,
hence its $\Sc$ and $\PA$ reducts are completely additive), hence $S\Nr_n\CA_{n+k+1}$ is not atom canonical \cite[proposition 2,88,
theorem, 2.96]{HHbook}.

\item $\A$ is dense in $\Cm\At\A$.

\item Completely additive varieties defined by Sahlqvist
equations are closed under \d\ completions \cite[theorem 2.96]{HHbook}.

\item This part is like the proof of item (3) in theorem \ref{can2}.
Assume that $\A$ has an $m=n+k+1$ smooth complete representation $\M$. $L(A)$ denotes the signature that contains
an $n$ ary predicate for every $a\in A$.
For $\phi\in L(A)_{\omega,\infty}^n$,
let $\phi^{M}=\{\bar{a}\in C^m(M):M\models_C \phi(\bar{a})\}$,
and let $\D$ be the algebra with universe $\{\phi^{M}: \phi\in L_{\infty,\omega}^m\}$ with usual
Boolean operations, cylindrifiers and diagonal elements, \cite[theorem 13.20]{HHbook}. The polyadic operations are defined
by swapping variables.
Define $\D_0$ be the algebra consisting of those $\phi^{M}$ where $\phi$ comes from $L_m$, that is it contains $m$ variables.
Assume that $M$ is $n$ square, then certainly $\D_0$ is a subalgebra of the $\sf Crs_m$ (the class
of algebras whose units are arbitrary sets of $n$ ary sequences)
with domain $\wp(C^m(M))$ so $\D_0\in {\sf Crs_m}$. The unit $C^m(M)$ of $\D_0$ is symmetric,
closed under substitutions, so
$\D_0\in \sf G_m$ (these are relativized set algebras whose units are locally cube, they
are closed under substitutions.)  Since $M$ is $m$ flat we
have that cylindrifiers commute by definition,
hence $\D_0\in \CA_m$.

Now suppose that  $M$ is $m$ smooth then it is infinitary $m$ flat.
Then one proves that  $\D\in \CA_m$ in exactly the same way.
Clearly $\D$ is complete. We claim that $\D$ is atomic.
Let $\phi^M$ be a non zero element.
Choose $\bar{a}\in \phi^M$, and consider the infinitary conjunction
$\tau=\bigwedge \{\psi\in L_{\infty}: M\models_C \psi(\bar{a})\}.$
Then $\tau\in L_{\infty}$, and $\tau^{M}$ is an atom, as required

Now defined the neat embedding by $\theta(r)=r(\bar{x})^{M}$.
Preservation of operations is straightforward.  We show that $\theta$ is injective.
Let $r\in A$ be non-zero. But $M$ is a relativized representation, so there $\bar{a}\in M$
with $r(\bar{a})$ hence $\bar{a}$ is a clique in $M$,
and so $M\models r(\bar{x})(\bar{a})$, and $\bar{a}\in \theta(r)$ proving the required.

We check that it is a complete embedding under the assumption that
$M$ is a complete relativized representation.
Recall that $\A$ is atomic. Let $\phi\in L_{\infty}$ be such that $\phi^M\neq 0$.
Let $\bar{a}\in \phi^M$. Since
$M$ is complete and $\bar{a}\in C^m(M$) there is $\alpha\in \At\A$, such
that $M\models \alpha(\bar{a})$, then $\theta(\alpha).\phi^C\neq 0.$
and we are done.
Now $\A\in S_c\Nr_n\CA_{n+k}$; it embeds completely into $\Nr_n\D$, $\D$ is complete, then
so is $\Nr_n\D$, and consequently $\Cm\At\A\subseteq \Nr_n\D$, which is impossible, because we know that
$\Cm\At\A\notin S\Nr_n\CA_{n+k}.$

\item By theorem \ref{OTT}
\end{enumarab}
\end{proof}
In theorem \ref{main} such relation algebras for $k=1$ will be provided.

\subsection{Omitting types in other contexts}

Here we prove a result mentioned in \cite{ANT} without proof,
namely, that the omitting types theorem fails for any finite first order 
definable extension of $L_n$, first order logic restricted to the first $n$ variables,
when $n>2$. We add that our result even extends to stronger logics like ones endowed with operations of transitive closure as well \cite{Maddux}.
We recall what we mean by first order definable algebraic operations.
Such operations, as the name suggests, are built using first order formulas, one for each formula.
Later, we will show that in the corresponding logic they correspond to newly added
connectives defined by such formulas, and the arity of connectives is determined
by the number of {\it relation symbols} viewed as formula schemes in the defining formula.

\begin{definition}
Let $\Lambda$ be a first order language with countably many
relation symbols, $R_0, \ldots, R_i,\ldots : i\in \omega,$
each of arity $n$.
Let $\Cs_{n,t}$ denote the following class of similarity type $t$:
\begin{enumroman}
\item $t$ is an expansion of $t_{\CA_n}.$
\item  $S\Rd_{ca}\Cs_{n,t}=\Cs_n.$ In particular, every algebra in $\Cs_{n,t}$ is a Boolean
field of sets with unit $^nU$ say,
that is closed under cylindrifiers and contains diagonal elements.
\item For any $m$-ary operation $f$ of $t$, there exists a first order formula
$\phi$ with free variables among $\{x_0,\ldots, x_n\}$
and having exactly $m,$ $n$-ary relation symbols
$R_0, \ldots, R_{m-1}$ such that,
for any set algebra ${\A}\in \Cs_{n,t}$
with base $U$, and $X_0, \ldots X_{m-1}\in {\A}$,
we have:
$$\langle a_0,\ldots, a_{n-1}\rangle\in f(X_0,\ldots, X_{m-1})$$
if and only if
$${M}=\langle U, X_0,\ldots, X_{n-1}\rangle\models \phi[a_0,\ldots, a_{n-1}].$$
Here $M$ is the first order structure in which for each $i\leq m$,
$R_i$ is interpreted as $X_i,$ and $\models$ is the usual satisfiability relation.
Note that cylindrifiers and diagonal elements are so definable.
Indeed, for $i,j<n$,  $\exists x_iR_0(x_0\ldots, x_{n-1})$
defines $\sf c_i$ and $x_i=x_j$ defines $\sf d_{ij}.$
\item With $f$ and $\phi$ as above,
$f$ is said to be a first order definable operation with $\phi$ defining $f$,
or simply a first order definable
operation, and $\Cs_{n,t}$ is said to be a first order definable
expansion of $\Cs_n.$
\item $\RCA_{n,t}$ denotes the class $SP\Cs_{n,t}$, i.e. the class of all subdirect products
of algebras
in $\Cs_{n,t}.$ We also refer to
$\RCA_{n,t}$ as a first order definable expansion of $\RCA_n.$
\end{enumroman}
\end{definition}

From now on, fix a {\it finite} $t$ (that is an expansion of $t_{\CA_n})$
and fix a first order language $\Lambda$
with countably many relation symbols each of
arity $n$.
For $\omega\geq m>0$ , let $\Fm_r^{\Lambda_{n+m}}$,
or $\Fm_r^{n+m}$, denote the set of {\it restricted} formulas
built up of $n+m$  variables.
Here restricted means that the
variables occurring in atomic subformulas, appear in their natural order. 
There is a one to one correspondence
between restricted formulas and $\CA$ terms \cite[\S 4.3]{tarski}.
We now turn to defining certain abstract algebras based on neat reducts allowing more operations 
defined via the spare dimensions. Such operations correspond precisely to the first order definable operations.

\begin{definition}
\begin{enumroman}
\item Let $\omega\geq m>0.$ For $\A\in \CA_{n+m},$
$\Nr_n{\A}$ denotes the neat $n$-reduct of $\A$.
Let $t$ be a finite expansion of $t_{\CA_n}$
such that $\Cs_{n,t}$ is a first order definable expansion
of $\Cs_n.$
Assume further, that for any $f\in t$ there is a formula $\phi(f)$
defining $f$ and having {\it free} variables among the first $n$. Fix such $\phi(f).$
We let ${\Ra}\A$
denote the following algebra of type $t$ which is an expansion of ${\Nr}_n{\A}$:
$${\Ra}{\A}=
\langle {\Nr}_n{\A}, \tau^{\A}(\phi(f))\rangle_{f\in t\smallsetminus t_{\CA_n}}.$$

\item For $m\geq 0$,  $\K_{n,m}$ denotes the following class of type $t$:
$$\K_{n,m}=S\{{\Ra}{\A}: {\A}\in \CA_{n+m}\}.$$
\end{enumroman}
\end{definition}

In the definition of ${\Ra}\A$, the extra operations in $t$
are defined by forming the term corresponding to the beforehand
fixed formula $\phi(f)$ defining (the first order definable operation) $f$.
Note that if
$\A$ is representable then the choice of the formula
defining $f$ is immaterial, any two such equivalent formulas give
the same thing when interpreted in the algebra.

In principle, this might not be the case when $\A$
is not representable and finite dimensional. That is assume that $\A\in \Nr_n\CA_{n+k}$.
Assume that $f\in t$
is definable in set algebras via $\phi$, say, and $\psi$ is equivalent to $\phi$,  and both formulas useless that $n+k$
variables, then
$\tau^{\A}(\psi)$ may not be equal to $\tau^{\A}(\phi.)$
(A counterexample, though is not easy, but the idea is that, though the two formulas use $< n+k$ variables,
the proof of their {\it equivalence} may  need more $> n+k$ variables.)

Note too, that because the free variables
occurring in every $\phi(f)$, are
among the first $n$,
we do have
$\tau^{\cal A}(\phi(f))$ is in $\Nr_n\A, $ so that in any event the definition of ${\Ra}{\A}$
is indeed sound.

\begin{definition}
Let $t$ be an expansion of $t_{\CA_n}$ such that $\Cs_{n,t}$ is a first order definable
expansion of $\Cs_n$. An algebra $\A$ of type $t$
is said to have the neat embedding property, if
$\A$ is a subalgebra of ${\Ra}\B$
for some ${\B}\in \CA_{\omega}.$
\end{definition}

Note that in principle $\B$ may not be unique. One can
construct $\B$ and ${\C}\in  \CA_{\omega}$ such that
$\Nr_3{\C}\cong \Nr_3{\B}$ but $\B$ and $\C$ are not isomorphic.

Actually this is quite easy to do, for any $n\geq 3$. Take $\A=\Nr_n\B$, where $\B$ is not locally finite.
Take $\C=\Sg^{\B}\A$, then $\C$ is locally finite, $\A=\Nr_n\C$ and $\B$ and $\C$ are not isomorphic.
This is the argument used to show
that $\Nr_n\CA_{\omega}=\Nr_n{\sf Lf}_{\omega}$.

In any case, what really counts at the end is the existence of at least
one such algebra
$\B$.
Note too that for such an algebra $\A$, its cylindric reduct is a cylindric algebra.

The following is proved in \cite{basim} reproving
Biro's result using only cylindric algebras, without reference to relation algebras.

\begin{theorem}
\begin{enumarab}
\item Let $t$ be an expansion of $t_{\CA_n}$ such that $\Cs_{n,t}$ is a first order definable
expansion of $\Cs_n$.
$\A$ of type $t$  is representable if and only if $\A$
has the neat embedding property.
\item Let $3\leq n\leq m\leq \omega.$
If $\K_{n,m}\supseteq \RCA_{n,t},$ then $m=\omega.$
In other words, $\K_{n,m}$ is properly contained in $\RCA_{n,t}$
when $m$ is finite.
\end{enumarab}
\end{theorem}

The condition on finiteness of models often provides us (in a certain sense) with more complex theories,
but the condition of finiteness for variable sets usually make theories simpler
(for example, algorithmically) nevertheless some desirable properties can be lost. Our next result is of such kind.
Letting $L_n$ denote first order logic with $n$ variables, we have:

\begin{theorem} For $n>2$, there is no finite first order definable expansion of $L_n$ that is both sound and complete
\end{theorem}
\begin{proof}\cite{basim} This was proved using Monk's algebras, 
so languages contained infinitely many countable relation symbols.
\end{proof}

Now using Maddux's algebras one can show that the above theorem holds
if we only have one binary relation symbol; the algebra $\A$ constructed in \cite{ANT} gives the same result but for the omitting types theorem.

\begin{theorem}\label{firstorderdefinable} The same result holds when we have one binary relation 
symbol; furthermore, the omitting types theorem fails.
\end{theorem}
\begin{proof}

Here we follow closely the formalism for relation algebras presented in \cite{Biro} which is that introduced in the Tarski Givant monograph; however,
we address $L_n$. Let $n$ be finite $n\geq 3$. We have $\L_n$ is defined in the following signature.
$E$ a binary symbol and $=$ are formulas.
If $\phi, \psi$ are formulas then so are $\phi\land \psi$, $\exists \phi$, $(\phi\land \psi)$ and
$\neg \phi$.

A model is a pair $\M=(M ,E)$
such that $E\subseteq {}M\times M$. Satisfiability is defined by recursion, the clause that
deserves attention is the cylindrifier clause which is defined, the usual way, 
like first order logic.

We define an auxiliary language $\L^*$ which countably many
relation symbols $R_0, R_1,\ldots, R_m\ldots $ in addition to $E$, each of arity $n$.

The language $\L^{f}$ is a first order expansion of $\L_n$
if the following hold:

\begin{enumarab}
\item It is obtained from $\L_n$ by adding finitely many new symbols
each with a non-negative rank, with formation rules as above, namely, if $c$ is a
new $m$ ary symbol and $\phi_0,\ldots, \phi_{m-1}$
are formulas in $\L^{f}$, then $c(\phi_0,\ldots, \phi_{m-1})$ is a formula.

\item The class of models of the expanded language is the same.

\item With every symbol $c$ of $\L^{f}$ we associate
a formula $\phi_c$ of $L^{*}$ such that
$\phi_c$ involves only $E, R_0,\ldots, R_{m-1}$ where
$m$ is the rank of $c$ and the free variables of $\phi_c$
are among $v_0, \ldots,  v_{m-1}$.
For $c\in \L^{f}$, the formula $\phi_c$ is defined as expected, example,
$\phi_E=: E(v_0v_1)$, $\phi_{=}=: v_0=v_1$ and
$\phi_{\sf c_k}=:\exists v_k(R_0(v_0v_2v_3)).$

\item We define
$$M\models c(\phi_0,\ldots \phi_{m-1})[\bar{s}]\Longleftrightarrow
(U, \phi^M\ldots, \phi_{m-1}^M)\models \phi_c[\bar{s}],$$
for every symbols different from
$E$ and $=$ and every $s\in {}^nM$.

\end{enumarab}

To deal with the proof theory of such a formalism one adjoins to $\L^f$ new meta variables symbols 
$X_1, X_2, \dots, $
called formula variables, thereby forming a new language $\L_{sh}$ (here these are sometimes called formula variables).
Indeed assume that $\Sigma$ is sound and complete in $\L^f$.
Now formulas are formed the same way treating the new formula
variables like $=$ and $E$.  Formulas are called {\it schemas}. An instance
of a schema $\sigma(X_1,\ldots, X_n)$ is obtained from $\sigma$ by
replacing  all occurrences of the formula variables
$X_1, \ldots, X_n$ in $\sigma$
by formulas in $\L^f$.

Let $C$ be the finite set of symbols different from $E$ and $=$.
Let $t'$ be a similarity type that has an operation symbol $H_c$ for any $c\in C$, and whose arity
is the rank $c$. Since $C$ is finite we may assume that
$t'$ is a finite expansion of $\CA_n$.
Let the set of variable symbols of
that algebraic language of $t'$ be $\{w_i: i\in \omega\}$.

For any schema $\theta$
define $\tau(\theta)$ by recursion, example,
$\tau(=)={\sf d}_{01}$, $\tau(E)=w_0$, $\tau(X_i)=w_{i+1}$
$\tau(c(\phi_0, \ldots, \phi_{m-1}))=H_c(\tau(\phi_0), \ldots, \tau(\phi_{m-1})).$

The operations on set algebra $\A\in {\sf Cs}_3$
are interpreted by;
for any $c\in C$ $\R_0,\ldots, \R_{m-1}\in \A$, set
$$H_c(\bar{\R})^{\A}=\{s\in ^nU:  (U, \R_0, \ldots, \R_{m-1})\models \phi_c[s]\}.$$

Now define a term $\tau(\theta)$ of the language $t'$
for each formula schema of the expanded language, exactly as before
Let $\Delta=\{\tau(\chi)=1: \chi\in \Sigma\},$
Then $\Delta$, hence $\Sigma,$ cannot be finite.

Now we show that the omitting types theorem fails for $\L_f$.

Let $\A\in \sf RCA_n\cap \Nr_n\CA_{n+k}$ be an algebra that is atomic but not completely representable. Exists.
We translate this algebra to an $\L$ theory, with a non principal type that cannot be
omitted. We have $\A=\Fm_T$, where $T$ is a countable $\L_n$ theory.

Expand the language of $T$ by including syntactically $\phi_c$ as a new connective, to obtain $\L^{f}$.
Let $\Gamma$ be the non principal type of co-atoms of $\A$; that is is $\Gamma=\{\neg\phi: \phi_T\in \At\A\}$.
Assume, for contradiction, that $\M=(M, E)$ is a model omitting $\Gamma$, in the expanded language, where $\phi_c$
is interpreted semantically, as defined above, by
$$\M\models {\phi}_c(\psi_1,\ldots \psi_{m-1})[s]\Longleftrightarrow (\M, \psi_1^{\M}, \ldots \psi_{m-1}^{\M})\models \phi_c[s].$$
Let $\B$ be the corresponding set algebra, with the semantics for $\phi_c$ defined as above, so that
we have
$$[\phi_c(R_1, \ldots, R_{m-1})]^{\B}=\{s\in {}^nM: (M, R_0, \ldots,  R_{n-1})\models \phi_c[s]\}.$$
Then $\B$ is closed under the operations (this can be proved by an easy induction exactly as above).

But the the reduct of $\B$ to the language of $\CA_n$ gives a complete representation of $\A$
which is impossible.

\end{proof}

When we have elimination of quantifiers and everything countable, things turn up drastically different.
Non-principal types {\it can be} omitted in {\it classical countable models}. 
Not only that, but if the non-principal types approached are {\it maximal}
and there are possibly {\it uncountably many of them} then, in case we have elimination of 
quantifiers, they can all be omitted simultaneously in a classical countable model, too.

We recall a tremendously 
deep result of Shelah, that we use to prove \cite[theorem 3.2.9]{Sayed}.

\begin{lemma} Suppose that $T$ is a theory,
$|T|=\lambda$, $\lambda$ regular, then there exist models $\M_i: i<{}^{\lambda}2$, each of cardinality $\lambda$,
such that if $i(1)\neq i(2)< \chi$, $\bar{a}_{i(l)}\in M_{i(l)}$, $l=1,2,$, $\tp(\bar{a}_{l(1)})=\tp(\bar{a}_{l(2)})$,
then there are $p_i\subseteq \tp(\bar{a}_{l(i)}),$ $|p_i|<\lambda$ and $p_i\vdash \tp(\bar{a}_ {l(i)})$ ($\tp(\bar{a})$ denotes the complete type realized by
the tuple $\bar{a}$).
\end{lemma}
\begin{proof} \cite{Shelah} Theorem 5.16, Chapter IV.
\end{proof}
\begin{corollary} For any countable theory, there is a family of $< {}^{\omega}2$ countable models that overlap only on principal types
\end{corollary}
We give a sketch of Shelah's ideas:

\begin{theorem}(Shelah) If $T$ is a countable theory and we have $<{}^{\omega}2$
many non principal maximal types, then they can be omitted in a countable model
\end{theorem}
\begin{proof} The idea is that one can build several models such that they overlap only on isolated types.
One can build {\it two} models so that every maximal type which is realized in both is isolated. Using the jargon of Robinson's finite forcing
implemented via games, the idea
is that one distributes this job of building the two models among experts, each has a role to play, and that all have
winning strategies. There is no difficulty in stretching the above idea to make the experts build three, four or any finite number of models
which overlap only at principal types.
With a pinch of diagonalisation we can extend the number to $\omega$.

To push it still further to $^{\omega}2$ needs uses ideas of Shelah, here it is a typical 'binary tree' construction, where at each root
reached at a finite stage
is split into two other models,  we end up with continuum many, these are not necessarily pairwise non isomorphic,
we mean here by a type overlapping in two distinct models with respect to an indexing set that has the power of the continuum,
but the map from this set to the class of models obtained is not necessarily injective.
(Otherwise, we would obtain that any countable theory has continuum many models,
which of course is entirely absurd.)

Now assume not.  Let ${\bold F}$ be the given set of non principal ultrafilters. Then for all $i<{}^{\omega}2$,
there exists $F$ such that $F$ is realized in $\B_i$. Let $\psi:{}^{\omega}2\to \wp(\bold F)$, be defined by
$\psi(i)=\{F: F \text { is realized in  }\B_i\}$.  Then for all $i<{}^{\omega}2$, $\psi(i)\neq \emptyset$.
Furthermore, for $i\neq j$, $\psi(i)\cap \psi(j)=\emptyset,$ for if $F\in \psi(i)\cap \psi(j)$ then it will be realized in
$\B_i$ and $\B_j$, and so it will be principal.  This implies that $|\bold F|={}^{\omega}2$ which is impossible.
\end{proof}

The following provides a proof  of a result actually stronger than that stated in \cite{Sayed} without a proof, namely, theorem
3.2.9, since it addresses a strictly larger class than $\Nr_n\A_{\omega}$ addressed in 
\cite[theorem 3.2.9]{Sayed}, namely, the class $S_c\Nr_n\CA_{\omega}.$
We will give many examples in theorem \ref{maintheorem} below
showing that indeed the inclusion $\Nr_n\CA_{m}\subseteq S_c\Nr_n\CA_{m}$ for all $m>n$, is proper,
This is not trivial, for example its relation algebra analogue for $m<\omega$ is an open problem, see \cite{r}. 

\begin{theorem}\label{uncountable} Let $\A=S_c\Nr_n\CA_{\omega}$. Assume 
that $|A|=\lambda$, where $\lambda$ is an uncountable
cardinal. Let  $\kappa< {}^{\lambda}2$,
and $(F_i: i<\kappa)$ be a system of non principal ultrafilters.
Then there exists
a set algebra $\C$ with base $U$ such that $|U|\leq \lambda$, $f:\A\to \C$ such that $f(a)\neq 0$ and for all $i\in \kappa$, $\bigcap_{x\in X_i} f(x)=0.$

\end{theorem}

\begin{proof} Assume that  $\A\subseteq_c \Nr_n\B$, where $\B$ is $\omega$ dimensional, locally finite and has the same cardinality as $\A$.
This is possible by taking $\B$ to be the subalgebra of which $\A$ is a strong neat reduct generated by  $A$, and noting that we gave countably many
operations. Then $\A=\Fm_T$, for some countable theory $T.$
Let $\Gamma_i=\{\phi/T: \phi\in F_i\}$. Then the $\Gamma_i$s  are maximal non principal types of $T$.
The model $\sf M$ given by the previous proof gives the desired
representation. 

\end{proof}

Now we give two metalogical readings of the last two theorems. The first is given in \cite[theorem 3.2.10]{Sayed}
but we include it, because it prepares for the other one, which is an entirely new omitting types theorems for cylindric
algebras of sentences. Cylindrifiers in such algebras can be defined because we include individual constants; the number of
these determines the dimension of the algebra in
question, this interpretation was given in  \cite{Amer}; in the context of representing
algebras of sentences as (full) neat reducts.

\begin{theorem}\label{Shelah1} Let $T$ be an $\L_n$ consistent theory that admits elimination of quantifiers.
Assume that $|T|=\lambda$ is a regular cardinal.
Let $\kappa<2^{\lambda}$. Let $(\Gamma_i:i\in \kappa)$ be a set of non-principal maximal types in $T$. Then there is a model $\M$ of $T$ that omits
all the $\Gamma_i$'s
\end{theorem}
\begin{demo}{Proof} If $\A=\Fm_T$ denotes the cylindric algebra corresponding to $T$, then since $T$ admits elimination of quantifiers, then
$\A\in \Nr_n\CA_{\omega}$. This follows from the following reasoning. Let $\B=\Fm_{T_{\omega}}$ be the locally finite cylindric algebra
based on $T$ but now allowing $\omega$ many variables. Consider the map $\phi/T\mapsto \phi/T_{\omega}$.
Then this map is from $\A$ into $\Nr_n\B$. But since $T$ admits elimination of quantifiers the map is onto.
The Theorem now follows.
\end{demo}

We now give another natural omitting types theorem for certain uncountable languages.
Let $L$ be an ordinary first order language with
a list $\langle c_k\rangle$ of individual constants
of order type $\alpha$. $L$ has no operation symbols, but as usual, the
list of variables is of order type $\omega$. Denote by $Sn^{L_{\alpha}}$ the set
of all $L$ sentences, the subscript $\alpha$ indicating that we have $\alpha$ many constants Let $\alpha=n\in \omega$.
Let $T\subseteq Sn^{L_0}$ be consistent. Let $\M$  be an $\L_0$  model of $T$.
Then any $s:n\to M$ defines an expansion of $\M$ to $L_n$ which we denote by $\M{[s]}$.
For $\phi\in L_n$ let $\phi^{\M}=\{s\in M^n: \M[s]\models \phi\}$. Let
$\Gamma\subseteq Sn^{L_{n}}$.
The question we address is: Is there a model $\M$ of $T$ such that for no expansion $s:n\to M$ we have
$s\in  \bigcap_{\phi\in \Gamma}\phi^{\M}$.
Such an $\M$ omits $\Gamma$. Call $\Gamma$ principal over $T$ if there exists $\psi\in L_n$ consistent with $T$ such that
$T\models \psi\to \Gamma.$
Otherwise, $\Gamma$ is not principal over $T$.

\begin{theorem}\label{Shelah2}  Let $T\subseteq Sn^{L_0}$ be consistent and assume that $\lambda$ is a regular cardinal, and $|T|=\lambda$.
Let $\kappa<2^{\lambda}$. Let $(\Gamma_i:i\in \kappa)$ be a set of non-principal maximal types in $T$.
Then there is a model $\M$ of $T$ that omits
all the $\Gamma_i$'s
That is, there exists a model $\M\models T$ such that there is no $s:n\to \M$
such that   $s\in \bigcap_{\phi\in \Gamma_i}\phi^{\M}$.
\end{theorem}
\begin{demo}{Proof}
Let $T\subseteq Sn^{L_0}$ be consistent. Let $\M$ be an $\L_0$  model of $T$.
For $\phi\in Sn^L$ and $k<\alpha$
let $\exists_k\phi:=\exists x\phi(c_k|x)$ where $x$ is the first variable
not occurring in $\phi$. Here $\phi(c_k|x)$ is the formula obtained from $\phi$ by
replacing all occurrences of $c_k$, if any, in $\phi$ by $x$.
Let $T$ be as indicated above, i.e, $T$ is a set of sentences in which no constants occur. Define the
equivalence relation $\equiv_{T}$ on $Sn^L$ as follows
$$\phi\equiv_{T}\psi \text { iff } T\models \phi\equiv \psi.$$
Then, as easily checked $\equiv_{T}$ is a
congruence relation on the algebra
$$\Sn=\langle Sn,\land,\lor,\neg,T,F,\exists_k, c_k=c_l\rangle_{k,l<n}$$
We let $\Sn^L/T$ denote the quotient algebra.
In this case, it is easy to see that $\Sn^L/T$ is a $\CA_n$, in fact is an $\RCA_n$.
Let $L$ be as described above. But now we denote it $L_n$,
the subscript $n$ indicating that we have $n$-many
individual constants. Now enrich $L_{n}$
with countably many constants (and nothing else) obtaining
$L_{\omega}$.
Recall that both languages, now, have a list of $\omega$ variables.
For $\kappa\in \{n, \omega\}$
let $\A_{\kappa}=\Sn^{L_{k}}/{T}$.
For $\phi\in Sn^{L_n}$, let $f(\phi/T)=\phi/{T}$.
Then, as easily checked, $f$ is an embedding  of $\A_{n}$
into $\A_{\omega}$. Moreover $f$ has the additional property that
it maps $\A_{n}$, into (and onto) the neat $n$ reduct of $\A_{\beta}$,
(i.e. the set of $\alpha$ dimensional elements of $A_{\beta}$).
In short, $\A_{n}\cong \Nr_{n}\A_{\omega}$. Now again putting $X_i=\{\phi/T: \phi\in \Gamma_i\}$ and using that the
$\Gamma_i$'s are maximal non isolated, it follows that the
$X_i's$ are non-principal ultrafilters
Since $\Nr_n\CA_{\omega}\subseteq S_c\Nr_n\CA_{\omega}$, then our result follows.
\end{demo}

Now when we drop the condition of maximality things also change dramaticaly but in a different way,
it can be the case that $<2^{\omega}$ many types cannot be omitted. But they can if we stipulate 
strong independent  set theoretic axioms like Martin's axiom. 
In the general case, however, the problem is independent from $ZFC$.

It is stated in \cite[theorem 3.2.8]{Sayed} without proof,
that it is possible that ${\sf covK}$ many non-isolated types cannot be omitted in $L_n$ countable theories
whose Lindenbaum Tarski algebras belong to
$\Nr_n\CA_{\omega}$.  Here ${\sf cov K}$ is the least cardinal $\kappa$ such that the real 
line can be covered by $\kappa$ nowhere dense sets.
This is known for usual first order logic as our argument to be presented clearly manifests.

We show, here that this also holds for $L_n$ $n>2$; this is the best estimate. By accomplishing this,
we  complete the independence proof in \cite{Sayed}.

Let $OTT(\kappa)$, $\kappa$ a cardinal be as in \cite{Sayed}, which says that $\kappa$ many non principal types can be omitted
in $L_n$ theories $T$, $n$ finite $>2$, 
when $\Fm_T\in \Nr_n\CA_{\omega}$.

\begin{theorem}\label{independence} The statement $OTT({\sf covK})$ 
is false
\end{theorem}
\begin{proof}
Let $n>2$. To show that $OTT({\sf covK})$ could be false, we adapt
an example in \cite{CF} p.242. Fix $n\geq 2$. There the example is constructed for $L_{\omega,\omega}$ to adapt
to $L_n$ some care is required.
Let $T$ be a theory such that for this given $n$, in $S^n(T)$, the Stone space of $n$ types,  the isolated points are not dense.
(In \cite{CF}, a theory $T$ is chosen which does not have a prime model. This implies that there is an $n$ such that the isolated types in
$S_n(T)$ are not dense;  {\it here we need a fixed $n$, given in advance, so not {\it any} theory without a
prime model will do, for the number witnessing its primeness could be greater than $n$}.).
It is easy to construct such theories, for any fixed $n$.
In \cite{CF}, a family $P$ is a family of non-principal types $|P|={\sf covK}$ 
was constructed 
for such $T$ that cannot be omitted.

Let $\A=\Fm/T$ and for $p\in P$ let $X_p=\{\phi/T:\phi\in p\}$. Then $X_p\subseteq \Nr_n\A$, and $\prod X_p=0$.
However, for any $0\neq a$, there is no set algebra $\C$ with countable base
$M$ and $g:\A\to \C$ such that $g(a)\neq 0$ and $\bigcap_{x\in X_i}f(x)=\emptyset$.
But in principle, if we take the  neat $n$ reduct, representations preserving meets can exist.
We exclude this possibility by showing that
such representations necessarily lift to
all $\omega$ dimensions.

Let $\B=\Nr_n\A$. Let $a\neq 0$. Assume, seeking a contradiction, that there exists
$f:\B\to \D'$ such that $f(a)\neq 0$ and $\bigcap_{x\in X_i} f(x)=\emptyset$. We can assume that
$B$ generates $\A$ and that $\D'=\Nr_n\D$ where $\D\in \Lf_{\omega}$. Let $g=\Sg^{\A\times \D}f$. We will show
that $g$ is a one to one function with domain $\A$ that preserves the $X_i$'s which is impossible (Note that by definition $g$ is a homomorphism).
We have
$$\dom g=\dom\Sg^{\A\times \D}f=\Sg^{\A}\Nr_{n}\A=\A.$$
By symmetry it is enough to show that $g$ is a function.  We first prove the following (*)
 $$ \text { If } (a,b)\in g\text { and }  {\sf c}_k(a,b)=(a,b)\text { for all } k\in \omega\sim n, \text { then } f(a)=b.$$
Indeed,
$$(a,b)\in \Nr_{n}\Sg^{\A\times \D}f=\Sg^{\Nr_{n}(\A\times \D)}f=\Sg^{\Nr_{n}\A\times \Nr_{n}\D}f=f.$$
Here we are using that $\A\times \D\in \Lf_{\omega}$, so that  $\Nr_{n}\Sg^{\A\times \D}f=\Sg^{\Nr_{n}(\A\times \D)}f.$
Now suppose that $(x,y), (x,z)\in g$.
Let $k\in \omega\sim n.$ Let $\Delta$ denote symmetric difference. Then
$$(0, {\sf c}_k(y\Delta z))=({\sf c}_k0, {\sf c}_k(y\Delta z))={\sf c}_k(0,y\Delta z)={\sf c}_k((x,y)\Delta(x,z))\in g.$$

Also,
$${\sf c}_k(0, {\sf c}_k(y\Delta z))=(0,{\sf c}_k(y\Delta z)).$$
Thus by (*) we have  $$f(0)={\sf c}_k(y\Delta z) \text { for any } k\in \omega\sim n.$$
Hence ${\sf c}_k(y\Delta z)=0$ and so $y=z$.
We conclude that there exists a countable $\B\in \Nr_n\CA_{\omega}$ and $(X_i:i<covK)$ such that $\prod X_i=0$ but there is no representation that
preserves the $X_i$'s. In more detail. Give any $a\in \B$, if $a$ is non zero, $\C$ is a set algebra with countable base
and $f:\B\to \C$ is a homomorphism such
that $f(a)\neq 0$, then there exists $i<\sf covK$, such that $\bigcap_{x\in X_i} f(x)\neq \emptyset.$
\end{proof}

The above result conforms to the accepted view that there are  two typical types of investigations in set theory, proper.
The first type consists of theorems demonstrating the independence of mathematical statements.
This type requires a thorough understanding of mathematics surrounding the
source of problem in question, reducing the ambient mathematical constructions to combinatorial statements about sets,
and finally using some method
(primarily forcing) to show  that such combinatorial statements
are independent.
A second type involves delineating the edges of the independence proofs, giving proof in $ZFC$ of statements
that on first sight would be suspected of being independent.
Proofs of this kind is often extremely subtle and surprising; very similar
statements are independent and it is hard to detect the underlying difference.

This discrepancy was witnessed here in the non-trivial context of `the number of non principal types that can be omitted'
when the maximality condition imposed on the $<{}^{\omega}2$ many 
non-principal types shifts us from the realm of independence to that of
provability from $ZFC$.

\subsection{Relativizing differently}

We have seen that when cylindrifiers commute even {\it locally} we lose the omitting types theorem. Now we show that a 
different relativization obtained by getting rid completely of full fledged commutativity of cylindrfiers 
renders an omitting types theorem for countable theories. 

It is known that the  atomic algebras addressed, in this context, 
are completely representable; this gives a Vaught's theorem (atomic theories have atomic models) 
regardless of cardinalities. But from this alone we cannot infer that the 
omitting types theorem holds, 
its usually  the other way round in the countable case, as indeed in usual first
order logic.

The notion of relativized representations constitute a huge topic in both algebraic and modal logic, see the introduction of 
\cite{1}, \cite{Fer}, \cite{v}.
Historically,  in \cite{tarski} square units got all the attention and relativization  was treated as a side issue.
However, extending original classes of models for logics to manipulate their properties is common. 
This is no mere tactical opportunism, general models just do the right thing.

The famous move from standard models to generalized models is 
Henkin's turning round  second  order logic into an axiomatizable two sorted first
order logic. Such moves are most attractive 
when they get an independent motivation. 

The idea is that we want to find a semantics that gives just the right action 
while additional effects of square set theoretic representations are separated out as negotiable decisions of formulation 
that can threaten completeness, decidability, and interpolation. 
(This comes across very much in cylindric algebras, especially in finite variable fragments of first order logic,
and classical polyadic equality algebras, in the context of Keisler's logic with equality.)

Using relativized representations Ferenczi \cite{Fer}, proved that if we weaken commutativity of cylindrifiers
and allow  relativized representations, then we get a finitely axiomatizable variety of representable 
quasi-polyadic equality algebras (analogous to the Andr\'eka-Resek Thompson ${\sf CA}$ version, cf. \cite{Sayedneat} and \cite{Fer}, 
for a discussion of the Andr\'eka-Resek Thompson breakthrough for cylindric-like algebras); 
even more this can be done without the merry go round identities.

This is in sharp contrast with the long list of  complexity results proved for the commutative case.
Ferenczi's results can be seen as establishing a hitherto fruitful contact between neat embedding 
theorems 
and relativized representations, with enriching 
repercussions and insights for both notions.

We will prove Ferenczi's result, and for that matter the Andr\'eka-Thompson-Resek result, 
using games \cite{games}. 
We therefore 
find it convenient to follow
the notation and terminology in \cite{AT} and \cite{Fer} deviating from our earlier notation only once.
For relations $R,S$, we write  $R|S=\{(a,b): \exists c[(a,c)\in R, (c,b)\in S]\}$, instead of $S\circ R.$
As we shall see this simplifies notation considerably.

We denote the variety specified by the equational axiomatization provided in \cite{AT}, by $\PTA_{\alpha}$, 
short for {\it partial transposition algebras}. Here only substitutions corresponding to replacements are involved. They are term definable
in the abstract algebras, using cylindrifiers and diagonal elemnts, 
and the top elements of the their concrete versions are closed with respect to such substitution 
operations when interpreted concretely.
That is to say, if $V$ is the top element and $s\in V$, and $i, j\in \alpha$, then we require that $s\circ [i|j]\in V$.

On the other hand, $\TA_{\alpha}$, like in \cite{Fer},
denotes the variety of transposition algebras.
Here all finite substitutions are involved.
However, in the abstract class only substitutions corresponding to replacements and transpositions are used
in the equational axiomatization.
But these are adequate in the sense that {\it all finite} substitutions are term definable from them.
If $V$ is the top element of a concrete set algebra of dimension $\alpha$, 
then we now require that if $s\in V$, and $\tau$ is any finite transformation on $\alpha$
then $s\circ \tau\in V$.

We start by recalling from \cite{AT}, 
and \cite{Fer}, respectively,  the equational 
axiomatization of each. 

\begin{definition}[Class $\PTA_{\alpha}$]
An algebra $\A=\langle A, +, \cdot, -, 0, 1, {\sf c}_{i}, {\sf d}_{ij}\rangle_{i,j\in\alpha}$, where $+$, $\cdot$ 
are binary operations, $-$, ${\sf c}_i$
 are unary operations and $0$, $1$, ${\sf d}_{ij}$ 
are constants for every $i,j\in\alpha$, is partial 
transposition algebra
\footnote{We call it partial transposition algebra to illustrate that MGR axiom gives partial transpositions. 
This class is different from the class of partial transposition algebras defined in \cite{ptaf}.} if it satisfies the following identities 
for every $i, j, k\in\alpha$.
\begin{description}
\item[$(C_{0})-(C_{3})$] $\langle A, +, \cdot, -, 0, 1, {\sf c}_i\rangle_{i\in\alpha}$ is a 
Boolean algebra with additive closure operators ${\sf c}_i$ such that the complements of ${\sf c}_i$-closed elements are $i$-closed,
 \item[$(C_4)^*$] ${\sf c}_i{\sf c}_j x\geq {\sf c}_j{\sf c}_i x\cdot {\sf d}_{jk}$ if $k\notin\{i, j\}$,
\item[$(C_5)$]${\sf d}_{ii}=1$,
\item[$(C_6)$]${\sf d}_{ij}={\sf c}_k({\sf d}_{ik}\cdot {\sf d}_{kj})$ if $k\notin\{i, j\}$,
\item[$(C_7)$]${\sf d}_{ij}\cdot{\sf c}_i({\sf d}_{ij}\cdot x)\leq x$ if $i\not= j$,
\item[(MGR)]for every $i,j\in\alpha$, $i\not=j$, let ${\sf s}^i_j x={\sf c}_i({\sf d}_{ij}\cdot x)$, $s^i_i x=x$. Then:
$$ {\sf s}_i^k{\sf s}_j^i{\sf s}_m^j{\sf s}_k^m{\sf c}_kx={\sf s}_m^k{\sf s}_i^m{\sf s}_j^i {\sf s}_k^j{\sf c}_kx\text{ if }k\notin \{i,j,m\}, m\notin \{i,j\}$$.
 \end{description}
\end{definition}
The axiomatization of quasi-polyadic equality algebras \cite[theorem 1]{ST} differs 
for the axiomatization of   $\TA_{\alpha}$ to be introduced next, 
in only one axiom, namely, the axiom
\begin{description}
\item[($F_5$)] ${\sf s}^i_j{\sf c}_k x={\sf c}_k{\sf s}^{i}_{j}x$ if $k\notin\{i,j\}$.
\end{description}

The following axiomatization is due to Ferenczi, abstracting
away from the class ${\sf G}_{\alpha}$ (meaning that the axioms all hold in
$\sf G_{\alpha}$). This
is a soundness condition. Ferenczi proves completeness of these
axioms, which we shall also prove in a minute  using the different technique of resorting
to games.
\begin{definition}[Class $\TA_{\alpha}$]
A transposition equality algebra of dimension $\alpha$ is an algebra
$$\A=\langle A, +, \cdot, -, 0, 1, {\sf c}_i, {\sf s}^{i}_{j}, {\sf s}_{ij}, {\sf d}_{ij}\rangle_{i,j\in\alpha},$$ where ${\sf c}_i, {\sf s}^{i}_{j}, 
{\sf s}_{ij}$ are unary operations, ${\sf d}_{ij}$ are constants, the axioms ($F_0$)-($F_9$) below are valid for every $i,j,k<\alpha$:
\begin{description}
\item[$(Fe_0)$] $\langle A, +, \cdot, -, 0, 1\rangle$ is a Boolean algebra, ${\sf s}^i_i={\sf s}_{ii}={\sf d}_{ii}=Id\upharpoonright A$ 
and ${\sf s}_{ij}={\sf s}_{ji}$,
\item[$(Fe_1)$] $x\leq{\sf c}_i x$,
\item[$(Fe_2)$] ${\sf c}_i(x+y)={\sf c}_i x+{\sf c}_i y$,
\item[$(Fe_3)$] ${\sf s}^{i}_{j}{\sf c}_i x={\sf c}_i x$,
\item[$(Fe_4)$] ${\sf c}_i{\sf s}^{i}_{j}x={\sf s}^{i}_{j}x$, $i\not=j$,
\item[$(Fe_5)^*$] ${\sf s}^{i}_{j}{\sf s}^{k}_{m}x={\sf s}^{k}_{m}{\sf s}^{i}_{j}x$ if $i,j\notin\{k,m\}$,
\item[$(Fe_6)$] ${\sf s}^{i}_{j}$ and ${\sf s}_{ij}$ are Boolean endomorphisms,
\item[$(Fe_7)$] ${\sf s}_{ij}{\sf s}_{ij}x=x$,
\item[$(Fe_8)$] ${\sf s}_{ij}{\sf s}_{ik}x={\sf s}_{jk}{\sf s}_{ij}x$, $i,j,k$ are distinct,
\item[$(Fe_9)$] ${\sf s}_{ij}{\sf s}^{i}_{j}x={\sf s}^{j}_{i}x$,
\item[$(Fe_{10})$] ${\sf s}^{i}_{j}{\sf d}_{ij}=1$,
\item[$(Fe_{11})$] $x\cdot{\sf d}_{ij}\leq{\sf s}^{i}_{j}x$.
\end{description}
For $\A\in \TA_{\alpha}$, its partial transposition reduct is
the structure $$\mathfrak{Rd}_{pt}\A=\langle A, +, \cdot, -, 0, 1, {\sf c}_i, {\sf s}^{i}_{j}, {\sf d}_{ij}\rangle_{i,j\in\alpha}.$$
\end{definition}

\par $(F_5)^*$ (and also $(F_5)^*$ and $(S_5)^*$) is obviously a weakening of $(F_5)$.
Also it is known that $\mathfrak{Rd}_{pt}\A\in \PTA_{\alpha}$, for any $\A\in \TA_{\alpha}$ \cite{ST}.
We consider as known the concept of the substitution operator ${\sf s}_{\tau}$ defined for any 
finite transformation $\tau$ on $\alpha$;
$\sf s_{\tau}$ can be introduced uniquely in $FPEA_{\alpha}$ and in $\TA_{\alpha}$, too. The existence of such an $\sf s_{\tau}$
follows from the proof of \cite[theorem 1(ii)]{ST}, it is easy to check that the proof works by only assuming $(F_5)^*$
instead of $(F_5)$ and (notationally) the composition operator $|$ instead of $\circ$.

Throughout this part we assume that the polyadic-like algebras
occurring here are equipped with the operator $\s_{\tau}$, where $\tau$ is finite.

Furthermore, $\s_{\tau}$ is assumed to have the following properties for arbitrary finite transformations $\tau$ and $\lambda$
and ordinals $i,j<\alpha$  (by \cite[p.542]{ST}):
\begin{enumerate}
\item $\s_{\tau|\lambda}=\s_{\tau}\s_{\lambda}$\footnote{This depends on ($F_7$) and ($F_8$).},
\item${\sf s}^{i}_{j}=\s_{[i|j]}$,
\item $\s_{\tau}{\sf d}_{ij}={\sf d}_{\tau i\tau j}$ (of course only in the class $\TA_{\alpha}$),
\item ${\sf c}_i\s_{\tau}\leq\s_{\tau}{\sf c}_{\tau-1}$, here $\tau$ is finite permutation.
\end{enumerate}
\begin{lemma}\label{lesa}
Let $\alpha$ be an ordinal, $\A\in \TA_{\alpha}$, $a\in\At\A$ and $i,j\in\alpha$. 
Then $$a\leq{\sf d}_{ij}\Longrightarrow\sf s_{[i,j]}a=a.$$
\end{lemma}
\begin{proof}
See \cite[p. 875]{Fer}.
\end{proof}

Now we focus on games and networks. Such games are not atomic, and the edges of 
networks can be labelled by any element of the algebra not just atoms. 
In what folows $n$ is finite $>1$. 
We will formulate the games between \pe\ and \pa.
Here we follow closely Hirsch-Hodkinson's techniques \cite{HHbook} adapted to the present situation.

We have some preparing to do.  Let $\overline{x}$, $\overline{y}$ be $n$-tuples of elements of some set.
We write $x_{i}$ for the $i$th element of $\overline{x}$, for $i<n$, so that $\overline{x}=(x_{0}, \cdots, x_{n-1})$.
For $i<n$, we write $\overline{x}\equiv_{i}\overline{y}$
if $x_{j}=y_{j}$ for all $j<n$ with $j\not=i$. The next two definitions are taken from \cite{HHbook}. A (relativized) network is a finite
approximation to a (relativized) representation.
It can be viewed as a hypergraph, of which only some (not necessarily all) of its hyperedges
are labelled by elements of the algebra.

\begin{definition} \
\begin{itemize}
\item Let $\A\in \PTA_{n}$. A relativized $\A$ pre-network is a pair $N=(N_1, N_2)$
where $N_1$ is a finite set of nodes $N_2:N_1^n\to \A$ is a partial map, such that if $f\in \dom N_2$,
and $i,j<n$ then $f_{f(j)}^i\in \dom N_2$. $N$ is atomic if $\rng N\subseteq \At\A$.
We write $N$ for any of $N, N_1, N_2$ relying on context, we write $\nodes(N)$ for $N_1$ and $\edges(N)$ for $\dom(N_2)$.
$N$ is said to be a network if
\begin{enumerate}
\item for all $\bar{x}\in \edges(N)$, we have $N(\bar{x})\leq \sf d_{ij}$ iff $x_i=x_j$.
\item if $\bar{x}\equiv_i \bar{y}$, then $N(\bar{x})\cdot {\sf c}_i N (\bar{y})\neq 0.$
\end{enumerate}
\item Let $\A\in \TA_{n}$. A relativized $\A$ pre-network is a pair $N=(N_1, N_2)$
where $N_1$ is a finite set of nodes $N_2:N_1^n\to \A$ is a partial map, such that if $f\in \dom N_2$,
and $\tau$ is a finite transformation then $\tau| f\in \dom N_2$. Again $N$ is atomic if $\rng N\subseteq \At\A$.
Also we write $N$ for any of $N, N_1, N_2$ relying on context, we write $\nodes(N)$ for $N_1$ 
and $\edges(N)$ for $\dom(N_2)$.
$N$ is said to be a network if
\begin{enumerate}
\item for all $\bar{x}\in \edges(N)$, we have $N(\bar{x})\leq \sf d_{ij}$ iff $x_i=x_j$,
\item if $\bar{x},\bar{y}\in \edges(N)$ and $\bar{x}\equiv_i \bar{y}$, then $N(\bar{x})\cdot {\sf c}_i N (\bar{y})\neq 0$,
\item $N([i,j]|\bar{x})=\s_{[i,j]}N(\bar{x})$, for all $\bar{x}\in \edges(N)$ and all $i,j<n$.
\end{enumerate}
\end{itemize}
\end{definition}
\begin{definition}Let $\A\in \PTA_n\cup \TA_n$. We define a game denoted by
$G_{\omega}(\A)$ with $\omega$ rounds, in which the players $\forall$ (male) and $\exists$ (female)
build an infinite chain of relativized $\A$ pre-networks
$$\emptyset=N_0\subseteq N_1\subseteq \ldots.$$
In round $t$, $t<\omega$, assume that $N_t$ is the current prenetwork, the players move as follows:

\begin{enumerate}
\item $\forall$ chooses a non-zero element $a\in \A$, $\exists$
must respond with a relativized prenetwork $N_{t+1}\supseteq N_t$ containing an edge $e$ with
$N_{t+1}(e)\leq a$,
\item $\forall$ chooses an edge $\bar{x}$ of $N_t$ and an element $a\in \A$. $\exists$ must respond with a pre-network
$N_{t+1}\supseteq N_t$ such that either $N_{t+1}(\bar{x})\leq a$ or $N_{t+1}(\bar{x})\leq -a$,
\item $\forall$ may choose an edge $\bar{x}$ of $N_t$ an index $i<n$ and $b\in \A$ with $N_t(\bar{x})\leq {\sf c}_ib$.
$\exists$ must respond with a prenetwork $N_{t+1}\supseteq N_t$ such that for some $z\in N_{t+1},$
$N_{t+1}(\bar{x}^i_z)=b$.
\end{enumerate}
$\exists$ wins if each relativized pre-network $N_0,N_1,\ldots$
played during the game is actually a relativized network.
Otherwise, $\forall$ wins. There are no draws.
\end{definition}
\begin{lemma}\label{lem}
Let $\A\in \PTA_{n}$ be atomic. For all $i,j\in n$, $i\neq j$, define $\t_j^ix=\sf d_{ij}\cdot {\sf c}_i x$ and $\t_i^ix=x$.
Then
\begin{enumerate}
\item $(\t_j^i)^{\A}: \At\A\to \At\A$
\item Let $\Omega=\{\t_i^j: i,j\in n\}^*$, where for any set $H$, $H^*$ denotes the free monoid generated by $H$. Let
$$\sigma=\t_{j_1}^{i_1}\ldots \t_{j_n}^{i_n}$$
be a word.
Then define for $a\in A$:
$$\sigma^{\A}(a)=(\t_{j_1}^{i_1})^{\A}((\t_{j_2}^{i_2})^{\A}\ldots (\t_{j_n}^{i_n})^{\A}(a)\ldots ),$$
and $$\hat{\sigma}=[i_1|j_1]| [i_2|j_2]\ldots|[i_n|j_n].$$
Then
$$\A\models \sigma(x)=\tau(x)\text { if } \hat{\sigma}=\hat{\tau}, \sigma,\tau\in \Omega.$$
That is for all $\sigma,\tau\in \Omega$, if $\hat{\sigma}=\hat{\tau}$, then for all $a\in A$, we have
$\sigma^{\A}(a)=\tau^{\A}(a)$.
\end{enumerate}
\end{lemma}
\begin{proof} \cite[lemma]{AT}. The $\textbf{MGR}$, merry go round identities
are used here. We note that Andr\'eka's proof of this lemma is long, but 
using fairly obvious results on semigroups a much shorter proof can be given.
\end{proof}
\begin{lemma}\label{soatom}
Let $\A\in \PTA_n$ be atomic. For all $i,j\in n$, $i\neq j$, let $\t_j^ix$ be as above.
The following hold for all $i,j,k,l\in n$:
\begin{enumerate}
\item $(\t^{i}_{j})^{\A}x\leq{\sf d}_{ij}$ for all $x\in\A$.
\item $(x\leq{\sf d}_{ij}\Rightarrow(\t^{k}_{i})^{\A}x\leq {\sf d}_{ij}\cdot{\sf d}_{ik}\cdot{\sf d}_{jk})$ for all $x\in\A$.
\item $(a\leq{\sf c}_i b\Leftrightarrow{\sf c}_i a={\sf c}_i b)$ for all $a,b\in\At\A$.
\item ${\sf c}_i (\t^i_j)^{\A} x={\sf c}_i x$ for all $x\in\A$.
\end{enumerate}
\end{lemma}
\begin{proof}
\begin{description}
\item[$(i)$] Follows directly from the definition of $(\t^{i}_{j})^{\A}$.
\item[$(ii)$]First we need to check the following for all $i,j,k<n$:
\begin{enumerate}
\item ${\sf c}_k{\sf d}_{ij}={\sf d}_{ij}$ if $k\notin\{i,j\}$. For, see \cite[Theorem 1.3.3]{tarski}, the proof doesn't involve $(C_4)$.
\item ${\sf d}_{ij}={\sf d}_{ji}$. For, see \cite[Theorem 1.3.1]{tarski}, the proof works, indeed it doesn't depend on $(C_4)$.
\item ${\sf d}_{ij}\cdot{\sf d}_{jk}={\sf d}_{ij}\cdot{\sf d}_{ik}$. 
A proof for such can be founded in \cite[Theorem 1.3.7]{tarski}.
\end{enumerate} Now we can proof ($ii$):
\begin{eqnarray}
\nonumber\t^{k}_{i}x&=&{\sf d}_{ik}\cdot{\sf c}_k x\\
\nonumber&\leq&{\sf d}_{ik}\cdot{\sf c}_k{\sf d}_{ij}\\
&=&{\sf d}_{ik}\cdot{\sf d}_{ij}\\
\nonumber&=&\sf d_{ki}\cdot{\sf d}_{ij}\\
&=&{\sf d}_{ki}\cdot{\sf d}_{kj}
\end{eqnarray}
From $(1)$, $(2)$ the required follows.
\item[$(iii)-(iv)$] See \cite[p. 675-676]{AT}
\end{description}
\end{proof}
\begin{definition}[Partial transposition network]
Let $\A$ be an atomic $\PTA_n$ and fix an atom $a\in \At\A$. 
Let $\bar{x}$ be any $n$-tuple (of nodes) such that $x_i=x_j$ if and only if $a\leq {\sf d}_{ij}$ for all $i,j<n$. 
Let $NSQ_{\bar{x}}=\{\bar{y}\in{^{n}\{x_0, x_1, \cdots , x_{n-1}\}}:|\rng(\bar{y})|<n\}$ 
($NS$ stands for non-surjective sequences). 
We define the partial transposition network ${PT_{\bar{x}}^{(a)}:NSQ_{\bar{x}}\rightarrow \At\A}$ as follows: 
If $\bar{y}\in NSQ_{\bar{x}}$, then $\bar{y}=\tr|\bar{x}$, for 
some  $\si$, $\sj$ $<n$. Let $PT_{\bar{x}}^{(a)}(\bar{y})=\tt a$. This is well defined by Lemma \ref{lem}
\end{definition}

\begin{definition}[Transposition network]
Let $\A$ be an atomic $\TA_n$ and fix an atom $a\in \At\A$. Let $\bar{x}$ be any $n$-tuple of nodes such that $x_i=x_j$
if and only if $a\leq {\sf d}_{ij}$ for all $i,j<n$.
Let $Q_{\bar{x}}={^{n}\{x_0, x_1, \cdots , x_{n-1}\}}$. Consider the following equivalence relation $\sim$ on $Q_{\bar{x}}$:
\begin{equation*}
\bar{y}\sim\bar{z} \text{ if and only if } \bar{z}=\tau|\bar{y} \text{ for some finite permutation }\tau,
\end{equation*}
$\bar{y},\bar{z}\in Q_{\bar{x}}$.\\
Let us choose and fix representative tuples for the equivalence classes concerning $\sim$
such that each representative tuple is of the form $\tr|\bar{x}$ for some $k\geq 0$, $\si, \sj< n$.
Such representative tuples exist. Indeed, for every $\bar{y}\in Q_{\bar{x}}$, $\exists\tau$ finite permutation,
$\exists k\geq 0$, $\si$, $\sj<n$ such that $$\bar{y}=\tau|\tr|\bar{x}.$$
Let $\bar{z}=\tau^{-1}|\bar{y}$, then $\bar{z}\sim\bar{y}$ 
and $\bar{z}=\tr|\bar{x}$. Let $\sf{Rt}$ denote this fixed set of representative tuples. 
We define the transposition network ${T_{\bar{x}}^{(a)}:Q_{\bar{x}}\rightarrow \At\A}$ as follows:
\begin{itemize}
\item If $\bar{y}\in\sf{Rt}$, then $\bar{y}=\tr|\bar{x}$, for some  
$\si$, $\sj$ $<n$. Let $T_{\bar{x}}^{(a)}(\bar{y})=\ttr a$. This is well defined by Lemma \ref{lem}.
\item If $\bar{z}=\sigma|\bar{y}$ for some finite permutation $\sigma$ 
and some $\bar{y}\in\sf{Rt}$, then let $T_{\bar{x}}^{(a)}(\bar{z})= \s_{\sigma}T_{\bar{x}}^{(a)}(\bar{y})$.
\end{itemize}
\end{definition}
\begin{lemma}
The above definition is unique.
\end{lemma}
\begin{proof}
The first part is well defined by Lemma \ref{lem}.
Now we need to prove that if $\sigma|\bar{y}=\tau|\bar{y}$ for some finite permutations $\sigma, \tau$ 
and some $\bar{y}\in\sf{Rt}$, then $\s_{\sigma}T^{(a)}_{\bar{x}}(\bar{y})=\s_{\tau}T^{(a)}_{\bar{x}}(\bar{y})$. First, we need the following:
\begin{athm}{Claim}
If $\bar{y}=\tau|\bar{y}$ for some finite permutation $\tau$ and some 
$\bar{y}\in\sf{Rt}$, then $$T^{(a)}_{\bar{x}}(\bar{y})=\s_{\tau}T^{(a)}_{\bar{x}}(\bar{y}).$$
\end{athm}
\begin{proof}
It suffices to show that if $\bar{y}=[i,j]|\bar{y}$ for some $i, j<n$ and some 
$\bar{y}\in\sf{Rt}$, then $T^{(a)}_{\bar{x}}(\bar{y})=\s_{[i,j]}T^{(a)}_{\bar{x}}(\bar{y})$. 
For, suppose that $\bar{y}=[i,j]|\bar{y}$ for some $i, j<n$ and some $\bar{y}\in\sf{Rt}$, 
then $y_i=y_j$ and then $T^{(a)}_{\bar{x}}(\bar{y})\leq\d_{ij}$ by Lemma \ref{soatom}. 
Hence by Lemma \ref{lesa}, $T^{(a)}_{\bar{x}}(\bar{y})=\s_{[i,j]}T^{(a)}_{\bar{x}}(\bar{y})$.
\end{proof}
Returning to our prove, assume that $\sigma|\bar{y}=\tau|\bar{y}$ 
for some finite permutations $\sigma, \tau$ and some $\bar{y}\in\sf{Rt}$. 
Then $(\tau^{-1}|\sigma)|\bar{y}=\bar{y}$ and so $\s_{(\tau^{-1}|\sigma)}T^{(a)}_{\bar{x}}(\bar{y})=T^{(a)}_{\bar{x}}(\bar{y})$. 
Therefore, $\s_{\tau^{-1}}\s_{\sigma}T^{(a)}_{\bar{x}}(\bar{y})=T^{(a)}_{\bar{x}}(\bar{y})$. 
Hence, $\s_{\sigma}T^{(a)}_{\bar{x}}(\bar{y})=\s_{\tau}T^{(a)}_{\bar{x}}(\bar{y})$, and we are done.
\end{proof}
\begin{lemma}\label{mynetw}\
\begin{enumerate}
\item Let $\A$ be an atomic $\PTA_n$ 
and $a\in \At\A$. 
Let $\bar{x}$ be any $n$-tuple of nodes such that $x_i=x_j$ if and 
only if $a\leq \sf d_{ij}$ for all $i,j<n$. 
Then $PT^{(a)}_{\bar{x}}$ is an atomic $\A$ network.
\item Let $\A$ be an atomic $\TA_n$ and $a\in \At\A$. 
Let $\bar{x}$ be any $n$-tuple of nodes such that $x_i=x_j$ 
if and only if $a\leq \sf d_{ij}$ for all $i,j<n$. Then $T^{(a)}_{\bar{x}}$ is an atomic $\A$ network.
\end{enumerate}
\end{lemma}
\begin{proof}
Straightforward from the above
\end{proof}

\begin{lemma} Let $\A\in \PTA_n\cup \TA_n$. Then $\exists$ can win any play of $G_{\omega}(\A)$.
\end{lemma}
\begin{proof} Let $\A^+$ be the canonical extension of $\A$.
Then $\A^+$ is atomic, and is inside the variety in which $\A$ is in. Of course any $\A$ pre-network is an $\A^+$ pre-network.
In each round $t$ of the game $G_{\omega}(\A)$, where $N_t$ is as above, $\exists$ constructs
an atomic $\A^+$ network $M_t$ satisfying
\begin{description}

\item
$M_t\supseteq N_t, \nodes(M_t)=\nodes(N_t), \edges(M_t)=\edges(N_t).$
\end{description}
Then if $\bar{x}\equiv_i \bar{y}$, we have
$$N_t(\bar{x})\cdot {\sf c}_iN_t(\bar{y})\geq M_t(\bar{x})\cdot {\sf c}_iM_t(\bar{y})\neq 0.$$
$\exists$ starts by $M_0=N_0=\emptyset$.
Suppose that we are in round $t$ and assume inductively that $\exists$ has managed to construct $M_t\supseteq N_t$ as
indicated above. We consider the possible moves of $\forall$.

\begin{enumerate}
\item Suppose that $\forall$ picks a non zero element $a\in \A$. $\exists$ chooses an atom $a^-\in \A^{+}$
with $a^-\leq a$. She chooses new nodes $x_0,\ldots x_{n-1}$ with
$x_i=x_j$ iff $a^-\leq {\sf d}_{ij}$.
\begin{description}
\item[If $\A\in \PTA_n$.]
She creates two new 
relativized networks $N_{t+1}$, $M_{t+1}$ with nodes those of $N_t$ 
plus $x_0,\ldots x_{n-1}$ and hyper\ those of $N_t$ together with $NSQ_{\bar{x}}$. The new hyper labels in
$M_{t+1}$ are defined as follows:
$$M_{t+1}=M_{t}\cup PT^{(a^-)}_{\bar{x}}.$$By Lemma \ref{mynetw} it follows that $M_{t+1}$ is an atomic $\A^+$ network.
Labels in $N_{t+1}$ are given by
\begin{itemize}
\item
$N_{t+1}(\bar{x})=a\cdot \prod_{i,j: x_i=x_j} {\sf d}_{ij}$.
\item $N_{t+1}(\bar{y})=\prod_{i,j:y_i=y_j}{\sf d}_{ij}$ for any other hyperedge $\bar{y}$.
\end{itemize}
\item[If $\A\in \TA_n$.]
She creates two new relativized networks $N_{t+1}$, $M_{t+1}$ 
with nodes those of $N_t$ plus $x_0,\ldots x_{n-1}$ and hyperedges those of $N_t$ together with $Q_{\bar{x}}$. The new hyper labels in
$M_{t+1}$ are defined as follows:
$$M_{t+1}=M_{t}\cup T^{(a^-)}_{\bar{x}}.$$By Lemma \ref{mynetw} it follows that $M_{t+1}$ is an atomic $\A^+$ network.
Labels in $N_{t+1}$ are given by
\begin{itemize}

\item
$N_{t+1}(\tau|\bar{x})=\s_{\tau}a\cdot \prod_{i,j: x_i=x_j} {\sf d}_{\tau i\tau j}$ for any finite permutation $\tau$.
\item $N_{t+1}(\bar{y})=\prod_{i,j:y_i=y_j}{\sf d}_{ij}$ 
for any other hyperedge $\bar{y}$.
\end{itemize}
\end{description}
$\exists$ responds to $\forall$'s move in round $t$ with $N_{t+1}$. One can check that $N_{t}\subseteq N_{t+1}\subseteq M_{t+1}$, as required.
-

\item If $\forall $ picks an edge $\bar{x}$ of $N_t$ and an element $a\in \A$, $\exists$ lets $M_{t+1}=M_t$ and lets
$N_{t+1}$ be the same as $N_t$ except that\begin{description}
\item[If $\A\in \PTA_n$.] $N_{t+1}(\bar{x})=N_t(\bar{x})\cdot a$ if $M_t(\bar{x})\leq a$
and
$N_{t+1}(\bar{x})=N_t(\bar{x}).-a$ otherwise. Because $M_t(\bar{x})$ is an atom in $\A^+$, 
it follows that if $M_t(\bar{x})\nleq a$, then $M_t(\bar{x})\leq -a$, so this is satisfactory.
\item[If $\A\in \TA_n$.]  for every finite permutation $\tau$, $N_{t+1}(\tau|\bar{x})=N_t(\tau|\bar{x})\cdot \s_{\tau}a$ if $M_t(\tau|\bar{x})\leq \s_{\tau}a$
and
$N_{t+1}(\tau|\bar{x})=N_t(\tau|\bar{x}).-\s_{\tau}a$ otherwise.
Because $M_t(\tau|\bar{x})$ is an atom in $\A^+$ for every finite permutation $\tau$, 
it follows that if $M_t(\tau|\bar{x})\nleq \s_{\tau}a$, then $M_t(\tau|\bar{x})\leq -\s_{\tau}a$, so this is satisfactory.
\end{description}

\item Alternatively $\forall$ picks $\bar{x}\in N_t,$ $i<n$ and $b\in \A$ such that $N_t(\bar{x})\leq {\sf c}_ib$.
Let $M_t(\bar{x})=a^-$. If there is $z\in M_t$ with $M_t(\bar{x}_z^i)\leq b$ then we are done.
In more detail, $\exists$ lets $M_{t+1}=M_t$ and define $N_{t+1}$ accordingly.
Else, there is no such $z$. We have ${\sf c}_i a^-\cdot b\not=0$ (inside $\A^+$), indeed
\begin{eqnarray*}
{\sf c}_i a^-\cdot b)&=&{\sf c}_i a^-\cdot {\sf c}_i b\text{  }\text{  }\text{  }\text{  }\text{  }\text{  }\text{  }\text{  }\text{  }\text{  }\text{  }\\
&=&{\sf c}_i(a^-\cdot {\sf c}_i b)\text{  }\text{  }\text{  }\text{  }\text{  }\text{  }\text{  }\text{  }\text{  }\\
&=&{\sf c}_i a^-\not=0.\text{  }\text{  }\text{  }\text{  }\text{  }\text{  }(\text{since }a^-\leq{\sf c}_i b)
\end{eqnarray*}
Choose an atom $b^-\in\A^+$ with $b^-\leq{\sf c}_i a^-\cdot b$. Then we have $b^-\leq b$ and $b^-\leq{\sf c}_i a^-$. 
But by Lemma \ref{soatom} ($iv$) we also have $a^-\leq{\sf c}_i b^-$.
\begin{description}
\item[If $\A\in \PTA_n$.]Let $G$ be the $\A^+$ network with nodes $\{x_0,\ldots x_{n-1}, z\}$, $z$ a new node,
and the hyperedges are the sequences in $NSQ_{\bar{x}}\cup NSQ_{\bar{t}}$ and $G=PT^{(a^-)}_{\bar{x}}\cup PT^{(b^-)}_{\bar{t}}$
where $\bar{t}=\bar{x}^i_z$.
\item[If $\A\in \TA_n$.]Let $G$ be the $\A^+$ network with nodes $\{x_0,\ldots x_{n-1}, z\}$, $z$ a new node,
and the hyperedges are the sequences in $Q_{\bar{x}}\cup Q_{\bar{t}}$ and $G=T^{(a^-)}_{\bar{x}}\cup T^{(b^-)}_{\bar{t}}$
where $\bar{t}=\bar{x}^i_z$.
\end{description} Again this is well defined. Then, 
it easy to check that  $M_t(i_1,\ldots, i_n)=G(i_1,\ldots, i_n)$ for all $i_1,\ldots, i_n\in \rng \bar{x}$.
That is the subnetworks of $M_t$ and $G$ with nodes
$\rng\bar{x}$ are isomorphic.
Then we can amalgamate $M_t$ and $G$ and define $M_{t+1}$ as the outcome.
The amalgamation here is possible, since the networks are only relativized,
we don't have all hyperedges.
By Lemma \ref{mynetw}, one can check that $M_{t+1}$ as so defined is an atomic $\A^+$ network.
Now define $N_{t+1}$ accordingly. That is $N_{t+1}$ has the same nodes and edges of $M_{t+1}$, with labelling as for $N_t$ except that
\begin{description}
\item[If $\A\in \PTA_n$.]
$N_{t+1}(\bar{t})=b$. The rest of the other labels are  defined to be $$N_{t+1}(\bar{y})=\prod_{i,j:y_i=y_{j}}{\sf d}_{ij}.$$
\item[If $\A\in \TA_n$.]
$N_{t+1}(\tau|\bar{t})=\s_{\tau}b$ for any finite permutation $\tau$. The rest of the other labels are  defined to
be $$N_{t+1}(\bar{y})=\prod_{i,j:y_i=y_{j}}{\sf d}_{ij}.$$
\end{description}
By Lemma \ref{soatom}, it is clear that $N_{t}\subseteq N_{t+1}\subseteq M_{t+1}$ and then $N_{t+1}$ is $\A$ network.
Hence $N_{t+1}$ is appropriate to be played by $\exists$ in response to $\forall$'s move.
\end{enumerate}
\end{proof}
Now we reap the harvest of our games. In all cases \pe\ can build the desired representation.

\begin{theorem}\label{11}\ 
Let $2\leq n<\omega$. 
\begin{enumarab}
\item If $\A\in \PTA_{n}$, then $\A\in {\sf D}_{n}$.
\item If $\A\in \TA_n$, then $\A\in {\sf G_n}$.
\end{enumarab}
\end{theorem}\begin{proof}\cite{games}.
 Let $\A\in \PTA_n$. We want to build an isomorphism from $\A$ to some $\B\in {\sf D}_{n}$. We can assume that $\A$ is countable.
If not just take a countable elementary subalgebra of it; representability of the latter implies the representability
of $\A$, since $\PTA_n$ is a variety. Consider a play
$N_0\subseteq N_1\subseteq \ldots $ of $G_{\omega}(\A)$ in which $\exists$ plays as in the previous lemma
and $\forall$ plays every possible legal move. Since the algebra is countable these can be scheduled.
The outcome of the play is essentially a relativized representation of $\A$ defined as follows.
Let $N=\bigcup_{t<\omega}\nodes(N_t)$, and $\edges(N)=\bigcup_{t<\omega}\edges(N_t)\subseteq {}^nN$.
By the definition of the networks, $\wp(\edges(N))\in D_{n}$. We make $N$ 
into a representation by
defining $h:\A\rightarrow\wp(\edges(N))$ as follows
$$h(a)=\{\bar{x}\in \edges(N): \exists t<\omega(\bar{x}\in N_t \& N_t(\bar{x})\leq a)\}.$$
$\forall$-moves of the second
kind guarantee that for any $n$-tuple $\bar{x}$ and any $a\in\A$, for sufficiently large $t$,
 we have either $N_{t}(\bar{x})\leq a$ or $N_{t}(\bar{x})\leq -a$. This ensures that $h$ preserves the Boolean operations.
$\forall$-moves of the third kind ensure that the cylindrifiers are respected by $h$.
Preserving diagonals follows from the definition of networks. The first kind of $\forall$-moves tell us that $h$ is one-one.
But the construction of the game under consideration ensures that $h$ is onto, too.
In fact, $\bar{h}$ is a representation from $\A$ onto $\B\in \sf D_n$. This easily follows from the definition of networks.

\end{proof}

Building representations can be implemented by the step-by-step method as in \cite{AT} and \cite{Fer}
which consists of treating defects one by one and then taking a limit where the contradictions disappear.
What can be done by step-by-step constructions, can be done by games
but not necessarily the other way round.
Games were introduced in algebraic logic by Hirsch and Hodkinson.
Such games, which are basically Banach-Mazur games in disguise, are games of infinite lengths 
between two players $\forall$
and $\exists$.
The real advantage of the game technique is that games do not only build representations,
when we know that such representations exist, but they also tell us  when such representations exist, if we do not know a priori that they do.
The translation, however, from step-by step techniques to games is not always a
purely mechanical process, even if we know that it can be done.
This transfer can well involve some ingenuity, in obtaining games that are transparent, intuitive and easy to grasp.

It is an unsettled (philosophical) question as to which is more intuitive, step-by-step techniques or games.
Basically, we believe that this depends on the context, but in all cases it is nice to have both available if possible,
when we know one exists. When we have a step-by-step technique, then we are sure that there is
at least one corresponding game. 
Choosing a simple game, which is not always an easy task, is what counts at the end.

We now show that the logics corresponding to the Andr\'eka-Thompson-Resek-Ferenzci
algebraic strong representability result, are not only complete, but also enjoy 
the standard version an omitting types theorem, manifesting another positive property. 
We use the representability result
restricted to the countable case. This offers a suitable $\omega$ dimensional dilation, since set algebras neatly 
embed into any number of extra dimensions. 
In particular, our proof does not imply the representability result.

In contrast to many  cylindric-like algebras, see item (2) in theorem \ref{SL}, 
it is known that the class of neat reducts for relativized set
algebras is closed under forming subalgebras; indeed it is a variety.

\begin{lemma}\label{g} Let $n>1$. If $\A\in \sf D_n$, then there exists $\B\in \sf D_{\omega}$
such that  $\A\subseteq \Nr_n\B$, $\A$ generates $\B$, and
$\A=\Nr_n\B$. A completely analogous result holds for $\sf G_n$.
\end{lemma}

Now we consider the logics corresponding to $\sf G_n$ and $\sf D_n$
when $n>2$. The algebraic property expressd in the previous lemma, 
enables us to prove an omitting types theorem for the $n$ variable logic corresponding to both $\sf D_n$ and $\sf G_n$.

We {\it do not} need the extra condition imposed and formulated for cylindric algebras in theorem \cite{Shelah}, namely, {\it complete}
neat embeddings. The neat embeddings can always
be chosen to be complete neat embeddings. This allows us to prove the following contrasting result.

\begin{theorem} Let $n>1$, let $\L_n$ be the logic corresponding
to $\sf D_n$ or $\sf G_n.$ Then $\L_n$ has the omitting types theorem. In fact, $< \sf covK$ many 
non-principal types in a countable $\L_n$ theory $T$ can be omitted in a $\sf G_n$ with countable base.

\end{theorem}

\begin{proof} We consider only ${\sf D}_n$. The other case is the same.
Let $\Sigma_n$ be the Andr\'eka-Thompson axiomatization for algebras of dimension $n$.
Let $\A=\Fm_T$, then we have $\A\in {\sf Mod}(\Sigma_n)$ is countable and $(X_i: i<\omega)$ a family of non principal types,
that is $\prod X_i=0$.
Since $\A$ is representable, then we can assume that $\A=\Nr_n\B$, and
we can also assume, since $A$ generates $\B$, that $\B\in {\sf Mod}(\Sigma_{\omega})$ is both countable
and locally finite; that is $\Delta x=\{i\in \omega: {\sf c}_ix\neq x\}$ is finite, for al $x\in \B$.
Now we work as above, but instead of one non principal type,
namely, the co-atoms, we now have countably many.
So here we can (and will) appeal to the Baire category theorem. First any finite substitution of $\omega$ is term definable in $\B$
(This is done like the cylindric case, it does not depend on commutativity of cylindrifiers \cite[definition 1.11.9, theorems 1.11.1-13]{tarski}.
Let $\sf adm$ be  the set of admissible substitutions.
where $\tau$ is admissible if $\dom\tau\subseteq n$ and $\rng\tau\cap n=\emptyset$.
Then we have
for all $i< \omega$ and $\sigma\in \sf adm$,
\begin{equation}\label{tarek1}
\begin{split}
\s_{\sigma}{\sf c}_{i}p=\sum_{j} \s_{\sigma}{\sf s}_j^ip
\end{split}
\end{equation}
This uses that ${\sf c}_k=\sum {\sf s}_i^k x$, which is proved like the cylindric (locally finite) case;
the proof depends on diagonal elements.

Let $(\Gamma_i: i\in \lambda\}$ with  $\lambda<{\sf covK}$ be the given family of non principal types,
and let $X_i=\{\neg \phi/T: \phi\in \Gamma_i\}$. 
By $\A=\Nr_{n}\B$, we also have $\prod^{\B}X_i=1$, because $\A$ is a complete subalgebra of $\B$ by the following argument.

Assume $S\subseteq A$ and $\sum ^{\A}S=y$, and for contradiction that $d\in B$ and
$s\leq d< y$ for all $s\in S$. Then $d$ uses finitely many dimensions not in $n$, say $m_1,\ldots, m_n$.
Let $\tau=y. -{\sf c}_{m_1}\ldots {\sf c}_{m_n}(-d)$ is in $\A$
and $s\leq \tau<y$ which is impossible.
Because substitutions are completely additive, we have
for all $\tau\in \sf adm$
\begin{equation}\label{t1}
\begin{split}
\prod {\sf s}_{\bar{\tau}}^{\B}X_i=0.
\end{split}
\end{equation}
Let $S$ be the Stone space of $\B$, whose underlying set consists of all Boolean ultrafilters of
$\B$, and let $F$ be a an ultrafilter chosen as usual.
For each $\tau\in \sf adm$ for each $i\in \kappa$, let
$$X_{i,\tau}=\{{\sf s}_{\tau}x: x\in X_i\}.$$
Then by complete additivity we have:
\begin{equation}\label{t2}\begin{split}
(\forall\tau\in \sf adm)(\forall  i\in \kappa)\prod{}^{\A}X_{i,\tau}=0
\end{split}
\end{equation}
Let $S$ be the Stone space of the Boolean part of $\A$, and for $x\in \A$, let $N_x$
denote the clopen set consisting of all
Boolean ultrafilters that contain $x$.
Then from \ref{t1}, \ref{t2}, it follows that for $x\in \A,$ $j<\alpha$ $i<\kappa$ and
$\tau\in \sf adm$, the sets
$$\bold G_{\tau,j,x}=N_{s_{\tau}{\sf c}_jx}\setminus \bigcup_{i} N_{s_{\tau}{\sf s}_i^jx}
\text { and } \bold H_{i,\tau}=\bigcap_{x\in X_i} N_{{\sf s}_{\bar{\tau}}x}$$
are closed nowhere dense sets in $S$.
Also each $\bold H_{i,\tau}$ is closed and nowhere
dense.
Let $$\bold G=\bigcup_{\tau\in adm}\bigcup_{j\in \beta}\bigcup_{x\in B}\bold G_{\tau, j,x}
\text { and }\bold H=\bigcup_{i\in \kappa}\bigcup_{\tau\in adm}\bold H_{i,\tau.}$$
By properties of $\sf covK$, it can be shown $\bold H$ is a
countable collection of nowhere dense sets.
By the Baire Category theorem  for compact Hausdorff spaces, we get that $X=S\smallsetminus \bold H\cup \bold G$ is dense in $S$.
Accordingly, let $F$ be an ultrafilter in $N_a\cap X$, then$F$ is perfect the required representation follows.

Then $F$ is a perfect ultrafilter. Because our algebras have diagonal algebras,
we have to factor our base by a congruence relation that reflects
equality.
Define an equivalence relation on $\Gamma=\{i\in \omega:\exists j\in n: {\sf c}_i{\sf d}_{ij}\in F\}$,
via $k\sim l$ iff ${\sf d}_{kl}\in F.$ Then $\Gamma\subseteq \omega$ and
the desired representation, which gives the required model, is defined on a $\sf D_n$
with base
$\Gamma/\sim$.
We omit the details which can be
easily distilled from \cite[p.128-129]{Sayedneat}.
\end{proof}

We can consider the fragment of first order logic restricted to the Thompson axioms as a guarded fragment of first order 
logic; call it $\L_n$, $n>1$ finite, where the existential quantifier is restricted to the unit, the guard.
It is also proved in \cite{games} that $\L_n$ has a really nice modal beheviour. It has the finite base property, hence is decidable;
in fact, the universal theory of $\sf D_n$ is decidable,
and it has a Craig interpoltaion theoirem and a Beth definability theorem. Our last results adds to the 
positive property of this multi-dimensional modal logic, that has quite an expreesive power.
This match is rare, and hard to achieve.

\section{Graphs and Strong representability}

\subsection{An extension of a result of Hirsch and Hodkinson's}

Here we extend Hirsch and Hodkinson's celebrated result that the class of strongly representable atom structures
of cylindric algebras of finite dimension $>2$ is not elementary, to any class of algebras with signature between
$\Df$ and $\PEA$. We also suggest a modification of their construction to obtain the result for relation algebras.
We obtain the two results in one go, meaning that the cylindric atom structure of dimension $n$
constructed,  will be the atom structure
consisting of the $n$ basic matrices of an atomic relation algebra. Also  both algebras
are based on one graph and both
are representable if and only if the chromatic number of the underlying graph is infinite.
Then we use, like they did, Erdos' probabilistic graphs.

Throughout this section, $n$ is a finite ordinal $>2$.

\begin{definition}
Let $\Gamma=(G, E)$ be a graph.
\begin{enumerate}
\item{A set $X\subset G$ is said to be \textit{independent} if $E\cap(X\times X)=\phi$.}
\item{The \textit{chromatic number} $\chi(\Gamma)$ of $\Gamma$ is the smallest $\kappa<\omega$
such that $G$ can be partitioned into $\kappa$ independent sets,
and $\infty$ if there is no such $\kappa$.}
\end{enumerate}
\end{definition}

\begin{definition}\ \begin{enumerate}
\item For an equivalence relation $\sim$ on a set $X$, and $Y\subseteq
X$, we write $\sim\upharpoonright Y$ for $\sim\cap(Y\times Y)$. For
a partial map $K:n\rightarrow\Gamma\times n$ and $i, j<n$, we write
$K(i)=K(j)$ to mean that either $K(i)$, $K(j)$ are both undefined,
or they are both defined and are equal.
\item For any two relations $\sim$ and $\approx$. The composition of $\sim$ and $\approx$ is
the set
$$\sim\circ\approx=\{(a, b):\exists c(a\sim c\wedge c\approx b)\}.$$\end{enumerate}
\end{definition}
\begin{definition}Let $\Gamma$ be a graph.
We define an atom structure $\eta(\Gamma)=\langle H, D_{ij},
\equiv_{i}, \equiv_{ij}:i, j<n\rangle$ as follows:
\begin{enumerate}
\item$H$ is the set of all pairs $(K, \sim)$ where $K:n\rightarrow \Gamma\times n$ is a partial map and $\sim$ is an equivalent relation on $n$
satisfying the following conditions \begin{enumerate}\item If
$|n\diagup\sim|=n$, then $\dom(K)=n$ and $\rng(K)$ is not independent
subset of $n$.
\item If $|n\diagup\sim|=n-1$, then $K$ is defined only on the unique $\sim$ class $\{i, j\}$ say of size $2$ and $K(i)=K(j)$.
\item If $|n\diagup\sim|\leq n-2$, then $K$ is nowhere defined.
\end{enumerate}
\item $D_{ij}=\{(K, \sim)\in H : i\sim j\}$.
\item $(K, \sim)\equiv_{i}(K', \sim')$ iff $K(i)=K'(i)$ and $\sim\upharpoonright(n\setminus\{i\})=\sim'\upharpoonright(n\setminus\{i\})$.
\item $(K, \sim)\equiv_{ij}(K', \sim')$ iff $K(i)=K'(j)$, $K(j)=K'(i)$, and $K(\kappa)=K'(\kappa) (\forall\kappa\in n\setminus\{i, j\})$
and if $i\sim j$ then $\sim=\sim'$, if not, then $\sim'=\sim\circ[i,
j]$.
\end{enumerate}
\end{definition}
It may help to think of $K(i)$ as assigning the nodes $K(i)$ of
$\Gamma\times n$ not to $i$ but to
the set $n\setminus\{i\}$, so long as its elements are pairwise non-equivalent via $\sim$.\\
For a set $X$, $\mathcal{B}(X)$ denotes the Boolean algebra
$\langle\wp(X), \cup, \setminus\rangle$. We write $a\cap b$ for
$-(-a\cup-b)$.
\begin{definition}\label{our algebra}
Let $\B(\Gamma)=\langle\B(\eta(\Gamma)), {\sf c}_{i},
{\sf s}^{i}_{j}, {\sf s}_{ij}, {\sf d}{ij}\rangle_{i, j<n}$ be the algebra, with
extra non-Boolean operations defined as follows:

${\sf d}{ij}=D_{ij}$,
${\sf c}_{i}X=\{c: \exists a\in X, a\equiv_{i}c\}$,
${\sf s}_{ij}X=\{c: \exists a\in X, a\equiv_{ij}c\}$,
${\sf s}^{i}_{j}X=\begin{cases}
{\sf c}_{i}(X\cap D_{ij}), &\text{if $i\not=j$,}\\
X, &\text{if $i=j$.}
\end{cases}$
For all $X\subseteq \eta(\Gamma)$.

\end{definition}

\begin{definition}
For any $\tau\in\{\pi\in n^{n}: \pi \text{ is a bijection}\}$, and
any $(K, \sim)\in\eta(\Gamma)$. We define $\tau(K,
\sim)=(K\circ\tau, \sim\circ\tau)$.\end{definition}

The proof of the following two Lemmas is straightforward.
\begin{lemma}\label{Lemma 1}\ \\
For any $\tau\in\{\pi\in n^{n}: \pi \text{ is a bijection}\}$, and
any $(K, \sim)\in\eta(\Gamma)$. $\tau(K, \sim)\in\eta(\Gamma)$.
\end{lemma}
\begin{lemma}\label{Lemma 2}\ \\For any $(K, \sim)$, $(K', \sim')$, and $(K'', \sim'')\in\eta(\Gamma)$, and $i, j\in n$:
\begin{enumerate}
\item$(K, \sim)\equiv_{ii}(K', \sim')\Longleftrightarrow (K, \sim)=(K', \sim')$.
\item$(K, \sim)\equiv_{ij}(K', \sim')\Longleftrightarrow (K, \sim)\equiv_{ji}(K', \sim')$.
\item If $(K, \sim)\equiv_{ij}(K', \sim')$, and $(K, \sim)\equiv_{ij}(K'', \sim'')$, then $(K', \sim')=(K'', \sim'')$.
\item If $(K, \sim)\in D_{ij}$, then \\
$(K, \sim)\equiv_{i}(K', \sim')\Longleftrightarrow\exists(K_{1},
\sim_{1})\in\eta(\Gamma):(K, \sim) \equiv_{j}(K_{1},
\sim_{1})\wedge(K', \sim')\equiv_{ij}(K_{1}, \sim_{1})$.
\item${\sf s}_{ij}(\eta(\Gamma))=\eta(\Gamma)$.
\end{enumerate}
\end{lemma}
The proof of the next lemma is tedious but not too hard.
\begin{theorem}\label{it is qea}
For any graph $\Gamma$, $\B(\Gamma)$ is a simple
$\PEA_n$.
\end{theorem}
\begin{proof}\
We follow the axiomatization in \cite{ST} except renaming the items by $Q_i$.
Let $X\subseteq\eta(\Gamma)$, and $i, j, \kappa\in n$:
\begin{enumerate}
\item ${\sf s}^{i}_{i}=ID$ by definition \ref{our algebra}, ${\sf s}_{ii}X=\{c:\exists a\in X, a\equiv_{ii}c\}=\{c:\exists a\in X, a=c\}=X$
(by Lemma \ref{Lemma 2} (1));\\
${\sf s}_{ij}X=\{c:\exists a\in X, a\equiv_{ij}c\}=\{c:\exists a\in X,
a\equiv_{ji}c\}={\sf s}_{ji}X$ (by Lemma \ref{Lemma 2} (2)).
\item Axioms $Q_{1}$, $Q_{2}$ follow directly from the fact that the
reduct $\mathfrak{Rd}_{ca}\mathfrak{B}(\Gamma)=\langle\mathcal{B}(\eta(\Gamma)), {\sf c}_{i}$, ${\sf d}{ij}\rangle_{i, j<n}$
is a cylindric algebra which is proved in \cite{hirsh}.
\item Axioms $Q_{3}$, $Q_{4}$, $Q_{5}$
follow from the fact that the reduct $\mathfrak{Rd}_{ca}\mathfrak{B}(\Gamma)$ is a cylindric algebra, and from \cite{tarski}
(Theorem 1.5.8(i), Theorem 1.5.9(ii), Theorem 1.5.8(ii)).
\item ${\sf s}^{i}_{j}$ is a Boolean endomorphism by \cite{tarski} (Theorem 1.5.3).
\begin{align*}
{\sf s}_{ij}(X\cup Y)&=\{c:\exists a\in(X\cup Y), a\equiv_{ij}c\}\\
&=\{c:(\exists a\in X\vee\exists a\in Y), a\equiv_{ij}c\}\\
&=\{c:\exists a\in X, a\equiv_{ij}c\}\cup\{c:\exists a\in Y,
a\equiv_{ij}c\}\\
&={\sf s}_{ij}X\cup {\sf s}_{ij}Y.
\end{align*}
${\sf s}_{ij}(-X)=\{c:\exists a\in(-X), a\equiv_{ij}c\}$, and
${\sf s}_{ij}X=\{c:\exists a\in X, a\equiv_{ij}c\}$ are disjoint. For, let
$c\in({\sf s}_{ij}(X)\cap {\sf s}_{ij}(-X))$, then $\exists a\in X\wedge b\in
(-X)$, such that $a\equiv_{ij}c$, and $b\equiv_{ij}c$. Then $a=b$, (by
Lemma \ref{Lemma 2} (3)), which is a contradiction. Also,
\begin{align*}
{\sf s}_{ij}X\cup {\sf s}_{ij}(-X)&=\{c:\exists a\in X,
a\equiv_{ij}c\}\cup\{c:\exists a\in(-X), a\equiv_{ij}c\}\\
&=\{c:\exists a\in(X\cup-X), a\equiv_{ij}c\}\\
&={\sf s}_{ij}\eta(\Gamma)\\
&=\eta(\Gamma). \text{ (by Lemma \ref{Lemma 2} (5))}
\end{align*}
therefore, ${\sf s}_{ij}$ is a Boolean endomorphism.
\item \begin{align*}{\sf s}_{ij}{\sf s}_{ij}X&={\sf s}_{ij}\{c:\exists a\in X, a\equiv_{ij}c\}\\
&=\{b:(\exists a\in X\wedge c\in\eta(\Gamma)), a\equiv_{ij}c, \text{ and }
c\equiv_{ij}b\}\\
&=\{b:\exists a\in X, a=b\}\\
&=X.\end{align*}
\item\begin{align*}{\sf s}_{ij}{\sf s}^{i}_{j}X&=\{c:\exists a\in {\sf }^{i}_{j}X, a\equiv_{ij}c\}\\
&=\{c:\exists b\in(X\cap {\sf d}{ij}),a\equiv_{i}b\wedge
a\equiv_{ij}c\}\\
&=\{c:\exists b\in(X\cap {\sf d}{ij}), c\equiv_{j}b\} \text{ (by Lemma
\ref{Lemma 2} (4))}\\
&={\sf s}^{j}_{i}X.\end{align*}
\item We need to prove that ${\sf s}_{ij}{\sf s}_{i\kappa}X={\sf s}_{j\kappa}{\sf s}_{ij}X$ if $|\{i, j, \kappa\}|=3$.
Let $(K, \sim)\in {\sf s}_{ij}{\sf s}_{i\kappa}X$ then
$\exists(K', \sim')\in\eta(\Gamma)$, and $\exists(K'', \sim'')\in X$
such that $(K'', \sim'')\equiv_{i\kappa}(K', \sim')$ and $(K',
\sim')\equiv_{ij}(K, \sim)$.\\
Define $\tau:n\rightarrow n$ as follows:
\begin{align*}\tau(i)&=j\\
\tau(j)&=\kappa\\
\tau(\kappa)&=i, \text{ and}\\
\tau(l)&=l \text{ for every } l\in(n\setminus\{i, j,
\kappa\}).\end{align*} Now, it is easy to verify that $\tau(K',
\sim')\equiv_{ij}(K'', \sim'')$, and $\tau(K',
\sim')\equiv_{j\kappa}(K, \sim)$. Therefore, $(K, \sim)\in
{\sf s}_{j\kappa}{\sf s}_{ij}X$, i.e., ${\sf s}_{ij}{\sf s}_{i\kappa}X\subseteq
{\sf s}_{j\kappa}{\sf s}_{ij}X$. Similarly, we can show that
${\sf s}_{j\kappa}{\sf s}_{ij}X\subseteq {\sf s}_{ij}{\sf s}_{i\kappa}X$.
\item Axiom $Q_{10}$ follows from \cite{tarski} (Theorem 1.5.7)
\item Axiom $Q_{11}$ follows from axiom 2, and the definition of $s^{i}_{j}$.
\end{enumerate}
Since $\Rd_{ca}\B$ is a simple $\CA_{n}$, by
\cite{hirsh}, then $\B$ is a simple $\PEA_n$. This follows from the fact that ideals $I$ is an ideal in $\Rd_{ca}\B$ if and only if it is an ideal in
$\B$.
\end{proof}
\begin{definition}
Let $\C(\Gamma)$ be the subalgebra of
$\B(\Gamma)$ generated by the set of atoms.
\end{definition}
Note that the cylindric algebra constructed in \cite{hirsh} is
$\Rd_{ca}\B(\Gamma)$ not
$\Rd_{ca}\C(\Gamma)$, but all results in
\cite{hirsh} can be applied to
$\Rd_{ca}\C(\Gamma)$. Therefore, since our
results depends basically on \cite{hirsh}, we will refer to
\cite{hirsh} directly when we apply it to get any result on
$\Rd_{ca}\C(\Gamma)$.

\begin{theorem}
$\C(\Gamma)$ is a simple $\PEA_{n}$ generated by the set of the
$n-1$ dimensional elements.
\end{theorem}
\begin{proof}
$\mathfrak{C}(\Gamma)$ is a simple $\PEA_{n}$ from  Theorem \ref{it
is qea}. It remains to show that $\{(K, \sim)\}=\prod\{{\sf c}_{i}\{(K,
\sim)\}: i<n\}$ for any $(K, \sim)\in H$. Let $(K, \sim)\in H$,
clearly $\{(K, \sim)\}\leq\prod\{{\sf c}_{i}\{(K, \sim)\}: i<n\}$. For the
other direction assume that $(K', \sim')\in H$ and $(K,
\sim)\not=(K', \sim')$. We show that $(K',
\sim')\not\in\prod\{{\sf c}_{i}\{(K, \sim)\}: i<n\}$. Assume toward a
contradiction that $(K', \sim')\in\prod\{{\sf c}_{i}\{(K, \sim)\}:i<n\}$,
then $(K', \sim')\in {\sf c}_{i}\{(K, \sim)\}$ for all $i<n$, i.e.,
$K'(i)=K(i)$ and
$\sim'\upharpoonright(n\setminus\{i\})=\sim\upharpoonright(n\setminus\{i\})$
for all $i<n$. Therefore, $(K, \sim)=(K', \sim')$ which makes a
contradiction, and hence we get the other direction.
\end{proof}
\begin{theorem}\label{chr. no.}
Let $\Gamma$ be a graph.
\begin{enumerate}\item Suppose that $\chi(\Gamma)=\infty$. Then $\mathfrak{C}(\Gamma)$
is representable.\item If $\Gamma$ is infinite and
$\chi(\Gamma)<\infty$ then $\Rd_{df}\C$
is not
representable
\end{enumerate}
\end{theorem}
\begin{proof}
\begin{enumerate}
\item We have $\Rd_{ca}\C$ is representable (c.f., \cite{hirsh}).
Let $X=\{x\in \C:\Delta x\not=n\}$. Call $J\subseteq \C$ inductive if
$X\subseteq J$ and $J$ is closed under infinite unions and
complementation. Then $\C$ is the smallest inductive
subset of $C$. Let $f$ be an isomorphism of
$\Rd_{ca}\C$ onto a cylindric set algebra with
base $U$. Clearly, by definition, $f$ preserves $s^{i}_{j}$ for each
$i, j<n$. It remains to show that $f$ preserves ${\sf s}_{ij}$ for every
$i, j<n$. Let $i, j<n$, since ${\sf s}_{ij}$ is Boolean endomorphism and
completely additive, it suffices to show that $f{\sf s}_{ij}x={\sf s}_{ij}fx$
for all $x\in \At\C$. Let $x\in \At\C$ and $\mu\in
n\setminus\Delta x$. If $\kappa=\mu$ or $l=\mu$, say $\kappa=\mu$,
then$$ f{\sf s}_{\kappa l}x=f{\sf s}_{\kappa
l}{\sf c}_{\kappa}x=f{\sf s}^{\kappa}_{l}x={\sf s}^{\kappa}_{l}fx={\sf s}_{\kappa l}fx.
$$
If $\mu\not\in\{\kappa, l\}$ then$$ f{\sf s}_{\kappa
l}x=f{\sf s}^{l}_{\mu}s^{\kappa}_{l}{\sf s}^{\mu}_{\kappa}{\sf c}_{\mu}x={\sf s}^{l}_{\mu}s^{\kappa}_{l}{\sf s}^{\mu}_{\kappa}{\sf c}_{\mu}fx={\sf s}_{\kappa
l}fx.
$$
\item Assume toward a contradiction that $\Rd_{df}\C$ is representable. Since $\Rd_{ca}\C$
is generated by $n-1$ dimensional elements then
$\mathfrak{Rd}_{ca}\C$ is representable. But this
contradicts Proposition 5.4 in \cite{hirsh}.
\end{enumerate}
\end{proof}
\begin{theorem}\label{el}
Let $2<n<\omega$ and $\mathcal{T}$ be any signature between $\Df_{n}$
and $\PEA_{n}$. Then the class of strongly representable atom
structures of type $\mathcal{T}$ is not elementary.
\end{theorem}
\begin{proof}
By Erd\"{o}s's famous 1959 Theorem \cite{Erdos}, for each finite
$\kappa$ there is a finite graph $G_{\kappa}$ with
$\chi(G_{\kappa})>\kappa$ and with no cycles of length $<\kappa$.
Let $\Gamma_{\kappa}$ be the disjoint union of the $G_{l}$ for
$l>\kappa$. Clearly, $\chi(\Gamma_{\kappa})=\infty$. So by Theorem
\ref{chr.
no.} (1), $\mathfrak{C}(\Gamma_{\kappa})=\mathfrak{C}(\Gamma_{\kappa})^{+}$ is representable.\\
\- Now let $\Gamma$ be a non-principal ultraproduct
$\prod_{D}\Gamma_{\kappa}$ for the $\Gamma_{\kappa}$. It is
certainly infinite. For $\kappa<\omega$, let $\sigma_{\kappa}$ be a
first-order sentence of the signature of the graphs. stating that
there are no cycles of length less than $\kappa$. Then
$\Gamma_{l}\models\sigma_{\kappa}$ for all $l\geq\kappa$. By
Lo\'{s}'s Theorem, $\Gamma\models\sigma_{\kappa}$ for all
$\kappa$. So $\Gamma$ has no cycles, and hence by, \cite{hirsh}
Lemma 3.2, $\chi(\Gamma)\leq 2$. By Theorem \ref{chr. no.} (2),
$\mathfrak{Rd}_{df}\mathfrak{C}$ is not representable. It is easy to
show (e.g., because $\mathfrak{C}(\Gamma)$ is first-order
interpretable in $\Gamma$, for any $\Gamma$) that$$
\prod_{D}\mathfrak{C}(\Gamma_{\kappa})\cong\mathfrak{C}(\prod_{D}\Gamma_{\kappa}).$$
Combining this with the fact that: for any $n$-dimensional atom
structure $\mathcal{S}$

$\mathcal{S}$ is strongly
representable $\Longleftrightarrow$ $\mathfrak{Cm}\mathcal{S}$ is
representable,
the desired follows.
\end{proof}

\subsection{The good and the bad}

Here we give a different approach to Hirsch Hodkinson's result.
We use the notation and the general ideas in \cite[lemmas 3.6.4, 3.6.6]{HHbook2}.
An important difference is that our cylindric algebras are binary generated,
and they their atom structures are the set of basic matrices on relation algebras satisfying the same properties.
We abstract the two Monk-like algebras dealt with in the proof of theorem \ref{hodkinson}.

Let $\G$ be a graph. One can  define a family of first order structures (labelled graphs)  in the signature $\G\times n$, denote it by $I(\G)$
as follows:
For all $a,b\in M$, there is a unique $p\in \G\times n$, such that
$(a,b)\in p$. If  $M\models (a,i)(x,y)\land (b,j)(y,z)\land (c,l)(x,z)$, then $| \{ i, j, l \}> 1 $, or
$ a, b, c \in \G$ and $\{ a, b, c\} $ has at least one edge
of $\G$.
For any graph $\Gamma$, let $\rho(\Gamma)$ be the atom structure defined from the class of models satisfying the above,
these are maps from $n\to M$, $M\in I(\G)$, endowed with an obvious equivalence relation,
with cylindrifiers and diagonal elements defined as Hirsch and Hodkinson define atom structures from classes of models,
and let $\M(\Gamma)$ be the complex algebra of this atom structure.

We define a relation algebra atom structure $\alpha(\G)$ of the form
$(\{1'\}\cup (\G\times n), R_{1'}, \breve{R}, R_;)$.
The only identity atom is $1'$. All atoms are self converse,
so $\breve{R}=\{(a, a): a \text { an atom }\}.$
The colour of an atom $(a,i)\in \G\times n$ is $i$. The identity $1'$ has no colour. A triple $(a,b,c)$
of atoms in $\alpha(\G)$ is consistent if
$R;(a,b,c)$ holds. Then the consistent triples are $(a,b,c)$ where

\begin{itemize}

\item one of $a,b,c$ is $1'$ and the other two are equal, or

\item none of $a,b,c$ is $1'$ and they do not all have the same colour, or

\item $a=(a', i), b=(b', i)$ and $c=(c', i)$ for some $i<n$ and
$a',b',c'\in \G$, and there exists at least one graph edge
of $G$ in $\{a', b', c'\}$.

\end{itemize}Note that some monochromatic triangles
are allowed namely the 'dependent' ones.
This allows the relation algebra to have an $n$ dimensional cylindric basis
and, in fact, the atom structure of $\M(\G)$ is isomorphic (as a cylindric algebra
atom structure) to the atom structure $\Mat_n$ of all $n$-dimensional basic
matrices over the relation algebra atom structure $\alpha(\G)$.

Our next idea  gives the result that strongly representable atom structures of both relation algebras
and cylindric algebras of finite dimension $>2$ in one go, using Erdos' graphs of large
chromatic number and girth.

\begin{theorem}\label{new}
$\alpha(\G)$ is strongly representable iff $\M(\G)$ is representable iff
$\G$ has infinite chromatic number.
\end{theorem}
\begin{proof}
The underlying idea here is that the shades of red (addressed in the proof of theorem \ref{hodkinson} for the two Monk-like algebras)
will appear in the {\it ultrafilter extension} of $\G$, if it has infinite chromatic number as a reflexive node, \cite[definition 3.6.5]{HHbook2}.
and its $n$ copies,
can be used to completely represent
$\M(\G)^{\sigma}$ (the canonical extension of $\M(\G)$).
Let $\M(\G)_+$ be the ultrafilter atom structure of $\M(\G)$

Fix $\G$, and let  $\G^* $ be the ultrafilter extension of $\G$.
First define a strong bounded morphism $\Theta$
form $\M(\G)_+$ to $\rho(I(\G^*))$, as follows:
For any $x_0, x_1<n$ and $X\subseteq \G^*\times n$, define the following element
of $\M(\G^*)$:
$$X^{(x_0, x_1)}=\{[f]\in \rho(I(\G^*)): \exists p\in X[M_f\models p(f(x_0),f(x_1))]\}.$$
Let $\mu$ be an ultrafilter in $\M(\G)$. Define $\sim $ on $n$ by $i\sim j$ iff $\sf d_{ij}\in \mu$.
Let $g$ be the projection map from $n$ to $n/\sim$.
Define a $\G^*\times n$ coloured graph with domain $n/\sim$ as follows.
For each $v\in \Gamma^*\times n$
and $x_0, x_1<n$, we let
$$M_{\mu}\models v(g(x_0), g(x_1))\Longleftrightarrow  X^{(x_0, x_1)}\in \mu.$$
Hence, any ultrafilter $\mu\in \M(\G)$ defines $M_{\mu}$ which is  a $\G^*$ structure.
If $\Gamma$ has infinite chromatic number, then $\G^*$ has a reflexive node, and this can be used
to completely represent $\M(\G)^{\sigma}$, hence represent  $\M(\G)$ as follows:
To do this one tries to show  \pe\ has a \ws\ in the usual $\omega$ rounded atomic game on networks \cite{HHbook2}, that test complete
representability.

Her strategy can be implemented using the following argument.
Let $N$ be a given $\M(\G)^{\sigma}$ network. Let $z\notin N$ and let $y=x[i|z]\in {}^n(N\cup \{z\}={}^nM$.
Write $Y=\{y_0,\ldots y_{n-1}\}$.  We need to complete the labelling of edges of $M$.
We have a fixed $i\in n$. Defines $q_j$ $:j\in n\sim \{i\}$, the unique label of any two distinct elements in $Y\sim y_j$,
if the latter elements are pairwise distinct, and arbitrarily
otherwise.
Let $d\in \G^*$ be a reflexive node in the copy that does not contain any of the $q_j$s (there number is $n-1$),
and define $M\models d(t_0, t_1)$ if $z\in \{t_0, t_1\}\nsubseteq Y$.
Labelling the  other edges are like $N$.
The rest of the proof would be similar to \cite[theorem  3.6.4 p.79]{HHbook2}.
\end{proof}

The idea above is essentially due to Hirsch and Hodkinson, it also works for relation and cylindric algebras, and this is the essence. For each graph
$\Gamma$, they associate a cylindric algebra atom structure of dimension $n$, $\M(\Gamma)$ such that $\Cm\M(\Gamma)$ is representable
if and only if the chromatic number of $\Gamma$, in symbols $\chi(\Gamma)$, which is the least number of colours needed, $\chi(\Gamma)$ is infinite.
Using a famous theorem of Erdos as we did above, they construct  a sequence $\Gamma_r$ with infinite chromatic number and finite girth,
whose limit is just $2$ colourable, they show that the class of strongly representable
algebras  is not elementary. This is a {\it reverse process} of Monk-like  constructions,
which gives a sequence of graphs of finite chromatic number whose limit (ultraproduct) has infinite
chromatic number.

In more detail, some statement fails in $\A$ iff $\At\A$
be partitioned into finitely many $\A$-definable sets with certain
`bad' properties. Call this a {\it bad partition}.
A bad partition of a graph is a finite colouring. So Monk's result finds a sequence of badly partitioned atom structures,
converging to one that is not.  This boils down, to finding graphs of finite chromatic numbers $\Gamma_i$, having an ultraproduct
$\Gamma$ with infinite chromatic number.

Then an  atom structure is {\it strongly representable} iff it
has {\it no bad partition using any sets at all}. So, here, the idea  is to {\it find atom structures, with no bad partitions
with an ultraproduct that does have a bad partition.}
From a graph Hirsch and Hodkinson constructed  an atom structure that is strongly representable iff the graph
has no finite colouring.  So the problem that remains is to find a sequence of graphs with no finite colouring,
with an ultraproduct that does have a finite colouring, that is, graphs of infinite chromatic numbers, having an ultraproduct
with finite chromatic number.

It is not obvious, a priori, that such graphs actually exist.
And here is where Erdos' methods offer solace.
Indeed, graphs like are found using the probabilistic methods of Erdos, for those methods
render finite graphs of arbitrarily large chromatic number and girth.
By taking disjoint unions as above, one {\it can get}
graphs of infinite chromatic number (no bad partitions) and arbitrarily large girth. A non principal
ultraproduct of these has no cycles, so has chromatic number 2 (bad partition).
This motivates:

\begin{definition}
\begin{enumarab}

\item A Monks algebra is good if $\chi(\Gamma)=\infty$

\item A Monk's algebra is bad if $\chi(\Gamma)<\infty$

\end{enumarab}
\end{definition}
It is easy to construct a good Monks algebra
as an ultraproduct (limit) of bad Monk algebras. Monk's original algebras can be viewed this way.
The converse is, as illustrated above, is much harder.
It took Erdos probabilistic graphs, to get a sequence of good graphs converging to a bad one.

Using our modified Monk-like algebras we obtain the result formulated in theorem \ref{el}.
The following corollary answers a question of Hodkinson's in \cite[p. 284]{AU}; though admittedly he proposed the line of argument.

On the other hand, we prove more than required. Hodkinson's question only addressed the two cases $\PA$ and $\PEA$.

\begin{corollary}\label{Hodkinson} For $\K$ any class between $\Df$ and $\PEA$ and any finite $n>2$,
the class $\K_s=\{\A\in \K_n: \Cm\At\A\in \sf RK_n\}$, is not elementary.
\end{corollary}

Recall that networks and atomic games are defined in \cite{HHbook}. They can be easily modified to address any signature between $\Df$ and $\PEA$.
For example for $\sf PEA$ we add the following clause  $s_{[i,j]}N(\bar{x})= N(\bar{x})\circ [i,j])$ for atomic networks,
and for $\sf Df$ we delete the condition addressing diagonal elements and we are
left only with the cylindrifier move. The games are the same,
which means that the information is coded in the networks and not
in the games.

A Lyndon $n$th condition, or simple a Lyndon condition, is a first order sentence that codes that \pe\ has a \ws\
in the usual atomic game on (atomic) networks having
$n$ rounds \cite{HHbook2}.

Again like in the case of rainbow algebras, define the polyadic operations corresponding to transpositions
using the notation in \cite[definition 3.6.6]{HHbook2}, in the context of defining atom structures from
clllases of models by
$R_{s_{[ij]}}=\{([f], [g]): f\circ [i,j]=g\}$. This is well defined.
In particular, we can (and will)
consider the Monk-like algebra $\M(\Gamma)$ as defined in \cite[top of p. 78]{HHbook2} as a polyadic equality
algebra.

\begin{theorem}\label{hhchapterbook}Let $\M(\Gamma)$ be the polyadic equality algebra defined above.
If $\chi(\Gamma)=\infty$, then $\M(\Gamma)$ is representable as a  polyadic
equality algebra. If $\chi(\Gamma)<\infty$,
then $\Rd_{df}\M(\Gamma)$ is not representable.
\end{theorem}
\begin{proof} Only networks are changed but the atomic game is the same, so clearly \pe\ has a \ws\
infinite (possibly transfinite)  game over $M(\Gamma)^{\sigma}$ which is
the canonical extension of $\M(\Gamma)$ \cite[lemma 3.6.4, lemma 3.6.7]{HHbook2}

In using the latter lemma the proof is unaltered, in the former lemma \pe\ s strategy to win the
$G^{|{\sf Uf}(M(\Gamma))|+\omega)}\M(\Gamma)^{\sigma}$ game is exactly the same
so  that  $\M(\Gamma)^{\sigma}$ is completely representable, hence $\M(\Gamma)$
is representable as a polyadic equality algebra.

The second part follows from the fact that $\M(\Gamma)$ is generated by elements whose dimension
sets $<n$,  \cite[theorem 5.4.29]{tarski}.
\end{proof}

\section*{Part II}

\section{Completions and decidability}

\subsection{Lifting results from atomic  $\RA$s to $\CA$s, via atom structures}

Throughout 
this section, unless otherwise explicity specified, $m>2$ will denote the dimension and $n$ will be always finite $>m.$
We address the decidability of the problem as to whether a finite $\CA_m$ is in $S\Nr_m\CA_n.$
We succeed to obtain a negative result but only for the lowest value of
$m$, namely, when $m=3$ and $n\geq m+3$.
This generalizes the result of Hodkinson in \cite{AU} proved for $\RCA_n$ but only for $n=3$. For higher dimensions, the problem to the best
of our knowledge remains unsettled. We also connect the syntactical notion of  algebras having a (complete) neat embedding property to
the semantical one of having (complete) relativized representations as defined in definition \ref{rel},
generalizing results proved
by Hirsch and Hodkinson for relation algebras, witness \cite[theorems 13.45, 13. 46]{HHbook}.

Our first theorem is conditional. 
It gives a sufficient condition of how to lift results on relation algebras to cylindric or polyadic ones, using {\it $n$ dimensional basis} 
over an atom structure
of a relation algebra, to form a cylindric-like $n$ dimensional atom structure; a theme 
in \cite{AU}.

Let $\sigma_k$ be the usual atomic $k$ rounded  game, which we view as being played on atomic networks 
of relation or cylindric atom structures.
Context will help which one
we mean.

We need to prepare some more for our next theorem. The next definition and lemma
are needed to transfer results on the existence of canonical axiomatizations, or rather 
the non-existence thereof, from  relation algebras to finite dimensional cylindric algebras of finite dimension $n>2$. 

Hodkinson proved that this is indeed the case with $\RCA_n$ without going through the route of relation algebras.
He actually proved more, namely, 
 any equational axiomatization of $\RCA_n$ must contain infinitely many non-canonical sentences, that is sentences that are not
preserved in canonical exiension. Closed under
canonical extensions,
this shows that $\RCA_n$ $(3<n<\omega)$ is only {\it barely} canonical.
We shall also use, like Hodkinson and Venema for relation algebras  and Hodkinson for finite dimensional 
cylindric algebras,  Erdos' probabilistic graphs.
The following two definitions are taken from \cite{HV}.
\begin{definition}
\begin{enumarab}
\item We say that a partially ordered set $(I, \leq)$ is directed if every finite subset of $I$ has an upper bound in $I$.
\item An inverse system of $L^a$ structures is a triple
$$D=((I,\leq)) (S_i:i\in I), (\pi_{ij}: i,j\in I, i\leq j)),$$
where $(I,\leq)$ is a directly partially ordered set, each $S_i$ is an $L^a$ structure, and for $i\leq j\in I$, $\pi_{ji}:S_j\to S_i$ is a surjective homomorphism,
such that whenever $k\geq j\geq i$ in $I$ then $\pi_{ii}$ is the identity map and $\pi_{ki}=\pi_{ji}\circ \pi_{kj}$.
\item We say that $D$ in an inverse system of finite structures if each $S_j$ is a finite structure, and an inverse system of bounded morphisms if each $\pi_{ij}$
is a bounded morphism.
\item The inverse limit $lim_{\leftarrow}D$ of $D$ is the substructure of $\prod_{i\in I}S_i$ with domain
$$\{\chi\in \prod_{i\in I}S_i: \pi_{ji}(\chi_j))=\chi(i)\text { whenever } j\geq i \text { in } I\}.$$
\item For any $i\in I$, the projection $\pi_i:lim_{\leftarrow}D\to S_i$ is defined by $\pi_i(\chi)=\chi(i)$
\end{enumarab}
\end{definition}

\begin{definition}
Let $D=((I,\leq) )S_i:i\in I), (\pi_{ji}: i,j\in I, i\leq j)$ be an inverse system of finite structure and bounded morphisms.
Let $I=lim_{\leftarrow} D$. For each $i\in I$ define $\pi_i^+:S_i^+\to I^+$ by $\pi_i^+(X)=\pi_i^{-1}[X]$, for $X\subseteq S_i$.
\end{definition}
Each $\pi_i^+$ is an algebra embedding $: S_i^+\to I^+$, and its range $\pi_i^+(S_i)$ is a finite subalgebra of $I^+$. It follows that
$\A_D=\bigcup_{i\in I}\pi_i^+(S_i^+)$ is a directed union of finite subalgebras of $I^+$, and we have
$$(\A_D)_+\cong lim_{\leftarrow}D=I$$
The following theorem is proved in \cite{HV} by adapting techniques of Erdos in constructing probabilistic
graphs with arbitrary large chromatic number and girth.

\begin{theorem} Let $k\geq 2$. There are finite graphs $G_0, G_1\ldots$ and surjective homomorphisms $\rho_i:G_{i+1}\to G_i$ for $i<\omega$
such that for each $i$, $\rho_i$ is a bounded morphism and
\begin{enumarab}
\item for each edge $xy$ of $G_i$ and each $x'\in \rho_i^{-1}(x)$, there is a $y'\in \rho_i^{-1}(y)$ such that
$x'y'$ is an edge of $G_{i+1}$,
\item $G_i$ has no odd cycles of length $\leq i$
\item $\chi(G_i)=k.$
\end{enumarab}
\end{theorem}
\begin{demo}{Proof} \cite{HV}
\end{demo}
Fix integers $k\geq m\geq 2$, and let $H_0, H_1\ldots$ and $\pi_i:H_{i+1}\to H_i$ be graphs and homomorphisms as above.
Now fix a complete graph $K_m$ with $m$ nodes, and for each $i<\omega$, let $G_i$ be the disjoint union of $H_i$ and $K_m$.
For each $i<j<\omega$ define $\rho_{ij}$ to be the identity on $G_i$, and
$\rho_{ji}:G_j\to G_i$ be defined by
\begin{equation*}
\rho_{ji}(x) =
\begin{cases}
\pi_i\circ \ldots \pi_{j-1} , & \hbox{if $x\in H_j$} \\
x,  & \hbox{if $x\in K_m$.} \end{cases}
\end{equation*}
Let $$D_m^k=((\omega, \leq), (G_i:i<\omega), (\rho_{ji}: i\leq j<\omega))$$
Then $lim_{\leftarrow}D_m^k$ is a graph with chromatic number $m$.

\begin{theorem}\label{decidability} Let $m\geq 3$. Assume that for any simple atomic relation algebra $\A$ with atom structure $S$,
there is a cylindric atom structure $H$, constructed effectively from $\At\A$,  such that:
\begin{enumarab}
\item If $\Tm S\in \sf RRA$, then $\Tm H\in \RCA_m$.
\item If $S$ is finite, then $H$ is finite
\item $\Cm S$ is embeddable in $\Ra$ reduct of $\Cm H$.
\item \pe\ has a \ws\ in $\sigma_k$ over $S$ iff \pe\ has a \ws\ in $\sigma_k$
over $H$, where $\sigma_k$ is the usual atomic $k$ rounded game.
\end{enumarab}
Then  for all $k\geq 3$, $S\Nr_m\CA_{m+k}$ is not closed under completions,
and  it is undecidable whether a finite cylindric algebra is in
and $S\Nr_m\CA_{m+k}$, and $\RCA_n$ cannot be axiomatizable by canonical sentences. In fact, any axiomatization of $\RCA_m$
must contain infinitely many non canonical sentences.
\end{theorem}

\begin{demo}{Proof} First a simple observation, note that (1) is a special case of (4).
\begin{enumarab}
\item For the first part. Let $S$ be a relation atom structure such that $\Tm S$ is representable while $\Cm S\notin \RA_6$.
Such an atom structure exists \cite[lemmas 17.34-35-36-37]{HHbook}.

We give a brief sketch at how such algebras are constructed by allowing complete irreflexive graphs having an arbitrary finite set
nodes, slightly generalizing the proof in {\it op.cit}, though the proof idea is essentially the same.

Another change is that we refer to non-principal ultrafilters (intentionally) by blurs to emphasize the connection with the blow up and
blur construction in \cite{ANT} as well as with the blow up and blur construction in the proof of theorems \ref{blowupandblur}
and \ref{blurs} above.
In all cases a finite algebra is blown up and blurred to give a representable algebra (the term algebra on the blown up and blurred finite atom
structure) whose \d\ completion does not have a neat embedding
property.

In the latter two cases the blurs were used as colourings to represent the term algebra; in the last we only needed only one shade of red,
witness theorems \ref{hodkinson}, \ref{blowupandblur}.
We use the notation of the above cited lemmas in \cite{HHbook2} without warning, and our proof will be very brief just stressing the main ideas.
Let $\R=\A_{K_m, K_n}$, $m>n>2$.
Let $T$ be the term algebra obtained by splitting the reds. Then $T$ has exactly two blurs
$\delta$ and $\rho$. $\rho$ is a flexible non principal ultrafilter consisting of reds with distinct indices and $\delta$ is the reds
with common indices.
Furthermore, $T$ is representable, but $\Cm\At T\notin  S\Ra\CA_{m+2}$, in particular, it is not representable. \cite [lemma 17.32]{HHbook2}

First it is obvious that  \pe\ has a \ws\ over $\At\R$ using in $m+2$ rounds,
hence $\R\notin \RA_{m+2}$, hence is not
in $S\Ra\CA_{m+2}$.  $\Cm\At T$ is also not in the latter class
$\R$ embeds into it, by mapping ever red
to the join of its copies. Let $D=\{r_{ll}^n: n<\omega, l\in n\}$, and $R=\{r_{lm}^n, l,m\in n, l\neq m\}$.
If $X\subseteq R$, then $X\in T$ if and only if $X$ is finite or cofinite in $R$ and same for subsets of $D$ \cite[lemma 17.35]{HHbook}.
Let $\delta=\{X\in T: X\cap D \text { is cofinite in $D$}\}$,
and $\rho=\{X\in T: X\cap R\text { is cofinite in $R$}\}$.
Then these are {\it the} non principal ultrafilters, they are the blurs and they are enough
to (to be used as colours), together with the principal ones, to represent $T$ as follows \cite[bottom of p. 533]{HHbook2}.
Let $\Delta$ be the graph $n\times \omega\cup m\times \{{\omega}\}$.
Let $\B$ be the full rainbow algebras over $\At\A_{K_m, \Delta}$ by
deleting all red atoms $r_{ij}$ where $i,j$ are
in different connected components of $\Delta$.

Obviously \pe\ has a \ws\ in $\EF_{\omega}^{\omega}(K_m, K_m)$, and so it has a \ws\ in
$G_{\omega}^{\omega}(\A_{K_m, K_m})$.
But $\At\A_{K_m, K_m}\subseteq \At\B\subseteq \At_{K_m, \Delta}$, and so $\B$ is representable.

One  then defines a bounded morphism from $\At\B$ to the the canonical extension
of $T$, which we denote by $T^+$, consisting of all ultrafilters of $T$. The blurs are images of elements from
$K_m\times \{\omega\}$, by mapping the red with equal double index,
to $\delta$, for distinct indices to $\rho$.
The first copy is reserved to define the rest of the red atoms the obvious way.
(The underlying idea is that this graph codes the principal ultrafilters in the first component, and the non principal ones in the second.)
The other atoms are the same in both structures. Let $S=\Cm\At T$, then $\Cm S\notin S\Ra\CA_{m+2}$ \cite[lemma 17.36]{HHbook}.

Note here that the \d\ completion of $T$ is not representable while its canonical extension is
{\it completely representable}, via the representation defined above.
However, $T$ itself is {\it not} completely representable, for a complete representation of $T$ induces a representation of its \d\ completion,
namely, $\Cm\At\A$.

\item Let $H$ be the $\CA_m$ atom structure provided by the hypothesis of the previous theorem.
Then $\Tm H\in \RCA_m$. We claim that $\Cm H\notin {\bf S}\Nr_m\CA_{m+k}$, $k\geq 3$.
For assume not, i.e. assume that $\Cm H\in {\bf S}\Nr_m\CA_{m+k}$, $k\geq 3$.
We have $\Cm S$ is embeddable in $\Ra\Cm H.$  But then the latter is in ${\bf S}\Ra\CA_6$
and so is $\Cm S$, which is not the case.

\item For the second part, suppose for contradiction that there is an algorithm ${\sf Al}$ to determine whether a finite
algebra is in $S\Nr_m\CA_{m+k}$. We claim that we can now decide effectively whether a finite simple atomic
relation algebra is in $S\Ra\CA_5$  by constructing the $\CA_m$, $\C=\Tm H$ which is finite,
and returning  the answer ${\sf Al}(\C)$.
This is the correct answer.  However,
this contradicts the result that it is undecidable to tell whether a finite relation algebra is in $S\Ra\CA_5$
\cite[theorem 18.13]{HHbook}.

\item The last item follows from lifting the proof for relation algebras. Let ${\GG}_n$ denote the games defined after 
$n$ rounds. We write, emphasizing the dependence of $H$ on $S$, 
$\Ca_t(S)$ for $H$. Here $t$ is the dimension; will be fixed to be 
finite and $>2$.  We also write $\At^+$ for the complex algebra $\Cm\At$ of
$\At$ and $\A^{\sigma}$ for $\Uf^+$, the canonical extension of $\A$.
Assume for contradiction that $\RCA_t$ has a canonical axiomatization.
Then there is an $n_0$ such that for any $n<\omega$, there is $n^*<\omega$, such that for any $\CA_t$ $\A$
if $\exists$ has a winning strategy in ${\GG}_{n^*}(\A)$ and ${\GG}_{n_0}(\A^{\sigma})$ then she has a winning strategy in
${\GG}_n(\A^{\sigma})$ \cite[proposition5.4]{HV}.
Let $e:\omega\to \omega$ be defined by $e(n)=2^{dn.4^n}$.

Then if $G$ is a graph and $\chi(G)\geq e(n)$ for some $n<\omega$, then \pe\ has a winning strategy
in ${\GG}_n(\alpha(G)^+)$ \cite[proposition 6.4]{HV}.
For $n<\omega$, let $n'<\omega$ be so large such that any colouring using $dn$ colours , of the edges of a complete graph with $n'$ nodes
has a monochromatic triangle. Let $u:\omega\to \omega$ be defined ny $u(n)=n'-2+n'(n'-1)(dn+1)$.
Then if $G$ is a graph with $\chi(G)\leq n<\infty$ and $|G|\geq n'-1$
then $\forall$ has a winning strategy in ${\GG}_{u(n)}(\alpha(G)^+)$\cite[proposition 6.6]{HV}.
Let $m=e(n_0)$ and $n=u(m)$ . Since the games played are determined,
there is $n^*<\omega$, such that for any cylindric algebra $\A$, such that $\exists$
has a winning strategy in  ${\GG}_{n_0}(\A^{\sigma})$, if $\forall$
has a winning strategy in ${\GG}_n(\A^{\sigma})$
then he has a winning strategy in ${\GG}_{n^*}(\A)$.

Let $k=e(n^*)$.
Let $D=D_m^k=((\omega, \leq), (G_i:i<\omega), (\rho_{ij}: i\leq j<\omega))$ be the inverse system as defined above.
We have $G=lim_\leftarrow D$ and $\chi(G_i)=k$ for all $i$ and $\chi(G)=m$.
Let $\alpha(D)=((\omega,\leq), \Ca_t((\alpha(G_i)):i<\omega), (\alpha^{\rho_{ji}}: i\leq j<\omega))$, where
$\rho_{ji}:\Ca_t(\alpha(G_j))\to \Ca_t(\alpha(G_i))$
is defined by $((a,k), (b,k),(c,k))\mapsto ((\rho_{ji}a,k),(\rho_{ji}b,k) (\rho_{ji}c,k))$. Then $\alpha(D)$ is an inverse system of cylindric algebra
atom structures and
bounded morphisms. Each $[\Ca_t(\alpha (G_i))]^+$ is a cylindric algebra. Write $\A$ for the algebra $\A_{\alpha(D)}$.
Then $[\Ca_t(\alpha (G_i))]^+\subseteq \A$ for all $i<\omega$ and $\A$ is the directed union $\bigcup_{i<\omega}[\Ca_t(\alpha (G_i)]^+$.
Then $\A$ is a cylindric algebra of dimension $t$.
It can be checked that $\A$ is atomic and $\alpha(lim_\leftarrow D)\cong lim_\leftarrow\alpha(D)$, hence $\A^{\sigma}\cong [\Ca_t(\alpha(G))]^+$.
Note that here $\Ca_t$ is applied to an infinite atom structure. But the resulting atom structure is that of a cylindric algebra of dimension $t$.
Now $G$ has chromatic number $m$ and is infinite. Then $\exists$ has a winning strategy
in ${\GG}_{n_0}(\A^{\sigma})$ while $\forall$ has a winning strategy in
${\GG}_{n}(\A^{\sigma})$. By choice of $n^*$, $\forall$ also has a winning strategy in ${\GG}_{n*}(\A)$ that
only uses finitely many elements $W\subseteq \A$.
We may choose $i<\omega$ such that $W\subseteq \alpha(G_i)^+$. Since the latter is a
subalgebra of $\A$, then this is a winning strategy  for $\forall$ in ${\GG}_{n^*}[\Ca_t(\alpha(G_i))]^+$.
But $\chi(G_i)=e(n^*)$, then $\exists$ has a winning strategy in this same game.
This is a contradiction that finishes the proof.

\end{enumarab}
\end{demo}

\subsection{Decidability}

Now we give an entirely different proof that the equational theory of several subvarieties of $\CA_n$, $n\geq 3$,
is undecidable. The proof is due to N\'emeti.

\begin{definition} $\pi$ is the pairing formula using partial function $p_0$ and $p_1.$
$\pi$ is defined as follows. $$\forall x \forall y \forall z[(p_0(x,y)\land p_0(x,z)\to y=z \land p_(x,y)\land p_1(x,z)\to y=z$$
$$\land \exists z(p_0(z, x)\land p_1(z,y).$$
$\pi'$ is the sentence that says that the witness of the quasi projections is unique (it uses $4$ variables),
and $\pi_{RA}$ is the quasi projective axiom
formulated in the language of relation algebras.
\end{definition}

A sentence $\lambda$ is inseparable if there is no set of formulas $T$ that recursively separates the theorems of $\lambda$
from the refutable sentences of $T$.

The next lemma  and theorem are due to N\'emeti using the pairing technique of Tarski's.
\begin{lemma} There exists an inseparable formula $\lambda$ and quasi projections
$p_0, p_1$ such that $\lambda \land \pi'$
is a valid in Peano arithmetic (these can be formulated in the language of set theory, that is we need
only one binary relation),
and there is a translation function $k: \Fm_{\omega}\to \Fm_4$
such that for any formulas $\psi$ and $\phi$,
whenever $\psi\models \phi$, then $k(\psi)\vdash_4 k(\phi)$.
\end{lemma}
\begin{demo}{Proof}  $\lambda$ is Robinson's arithmetic interpreted in set theory, so it is equivalent to $ZFC$ without the axiom of infinity,
 and $p_0$ and $p_1$ are the genuine projection function defined
also (by brute force) in the language of set theory. The translation function is implemented using three
functions $h$, $r$ and $f$; it is defined at $\phi$
by $h[rf(\phi).\pi_{RA}]$. Here $f$ is
Tarski function $f:\Fm_{\omega}^0\to \Fm_3^0$, so that $\pi\models \phi \leftrightarrow f(\phi)$
and $f$ preserves meaning,
$r$ is a recursive function translating $3$ formulas to $\RA$ terms
recursively, also preserving meaning, and $h$ is the natural
homomorphism from $\Fr_1\RA \to \Fm_4$, taking the generator to $E$ the only binary relation.
\end{demo}
The next theorem, as stated above,  is due to N\'emeti.
\begin{theorem}\label{pairing} ${\sf Eq}S\Nr_3\CA_n$  for any $n\geq 4$ is undecidable, this is also true for $S\Nr_n\CA_m$ for any $2<n<m$.
\end{theorem}

\begin{proof} Let $\K$ be any of the above classes.
Let $\lambda$ be an inseparable formula such that $\M\models \lambda\land \pi'$, where $\M$ is a model of Peano
arithmetic.
Then $\M\models k(\lambda)$. Let $\tau$ be the translation function of $n$ variable formulas to $\CA_n$ terms.
Let $T=\{\phi \in \Fm_{\omega}^0: \K\models \tau(k(\lambda))\leq \tau(k(\phi))\}$.
Assume that $\lambda \models \phi$. Then $k(\lambda)\vdash_4 k(\phi).$
Hence $\phi\in T$.
Assume now that $\lambda\models \neg \phi$. Then $k(\lambda)\vdash _4 \neg k(\phi)$.
By $\M\models k(\lambda)$ and ${\sf Cs}_n^{\M}\in K$, we have
it is not the case that $\K\models \tau(k(\lambda)=0$,
hence $\phi\notin T$.
We have seen that $T$ separates the theorems of $\lambda$ from the refutable
sentences of $\lambda$, hence $T$ is not recursive, but
$\tau$ and $k$ are recursive, hence ${\sf Eq K}$ is not recursive.
\end{proof}

\begin{lemma}\label{Andrekaetal} Let $\K$ be a class of relation algebras satisfying, that for every $n<\omega$ there is
a simple algebra in $\K$ with at least $n$ elements below the identity, then ${\sf Eq K}$
is undecidable.
\end{lemma}
\begin{proof} \cite{AGMNS}
\end{proof}

The following corollary is a consequence of the above lemma and of the undecidability result that we have just proved, witness
\cite[corollary 18.16, theorem 18.27]{HHbook} for similar results for relation algebras.

For a class $\sf K$ of algebras, the class ${\sf \K}\cap \sf Fin$ denotes the class of finite members of $\K$.

\begin{corollary}\label{undecidability} Let $n\geq 6$. Then the following hold:
\begin{enumarab}

\item The set of isomorphism types of algebras in $S\Nr_3\CA_n$ with
only infinite flat representations is not recursively enumerable

\item The equational theory of $S\Nr_3\CA_n$ is undecidable

\item The equational theory of $S\Nr_3(\CA_n\cap \sf Fin)$ is undecidable

\item The variety $S\Nr_3\CA_n$  is not finitely axiomatizable even
in $m$th order logic.

\item For every $n\geq 6$, there exists a finite algebra in $S\Nr_3\CA_n$ that does not have a finite $n$ dimensional hyperbasis

\end{enumarab}
\end{corollary}

\begin{proof}
\begin{enumarab}

\item  This follows from the following reasoning. The set $A$ of isomorphism types of finite
algebras is recursively enumerable, and so the set $B$ of isomorphism types of finite relation algebras
that are in $S\Nr_3\CA_n$ having a finite relativized representation is also
recursively enumerable,  hence $A$ and $B$ together with the set
of isomorphism types of of algebras in $S\Nr_3\CA_n$ with no finite recursive representation
are recursive iff $C$ is recursively enumerable.

\item Any such axiomatization will give a decision procedure for the class of finite algebras in $S\Nr_3\CA_n$.

\item  Let $\K=S\Nr_3\CA_n$.
We reduce the problem of telling if a finite simple algebra is not in
$\K$ to the problem of telling if an equation in the language of $\CA_3$ is valid in
$\K$. Let $\A$  be a finite relation algebra. Form $\Delta(\A)$, the diagram of $\A$,
but using variables instead of constants; define its conjunction.
Consider $g=\exists_{\bar{a}\in \A}x_a \Delta(A)$ relative to some enumeration $\bar{a}$ of $A$.
Then $\B\models g$ iff $A\nsubseteq \B$. Since  $\K$ is a discriminator variety,
the quantifier free $\neg \Delta(\A)$ is equivalent over
simple algebras to an equation $e(\A)$, which is effectively constructed
from $\A$. So, in fact, we have
$\B\models e(\A)$ iff $\A\nsubseteq \B$.
So $e(\A)$ is valid over simple $\K$ algebra iff for all $\B$ in $\K$,
$\A\nsubseteq \B$. Since $\A$ is simple, this happens if and only if $\A\notin \K$.
But $\K$ is a discriminator variety, so $e(\A)$ is valid in all algebras if
and only if $\A\notin \K$. Thus an
algorithm deciding validity of equations,
gives an algorithm to decide non membership for $\K$ for finite
algebras,  contradiction.

\item For every $m$, let $\A\in {\sf Cs_m}$ be a finite  cylindric set algebra, such that
$X=\{s\in {}^nU:  s_0=s_1\}|=m$. Now $\A$ is simple. Thus $\Ra\A$ is simple, for if $I$ is a proper ideal in $\Ra\A$, then $\Ig^{\A}I$
is proper in $\A$.
Further more $\A$ has $2$ dimensional $m$ elements below ${\sf d}_{01}$, and so $\Ra\A$ has $m$ elements below the identity.
Hence by Andr\'eka et al \cite{AGMNS}, see lemma \ref{Andrekaetal}, the equational theory of $S\Ra(\CA_m\cap {\sf Fin})$
is undecidable. Hence the equational theory of
$S\Nr_3(\CA_m\cap {\sf Fin})$ is undecidable.

\item  Assume for contradiction that every finite algebra in $S\Nr_3\CA_n$ has a finite
$n$ dimensional hyperbasis.
We claim that there is an algorithm that decides membership in $S\Nr_3\CA_n$ for finite algebras:
\begin{itemize}
\item Using a recursive axiomatization of $S\Nr_3\CA_n$ (exists), recursively enumerate all isomorphism types of
finite $\CA_3$s that are not in $S\Nr_3\CA_n.$

\item Recursively enumerate all finite algebras in $S\Nr_3\CA_n$.
For each such algebra, enumerate all finite sets of $n$ dimensional hypernetworks over $\A$,
using $\N$ as hyperlabels, and check  to
see if it is a hyperbasis. When a hypebasis is located specify $\A$.
This recursively enumerates  all and only the finite algebras in $S\Nr_3\CA_n$.
Since any finite $\CA_3$ is in exactly one of these enumerations, the process will decide
whether or not it is in  $S\Nr_3\CA_n$ in a finite time.
\end{itemize}

\end{enumarab}
\end{proof}

\begin{theorem}  
\begin{enumarab}\label{smoothcompleteness2}
\item Let $m\geq 3$. 
If $\A\in \CA_m$ is finite, and has an $n$ square relativized representation, with $m<n\leq \omega$,
then it has a {\it finite} $n$ square relativized
representation.  In particular, for every $2<m <n$, the variety ${\sf CB}_{m,n}$ 
has finite algebra on finite base property for $n$ square representations, and its universal theory is decidable.

\item If $n\geq 6$,
this is not true for $n$ smooth relativized representations for finite $\CA_3$s,
unless we restrict neat embeddings to include only finite algebras in $\CA_3$s.
\end{enumarab}
\end{theorem}
\begin{demo}{Proof} The first part follows from the fact that the first order theory coding the existence of square representations
can be easily coded in the clique guarded fragment of first order logic,
and indeed in the loosely guarded fragment of first order logic which
has the finite base property \cite[corollary 19.7]{HHbook}.

The second part follows from the fact that  the problem of deciding whether a finite $\CA_3$ is in
$S\Nr_3\CA_n$, when $n\geq 6$, is undecidable,
from which one can conclude that there are finite algebras in $S\Nr_3\CA_n$, $n\geq 6$
that do not have a finite $n$ dimensional hyperbasis \cite[corollary 18.4]{HHbook} and
these cannot possibly have finite representations. The last part follows from the following theorem.
\end{demo}

A contrasting result is:

\begin{theorem}\label{smoothcompleteness1}
Let $m$ be finite $>2$. Let $\A$ be a finite $\CA_m$. Then the following are equivalent
\begin{enumarab}
\item $\A$ has a finite $n$ smooth relativized representation
\item $\A\subseteq \Nr_n\B$, $\B\in \CA_n$ is finite.
\item $\A$ has a finite $n$ dimensional hyperbasis.
\end{enumarab}
\end{theorem}

\begin{proof} (1) to (2) to (3) is exactly as above, by noting that if $\A$ is finite then $\D$ as defined in corollary \ref{flat},
is also finite,
and that this gives necessarily a finite
hyperbasis, witness the proof of theorem \ref{smooth}.

It remains to show that  (1) implies (3). Let $H$ be a finite $n$ dimensional hypebasis;
we can assume
that is symmetric, that is, closed under substitutions. This does not
affect finiteness. Let $L$ be the finite signature consisting of an $m$-ary
relation symbol for every element of $\A$, together with an $n$ ary relation symbol $R_N$
for each $N\in H$. Define an $L$ structure by
$$M\models a(\bar{x})\text { iff } 1(\bar{x})\text { and }M(\bar{x})\leq a,$$
and
$$M\models R_N(x_0,\ldots, x_{n-1})\text { iff } M(x_0,\ldots, x_{n-1})=N$$
for all $a\in A$ and $N\in H.$ Then it is not hard to show that $M$
satisfies the properties postulated in theorem \ref{step} and these
can be coded as a fragment of the loosely guarded fragment of first order logic \cite[theorem 19.20]{HHbook}.
Recall that the equational theory of this class is undecidable though it has the finite base property.
\end{proof}

Concerning theorem \ref{decidability}, we note that Monk and Maddux constructs such an $H$ for $n=3$ and
Hodkinson constructs  an $H$ for arbitrary dimensions $>2$, but
unfortunately the relation algebra
does not embed into $\Cm H$, \cite[p.266]{AU} for $n>5$, and, to the best of our knowledge
it is not known whether it does embed when  $n\in \{4,5\}$, but it seems highly unlikely.

\section{Finite axiomatizability, first order definability, via rainbows and Monk algebras}

\subsection{Neat embedding theorems for relativized representations}

Here we prove a `completeness' theorem providing relativized semantics for algebras having a weak neat embedding property.
Essentially a completeness theorem, or rather infinitely many completeness
theorems, such results generalize the very well known result in cylindric algebra theory referred to as the 
{\it Neat Embedding Theorem of Henkin.}
This theorem is one of the earliest cornerstones in algebraic logic, and as its name suggests 
it was proved by Leon Henkin.
It is a fairly striaghtforward algebraic extension of
both the  completenes theorem of first order logic and his celebrated proof of this 
theorem. The difference here is that the $\omega$ extra dimensions, or for that matter witnesses,
in Henkin's proof is truncated to be finite. (This gives infinitely many cases, by theorem \ref{thm:cmnr}).
Accordingly representations are relativized.

Recall that  ${\sf Crs_n}$ denotes the class of relativized cylindric set algebras, 
there units are abitrary sets of sequences, ${\sf D_n}$ is the subclass of ${\sf Crs_n}$, 
whose units are closed under substitutions corresponding to replacements, 
while ${\sf G_n}$ are those ${\sf D_n}$s whose units are closed under
all finite substitutions.

\begin{theorem}\label{smoothcompleteness2} Let $1<m<n<\omega$.
Then the following are equivalent for an atomic algebra $\A$:
\begin{enumarab}
\item $\A\in S\Nr_m(\CA_n \cap  {\sf G}_n)$
\item $\A\in S\Nr_m(\CA_n\cap {\sf D_n})$
\item $\A\in S\Nr_n(\CA\cap {\sf Crs_n})$
\item $\A\in S\Nr_m\CA_n$
\item $\A$ has an smooth representation
\item  $\A$ has has $n$ infinitary $n$ flat representation
\item $\A$ has an $n$ flat representation
\end{enumarab}
\end{theorem}

\begin{proof}
\begin{enumarab}

\item $(1)\implies (2)\implies (3) \implies (4)$ is trivial.

\item  (4) to (5).  Assume that $\A\in S\Nr_n\CA_m$. Then $\A^+\in S_c\Nr_n\CA_m$
has a hyperbasis by the argument in \ref{smooth}, hence treating this hyperbasis as a saturated set
of mosaics, we can obtain $n$ smooth relativized
representation of $\A$, which is infinitary $n$ flat, witness theorems, \ref{step} and
\ref{smooth}.

\item (5) to (6). Here we use the elementary view to relativized representation. We translate the existence of an $n$
smooth representation to a first order theory then we use
a saturation argument, that is the base of the relativized representation will be an $\omega$ saturated model
of this theory. This idea of buiding 'saturated' representations, for example complete representations of canonical extensions 
of representable algebras, is due to Hirsch and Hodkinson.

Assume that $\A$ has an $n$ smooth representation.
Let ${\sf Clique}(\bar{x})$ be the formula
$$\bigwedge_{{i_0,\ldots  i_{m-1}<n}}1(x_{i_0},\ldots x_{i_{m-1}})$$
Let $T(\A)$ be the first order theory in the signature obtained by adding an $n$ ary relation symbol for every element of $\A$
and stating the first order axioms capturing general relativized representations \cite[definition 5.1]{HHbook}, for the relation algebra case.
For example, corresponding to the join operation, we have ($r,s\in A)$
$$\forall\bar{x}([r\lor s](\bar{x})\longleftrightarrow r(\bar{x})\lor s(\bar{x})).$$
And for cylindrifiers, the most important axiom, we {\it relativize} to the unit, namely,
the corresponding formula we stipulate is
$$\forall \bar{x}(1^{\A}(\bar{x})\to ({\sf c}_kr(\bar{x})\to (\exists \bar{z})(\bar{z}\equiv_k\bar{x} \land r(\bar{z})))$$
We extend the theory $T(A)$ by, for all $k<n,$
$$\forall \bar{x}({\sf Clique}(\bar{x})\land {\sf c}_ka(x_{i_0},\ldots x_{i_{m-1}})$$
$$\implies \exists x_{i_k}({\sf Clique}({\bar{x}})
\land  a(\bar{z})  \land (\bar{z})_k=x_{i_k}\land \bar{z}\equiv_k (x_{i_0}\ldots x_{i_{m-1}})).$$
Here $\bar{x}$ is of length $n$ and $\bar{y}$ is subtuple of $\bar{x}$ of length $m$. This is
for all $i_0, \ldots i_{m-1}, i_k<n$ and all $a\in A$,
$i_k\notin \{i_0,\ldots i_{m-1}\}.$
This new theory $Sq^n(\A)$ stipulates the existence of $n$ square relativized representation.

Now extend $L(\A)$ to the language $L(\A, E)$ by adding $2l$ many ary
predicates for $0<l\leq n$ and extend the theory $Sq^n(\A)$ to cf, \cite[definition 13.7]{HHbook},
$$\forall \bar{x}\bar{y}\bar{z}({\sf Clique} (\bar{x})\longleftrightarrow E^l(\bar{x}, \bar{x})\land
[E^j(\bar{x},\bar{y})\land E^l(\bar{y},\bar{z})\to E^l(\bar{x},\bar{z})]$$
and the other axioms are defined the obvious way:
$$\forall \bar{x}\bar{y} (E^l(\bar{x},\bar{y})\to E^j(x\circ \theta, y\circ \theta)),\ j,l<n, \theta:j\to l$$
$$\forall \bar{x}\bar{y}(E^l(\bar{x}, \bar{y}) \land r(\bar{x})\to r(\bar{y})), r\in \A$$
$$\forall \bar{x}\bar{y}(E^{n-2}(\bar{x}, \bar{y}) \land {\sf Clique} (\bar{x}x)\land {\sf Clique} (\bar{y}y)
\to \exists z(E^{n-1}(\bar{x}x, \bar{y}z)\land {\sf Clique}(\bar{y}yz))$$

Assume that obtained new theory has an $n$ smooth relativization, then it is consistent.
Let $M$ be an $\omega$ saturated model,
of this theory, then we show that it is a complete $n$ smooth relativized representation.
We will define an injective homomorphism $h: \A^+\to \wp(^nM)$

First note that the set $f_{\bar{x}}=\{a\in A: a(\bar{x})\}$
is an ultrafilter in $\A$, whenever $\bar{x}\in M$ and $M\models 1(\bar{x})$
Now define
$$h(S)=\{\bar{x}\in 1: f_{\bar{x}}\in S\}.$$
We check only injectivity uses saturation.
The rest is straightforward.
It suffices to show that for any ultrafilter $F$ of $\A$ which is an atom in $\A^+$, we have
$h(\{F\})\neq 0$.
Let $p(\bar{x})=\{a(\bar{x}): a\in F\}$. Then this type is finitely satisfiable.
Hence by $\omega$ saturation $p$ is realized in $M$ by $\bar{y}$, say .
Now $M\models 1(\bar{y})$ and $F\subseteq f_{\bar{x}}$,
since these are both ultrafilters,  equality holds.
Note that any partial isomorphism of $M$ is also a partial isomorphism of $M$ regraded
as a complete representation of $\A^+$.
Hence $\A^+$ has an infinitary $n$ flat, hence $\A$ has infinitary
$n$ flat representation.

\item (6) to (7) trivial.
\item (7) to (1). Exactly like the proof of item (5)  \ref{sah}.

\end{enumarab}
\end{proof}

\subsection{Rainbows, Monk-like  algebras  and relativized representations}

We use rainbows and 
Monk-like algebras to prove several results on finite axiomatizability and first order definability of several classes introduced 
ealier on in the paper. Throughout this subsection, unless otherwise specified $m$ will denote the dimension and $n$ will be reserved
for nodes or number of pebbles in play.  
It will be always the case that $2<m<n$, and $m$ is finite.

\begin{theorem}\label{nofinite}
\begin{enumarab}
\item  Let $2<m<n<\omega$. Then the class of algebras having an $n+1$ 
smooth representations is a variety, that  is not finitely axiomatizable over the class
having $n$ smooth representations. 

\item The class of algebras having complete $n$ smooth
representations, when $n\geq m+3$
is not even elementary.
\end{enumarab}
\end{theorem}
\begin{proof}
\begin{enumarab}
\item For brevity, we denote the class of algebras having an $n$ smooth representation by $\RCA_{n,s}$. The dimension is $m$ and
$n>m>2$.
For the first part. Assume that $\A$ has an $n$ smooth  representation $\M$. As in the above proof, take
$\D_0$ be the algebra consisting of those $\phi^{M}$ where $\phi$ comes from $L^n$. Since $M$ is $n$ flat we
have that cylindrifiers commute by definition,
hence $\D_0\in \CA_n$.
Now as above the map  $\theta(r)=r(\bar{x})^{M}$ is a neat embedding, hence $\A\subseteq \Nr_m\CA_n$, so that
$\A\in S\Nr_m\CA_n$. Conversely, if $\A\in S\Nr_m\CA_n$, then one can build an $n$ smooth representation
by showing that $\A^+$ has an $n$ dimensional hyperbasis, see
theorem \ref{smooth}.
Hence $\RCA_{n,s}\subseteq S\Nr_n\CA_m$. But by the celebrated result
of Hirsch and Hodkinson \cite[theorem 15.1(4)]{HHbook}, we have that for $k\geq 1$,
$S\Nr_m\CA_{m+k+1}$ is not finitely axiomatizable
over $S\Nr_m\CA_{m+k}$ and we are done.
This same result holds for $\Sc$s $\PA\s$ and $\PEA$s as well \cite{t}, witness also theorem \ref{thm:cmnr}.

\item For the second part, by theorem \ref{neat}, by noting that every classical 
complete representation is $k$ smooth for any finite $k$ 
it suffices to show that $\sf RCA_{m+3,s}\subseteq S_c\Nr_n\CA_{m+3}$.
Now assume that $\A$ has a complete $m$ flat representation.
But the above, using the same notation, the neat  embedding is a complete embedding under the assumption that
$M$ is a complete $n$ smooth representation of $\D$,
and we are done.

\end{enumarab}

\end{proof}

In the  following theorem giving several caharcaterizations of algebras possesing $n$ smooth
representation. 
For a similar characterzation for relation algebras, witness \cite[theorem 13.35]{HHbook2}.
\begin{theorem}
\begin{enumarab}
\item $\A$ is atomic and has an $n$ dimensional hyperbasis

\item $\A$ has a complete infinitary $n$ flat relativized representation
\item $\A\subseteq _c \Nr_m\C$ for some atomic $\C\in \CA_n$
\item $\A$ has a complete smooth relativized representation
\end{enumarab}
\end{theorem}
\begin{proof} The proof is left to the reader. it can be easily destilled from previous proofs. 
The proof of theorem \cite[theorem 13.45]{HHbook2} could be helpful.
\end{proof}

For any cardinal $\kappa$, $K_{\kappa}$ will denote the complete irreflexive graph with $\kappa$ nodes.
Let $p<\omega$, and $I$ a linearly irreflexive ordered set, viewed as model to a signature containing a binary relation $<$.
$M[p,I]$ is the disjoint union of $I$ and the complete graph $K_p$ with $p$ nodes.
$<$ is interpreted in this structure as follows $<^{I}\cup <{}^{K_p}\cup I\times K_p)\cup (K_p\times I)$
where the order on $K_p$ is the edge relation.

For finite $m>1$, let ${\sf CAB}_{m,n}$ be the class of cylindric algebras of dimension $m$ 
algebras having $n$ dimensional 
basis. Then $S{\sf CAB}_{m, n}={\sf CB}_{m,n}.$
Let ${\sf CAF}_{m, n}$ be the class of algebras having complete $n$ flat representations, and $\sf CRA_{m,n}$ the class of
algebras having $n$ smooth representations. Then we have 
$${\sf CRA}_{m.n}\subseteq {\sf CAF}_{m,n}\subseteq {\sf CAB}_{m,n}.$$ 
We know from theorem \ref{nofinite} 
that the first class is not elementary; we show that the second and third are not elementary 
either and we give another proof of the non-first order definability of the class $\sf CRA_{m, n}$.

\begin{theorem}\label{completerepresentation} Assume that $m$ is finite and $m>2$. Then the following hold
\begin{enumarab}
\item For $n\geq m+3$, the class ${\sf CAB}_{m,n}$ is not elementary.
\item For $n\geq m+3$, the class ${\sf CAF}_{m,n}$ is not elementary.
\end{enumarab}
\end{theorem}
\begin{proof}
\begin{enumarab}

\item We prove (1) and (2) then we give a different proof that ${\sf CRA}_{m,n}$ is not elementary. 
A rainbow argument can and will be used  lifting winning strategies from the \ef\ pebble game to rainbow atom structures.
Let $A=M[n-4, \Z]$ and $B=M[n-4, \N]$, then it can be shown  \pe\ has a \ws\ for all finite rounded games (with $n$ nodes)
on $\CA_{A,B}$, namely, in $G^n_r$ for all $r>n$, so she has a \ws\ in $G^n_{\omega}$
on any non trivial ultrapower, from which an elementary countable
subalgebra $\B$ can be extracted (using an elementary chain argument)
in which \pe\ also has \ws\ in $G_{\omega}^n$, so that $\B$ has
an $n$ smooth complete representation.

But it can also be shown that \pa\ can  the $\omega$ rounded game on $\CA_{A,B}$, also with $n$ nodes,
hence the latter does not have an $n$ square relativized representation,
but is elementary equivalent to one that has an $n$ smooth one. Since $n$ flatness implies $n$ smoothness,
we get the required.

To show that the class of $\CA_m$s ($m$ finite $>2$), having a classical 
complete representability is not elementary,  the usual atomic $\omega$ rounded game was played on $\CA_{\Z,\N}$. 
There was no limit on the number of pebbles used by \pa\ .
Now that we are dealing with $n$ relativized complete square representations, 
so it is reasonable to expect 
that we use only $n$ pebbles, and the game used is $G_r^n$ ( $n\geq m$, $r\leq n)$.

Let ${\sf EF}_r^p[A, B]$ denote the \ef\ pebble forth game defined in \cite{HHbook} between two structures $A$ and $B$ with
$p$ pebbles and
$r$ rounds, which each player will use as a private game to guide her/him in the rainbow game on coloured graphs.

We describe winning strategies in the private \ef\ game. 
In his private game, \pa\ always places the pebbles on distinct elements of $\Z$.
She uses rounds $0,\ldots n-3$, to cover $n-4$ and first two elements of $\Z$.
Because at least two out of three distinct colours are related by $<$, \pe\ must respond by pebbling
$n-4 \cup \{e,e'\}$ for some  $e,e'\in \N$.
Assuming that \pa\ has not won, then he has at least arranged that two elements of $\Z$
are pebbled, the corresponding pebbles in $B$ being in $\N$.
Then \pa\ can force \pe\ to play a two pebble game of length $\omega$ on $\Z$, $\N$
which he can win.

In her private game, \pa\ picks up a spare pebble pair and place the first pebble of it on $a\in A$.
By the rules of the game, $a$ is not currently occupied by a pebble. \pe\ has to choose which element of $B$ to put the
pebble on.  \pe\ chooses an unoccupied element in $n-4$, if possible. If they are all already occupied,
she chooses $b$ to be an arbitrary element
$x\in \N$. Because there are only $n-3$ pebble pairs, \pe\ can always implement this strategy and win.

We work with polyadic equality rainbow algebras. 
We can now lift \ws\ s to the game but now played on coloured
graphs, the atoms of $\PEA_{A,B}$.

In this game, like in the private \ef\ game, 
 \pe\ has a \ws\ in all finite rounded games, but \pa\ can win the
$\omega$ rounded game.

First a definition. Let $\PEA_{A,B}$ be a rainbow algebra.
A red clique $\sf R$ is a coloured graph all of whose edges are red. The index of a node $\alpha$ in a red clique $\sf R$ is defined by
$\mu(\alpha)=b\in B$ where $\sf R(\alpha, \beta)=\r_{bb'}$, for some $\beta\in R$ and $b'\in B$. 
This is well defined by the consistency rules on triples of reds 
(will be defined shortly).

We denote the dimension by $m$. We have $2<m<\omega$.
We show:
\begin{enumerate}
\item \pe\ has a \ws\ in the game $G_r^n(\PEA_{A,B})$ for every $r\geq n$.
\item  \pa\ has a \ws\ in $G_{\omega}^n(\Rd_{sc}\PEA_{A,B}).$
\end{enumerate}

We lift \ws\ s on coloured
graphs, the atoms of $\PEA_{A,B}$.

The colours used are:

\begin{itemize}

\item greens: $\g_i$ ($1\leq i<m-2)$, $\g_0^i$, $i\in A$.

\item whites : $\w_0, \w_i: i<m-2$

\item reds:  $\r_{ij}$ $(i,j\in B)$,

\item shades of yellow : $\y_S: S\subseteq_{\omega}B$, $S=B.$

\end{itemize}

Coloured graphs are complete graphs whose edges are labelled by any colour other than shades of yellow, 
some $m-1$ hyperedges are also labelled
by the shades of yellow.

An $i$ cone, $i\in A$, is a finite coloured graph $M$ consisting of $m$ nodes
$x_0,\ldots,  x_{m-2}, z$ such that $M(x_0, z)=\g^0_i$   
and for every $1\leq j\leq m-2,$ $M(x_j, z)=\g_j$, 
and no other edge of $M$
is coloured green.
$(x_0,\ldots, x_{m-2})$ 
is called the centre of the cone, $z$ the apex of the cone
and $i$ the tint of the cone.

We do not consider all coloured graphs; we define $\K$ a specific class of coloured graphs
From the model theoretic  point of view, such coloured graphs  are the models of the $L_{\omega_1, \omega}$ rainbow theory formulated 
in the rainbow signature as defined in \cite{HHbook2}. 
An edge $(x,y)$ is coloured by a green $\g$ means 
that $(x,y)$ satisfies the green binary relation.
Now $\K$ consists of the following coloured graphs.

\begin{itemize}

\item $M$ is a complete graph.

\item $M$ contains no triangles (called forbidden triples) 
of the following types:
\vspace{-.2in}
\begin{eqnarray}
&&\nonumber\\
(\g, \g^{'}, \g^{*}), (\g_i, \g_{i}, \w), 
&&\mbox{any }i\in m-1\;  \\
(\g^j_0, \g^k_0, \w_0)&&\mbox{ any } j, k\in A\\
\label{forb:pim}(\g^i_0, \g^j_0, \r_{kl})&&\mbox{unless } \set{(i, k), (j, l)}\mbox{ is an order-}\\
&&\mbox{ preserving partial function }A\to B\nonumber\\
\label{forb:match}(\r_{ij}, \r_{j'k'}, \r_{i^*k^*})&&\mbox{unless }i=i^*,\; j=j'\mbox{ and }k'=k^*
\end{eqnarray}
and no other triple of atoms is forbidden.

\item If $a_0,\ldots   a_{m-2}\in M$ are distinct, and no edge $(a_i, a_j)$ $i<j<n$
is coloured green, then the sequence $(a_0, \ldots a_{n-2})$
is coloured a unique shade of yellow.
No other $(n-1)$ tuples are coloured shades of yellow.

\item If $D=\set{d_0,\ldots  d_{m-2}, \delta}\subseteq M$ and
$\Gamma\upharpoonright D$ is an $i$ cone with apex $\delta$, inducing the order
$d_0,\ldots  d_{m-2}$ on its base, and the tuple
$(d_0,\ldots d_{m-2})$ is coloured by a unique shade
$y_S$ then $i\in S.$

\end{itemize}

The game on coloured graphs between \pe\ and \pa\ is obtained from the game $G^n_r$ 
played on atomic networks is exactly like the graph game in \cite{HH}, 
except that the size of graphs are restricted to $n$ nodes.

We start by \pe\ . Let $M$ be a coloured graph built at some stage $t<r$.

We assume inductively that \pe\ has never chosen $\g$ or $\w$  and if $F$ is a face in $M_t$, $|F|=m-1$
$\beta, \delta$, are apexes of two cones inducing the same order on $F$, with $\beta\in M_t\sim F$ and $\delta\notin M_t$
then the red clique obtained by considering the reds labelling each two distinct apexes of all such cones, played so far,
has at least $2$ elements.

We also assume inductively, that at round $t-1$, there is an associated private
game which we denote by ${\sf EF}[c,c']$, for $c, c'$ are distinct elements
such that there is a one to one correspondence between the pair of pebbles currently in play of
${\sf EF}[c,c']$, where $c, c'\in N_{t-1}$
lie in $F$, a base of a cone,  and the nodes of $R_{N_{t-1}}(F)$, the red clique in $N_{t-1}$,
consisting of the apexes of cones based on $F$, in such a way, that for each
such pair if the elements of $A$ and $B$ covered by the pebbles in the pair are
$a\in A$ and $b\in B$,  then the corresponding unique node $n\in R_{N_{t-1}}(F)$
satisfies $N_{t-1}(c, n)=\g_0^a$  and $\beta(n, R_{N_{t-1}}(F))=b$.

Last inductive hypothesis
that there are at least as many rounds remaining in the play of any such private game
as there are rounds remaining in the rainbow game on the algebra.

Let $3\leq t<r$. We assume that the current coloured graph is $M$ and that   \pa\ chose the graph $\Phi$, $|\Phi|=m$
with distinct nodes $F\cup \{\delta\}$, $\delta\notin M$, and  $F\subseteq M$ has size
$m-1$, in round $t-1$. We can  view \pe\ s
move as building a coloured graph $M^*$ extending $M$
whose nodes are those of $M$ together with the new node $\delta$ and whose edges are edges of $M$ together with edges
from $\delta$ to every node of $F$.

Now \pe\ must extend $M^*$ to a complete graph $M^+$ on the same nodes and
complete the colouring giving  a graph $M^+$ in $\K$ (the latter is the class of coloured graphs).
In particular, she has to define $M^+(\beta, \delta)$ for all nodes
$\beta\in M\sim F$.
\begin{enumerate}
\item  If $\beta$ and $\delta$ are both apexes of two cones on $F$; this is the hardest case.
Note that $\delta$ is a new node and $\beta\in M\sim F$.
Assume that the tint of the cone determined by $\beta$ is $a\in A$, and the two cones
induce the same linear ordering on $F$. Recall that we have $\beta\notin F$, but it is in $M$, while $\delta$ is not in $M$,
and that $|F|=m-1$.
Now \pe\ , by the rules of the game,  has no choice but to pick a  red colour, she does this as follows:
Let $$R_{M}(F)=\{x\in M: M(x,\beta)=\g_0^i,\text { for some $i\in A$},  F\cup \{x\}$$
$$\text { is an $i$ cone with base $F$ and appex $x$}\}.$$
This is a red clique, it is basically the apexes of cones in $M$, inducing the same order on $F$.
Inductively, we have  $|R_{M}(F)|< n-2$, hence there are  fewer than $n-2$ pairs of pebbles in play.
\pe\ picks up a spare pebble  pair, and playing
the role of \pa\ places one of the pebbles in the pair on $a$.
She uses her \ws\ to respond by placing the other
one on $b\in B$.
She then labels the edge between $\beta$ and $\delta$ with $\r_{\mu(\beta), b}$.

\item If this is not the case, and  for some $0<i<m-1$ there is no $f\in F$ such
that $M^*(\beta, f), M^* (f,\delta)$ are both coloured $\g_i$ or if $i=0$, they are coloured
$\g_0^l$ and $\g_0^{l'}$ for some $l$ and $l'$.
She chooses $\w_i$, for $M^+{(\beta,\delta)}$, for definiteness let it be the least such $i$.

\item Otherwise, if there is no $f\in F$,
such that $\Gamma^*(\beta, f), \Gamma^*(\delta ,f)$ are coloured $\g_0^t$ and $\g_0^u$
for some $t,u$, then \pe\ defines $\Gamma^+(\beta, \delta)$ to  be $\w_0$.

\end{enumerate}
Now we turn to colouring of $m-1$ tuples. Let $\Phi$ be the graph chosen by \pa\, it has set of note $F\cup \{\delta\}$.
For each tuple $\bar{a}=a_0,\ldots a_{m-2}\in {M^+}^{m-1}$, $\bar{a}\notin M^{m-1}\cup \Phi^{m-1}$,  with no edge
$(a_i, a_j)$ coloured green, then  \pe\ colours $\bar{a}$ by $\y_S$, where
$$S=\{i\in A: \text { there is an $i$ cone in $M^*$ with base $\bar{a}$}\}.$$

We need to check that such labeling works, that 
is, the above strategy adopted by \pa\ is indeed a winning one.

Consider a triangle of nodes $(\beta, y, \delta)$ in the graph $M_{r+1}=M^+$.
The only possible potential problem is that the edges $M^+(y,\beta)$ and $M^+(y,\delta)$ are coloured green with
distinct superscripts $p, q$ but this does not contradict
forbidden trianlges of the form involving $(\g_0^p, \g_0^q, \r_{kl})$.

Assume that $M^+(y,\delta)=\g_0^a$, $M^+(y, \beta)=\g_0^{a'}$ and $M^+(\delta, \beta)=\r_{bb'}$.
for  some $a, a'\in A$, $b, b'\in B$. Recall that \pa\ labelled the sides of the cones whose tints are $a$ and $a'$ by the greens,
\pe\ s strategy ensures that $\{(a,b), (a', b')\}$
is part of a position in the private game in which she is using her \ws\, so this map is a partial
homomorphism, and there is no inconsistency.
Now suppose that $M(\beta, y)$ and $M_{r+1}(y, \delta)$ are both red (some $y\in \nodes(M)$).
Then \pe\ chose the red label $N_{r+1}(y,\delta)$, for $\delta$ is a new node.
We can assume that  $y$ is the appex of a  cone with base $F$ in $M_r$. If not then $N_{r+1}(y, \delta)$ would be coloured
$\w$  by \pe\   and there will be no problem (see the next item). All properties will be maintained.
Hence $y, \delta ,\beta$ are in the same red clique, namely, in $R_{M^+}(F),$
so $M^+(\beta,\delta)=\r_{\mu(\beta), b}$ and $M^+(y, \delta)=\r_{\mu(y), b}$,
for the $b$ determined by her strategy in
the private game.
By definition of indices $\mu(\beta)$, $\mu(y)$,
we have $M^+(\beta, y)=M(\beta, y)=\r_{\mu(\beta), \mu(y)}$.

This is consistent triple, and so have shown that
forbidden triples of reds are avoided.
The colouring of yellow shades also avoids any inconsistencties; this can be proved 
exactly using the same argument in \cite[p.16]{Hodkinson},\cite[p.844]{HH}
and \cite{HHbook2}.

For the second part, in his private game, \pa\ always place the pebbles on distinct elements of $A$.
In the graph game, he bombards \pe\  with cones having the same base
and different green tints whose indices are the distinct elements of $A$ dictated by his private game.
In this way he forces \pe\ to play reds (to respect consistency rules) until she runs out of them; 
this way winning the rainbow game
on a red clique.  He can do this because the game now is not truncated 
to finitely many rounds; it is $\omega$ rounded, and $\Z$ is `longer' than $\N$.

In more detail, assume that $p\geq 3$ (for $p\leq 2$, the game is rather degenerate). In round $t$, $3\leq t<1+r$,
assume inductively that these are nodes $n_0,\ldots n_{q-1}$,
$q<\omega$, $2\leq q\leq p$, added for cones  with tints $j$, inducing the same order on one face $F$,
and  $a_j\in A$, $j<p$, $a_j$ pairwise distinct  and the indices on the nodes $(n_j, n_k)$ must be red, $j,k<q$, so that $\{n_j: j<q\}$ forms a red clique.
Each node $n_j$ has an index $\beta(n_j)\in B$.

As part of the inductive hypothesis,
suppose that at the start of round $t-1$ of his private game, \pe\ has not lost yet, so that \pa\ is still using his \ws\  and
the situation corresponds to the situation in round $t$ of the graph game. That is there is
a pair of pebbles on $(a_j, \beta(n_j))$) for each $j<q.$

Now we can assume that \pa\ only removes a single pair of pebbles and only when he has to.
So his \ws\ will make him remove a pair of pebbles, say the
pair on $(a_j, \beta(n_j))$ for some
$j<q$. In this case there must be $n$ nodes in the
current network $N_{t-1}$ and \pa\ removes the corresponding node $n_j$.
Since $p\geq 3$, removing the pair above leaves at least two
distinct pairs of pebbles, so the private game goes on for another
round. Also removing $n_j$ still leaves
at least two nodes in the red clique, retaining their indices.
But we know that  \pa\  has a winning strategy for the $\omega$ rounded private game.

If \pa\ s strategy in this game tells him to place a pebble $a$ on $A$, then in the graph game he plays the  cone with same
base $F$, $|F|=m-1$, and tint $a$.
Since $a$ is distinct from all other nodes in $\A$ covered by pebbles, so this forces
\pe\ to add a new node $n$ to the graph. Then $n$ must be part of a red clique.

So $n$ has an index $\beta(n)\in B$. In \pa\ private game he lets \pe\ place her
corresponding pebble on $\beta(n).$ Because \pa has a \ws\ in his private game,
eventually he will place a pebble $a\in A$, but there is nowhere in $B$ for \pe\ to place the other pebble.
But this means that $\{a_j, \beta(n_j)), (a_j' ,\beta(n_j')\}$ is not a partial homomorphism.
Hence $(n_j, n_j', c)$ is not consistent, and \pa\ has won.

\item We now give a different proof that ${\sf CRA}_{n,m}$ is not elementary. 
The proof in the first item, slightly modified, can work.
However, we'd rather give another nice argument using instead a {\it Monk-like algebra}. 
This example is a modification of \cite[exercise 1, p. 485]{HHbook2}, 
by lifting the relation algebra construction therein to the cylindric
case.

The algebra $\A(n)$ we use is a limiting case of the algebras $\A(n-1,r)$ in \cite{HHbook2} by allowing $r$ to go to infinity,
and splitting every atom to uncountable many. So $\A(n)$ is an infinite atomic relation algebra in which 
$\A(n-1, r)$ embed the obvious way.
In particular, the finite algebras $\A(n-1,r)$ for each $r\in \omega$ are not blurred after the splitting.
The atoms of $\A(n)$ are $Id$ and $a^k(i,j)$ for each $i<n-2$, $j<\omega$ $k<\omega_1$.
\begin{enumerate}
\item All atoms are self converse.
\item We list the forbidden triples $(a,b,c)$ of atoms of $\A(n)$- those such that
$a.(b;c)=0$. Those triples that are not forbidden are the consistent ones.
This defines composition: for $x,y\in A(n)$ we have
$$x;y=\{a\in \At(\A(n)); \exists b,c\in \At\A: b\leq x, c\leq y, (a,b,c) \text { is consistent }\}$$
Now all permutations of the triple $(Id, s,t)$ will be inconsistent unless $t=s$.
Also, all permutations of the following triples are inconsistent:
$$(a^k(i,j), a^{k'}(i,j), a^{k''}(i,j')),$$
if $j\leq j'<\omega$ and $i<n-1$ and $k,k', k''<\omega_1.$
All other triples are consistent.
\end{enumerate}
Then for any $r\geq 1$, $\A(n-1,r)$ embeds completely in $\A(n)$ the obvious way,
hence, the latter has no $n$ dimensional hyperbasis, because the former does not.

Then $\A(n)$ has an $m$ dimensional hyperbasis for each $m<n-1$, by proving that \pe\ has
\ws\ in the hyperbasis game $G_r^{m,n}(\A(n), \omega)$, for any $r<\omega$, that is, for any finite rounded game.

We now show that $\A(n)$ has an $m$ dimensional hyperbasis for each $m<n-1$, by proving that \pe\ has
\ws\ in the hyperbasis game $G_r^{m,n}(\A(n), \omega)$ \cite[definition 12.26]{HHbook2} for any $r<\omega$, that is, for any finite rounded game.
The strategy is very similar to \pe\ strategy given for the relation algebra $\A(n, r)$ defined on p. 466 \cite[definition 15.2]{HHbook2}.
The latter is given in \cite[sec 15.5 p.476 - 483]{HHbook2}
that is, it takes 7 pages so we'd better be sketchy.
Notice that our relation algebra is obtained from $\A(n,r)$ by taking a limiting case making $r$ and $k$ infinite.

There are three kinds of moves in this game,
$m$ dimensional moves, triangle moves and amalgamation moves.

However, some of \pe\ s moves are forced by \pa\ . For example if he plays an amalgamation move
$(L, M, x, y)$ in round $t$, then \pe\ has no choice but to reply with a hypernetwork $N_t$ such that
$L\equiv_x N_t\equiv_y M$. So if $\bar{x}\leq {}^{n}(n\smallsetminus \{x\})$ , \pe\ sets $N_t(\bar{x})=L(\bar{x}).$
The label on such an edge and any $N_t$ equivalent edge is determined
by \pa\ s move. \pe\ labels
all remaining edges and hyperedges of $N_t$ one at a time in some order, as follows:
If $\bar{x}$ is the next hyperedge,
\begin{enumarab}

\item If $|x|=|y|$ and there exists for all $i<|x|$,  $\bar{z_i}$ such that
$N(x_i,y_i, \bar{z_i})\leq Id$
for some already labelled hyperedge $\bar{y}$,
then she  lets $N_t(\bar{x})=N_t(\bar{y})$.

\item Otherwise, she lets
$N_t(\bar{x})$ be some new label that
is not a label of any hyperedge that is
already labelled in $N_t$ nor
is the label of any hyperedge
in any hypernetwork $N_u$ for $u<t$. Here the number of labels
is infinite so this is possible.

\end{enumarab}
We are left with:

\begin{enumarab}
\item  Suppose that \pa\ makes an $m$ dimensional move in round $t$,
choosing an $n$ wide $m$ dimensional hypernetwork $M$ over $\A(n)$.
\pe\ lets $N_t$ be an $n$ dimensional  hypernetwork
such that $N_t|_m^{n+1}=M$ and $N_t(j,0)=Id$ for each $j $ with
$m\leq j<n$.

\item Suppose that \pa\ makes a triangle move $(N_u, x, y, z, a ,b)$
for $u<t$, $x,y,z\leq n-1$, $z\notin \{x,y\}$
and $a,b\in \At(\A(n))$ with $N_u(x,y)\leq a;b$.
We can assume that $a,b\neq Id.$

Then \pe\ enumerates the $n-1$ nodes $w$ with $w\leq n-1$ and $w\neq z$, then she labels
the edges $(w_l, z)$ one by one $(l<n-1)$. Assume that $(w_l, z)$ be the next edge to label.

If $N_u(w_l, w_{l'})=Id$ for some $l'<l$, she defines $N_t(w_l,z)=N_t(w_{l'}, z).$

Otherwise, if $l<n-2$, she let $N_t(w_l,z)=a^{0}(i, r-1)$ for some $i<n-1$
such that each previously labelled edge $(w_l', x)$ for $l'<l$, including $(x,z)$ and
$(y,z)$ satisfies $N_t(w_l, z)$ is not less that $a(i)=\sum_{j<r} a(i,j)$.
She has labelled at most $n-3$ previous edges and there are
$n-2$ possible values of $i$ to choose from,
this can be done.

Finally, \pe\ lets $N_t(w_{n-2}, z)=a^0(i,j)$ for some $i<n-2$ such that no already
labelled edge $(w_{l'}, z)$, except possibly $(w_{n-2}, z)$ has a label below
$\sum_{j<r} a(i, j)$  (there are enough values of such $i$),  and with $j=r-1$
unless $N_u(w_{n-2}, w_{n-3}), N_t(w_{n-2}, z)\leq a^1(i, r-1)$ in which case $j=r-2\geq 0$.
Hence by the choice of $i$ and $l<n-3$ the triangle $(w_{n-3}, w_l, z)$ of $N_t$ is
consistent.  Now we check consistency of triples in the new hypernetwork

\begin{itemize}

\item If $N_u(w_{n-2}, w_{n-3})$, $N_t(w_{n-3}, z)\leq a(i, r),$ for some $i<n-2$ then
\pe\ chose $N_t(w_{n-1}, z)= a^0(i, r-2)$, so $(w_{n-2},w_{n-3}, z)$ is consistent.

\item Otherwise, if $N_u(w_{n-2}, w_{n-3})$ and $ N_t(w_{n-3}, z)\leq \sum_{j<r}a(i, j)$,
then as $n\geq 4$, we have
$w_{n-2}\neq x,y$, we must have $N_t(w_{n-2}, w_{n-3})\leq a(i,j)$ for some $j<r-1$.
But \pe\ chose $N_t(w_{n-1}, z)=a^{0}(i, r-1)$, thus $(w_{n-3}, w_{n-3}, z)$ is consistent.

\item Finally, if $N_u(w_{n-2}, w_{n-3})$ is not less than $a(i)=\sum_{j<r}a(i,j)$ or $N_t(w_{n-2}, z)$ is
not less than $a(i)$, consistency is clear.
\end{itemize}

\item Assume that \pa\ plays an amalgamation move say $(M, N, x, y)$ in round $t$, with $M, N\in \{N_u: u<t\}$ with
$M\equiv_{xy}N$ for distinct $x,y\in n$.

\pe\ s strategy should select a hypernetwork $N_t$ such that $M\equiv_x N_t\equiv_y N$.
If there are $u<t$ and maps $f_1, f_2:n\to n$ such that

\begin{itemize}

\item $M\equiv x N_{u}f_1$ and $N\equiv_y N_{u}f_2$

\item $N_u(f_1(z), f_2(z)=Id$ for all $z\in n\smallsetminus \{x, y\}$
\end{itemize}
Then she chooses any such $f_1, f_2$, $u$ as follows.
Let $f:n+1\to n+1$ be defined by $f(x)=f_2(x)$ and $f(z)=f_1(z)$ if $z\neq x$,
Then $N_t(y)$ is defined to be $N_{u}f$, is as required.
Otherwise she selects an atom $a^0(i,j)$ to label  $(x,y)$ such that

(a) Every triangle $N_t$ of the form $(x,y, z)$ is consistent, i.e there is no $z\leq n$ with $z\neq x,y$ such that
the triple $(a^0(i,j), M(y,z), N_t(z, x))$ violates the rules of composition.

(b) $r-1-t\leq j<r$

Then she uses hyperlabels using the default labeling.
This complete her strategy.
\end{enumarab}

Now we lift the construction to cylindric algebras. We basically use the argument in
\cite[sec 15.6]{HHbook2}. We want a cylindric algebra that is not in $S_c\Nr_m\CA_n$ with an ultrapower in
$S_c\Nr_m\CA_n$.

Write $T$ for $G_r^{m,n}$. Consider the cylindric algebra $\C=\Ca(H_m^{n}(\A(n), \omega))$. We show that $\C$ is as required.
Now $\C\notin S_c\Nr_m\CA_n$, for else $\Ra\C\in S_c\Ra\CA_n$,
and $\A(n)$ embeds completely into $\Ra\C$ which is impossible.
Now we lift the construction to cylindric algebras. We basically use the argument in
\cite[sec 15.6]{HHbook2}. We want a cylindric algebra that is not in $S_c\Nr_m\CA_n$ with an ultrapower in
$S_c\Nr_m\CA_n$.

Write $T$ for $G_r^{m,n}$. Consider the cylindric algebra $\C=\Ca(H_m^{n}(\A(n), \omega))$. We show that $\C$ is as required.
Now $\C\notin S_c\Nr_m\CA_n$, for else $\Ra\C\in S_c\Ra\CA_n$,
and $\A(n)$ embeds completely into $\Ra\C$ which is impossible.

Consider $\M=\M(\A(n), m, n, \omega, \C)$ as a
$5$ sorted structure with sorts $\A(n), \omega$, $H_m^{n}(\A(n), \omega)$, $H_{n-1}(\A(n), \omega)$.
Then \pe\ has a \ws\ in $G(T\upharpoonright 1+2r, \M)$ \cite[lemma 12.29]{HHbook2} so that
\pe\ has a \ws\ in $G(T\upharpoonright r, \M)$ for all finite $r>0$. So \pe\ has a \ws\ in
$G(T, \prod_D \M)$, for any non principal ultrafilter on $\omega$.

Hence there is a countable elementary subalgebra $\E$ of $\prod_D\M$ such
that \pe\ has a \ws\ in $G(T,\E).$ So therefore $\E$ must have the
form $\M(\B, m, n, \Lambda, \D)$  for some atomic $\B\in \RA$ countable set
$\Lambda$ and countable atomic $m$ dimensional $\D\in \CA_m$ such
that $\At\D\cong \At\Ca(H_m^n(\B, \Lambda))$.
Furthermore, we have  $\B\prec \prod _D\A(n)$  and $\D\prec \prod_D\C$.
Thus \pe\ has a \ws\ in $G(T, \M, m, n, \Lambda, \D)$
and so she also has a \ws\ in $G_{\omega}^{m,n}(\B,\Lambda)$.

So $\B\in S\Ra\CA_n$ and $\D$ embeds completely into $\Ca(H_m^n(\B, \Lambda))\in \Nr_m\CA_n$
and we are done.
In fact, one can show that both $\D$ and $\B$ are actually representable, by finding a representation of
$\prod \A(n)/F$ (or an elementary countable subalgebra of it it) embedding every
$m$ hypernetwork in such a way so that
the embedding respects $\equiv_i$ for
every $i<m$, but we do not need that much.

\end{enumarab}
\end{proof}

\begin{theorem}\label{squarerepresentation}
Assume that the dimension $m$ is finite and as usual $>2$.
\begin{enumarab}
\item The class of algebras having $k$ square representations is a variety. For $2<m<n<\omega$
the variety of $\CA_m$s 
that have $n+1$ square representations  is not finitely axiomatizable
over that consisiting of algebras having $n$ square representations.
\item For $2<m<n$, the class of $\CA_m$s having 
$n$ square representations is not finitely axiomatizable over the class that has
$n$ smooth nor $n$ flat representations.
\end{enumarab}
\end{theorem}
\begin{proof} 
For the first part. The addressed class coincides with ${\sf CB}_{n,m}$ by theorem \ref{can2}, hence it is a variety.

For the second part of (1). 
In the absence of amalgamation moves in games characterizing ${\sf CB}_{n,m}$, 
rainbows work best, so we proceed as folows:

Let $\A_r^n$, be the rainbow cylindric algebra based on $A=M[n-3, 2^{r-1}],$
and $B=M[n-3, 2^{r-1}-1]$, as defined in \cite[ lemma 17.15, 17.16, 17.17]{HHbook}.
These structures were used by Hirsch and Hodkinson to show that ${\sf RA}_n$
is not finitely axiomatizable over
${\sf RA}_{n+1}$; here ${\sf RA}_m$ is the variety of relation algebras with  $m$ dimensional relational basis.

Now we put them to a different use, lifting them to cylindric algebras using coloured graphs the usual way.
We use the game $G_r^n$, but now played on coloured graphs. 
The following can be proved using the lifting argument given in detail in theorem \ref{completerepresentation}.
This will be sketched in a while.
\begin{enumerate}
\item \pe\ has a \ws\ in the game $G_{\omega}^{n}(\A_r^n),$
\item \pe\ has a \ws\ in $G_r^{\omega}(\A_r^n),$
\item \pa\ has a \ws\ in $G_{\omega}^{n+1}(\A_r^n).$
\end{enumerate}

We have $\A_r^n\in \sf {\sf CB}_n\sim \sf CB_{n+1}$ by items $1$ and $3$. But 
\pe\ has a \ws\ in $G_r^{n+1}(\A_r^n)$ for all finite $r$ by item $2$, because the last game 
is easier for \pe\ to win than $G_r^{\omega}(\A_{n,r})$.
Indeed, in the former game  \pa\ s 
moves are restricted to only finitely many pebbles.  

The lifting argument is similar to that described  in some detail 
in the previous item. Note that in item 2, the game is the usual $r$ rounded atomic game since we have an infinite 
supply of pebbles. We omit this part of proof.

We sketch the lifting argument in item $1$ and $3$. In both cases the rounds are $\omega$.
In the first case the number of pairs of pebbles is only $n-2$, while in the last it is $n-1$.
{\it Only one pair of pebbles makes all the difference}.

The basic idea is that this extra pair in the second game, 
enables  \pe\ to defer the  \ef\ forth pebble game only on the `second part' of the structures which are 
ordinals, and because $2^{r-1}> 2^{r-1}-1$, he will win.
This is reflected in the graph game by playing cones with green tints coming form $2^{r-1}$ 
and because the indices of reds come from $2^{r-1}$,  \pa\ will win on a red clique obtained by bombarding \pe\ with cones having
green tints dictated by his private game.
There will come a point where \pe\ will be unable to deliver a consistent triple of reds. 
At this point no matter which red she chooses (forced one) the indices  will not (and cannot) 
match. 

He cannot do this in case there are $n-2$ pairs of pebbles by transferring the game to the 'second part' 
of the structures which he could do when there was one
extra pair of pebbles. 

From the `graph game' point of view  there will always be at least one pair outside the 
`threatning red clique', denoted below by $R_M(F)$, where $M$ is the current coloured graph in play,
and $F$ is the face chosen by \pa\ . Now  \pe\ has to choose a red label for the edge between $\beta$ and $\delta$
which are appexes of the two last 
cones with common base $F$ played by \pa\ . 

She uses this spare pair as follows.
She places one of the pebbles in this pair on $a$ where $a$ is the tint of the last cone with base 
$F$ played by \pe\ .
Then she places the other
on an {\it unoccupied element} of $n-3$ if there is one, else on an arbitrary element $x$ say 
of $2^{r-1}-1$.  By the fact that there are only $n-2$ pairs of pebbles there will always be
such an element. This is because in the \ef\ forth pebble game, $K_{n-2}$ is isomorphic to $S=n-3 \cup \{x\}$
hence $B\models y<y'$ for any distinct $y, y'\in S$, so 
that the position of pebbles always 
defines a homomorphism from $A$ to $B$.
She then labels the edge between $\beta$ and $\delta$ with $\r_{\mu(\beta), b}$.
We need to check that this startegy works.

Lets get more formal. 
Fix $r\geq m$. We check only the critical part, 
namely, when in the graph game, \pe\  has to label the edge between 
$\beta$ and $\delta$ that are both apexes of two cones on $F$, where $F$ is the face played by 
\pa\ , and she is forced to play a red. She has to do that in such a way that indices match.

Assume that the tint of the cone determined by $\beta$ is $a\in A$, and the two cones
induce the same linear ordering on $F$. Recall that we have $\beta\notin F$, but it is in $M$, while $\delta$ is not in $M$,
and that $|F|=m-1$.
Now \pe\ has no choice but to pick a  red colour, so she does this as follows. First 
she  defines the red clique $$R_{M}(F)=\{x\in M: M(x,\beta)=\g_0^i,\text { for some $i\in A$},  F\cup \{x\}$$
$$\text { is an $i$ cone with base $F$ and appex $x$}\}.$$
We have $|R_{M}(F)|< n-2$, hence there are  fewer than $n-2$ pairs of pebbles in play.
As before \pe\ picks up a spare pebble  pair, and playing
the role of \pa\ places one of the pebbles in the pair on $a$.
She uses her \ws\ to respond by placing the other
one on $b\in B$. Then she labels the edge between $\beta$ and $\delta$ with $\r_{\mu(\beta), b}$.
Checking that this is a \ws\  for \pe\ is like before. Notice that here $|R_{M}(F)|< n-2$ for any
{\it finite $r$},  so this \ws\ works for all 
$\A^n_r$s.
 
On the other hand, \pa\ has a \ws\ in his private \ef\ game game $EF_{\omega}^{n-1}(A, B)$.  He always 
place the pebbles on distinct elements of $A$.
He uses rounds $0,\ldots n-2$, to cover $n-3$ and top two elements
of $2^{r-1}$ with $n-1$ pebbles.  Then \pa\ can force \pe\ to play a two pebble game of
length $\omega$ on $2^{r-1}$ and $2^{r-1}-1$, which he can win because the former is longer than the latter.
This is reflected in the graph game 
by bombarding \pe\ with cones having the same base, namely, the induced face by \pa\ 's move, and 
different green tints dictated by his  private game. Because $2^{r-1}>2^{r-1}-1$, and the length of the game is $\omega$,
there will be a point, no matter how large is $r$ (remember it is finite), 
where \pe\ is forced an inconsistent red.

We conclude, that 
\pe\ has a \ws\ in $G_{\omega}^{n+1}(\prod_{r}\A_r^n/D)$, for any non principal ultrafilter $D$,
so the latter is in ${\sf  CB}_{n+1}.$

The second  item follows  
from theorem \ref{infinitedistance} below.
\end{proof}

\subsection{Relativized smooth complete representations}

Recall the definition of relativized representation, definition \ref{rel}.
Here $m<\omega$ denotes the dimension and ${\sf CRA}_{m,n}$ denotes the class of $\CA_m$s with an $n$ relativized
smooth representation.
There is no restriction whatsoever on $n$ except that it is $>m$. In particular $n$ can be infinite.

\begin{theorem}\label{longer} Regardless of cardinalities, $\A\in {\sf CRA}_{m, \omega}$
iff \pe\ has a \ws\ in $G_{\omega}^{\omega}$.
\end{theorem}
\begin{proof} One side is obvious. Now assume that \pe\ has a \ws\ in the $\omega$ rounded game, using $\omega$ many pebbles.
We need to build an $\omega$ relativized complete representation.
The proof goes as follows. First the atomic networks are finite, so we need to convert them into $\omega$ dimensional
atomic networks. For a network $N$, and  a map $v:\omega\to N$,
let $Nv$ be the network  induced by $v$, that is $Nv(\bar{s})=N(v\circ \bar{s})$.
let $J$ be the set of all such $Nv$, where $N$ occurs in some play
of $G_{\omega}^{\omega}(\A)$ in which \pe\ uses his \ws\ and $v:\omega\to N$ (so via these maps we are climbing up $\omega$).

This  can be checked to be  an $\omega$ dimensional hyperbasis (extended to the cylindric case the obvious way).
We can use that the basis consists of $\omega$ dimensional atomic networks, such that for
each such network, there is a finite bound on the size of its strict networks.
Then a complete $\omega$ relativized representation  can be obtained in a step by step way, requiring inductively
in step $t$, there for any finite clique $C$ of $M_t$, $|C|<\omega$, there is
a network in the base, and an embedding $v:N\to M_t$ such that $\rng v\subseteq C$.
Here we consider finite sequences of arbitrarily large length, rather than fixed length $n$ tuples.
Now that we have an $\omega$ dimensional hyperbases, we proceed as follows, we build a complete $\omega$ relativized representation $M$
in a step- by-step fashion; it will be a labelled complete hypergraph. We distinguish between $m$ hyperedges and
hyperedges of arbitrary finite length $\neq m$, the labels are different, they are {\it not } atoms.
The proof is similar to the proof of theorem \ref{step} except that hyperedges can have arbitrary finite length.
So in this new situation $M$ is required to  satisfy  the following properties:

\begin{enumarab}

\item Each $m$ hyperedge of of $M$ is labelled by an atom of $\A.$

\item $M(\bar{x})\leq {\sf d}_{ij}$ iff $x_i=x_j$. (In this case, we say that $M$ is strict).

\item For any clique $\bar{x}\in M$ of arbitrary finite length $>n$
there is a unique $N\in H$ such that $\bar{x}$ is labelled by $N$ and we write this as
$M(\bar{x})=N.$

\item  If $x_0,\ldots, x_{l-1}\in M$ and $M(\bar{x})=N\in H$, $l\geq m$
then for all $i_0\ldots i_{m-1}<n$, $(x_{i_0}\ldots x_{i_{m-1}})$
is a hyperedge and $M(x_{i_0}, x_{i_1}, \ldots,  x_{i_{m-1}})=N(i_0,\ldots, i_{m-1})\in \At\A$.

\item $M$ is symmetric; it closed under substitutions.
\item
If $\bar{x}$ is a clique of arbitrary length  $k<|x|$ and $N\in H$, then $M(\bar{x})\equiv_k N$ if
and only if there is
$y\in M$ such that $M(x_0,\ldots, x_{k-1}, y, x_{k+1}, \ldots, x_l, \ldots )=N.$

\item For every $N\in H$, there are $x_0,\ldots,  x_{l-1}\in M$, $M(\bar{x})=N.$

\end{enumarab}

We build a chain of hypergraphs $M_t:t<\omega$ their limit (defined in a precise
sense) will as required.  Each $M_t$ will have hyperedges of length $m$ labelled by
atoms of $\A$, the labelling of other hyperedges will be done later.
We require inductively that $M_t$ also satisfies:

Any clique in $M_t$ is contained in $\rng(v)$ for some $N\in H$ and some embedding $v:N\in M_t$, that is
such that $v:\dom(N)\to \dom(M_t)$
and this embedding satisfies the following two conditions:

(a) if $(v(i_0), \ldots, v(i_{m-1}))$ is a an atom  hyperedge of $M_t$, then
$M_t(\bar{v})=N(i_0,\ldots, i_{n-1}).$

(b) Whenever $a$ is a sequence of arbitrary length with $|a|\neq m$, then $v(a)$ is a
hyperedge of $M_t$ and is labelled by $N(\bar{a})$.

These tasks can be carried out obtaining the limit hypergraph $M$.
Let $L(A)$ be the signature obtained by adding an $n$ ary relation symbol for each element
of $\A$.
Define $M\models r(\bar{x})$ if $\bar{x}$ is an atom  hyperedge and $M(\bar{x})\leq r$
Then it can be checked with minor straightforward changes in the proof of theorem \ref{step} that
$M$ is the required complete relativized representation.
The idea here  is because an $\omega$ relativized representation only requires cylindrifier witnesses over
finite sized  cliques, that can get arbitrarily long, and they are not necessarily cliques that are uniformly bounded.
The last condition is not enough.
\end{proof}

In the following theorem and elsewhere, ${\bf K}_n$ for $n\leq \omega$ stands for the complete irreflexive graph with $n$ nodes.

\begin{theorem}\label{rel}
\begin{enumarab}
\item ${\sf CRA_m}\subset {\sf CRA}_{m,\omega}$, the strict inclusion can be only witnessed on
uncountable algebras. Furthermore, the class ${\sf CRA}_{m,\omega}$ is not elementary.
The classes $S_c\Nr_m\CA_{\omega}$, ${\sf CRA}_m$, ${\sf CBA}_{m,\omega}$ and ${\sf CRA}_{m, \omega}$,
coincide on atomic countable
algebras.
\item ${\sf CBA}_{m,\omega}\subset \bigcap_{n\in \omega}{\sf CBA}_{m,n}.$
\end{enumarab}
\end{theorem}

\begin{proof} The first item. 
That ${\sf CRAS}_{m,\omega}$ is not elementary is witnessed by the rainbow algebra
$\PEA_{{\bf K}_{\omega}, {\bf K}}$, where the latter
is a disjoint union of ${\bf K}_n$, $n\in \omega$. \pe\ has a \ws for all finite rounded 
games, but \pa\ can win the infinite rounded game.

And so the usual argument \pe\ can win the transfinite game on an uncountable non-trivial ultrapower of $\A$,
and using elementary chains one can find an elementary countable
subalgebra $\B$ of this ultrapower such that \pe\ has a \ws\ in the $\omega$ rounded game.
This $\B$ will have an $\omega$ square representation
hence will be in $\sf CRA_{m,\omega}$, and $\A$ is not in the latter class.

Indeed \pa\ can win $G_{\omega}$ as follows. In the $\omega$ rounded game \pa\ bombards \pe\ with cones have the same base,
and tints $\{\g^0_i: i<\omega_1\}$. Now \pe\ has to label any edge within
the red clique formed by $\r_{ij}$ for some distinct $i,j<\omega$, and
these indices must match in a triangle within the clique.
But an uncountable red clique is impossible because $\omega$ is
countable.

The three classes coincide on countable atomic algebras
with the class of completely representable algebras \cite[theorem 5.3.6]{Sayedneat}, by noting
that on countable and atomic algebras, an $\omega$ complete relativized representation is a  complete relativized representation.

For the second part the same rainbow algebra shows strictness of the inclusion
because obviously \pe\ can always win the \ef\ forth game  $EF_{\omega}^n({\bf K}_{\omega}, {\bf K})$, with $n$ nodes and $\omega$ rounds
(here subscripts and superscripts used in \ef\ game 
above are reversed) by placing a pebble on $\bf K_n$, hence it has a \ws\ in 
$G_{\omega}^n(\CA_{{{\bf K}_{\omega}}, \bold K})$. But we know that \pa\ can win the $\omega$ rounded game, and we are done.

\end{proof}

\section{Neat atom structures}

\subsection{Neat games}

Next  we introduce several definitions an atom structures concerning neat embeddings.
Here we denote the dimension by $n$, where $n$ is finite. $n$ will be always $>1$ and more often than not greater than $2$.

We need to recall, see the line before theorem \ref{neat11}, and highlight the following definition that will be used frequently in what follows:

\begin{definition}\label{s} Let $\K$ be a class of algebras having a Boolean reduct and assume that $\A$ and $\B$ are in $\K$.
$\A$ is a {\it strong subalgebra} of $\B$, if $\A$ is a subalgebra of
 $\B$, with the additional condition that whenever $X\subseteq \A$ is such that
$\sum ^{\A}X=1$, then $\sum ^{\B}X=1$. We write $S_c\K$ for the class of strong subalgebras of algebras in $\K$.
We write $\A\subseteq _c \B$ if $\A\in S_c\{\B\}$ or more simply if $\A$
is a strong subalgebra of $\B$.
\end{definition}

For example if $\A$ is a finite subalgebra of $\B$,  then it is a strong subalgebra of $\B$.
Another example is that $\A$ is a strong subalgebra of its \d\ completion.

\begin{definition}
\begin{enumarab}

\item Let $1\leq k\leq \omega$.  An atom structure $\alpha$ {\it $k$ weakly representable}, if the term algebra is in $\RCA_n$ but
the complex algebra is not in $S\Nr_n\CA_{n+k}$.

\item Let $1\leq k\leq \omega$. An atom structure $\alpha$ is {\it $k$ weakly  neat representable},
if the term algebra is in $\RCA_n\cap \Nr_n\CA_{n+k}$, but the complex algebra is not representable.
(This means that $\alpha$ is also $\omega$ weakly representable).

\item Let $m>n$. An atom structure  $\alpha$ is {\it $m$ neat}, $m>n$,  
if there is an atomic algebra $\A$, such that $\At\A=\alpha$ and $\A\in \Nr_n\CA_{m}.$
\item 
Let $m>n$. An atom structure $\alpha$ is {\it $m$ complete}
if there exists $\A$ such that $\At\A=\alpha$ and $\A\in S_c\Nr_n\CA_{m}$.
\end{enumarab}
\end{definition}
\begin{definition} Let $\K\subseteq \CA_n$, and $\L$ be an extension of first order logic.
$\K$ is {\it detectable} in $\L$, if for any $\A\in \K$, $\A$ atomic, and for any atom structure
$\beta$ such that $\At\A\equiv_{\L}\beta$,
if $\B$ is an atomic algebra such that $\At\B=\beta$, then $\B\in \K.$
\end{definition}
Roughly speaking, $\K$ is detectable in $\L$ if whenever an atomic algebra is not in $\K$, then $\L$ can {\it witnesses this}.
In particular, a class that is not detectable
in first order logic is simply non-elementary. A class that is not detectable by quasi (equations) is not a (quasi) variety.

We investigate the existence of such structures and the interconnections.
Note that if $\L_1$ is weaker than $\L_2$ and $\K$ is not detectable
in $\L_2$, then it is not detectable in $\L_1$.
We also present several $\K$s and $\L$s as in the second definition.
Note that $m$ neat implies $m$ complete,
but the converse is false as will be shown in a minute, witness item (4) in the following theorem
\ref{neat}.
Though formulated only for $\CA$s,
all our coming results readily extend to Pinter's algebras and polyadic algebras with and without equality
(and all signatures in between).

Recall that an atomic algebra is atomic strongly
representable if its \d\ completion is representable.

We now prove:
\begin{theorem}\label{main}
\begin{enumerate}
\item There exists a $k$ weakly representable atom structure, for all $k\geq 4$.

\item  For $k\geq 1$, there exists a $k+1$ weakly representable countable atom structure that is $k$ weakly neat 
representable.  In particular, there is a consistent countable $L_n$ theory $T$, and a non-principal type $\Gamma$, 
realized in all $n+k+1$ smooth models, but does not have
an $n+k$ witness.

\item Let $n$ be finite $n\geq 3$. Then there exists a countable weakly $k$ neat atom structure 
of dimension $n$ if and only if $k<\omega.$

\item There is an $\omega$ rounded game $J$ that determines $\omega$ neat atom structures, in the sense that if $\A$ is atomic and \pe\ has a 
\ws\ in $J$, 
then $\At\A\in \At\Nr_n\CA_{\omega}$. Furthermore, there is  an atom structure that is not $n+3$ complete, hence not $n+k$ neat for all $k\geq 3$,
but is elementary equivalent to one that is $\omega$ neat, hence $m$ neat for any 
$m\geq n$.
\end{enumerate}
\end{theorem}

\begin{proof}

\begin{enumerate}

\item  From theorems \ref{can}, \ref{smooth}, \ref{blowupandblur}

\item  Let $m=n+k$. Then the required follows, using theorem \ref{blurs}, by noting that the Hirsch Hodkinson Monk-like algebras 
$\A(m,r)$, for $1\leq r<\omega$ of \cite[definition 15.2]{HHbook2} has a 
complex set of $m$ blurs. (Such relation algebras will be put to an entirely different use 
in theorem \ref{2.12}.) The rest will follow by the reasoning in theorem 
\ref{OTT}.

Let us check this. For definitness, we take the parameter $r$ to be $1$. Set 
$\R=\A(m, 1)$. 

Let $E$ be a tenary relation on $\omega$, where 
$E(i,j,k)$ abbreviates that $i,j,k$ are {\it evenly distributed}, that is, 
$$E(i,j,k)\text { iff } (\exists p,q,r)\{p,q,r\}=\{i,j,k\}, r-q=q-p.$$

Using the notation in the proof of theorem \ref{blurs}, 
by taking $l\geq 2m-1$ and $|\At\R|\geq 2(m-1)l$, we have {\it enough $m$} blurs, namely,
$(J,E)$, or simply $J$ as defined in {\it op.cit} (because we have only one tenary relation).

In other words, $J$
is a complex blur for $\R$, cf. \cite[definition 3.1, together with
the condition $(J5)_n$ p. 79]{ANT}. 
Here $J$ is the set of all subsets of 
$\At\R$ whose size is $l$.

To prove  complex blurness as claimed, we 
have to check the following, using the notation of \cite{ANT}:

\begin{itemize}
\item Each element of $J$ is non empty
\item $\bigcup J=I$
\item $(\forall P\in I)(\forall W\in J)(I\subseteq P;W)$
\item $(\forall V_1,\ldots, V_m, W_2, \ldots, W_m\in J)(\exists T\in J)(\forall 2\leq i\leq m)
{\sf safe}(V_i,W_i,T)$, meaning that there is no inconsistent triangles whose elements are taken from $(V_i, W_i, T)$, respectively.
that is to say $a\leq b;c$, for all $a\in V_i$, $b\in W_i$, and $c\in T$.

\item $(\forall P_2,\ldots, P_m, Q_2,\ldots, Q_m\in I)(\forall W\in J)W\cap P_2;Q_m, \cap \ldots, P_m;Q_m\neq \emptyset$.
\end{itemize}

The first three itmes are straightforward. 

We check the fourth item. Let 
$$V_1,\ldots, V_n, W_2,\ldots W_m\in J.$$
Then $U=\bigcup \{V_i\cup W_i: 2\leq i\leq m\}$ has cardinality at
most $(2m-2)l$, hence the cardinality of $I-U$ is $\geq k-(2m-2)l$, and ${\sf safe}(V_i,W_i,T)$
because $V_i\cap W_i\cap T=\emptyset$ for all $2\leq i\leq m$.
Now we check the final condition. Let $P_2,\ldots, P_m, Q_2, \ldots, Q_m\in I$
be arbitrary, then $U=\{P_2,\ldots, Q_m\}$ has cardinality $\leq 2m-2$, and
so each $W\in J$ contains an $S\notin U$, by $l\geq 2m-1$, and $S\leq P_2;Q_2\ldots P_m;Q_m$ by the
definition of composition in $\A(m,1)$.

\item  The if part follows from \cite[theorem 2.1]{ANT}, 
see also theorem \ref{blurs} above. 

The only if part
says that $k$ cannot be infinite. Indeed, if it were,
then the term algebra will be in $\Nr_n\CA_{\omega}$, hence by \cite[5.3.6]{Sayedneat}
would be completely representable, which makes its atom
structure $\At$ strongly representable \cite[3.5.1]{HHbook2},
but then $\Cm\At$ will be representable. So in this sense the result is best possible,
but it can be strengthened in other ways, witness theorem \ref{blurs}.

\item  We describe an $\omega$ rounded game $J$, cf. \cite{r} and \cite{recent}.
We call $J$ a {\it neat } game, since it aims at capturing a two sorted structure, namely, 
a neat reduct. Since it is an {\it atomic} game played on atomic hypernetworks, 
it will turn out that it captures {\it only} the {\it atom structures} of neat reducts. This is
substantially weaker (see item (1) of theorem \ref{SL}), than capturing a neat reduct.  

However, in item (4) of theorem \ref{SL} 
we devise another  stronger {\it neat game}, denoted by $H$, 
that captures a neat reduct. 
$H$ will {\it not} be an atomic game, 
for \pe\ is allowed sometimes to choose an arbitrary element of the algebra
on which the game is played to label hyperedges of certain networks with no consistency conditions - these do not exist in the game $J$ -
played by \pa\, that is not necessarily an atom. However, in the hypernetwork part of the game, like in the weaker neat game $J$,
\pa\ can deliver only atoms and labels for hyperedges.

More succintly, if \pe\ has a \ws\ in the atomic game $J_k$  
($J$ truncated to $k$ rounds) for all $k$ on an atomic 
algebra $\A$ with countably many atoms, 
then this algebra will have an elementary equivalent algebra $\B$ such that 
$\At\B\in \At\Nr_n\CA_{\omega}.$

On the other hand,  if she has a \ws\ on $H_k$ ($H$ truncated to $k$ rounds) for all $k$, 
on such an $\A$, then $\A$ will have an elementary equivalent algebra $\B$,
such that $\B\in \Nr_n\CA_{\omega}$. 
This is much stronger, for like the case of representable algebras, and unlike the case of {\it completely}  representable algebra
even in the relativized sense, 
we may well have algebras $\A, \B\in \CA_n$ $(n>1$), 
and even more in $\RCA_n$ 
such that $\At\A=\At\B$, $\A\in \Nr_n\CA_{\omega}$ but $\B\notin \Nr_n\CA_{n+1}$, {\it a fortiori}
$\B\notin \Nr_n\CA_{\omega}$, witness
item (1) in \ref{SL}, and item(9) in theorem \ref{maintheorem}.

Using the terminology in item (5) of definition \ref{grip}, 
$H$ is gripping for $\Nr_n\CA_{\omega}$ but $J$ is not.

For an atomic network and for  $x,y\in \nodes(N)$, we set  $x\sim y$ if
there exists $\bar{z}$ such that $N(x,y,\bar{z})\leq {\sf d}_{01}$.
The equivalence relation $\sim$ over the set of all finite sequences over $\nodes(N)$ is defined by
$\bar x\sim\bar y$ iff $|\bar x|=|\bar y|$ and $x_i\sim y_i$ for all
$i<|\bar x|$.(It can be checked that this indeed an equivalence relation.)

A \emph{ hypernetwork} $N=(N^a, N^h)$ over an atomic polyadic equality algebra $\C$
consists of a network $N^a$
together with a labelling function for hyperlabels $N^h:\;\;^{<
\omega}\!\nodes(N)\to\Lambda$ (some arbitrary set of hyperlabels $\Lambda$)
such that for $\bar x, \bar y\in\; ^{< \omega}\!\nodes(N)$
\begin{enumerate}
\renewcommand{\theenumi}{\Roman{enumi}}
\setcounter{enumi}3
\item\label{net:hyper} $\bar x\sim\bar y \Rightarrow N^h(\bar x)=N^h(\bar y)$.
\end{enumerate}
If $|\bar x|=k\in \N$ and $N^h(\bar x)=\lambda$ then we say that $\lambda$ is
a $k$-ary hyperlabel. $(\bar x)$ is referred to a a $k$-ary hyperedge, or simply a hyperedge.
(Note that we have atomic hyperedges and hyperedges)
When there is no risk of ambiguity we may drop the superscripts $a,
h$.
There are {\it short} hyperedges and {\it long} hyperedges (to be defined in a while). The short hyperedges are constantly labelled.
The idea (that will be revealed during the proof), is that the atoms in the neat reduct are no smaller than the atoms
in the dilation. (When $\A=\Nr_n\B,$ it is common to call $\B$ a dilation of $\A$.)

We know that there is a one to one correspondence between networks and coloured graphs.
If $\Gamma$ is a coloured graph, then by $N_{\Gamma}$
we mean the corresponding network defined on $n-1$ tuples of the nodes of $\Gamma$ to
to coloured graphs of size $\leq n$.
\begin{itemize}
\item A hyperedge $\bar{x}\in {}^{<\omega}\nodes (\Gamma)$ of length $m$ is {\it short}, if there are $y_0,\ldots, y_{n-1}\in \nodes(N)$, such that
$$N_{\Gamma}(x_i, y_0, \bar{z})\leq {\sf d}_{01}\text { or } N(_{\Gamma}(x_i, y_1, \bar{z})\ldots$$ 
$$\text { or }N(x_i, y_{n-1},\bar{z})\leq {\sf d}_{01}$$ 
for all $i<|x|$, for some (equivalently for all)
$\bar{z}.$

Otherwise, it is called {\it long.}
\item A hypergraph $(\Gamma, l)$
is called {\it $\lambda$ neat} if $N_{\Gamma}(\bar{x})=\lambda$ for all short hyper edges.
\end{itemize}
This game is similar to the games devised by Robin Hirsch in \cite[definition 28]{r}, played on relation algebras.
However, lifting it to cylindric algebras is not straightforward, for in this new context the moves involve hyperedges of length $n$ (the dimension),
rather than edges. In the $\omega$ rounded game $J$,  \pa\ has three moves.
$J_n$ will denote the game $J$ truncated to $n$ rounds.

The first is the normal cylindrifier move. As usual, there is no polyadic move for the polyadic information is coded in the networks.
The next two are amalgamation moves.
But the games are not played on hypernetworks, they are played on coloured hypergraphs, consisting of two parts,
the graph part
that can be viewed as an $L_{\omega_1, \omega}$ model for the rainbow signature, and the part dealing with hyperedges with a
labelling function.
The amalgamation moves roughly reflect the fact, in case \pe\ wins, then for every $k\geq n$ there is a $k$ dimensional hyperbasis,
so that the small algebra embeds into cylindric algebras of arbitrary large dimensions.
The game is played on $\lambda$ neat hypernetworks,  translated to $\lambda$ neat hypergraphs,
where $\lambda$ is a label for the short hyperedges.

For networks $M, N$ and any set $S$, we write $M\equiv^SN$
if $N\restr S=M\restr S$, and we write $M\equiv_SN$
if the symmetric difference
$$\Delta(\nodes(M), \nodes(N))\subseteq S$$ and
$M\equiv^{(\nodes(M)\cup\nodes(N))\setminus S}N.$ We write $M\equiv_kN$ for
$M\equiv_{\set k}N$.

Let $N$ be a network and let $\theta$ be any function.  The network
$N\theta$ is a complete labelled graph with nodes
$\theta^{-1}(\nodes(N))=\set{x\in\dom(\theta):\theta(x)\in\nodes(N)}$,
and labelling defined by
$$(N\theta)(i_0,\ldots, i_{n-1}) = N(\theta(i_0), \theta(i_1), \ldots,  \theta(i_{n-1})),$$
for $i_0, \ldots, i_{n-1}\in\theta^{-1}(\nodes(N))$.

We call this game $H$. It is $\omega$ rounded. Its first move by \pa\ is the usual cylindrifier move (equivalently)
\pa\ s move in $G^{\omega}$),  but \pa\ has more moves which makes it harder for
\pe\ to win. 

\pa\ can play a \emph{transformation move} by picking a
previously played hypernetwork $N$ and a partial, finite surjection
$\theta:\omega\to\nodes(N)$, this move is denoted $(N, \theta)$.  \pe\
must respond with $N\theta$.

Finally, \pa\ can play an
\emph{amalgamation move} by picking previously played hypernetworks
$M, N$ such that $M\equiv^{\nodes(M)\cap\nodes(N)}N$ and
$\nodes(M)\cap\nodes(N)\neq \emptyset$.
This move is denoted $(M,
N)$. Here there is {\it no restriction} on the number of overlapping nodes, so in principal
the game is harder for \pe\ to win, if there was (Below we will encounter such restricted 
amalgamation moves).

To make a legal response, \pe\ must play a $\lambda_0$-neat
hypernetwork $L$ extending $M$ and $N$, where
$\nodes(L)=\nodes(M)\cup\nodes(N)$.

The following is a cylindric-like variation on \cite[theorem 39]{r}, formulated for polyadic equality algebras,
but it works for many cylindric-like algebras like $\sf Sc$, $\sf PA$ and $\sf CA$.
But it does not apply to $\Df$s because the notion of neat reducts for $\Df$s is trivial.
We abbreviate it by (*) because it will be referred to below:

{\it Let $\A$ be an atomic polyadic equality algebra with a countable atom structure $\alpha$.
If \pe\ can win the $\omega$ rounded game $J$ on $\alpha$,
then there exists a locally finite $\PEA_{\omega}$ such that
$\At\A\cong \At\Nr_n\C$. 
Furthermore, $\C$ can be 
chosen to be complete, in which case we have $\Cm\At\A=\Nr_n\C$.}

We give the proof which is similar to that that of \cite[theorem 39]{r}, by using a first order language
that contains $n$ -ary predicate symbols for elements of the algebra, rather than binary one. Other than that the proof
is almost the same.  In a while, in theorem \ref{j},  we show that the on the rainbow atom structure $\PEA_{\Z,\N}$, \pe\ does have
a \ws\ .

Fix some $a\in\alpha$. Using \pe\ s \ws\ in the game of neat hypernetworks, one defines a
nested sequence $N_0\subseteq N_1\ldots$ of neat hypernetworks
where $N_0$ is \pe's response to the initial \pa-move $a$, such that
\begin{enumerate}
\item If $N_r$ is in the sequence and
and $b\leq {\sf c}_lN_r(f_0, \ldots, x, \ldots  f_{n-2})$,
then there is $s\geq r$ and $d\in\nodes(N_s)$ such
that 
$$N_s(f_0,\ldots, f_{l-1}, d, f_{l+1}, \ldots, f_{n-2})=b.$$
\item If $N_r$ is in the sequence and $\theta$ is any partial
isomorphism of $N_r$ then there is $s\geq r$ and a
partial isomorphism $\theta^+$ of $N_s$ extending $\theta$ such that
$\rng(\theta^+)\supseteq\nodes(N_r)$.
\end{enumerate}
Now let $N_a$ be the limit of this sequence, that is $N_a=\bigcup N_i$, the labelling of $n-1$ tuples of nodes
by atoms, and the hyperedges by hyperlabels done in the obvious way.
This limit is well-defined since the hypernetworks are nested.
We shall show that $N_a$ is the base of a weak set algebra having unit  $V={}^{\omega}N_a^{(p)}$,
for some fixed sequence $p\in {}^{\omega}N_a$.

Let $\theta$ be any finite partial isomorphism of $N_a$ and let $X$ be
any finite subset of $\nodes(N_a)$.  Since $\theta, X$ are finite, there is
$i<\omega$ such that $\nodes(N_i)\supseteq X\cup\dom(\theta)$. There
is a bijection $\theta^+\supseteq\theta$ onto $\nodes(N_i)$ and $\geq
i$ such that $N_j\supseteq N_i, N_i\theta^+$.  Then $\theta^+$ is a
partial isomorphism of $N_j$ and $\rng(\theta^+)=\nodes(N_i)\supseteq
X$.  Hence, if $\theta$ is any finite partial isomorphism of $N_a$ and
$X$ is any finite subset of $\nodes(N_a)$ then
\begin{equation}\label{eq:theta}
\exists \mbox{ a partial isomorphism $\theta^+\supseteq \theta$ of $N_a$
 where $\rng(\theta^+)\supseteq X$}
\end{equation}
and by considering its inverse we can extend a partial isomorphism so
as to include an arbitrary finite subset of $\nodes(N_a)$ within its
domain.
Let $L$ be the signature with one $n$ -ary relation symbol ($b$) for
each $b\in\alpha$, and one $k$-ary predicate symbol ($\lambda$) for
each $k$-ary hyperlabel $\lambda$. We work in usual first order logic.

Here we have a sequence of variables of order type $\omega$, and two 'sorts' of formulas,
the $n$ relation symbols using $n$ variables; roughly
these that are built up out of the first $n$ variables will determine the atoms of
neat reduct, the $k$ ary predicate symbols
will determine the atoms of algebras of higher dimensions as $k$ gets larger;
the atoms in the neat reduct will be no smaller than the atoms in the dilations.

This process will be interpreted in an infinite weak set algebra with base $N_a$, whose elements are
those  assignments satisfying such formulas.

For fixed $f_a\in\;^\omega\!\nodes(N_a)$, let
$U_a=\set{f\in\;^\omega\!\nodes(N_a):\set{i<\omega:g(i)\neq
f_a(i)}\mbox{ is finite}}$.

Notice that $U_a$ is weak unit (a set of sequences agreeing cofinitely with a fixed one.)
We can make $U_a$ into the universe an $L$ relativized structure ${\cal N}_a$;
here relativized means that we are only taking those assignments agreeing cofinitely with $f_a$,
we are not taking the standard square model.
However, satisfiability  for $L$ formulas at assignments $f\in U_a$ is defined the usual Tarskian way, except
that we use the modal notation, with restricted assignments on the left:

We can make $U_a$ into the base of an $L$-structure ${\cal N}_a$ and
evaluate $L$-formulas at $f\in U_a$ as follow.  For $b\in\alpha,\;
l_0, \ldots l_{\mu-1}, i_0 \ldots, i_{k-1}<\omega$, \/ $k$-ary hyperlabels $\lambda$,
and all $L$-formulas $\phi, \psi$, let
\begin{eqnarray*}
{\cal N}_a, f\models b(x_{l_0}\ldots,  x_{l_{n-1}})&\iff&N_a(f(l_0),\ldots,  f(l_{n-1}))=b\\
{\cal N}_a, f\models\lambda(x_{i_0}, \ldots, x_{i_{k-1}})&\iff&  N_a(f(i_0), \ldots, f(i_{k-1}))=\lambda\\
{\cal N}_a, f\models\neg\phi&\iff&{\cal N}_a, f\not\models\phi\\
{\cal N}_a, f\models (\phi\vee\psi)&\iff&{\cal N}_a,  f\models\phi\mbox{ or }{\cal N}_a, f\models\psi\\
{\cal N}_a, f\models\exists x_i\phi&\iff& {\cal N}_a, f[i/m]\models\phi, \mbox{ some }m\in\nodes(N_a)
\end{eqnarray*}
For any $L$-formula $\phi$, write $\phi^{{\cal N}_a}$ for the set of all $n$ ary assignments satisfying it; that is
$\set{f\in\;^\omega\!\nodes(N_a): {\cal N}_a, f\models\phi}$.  Let
$D_a = \set{\phi^{{\cal N}_a}:\phi\mbox{ is an $L$-formula}}.$
Then this is the universe of the following weak set algebra
\[\D_a=(D_a,  \cup, \sim, {\sf D}_{ij}, {\sf C}_i)_{ i, j<\omega}\]
then  $\D_a\in\RCA_\omega$. (Weak set algebras are representable).

Now we consider formulas in more than $n$ variables, corresponding to the $k$ hyperlabels.
Let $\phi(x_{i_0}, x_{i_1}, \ldots, x_{i_k})$ be an arbitrary
$L$-formula using only variables belonging to $\set{x_{i_0}, \ldots,
x_{i_k}}$.  

Let $f, g\in U_a$ (some $a\in \alpha$) and suppose that $\{(f(i_j), g(i_j): j\leq k\}$
is a partial isomorphism of ${\cal N}_a$ (viewed as a weak mode), then one can easily prove by induction over the
quantifier depth of $\phi$ and using (\ref{eq:theta}), that
\begin{equation}
{\cal N}_a, f\models\phi\iff {\cal N}_a,
g\models\phi\label{eq:bf}\end{equation}

For any $L$-formula $\phi$, write $\phi^{{\cal N}_a}$ for
$\set{f\in\;^\omega\!\nodes(N_a): {\cal N}_a, f\models\phi}$.  Let
$Fm^{{\cal N}_a} = \set{\phi^{{\cal N}_a}:\phi\mbox{ is an $L$-formula}}$
and define a cylindric algebra
\[\D_a=(Fm^{{\cal N}_a},  \cup, \sim, {\sf D}_{ij}, {\sf C}_i, i, j<\omega)\]
where ${\sf D}_{ij}=(x_i= x_j)^{{\cal N}_a},\; {\sf C}_i(\phi^{{\cal N}_a})=(\exists
x_i\phi)^{{\cal N}_a}$.  Observe that $\top^{{\cal N}_a}=U_a,\; (\phi\vee\psi)^{{\cal N}_a}=\phi^{\c
N_a}\cup\psi^{{\cal N}_a}$, etc. Note also that $\D$ is a subalgebra of the
$\omega$-dimensional cylindric set algebra on the base $\nodes(N_a)$,
hence $\D_a\in {\sf Lf}_{\omega}\cap {\sf Ws}_\omega$, for each atom $a\in \alpha$, and is clearly complete.

Let $\C=\prod_{a\in \alpha} \D_a$. (This is not necessarily locally finite).
Then  $\C\in\RCA_\omega$, and $\C$ is also complete, will be shown to be is the desired generalized weak set algebra,
that is the desired dilation.
Note that unit of $\C$ is the disjoint union of the weak spaces.

An element $x$ of $\C$ has the form
$(x_a:a\in\alpha)$, where $x_a\in\D_a$.  For $b\in\alpha$ let
$\pi_b:\C\to \D_b$ be the projection map defined by
$\pi_b(x_a:a\in\alpha) = x_b$.  Conversely, let $\iota_a:\D_a\to \c
C$ be the embedding defined by $\iota_a(y)=(x_b:b\in\alpha)$, where
$x_a=y$ and $x_b=0$ for $b\neq a$.  Evidently $\pi_b(\iota_b(y))=y$
for $y\in\D_b$ and $\pi_b(\iota_a(y))=0$ if $a\neq b$.

Suppose $x\in\Nr_n\C\setminus\set0$.  Since $x\neq 0$,
it must have a non-zero component  $\pi_a(x)\in\D_a$, for some $a\in \alpha$.
Assume that $\emptyset\neq\phi(x_{i_0}, \ldots, x_{i_k})^{\D_a}= \pi_a(x)$ for some $L$-formula $\phi(x_{i_0},\ldots, x_{i_k})$.  We
have $\phi(x_{i_0},\ldots, x_{i_k})^{\D_a}\in\Nr_{n}\D_a$.  Let
$f\in \phi(x_{i_0},\ldots, x_{i_k})^{\D_a}$ and let
$$b=N_a(f(0), f(1), \ldots f(n-1))\in\alpha.$$
We first show that
$b(x_0, x_1, \ldots, x_{n-1})^{\D_a}\subseteq
 \phi(x_{i_0},\ldots, x_{i_k})^{\D_a}$.  Take any $g\in
b(x_0, x_1\ldots x_{n-1})^{\D_a}$,
so $N_a(g(0), g(1)\ldots, g(n-1))=b$.  The map $$\set{(f(0),
g(0)), (f(1), g(1))\ldots (f(n-1), g(n-1))}$$ is
a partial isomorphism of $N_a$.  By
 (\ref{eq:theta}) this extends to a finite partial isomorphism
 $\theta$ of $N_a$ whose domain includes $f(i_0), \ldots, f(i_k)$. Let
 $g'\in U_a$ be defined by
\[ g'(i) =\left\{\begin{array}{ll}\theta(i)&\mbox{if }i\in\dom(\theta)\\
g(i)&\mbox{otherwise}\end{array}\right.\] By (\ref{eq:bf}), ${\cal N}_a,
g'\models\phi(x_{i_0}, \ldots, x_{i_k})$. Observe that
$g'(0)=\theta(0)=g(0)$ and similarly $g'(n-1)=g(n-1)$, so $g$ is identical
to $g'$ over $n$ and it differs from $g'$ on only a finite
set of coordinates.  Since $\phi(x_{i_0}, \ldots, x_{i_k})^{\c
\ D_a}\in\Nr_{n}(\C)$ we deduce ${\cal N}_a, g \models \phi(x_{i_0}, \ldots,
x_{i_k})$, so $g\in\phi(x_{i_0}, \ldots, x_{i_k})^{\D_a}$.  This
proves that $b(x_0, x_1\ldots x_{n-1})^{\D_a}\subseteq\phi(x_{i_0},\ldots,
x_{i_k})^{\D_a}=\pi_a(x)$, and so
$$\iota_a(b(x_0, x_1,\ldots, x_{n-1})^{\c \ D_a})\leq
\iota_a(\phi(x_{i_0},\ldots, x_{i_k})^{\D_a})\leq x\in\c
C\setminus\set0.$$
Hence every non-zero element $x$ of $\Nr_{n}\C$
is above a non-zero element
$$\iota_a(b(x_0, x_1\ldots )^{\D_a})$$
(some $a, b\in \alpha$) and these latter elements are the atoms of $\Nr_{n}\C$.  So
$\Nr_{n}\C$ is atomic and $\alpha\cong\At\Nr_{n}\C$ --- the isomorphism
is $b \mapsto (b(x_0, x_1,\dots, x_{n-1})^{\D_a}:a\in A)$.

Now we can work in $L_{\infty,\omega}$ so that $\C$ is complete
by changing the defining clause for infinitary disjunctions to
$$N_a, f\models (\bigvee_{i\in I} \phi_i) \text { iff } (\exists i\in I)(N_a,  f\models\phi_i)$$

By working in $L_{\infty, \omega},$ we assume that arbitrary joins hence meets exist,
so $\C_a$ is complete, hence so is $C$. But $\Cm\At\A\subseteq \Nr_n\C$ is dense and complete, so
$\Cm\At\A=\Nr_n\C$.

For the second part, the rainbow algebra $\Rd_{sc}\PEA_{\Z,\N}$
is not in $S_c\Nr_n\Sc_{n+3}$ hence its rainbow atom structure is not
$n+3$ complete.
However, it is elementary equivalent to a countable  atomic algebra $\B$ whose atom structure carries an algebra in $\Nr_n\CA_{\omega}$
(the algebra itself, though necessarily representable, may not even be in
$\Nr_{n}\CA_{n+1}$. 

Indeed, \pe\ has a \ws\ in $J_k$ for all $k\geq n$ \cite[lemma 4.2, for the relation algebra case]{r},
to be proved in the next item, and
so we have $\B\in S_c\Nr_n\PEA_{\omega}$.
Hence, $\At\PEA_{\Z, \N}$ satisfies the above, namely, \pe\ has a \ws\ in $J_k$ for all finite $k\geq n$,
but \pa\ can win the $F^3$ game. 
Hence it is not $n+3$ complete but is elementary equivalent to 
an $\omega$ neat one. 

\end{enumerate}

\end{proof}

\begin{theorem} The class of completely representable algebras, and strongly representable
ones of dimension $>2$, are not detectable in $L_{\omega,\omega}$, while the class
$\Nr_n\CA_m$ for any ordinals $1<n<m<\omega$, is not detectable
even in $L_{\infty,\omega}.$ For  infinite $n$ and $m>n$, $\Nr_n\CA_m$
is not detectable in first order logic nor in the quantifier 
free reduct of $L_{\infty,\omega}$.
\end{theorem}
\begin{proof} Follows from \cite{hirsh}, generalized in theorem \ref{el} above, \cite[theorem 3.6.11, corollary 3.7.1]{HHbook2},
and \cite[theorems 5.1.4, 5.1.5]{Sayedneat}.
\end{proof}

The next corollary is an easy consequence of what we have established so far.
\begin{corollary}
\begin{enumarab}
\item Let $n\in\omega$. Let $K$ be a signature between $\Df_n$ and $\QEA_n$.
Then the class of completely representable $K$ algebras
is elementary
if and only if $n\leq 2$, in which case this class coincides with the (elementary) class of atomic representable algebras.
\item For each finite $n\geq 3$, there exists a simple countable atomic representable polyadic equality
algebra of dimension $n$ whose $\Df$ reduct is not completely representable.
\end{enumarab}
\end{corollary}
\begin{proof}
\begin{enumarab}
\item The if part for $n=2$ can be distilled from \cite[theorem 3.2.65, lemma 5.1.45, theorem 5.4.33]{tarski}.
A simpler proof is given in \cite[proposition 3.8.1]{HHbook2}
using duality theory of modal logic, but it works only  for
$\RCA_2$ and $\sf RPEA_2$ since they are completely
additive. For $n=0$ it reduces to the Boolean case and this is proved in \cite{HH},
and $n=1$ is also proved in \cite[proposition 3.8.1]{HHbook2}, too.

For the only if part. Let $\A=\PEA_{\Z, \N}$. Then \pe\ has a \ws\ for all finite atomic rounded game but \pa\ can win the $\omega$
rounded game. Hence using ultrapowers and an elementary chain argument we get a countable completely representable algebra $\B$
such that $\A\equiv \B$. Having countably many atoms, $\A$ is not completely representable, nor its $\Df$ reduct,
for it is generated by elements whose dimension sets $<n$. In fact, as indicated above
we have $\Rd_{\Sc}\A\notin S_c\Nr_n\Sc_{n+3}$, since \pa\ has a \ws\ in $F^{n+3}$.

\item  Let $\A$ be a countable atomic representable polyadic algebra, that is not necessarily simple, satisfying that its $\sf Df$ reduct is not
completely representable. Exists by the above. Consider the elements $\{{\sf c}_{(n)}a: a\in \At\A\}$. Then every simple component $S_a$ of $\A$ can be
obtained by relativizing to ${\sf c}_{(n)}a$ for an atom $a$.
Then one of the $S_a$'s should have no complete representation.
Else for each atom $a$, $S_a$ has a complete representation $h_a$.
From those one constructs a complete representation for $\A$ as follows.
The domain of the representation will be the disjoint union of the domains of $h_a$, and now represent $\A$ by
$h(\alpha)=\bigcup_{a\in \At \A}\{h(\alpha\cdot {\sf c}_{(n)}a)\}.$
\end{enumarab}

\end{proof}

The following theorem is proved in \cite{recent}, but we have decided to include the proof, 
since it gives a third example after theorems, \ref{completerepresentation}, \ref{squarerepresentation}, 
of how rainbows on relation algebras lift to rainbows on cylindric
ones proving the analagous $\CA$ or $\PEA$ result. The proof complete the proof of theoirem \ref{neat11}. 
As usual games on atomic networks are translated to games on coloured graphs. But here the game  translates 
to playing on {\it hyper-graphs};
it is the game $J_k$, which is the neat game $J$ truncated to $k$ rounds, played on hyper-networks.

\begin{theorem}\label{j} On the $n$ dimensional 
$\PEA_{\Z, \N}$,  \pe\ can win $J_k$, the game $J$ truncated to $k$ rounds 
for every $k>n$
\end{theorem}
\begin{proof}
\begin{enumarab}
\item  Cf. \cite{recent}. Let $k>n$ be given. We start by showing that \pe\ has a \ws\ in the usual graph game in $k$ rounds 
(there is no restriction here on the size of the graphs).
We define \pe\ s strategy for choosing labels for edges and $n-1$ tuples in response to \pa\ s moves.
Assume that we are at round $r+1$. In what follows we adapt Hirsch's arguments \cite[lemmas, 41-43]{r}.

Let $M_0, M_1,\ldots M_r$ be the coloured graphs at the start of a play of $J_k(\alpha)$ just before round $r+1$.
Assume inductively that \pe\ computes a partial function $\rho_s:\Z\to \N$, for $s\leq r$, that will help her choose
the suffices of the chosen red in the critical case.  Inductively
for $s\leq r$:

\begin{enumarab}
\item  If $M_s(x,y)$ is green then $(x,y)$ belongs  \pa\ in $M_s$ (meaning he coloured it).

\item $\rho_0\subseteq \ldots \rho_r\subseteq\ldots$
\item $\dom(\rho_s)=\{i\in \Z: \exists t\leq s, x, x_0, x_1,\ldots, x_{n-2}
\in \nodes(M_t)\\
\text { where the $x_i$'s form the base of a cone, $x$ is its appex and $i$ its tint }\}.$

The domain consists of the tints of cones created at an earlier stage.

\item $\rho_s$ is order preserving: if $i<j$ then $\rho_s(i)<\rho_s(j)$. The range
of $\rho_s$ is widely spaced: if $i<j\in \dom\rho_s$ then $\rho_s(i)-\rho_s(j)\geq  3^{m-r}$, where $m-r$
is the number of rounds remaining in the game.

\item For $u,v,x_0\in \nodes(M_s)$, if $M_s(u,v)=\r_{\mu,\delta}$, $M_s(x_0,u)=\g_0^i$, $M_s(x_0,v)=\g_0^j$,
where $i,j$ are tints of two cones, with base $F$ such that $x_0$ is the first element in $F$ under the induced linear order,
then $\rho_s(i)=\mu$ and $\rho_s(j)=\delta.$

\item $M_s$ is a a  coloured graph.

\item If the base of a cone $\Delta\subseteq M_s$  with tint $i$ is coloured $y_S$, then $i\in S$.

\end{enumarab}

To start with if \pa\ plays $a$ in the initial round then $\nodes(M_0)=\{0,1,\ldots n-1\}$, the
hyperedge labelling is defined by $M_0(0,1,\ldots n)=a$.

In response to a cylindrifier move for some $s\leq r$, involving a $p$ cone, $p\in \Z$,
\pe\ must extend $\rho_r$ to $\rho_{r+1}$ so that $p\in \dom(\rho_{r+1})$
and the gap between elements of its range is at least $3^{m-r-1}$. Properties (3) and (4) are easily
maintained in round $r+1$. Inductively, $\rho_r$ is order preserving and the gap between its elements is
at least $3^{m-r}$, so this can be maintained in a further round.
If \pa\ chooses a green colour, or green colour whose suffix
already belong to $\rho_r$, there would be fewer
elements to add to the domain of $\rho_{r+1}$, which makes it easy for \pe\ to define $\rho_{r+1}$.

Now assume that at round $r+1$, the current coloured graph is $M_r$ and that   \pa\ chose the graph $\Phi$, $|\Phi|=n$
with distinct nodes $F\cup \{\delta\}$, $\delta\notin M_r$, and  $F\subseteq M_r$ has size
$n-1$.  We can  view \pe\ s move as building a coloured graph $M^*$ extending $M_r$
whose nodes are those of $M_r$ together with the new node $\delta$ and whose edges are edges of $M_r$ together with edges
from $\delta$ to every node of $F$.

Now \pe\ must extend $M^*$ to a complete graph $M^+$ on the same nodes and
complete the colouring giving  a graph $M_{r+1}=M^+$ in $\K$ (the latter is the class of coloured graphs).
In particular, she has to define $M^+(\beta, \delta)$ for all nodes
$\beta\in M_r\sim F$, such that all of the above properties are maintained.

\begin{enumarab}

\item  If $\beta$ and $\delta$ are both apexes of two cones on $F$.

Assume that the tint of the cone determined by $\beta$ is $a\in \Z$, and the two cones
induce the same linear ordering on $F$. Recall that we have $\beta\notin F$, but it is in $M_r$, while $\delta$ is not in $M_r$,
and that $|F|=n-1$.
By the rules of the game  \pe\ has no choice but to pick a red colour. \pe\ uses her auxiliary
function $\rho_{r+1}$ to determine the suffices, she lets $\mu=\rho_{r+1}(p)$, $b=\rho_{r+1}(q)$
where $p$ and $q$ are the tints of the two cones based on $F$,
whose apexes are $\beta$ and $\delta$. Notice that $\mu, b\in \N$; then she sets $N_s(\beta, \delta)=\r_{\mu,b}$
maintaining property (5), and so $\delta\in \dom(\rho_{r+1})$
maintaining property (4). We check consistency to maintain property (6).

Consider a triangle of nodes $(\beta, y, \delta)$ in the graph $M_{r+1}=M^+$.
The only possible potential problem is that the edges $M^+(y,\beta)$ and $M^+(y,\delta)$ are coloured green with
distinct superscripts $p, q$ but this does not contradict
forbidden triangles of the form involving $(\g_0^p, \g_0^q, \r_{kl})$, because $\rho_{r+1}$ is constructed to be
order preserving.  Now assume that
$M_r(\beta, y)$ and $M_{r+1}(y, \delta)$ are both red (some $y\in \nodes(M_r)$).
Then \pe\ chose the red label $N_{r+1}(y,\delta)$, for $\delta$ is a new node.
We can assume that  $y$ is the apex of a $t$ cone with base $F$ in $M_r$. If not then $N_{r+1}(y, \delta)$ would be coloured
$\w$ by \pe\   and there will be no problem. All properties will be maintained.
Now $y, \beta\in M$, so by by property (5) we have $M_{r+1}(\beta,y)=\r_{\rho+1(p), \rho+1(t)}.$
But $\delta\notin M$, so by her strategy,
we have  $M_{r+1}(y,\delta)=\r_{\rho+1(t), \rho+1(q)}.$ But $M_{r+1}(\beta, \delta)=\r_{\rho+1(p), \rho+1(q)}$,
and we are done.  This is consistent triple, and so have shown that
forbidden triples of reds are avoided.

\item If this is not the case, and  for some $0<i<n-1$ there is no $f\in F$ such
that $M^*(\beta, f), M^* (f,\delta)$ are both coloured $\g_i$ or if $i=0$, they are coloured
$\g_0^l$ and $\g_0^{l'}$ for some $l$ and $l'$.
She chooses $\w_i$, for $M^+{(\beta,\delta)}$, for definiteness let it be the least such $i$.
It is clear that this also avoids all forbidden triangles (involving greens and whites).

\item Otherwise, if there is no $f\in F$,
such that $M^*(\beta, f), M*(\delta ,f)$ are coloured $\g_0^t$ and $\g_0^u$
for some $t,u$, then \pe\ defines $M^+(\beta, \delta)$ to  be $\w_0$.
Again, this avoids all forbidden triangles.
\end{enumarab}

She has not chosen green maintaining property one.  Now we turn to colouring of $n-1$ tuples,
to make sure that $M^+$ is a coloured graph maintaining property (6).

Let $\Phi$ be the graph chosen by \pa\, it has set of node $F\cup \{\delta\}$.
For each tuple $\bar{a}=a_0,\ldots a_{n-2}\in {M^+}^{n-1}$, $\bar{a}\notin M^{n-1}\cup \Phi^{n-1}$,  with no edge
$(a_i, a_j)$ coloured green (we already have all edges coloured), then  \pe\ colours $\bar{a}$ by $\y_S$, where
$$S=\{i\in A: \text { there is an $i$ cone in $M^*$ with base $\bar{a}$}\}.$$
We need to check that such labeling works, namely that last property is maintained.

Recall that $M$ is the current coloured graph, $M^*=M\cup \{\delta\}$ is built by \pa\ s move
and $M^+$ is the complete labelled graph by \pe\, whose nodes are labelled by \pe\ in response to \pa\ s moves.
We need to show that $M^+$ is labelled according to
the rules of the game, namely, that it is in $\K$.
As before, $(n-1)$ tuples are labelled correctly, by yellow colours.
(cf. \cite[p.16]{Hodkinson} \cite[p.844]{HH} 
and \cite{HHbook2}.)

\item In the game $J_n$ on the rainbow algebra $\PEA_{\Z,\N}$,
recall that   \pa\ has two more  amalgamation moves (together with the usual cylindrifier move), but now the game
is played on {\it hypergraphs}, generalizing the game played by Hirsch on {\it hypernetworks}.

Recall that a hypergraph has two parts, a coloured graph,
and hyperedges of arbitrarily long lengths and these are labelled. The game aims to capture a two sorted structure,
namely, a neat reduct.

Some of the hyperedges are called short, the others are called long.
The short hyperedges are called $\lambda$ neat, and they are constantly labelled to ensure that the
atoms in the small algebra are no bigger than those in the dilation (the big algebra in which  the small algebra neatly embeds).
Games are played on $\lambda$ neat hypergraphs, that is those hypergraphs in which the short
edges are labelled by the constant label $\lambda$.

Now \pe\ has two more amalgamation moves to respond to.
The first cylindrifier move is as before, except that now \pe\ has to provide labels for short hyperedges and long ones.
Roughly, in the former case she has 
no choice but to choose the label $\lambda$ and in the second case her strategy is straightforward, naturally
dictated by \pa\ s previous move.

In the first amalgamation move \pa\ is forced a response. In the second she has to amalgamate two given hypergraphs, by labelling edges and
$n-1$ hyperedges of the graphs and giving labels to the hyperedges labelling the short ones with $\lambda$.
The last strategy is easy. The response to the remaining amalgamation moves are very similar to
cylindrifier moves; this can be done by contracting the nodes, via the notion of {\it envelope} \cite{r}, to labelling edges done exactly
like the cylindrifier move.

Let us give more details. We show how \pe\ can win the $k$ rounded games of $J_k$, for every finite $k$ as follows.
We have already dealt with the graph part.
We now have to extend his strategy dealing with $\lambda$ neat hypernetworks, where $\lambda$ is constant label.

In a play, \pe\ is required to play $\lambda$ neat hypernetworks, so she has no choice about the
hyperedges for short edges, these are labelled by $\lambda$. In response to a cylindrifier move by \pa\
extending the current hypergraph providing a new node $k$,
and a previously played coloured graph $M$
all long hyperedges not incident with $k$ necessarily keep the hyperlabel they had in $M$.

All long hyperedges incident with $k$ in $M$
are given unique hyperlabels not occurring as the hyperlabel of any other hyperedge in $M$.
(We can assume, without loss of generality, that we have infinite supply of hyperlabels of all finite arities so this is possible.)

In response to an amalgamation move, which involves two hypergraphs required to be amalgamated, say $(M,N)$
all long hyperedges whose range is contained in $\nodes(M)$
have hyperlabel determined by $M$, and those whose range is contained in nodes $N$ have hyperlabel determined
by $N$. If $\bar{x}$ is a long hyperedge of \pe\ s response $L$ where
$\rng(\bar{x})\nsubseteq \nodes(M)$, $\nodes(N)$ then $\bar{x}$
is given
a new hyperlabel, not used in any previously played hypernetwork and not used within $L$ as the label of any hyperedge other than $\bar{x}$.
This completes her strategy for labelling hyperedges.

Now we turn to amalgamation moves. We need some notation and terminology taken from \cite{r}; they are very useful to economize on proofs.
Every edge of any hypernetwork has an owner \pa\ or \pe\ , namely, the one who coloured this edge.
We call such edges \pa\ edges or \pe\ edges. Each long hyperedge $\bar{x}$ in a hypernetwork $N$
occurring in the play has an envelope $v_N(\bar{x})$ to be defined shortly.

In the initial round of \pa\ plays $a\in \alpha$ and \pe\ plays $N_0$
then all irreflexive edges of $N_0$ belongs to \pa\ . 
There are no long hyperedges in $N_0$. If in a later move,
\pa\ plays the transformation move $(N,\theta)$
and \pe\ responds with $N\theta$ then owners and envelopes are inherited in the obvious way.
If \pa\ plays a cylindrifier move requiring a new node $k$ and \pe\ responds with $M$ then the owner
in $M$ of an edge not incident with $k$ is the same as it was in $N$
and the envelope in $M$ of a long hyperedge not incident with $k$ is the same as that it was in $N$.
The edges $(f,k)$ belong to \pa\ in $M$, all edges $(l,k)$
for $l\in \nodes(N)\sim \phi\sim \{k\}$ belong to \pe\ in $M$.
if $\bar{x}$ is any long hyperedge of $M$ with $k\in \rng(\bar{x})$, then $v_M(\bar{x})=\nodes(M)$.
If \pa\ plays the amalgamation move $(M,N)$ and \pe\ responds with $L$
then for $m\neq n\in \nodes(L)$ the owner in $L$ of a edge $(m,n)$ is \pa\ if it belongs to
\pa\ in either $M$ or $N$, in all other cases it belongs to \pe\ in $L$.

If $\bar{x}$ is a long hyperedge of $L$
then $v_L(\bar{x})=v_M(x)$ if $\rng(x)\subseteq \nodes(M)$, $v_L(x)=v_N(x)$ and  $v_L(x)=\nodes(M)$ otherwise.
This completes the definition of owners and envelopes.

The next claim, basically, reduces amalgamation moves to cylindrifier moves.
By induction on the number of rounds one can show:

\begin{athm}{Claim}\label{r} Let $M, N$ occur in a play of $H_k(\alpha)$ in which \pe\ uses the above labelling
for hyperedges. Let $\bar{x}$ be a long hyperedge of $M$ and let $\bar{y}$ be a long hyperedge of $N$.
\begin{enumarab}
\item For any hyperedge $\bar{x}'$ with $\rng(\bar{x}')\subseteq v_M(\bar{x})$, if $M(\bar{x}')=M(\bar{x})$
then $\bar{x}'=\bar{x}$.
\item if $\bar{x}$ is a long hyperedge of $M$ and $\bar{y}$ is a long hyperedge of $N$, and $M(\bar{x})=N(\bar{y})$
then there is a local isomorphism $\theta: v_M(\bar{x})\to v_N(\bar{y})$ such that
$\theta(x_i)=y_i$ for all $i<|x|$.
\item For any $x\in \nodes(M)\sim v_M(\bar{x})$ and $S\subseteq v_M(\bar{x})$, if $(x,s)$ belong to \pa\ in $M$
for all $s\in S$, then $|S|\leq 2$.
\end{enumarab}
\end{athm}

Again we proceed inductive, with the inductive hypothesis exactly as before except that now each $N_r$ is a
$\lambda$ neat hypergraph. All other inductive conditions are the same (modulo this replacement). Now,
we have already dealt with hyperlabels for long and short
hyperedges, we dealt with the graph part of the first hypergraph move.

All what remains is the amalgamation move. With the above claim at hand,
this turns out an easy task to implement guided by \pe\ s
\ws\ in the graph part.

We consider an amalgamation move $(N_s,N_t)$ chosen by \pa\ in round $r+1$.
We finish off with edge labelling first. \pe\ has to choose a colour for each edge $(i,j)$
where $i\in \nodes(N_s)\sim \nodes(N_t)$ and $j\in \nodes(N_t)\sim \nodes(N_s)$.
Let $\bar{x}$ enumerate $\nodes(N_s)\cap \nodes(N_t)$
If $\bar{x}$ is short, then there are at most two nodes in the intersection
and this case is similar to the cylindrifier move, she uses $\rho_s$ for the suffixes of the red.

if not, that is if $\bar{x}$ is long in $N_s$, then by the claim
there is a partial isomorphism $\theta: v_{N_s}(\bar{x})\to v_{N_t}(\bar{x})$ fixing
$\bar{x}$. We can assume that
$$v_{N_s}(\bar{x})=\nodes(N_s)\cap \nodes (N_t)=\rng(\bar{x})=v_{N_t}(\bar{x}).$$
It remains to label the edges $(i,j)\in N_{r+1}$ where $i\in \nodes(N_s)\sim \nodes (N_t)$ and $j\in \nodes(N_t)\sim \nodes(N_s)$.
Her strategy is similar to the cylindrifier move. If $i$ and $j$ are tints of the same cone she choose a red using $\rho_s$,
If not she  chooses   a white.
She never chooses a green.
Then she lets $\rho_{r+1}=\rho_r$ and maintaining the inductive hypothesis.

Concerning the last property to be maintained, and that is
colouring $n-1$ types. Let $M^+=N_s\cup M_s$, which is the graph whose edges are labelled according to the rules of the game,
we need to label $n-1$ hyperedges by shades of yellow.
For each tuple $\bar{a}=a_0,\ldots a_{n-2}\in {M^+}^{n-1}$, $\bar{a}\notin N_s^{n-1}\cup M_s^{n-1}$,  with no edge
$(a_i, a_j)$ coloured green (we have already labelled edges), then  \pe\ colours $\bar{a}$ by $\y_S$, where
$$S=\{i\in \Z: \text { there is an $i$ cone in $M^*$ with base $\bar{a}$}\}.$$
As before, see the references in 
proof of theorem \ref{completerepresentation} concerning the colouring by the shades of yellow, this can be checked to be $OK$.

\end{enumarab}
\end{proof}

\subsection{Some further results on neat embeddings}

Now we give a solution to the famous problem 2.12
lifting it to infinite dimensions, using a lifting
argument due to Monk. A weaker solution of this problem is credited to Pigozzi in \cite{tarski},
but we have not seen a published proof. 

A stronger result obtained jointly with Robin Hirsch proves this result
for many cylindric-like algebras including various diagonal-free reducts of
cylindric and polyadic equality algebras.

Here  we use {\it existing} finite dimensional algebras in the literature, 
lifting the result to infinite dimensions. The lifting argument is due to Monk, and
it was also used in \cite{t}. 

There is a slight novelty here. 
The constructed algebras  $\A^r$, for every $r\in \omega$  are in $\Nr_{\alpha}\CA_{\alpha+k}$, that is we drop the $S$.
Finally, in the infinite dimensional case the ultraproduct is {\it representable}; this is stronger than the result proved in \cite{t} restricted to 
$\CA$s. The proof also generalizes verbatim to $\sf QEA$, by noting that the finite dimensional algebras used 
are in fact quasi-polyadic equality algebras, that is, they allow the definition of substitution operations in an obvious way.
We do not know whether an anlogous stronger result holds for diagonal free reducts of such algebras 
(like Pinter's algebras or quasi-polyadic algebras) in the infinite dimensional case.

The result proved in \cite{t} addressing such algebras is weaker {\it when restricted to the $\CA$ case};
there  the ultraproduct could  capture only one extra dimension not
$\omega$ many. 

\begin{theorem}\label{2.12}
For $\alpha>2$ (infinite included), and $k\in \omega$ and $r\in \omega$,
there is an $\A^r\in \Nr_{\alpha}{\sf CA}_{\alpha+k}\sim S\Nr_{\alpha}{\sf CA}_{\alpha+k+1},$
such that $\prod \A^r/F\in \RCA_{\alpha}$.
In particular,  the variety
$S\Nr_{\alpha}\CA_{\alpha+k+1}$ is not axiomatizable by a finite schema over
$S\Nr_{\alpha}\CA_{\alpha+k}$.
\end{theorem}
\begin{proof}
Assume that $3\leq m\leq n$, and let
$$\C(m,n,r)=\Ca(H_m^{n+1}(\A(n,r),  \omega)),$$
be as defined in \cite[definition 15.4]{HHbook}.
Here $\A(n,r)$ is a finite Monk-like relation algebra \cite[definition 15.2]{HHbook}
which has an $n$ dimensional hyperbasis $H_m^{n+1}(\A(n,r), \omega)$
consisting of all $m$, $n+1$ wide hypernetworks \cite[definition 12.21]{HHbook}.
We used them before in the second item of theorem  \ref{main}, proving a strong negative omiiting types theorem formulated 
conditionally in theorem \ref{OTT}. 
The hyperedges are labelled by $\omega$; it is an infinite cylindric algebra of dimension $m$, and it is a neat reduct.

In fact we have:
\begin{enumarab}
\item For any $r$ and $3\leq m\leq n<\omega$, we
have $\C(m,n,r)\in \Nr_m{\sf PEA}_n$.

\item $\C(m,n,r)\notin S\Nr_n\CA_{n+m+1}$ but $\prod_r \C(m,n,r)/F\in \RCA_n$ for any non principal ultrafilter
on $\omega$.
\item  For $m<n$ and $k\geq 1$, there exists $x_n\in \C(n,n+k,r)$ such that $\C(m,m+k,r)\cong \Rl_{x}\C(n, n+k, r).$
An analogous result holds for quasi-polyadic equality algebras.
\end{enumarab}
Let us check this properties:
\begin{enumarab}
\item  $H_n^{n+1}(\A(n,r), \omega)$ is a wide $n$ dimensional $\omega$ symmetric hyperbases, so $\Ca H\in {\sf PEA}_n.$
But $H_m^{n+1}(\A(n,r),\omega)=H|_m^{n+1}$.
Thus
$$\C_r=\Ca(H_m^{n+1}(\A(n,r), \omega))=\Ca(H|_m^{n+1})\cong \Nr_m\Ca H$$

\item \cite[theorem 15.6]{HHbook}.
Let $G^k$ be the usual atomic game defined on atomic networks with $k$ rounds \cite[definition 11.1, lemma 11. 2, definition 11.3]{HHbook}.
In this game \pa\ is allowed only a triangle move.
We first address the relation algebras on which the $\C_r$s are based.
Given $k$, then we show that for any $r\geq  k^2$, we have \pe\ has a \ws\ in $G^k$
in $\A(n,r)$. This implies using ultraproducts and an elementary chain argument that \pe\
has a \ws\ in the $\omega$ rounded game in an elementary substructure of $\Pi\A(n,r)/F$,
hence the former is representable and then so is the latter because
${\sf RRA}$ is a variety.

Consider the only move by \pa\, namely, a triangle move $N(x,y,z,a,b).$
To label the edges $(w,z)$ where $w\neq k$, $w\neq x,y,z$, \pe\ uses $a^0(i, j_{w})$ where $i<n-1$
and $a,b$ not less that $a(i)$, where
the numbers $j_w(w\in k\in \{x,y,z\}$ are distinct elements
of $\{j<r: \neg \exists u,v\in k\sim {z}(N(u,v)\leq a(i,j))\}$. Because
$r$ {\it is} large enough,  this set has size at least $|k\sim \{x,y,z\}|$ and so
it  is impossible to find $j_w$. Then
any triangle labelled by $a^k(i,j)$ $a^{k'}(i,j')$ and
$a^{k''}(i, j'')$, the indices $j', j, j''$ are distinct and so
the triangle is consistent.

To show that $\Pi \C^r/F$ is also representable, it suffices to find a representation of
an algebra $\A\prec \Pi \A(n,r)/F$  that embeds {\it all}
$m$ dimensional hypernetworks, respecting $\equiv_i$ for all $i<m$ \cite[exercise 2, p.484]{HHbook}.  

We construct such a representation in a step by step manner.
The technique was used before, in theorems \ref{step} and \ref{longer} so we we will
be sketchy. 
The details can be easily recovered from the proof of either theorem referred to.
Assume that $M$ is a relation algebra representation of $\A$. 
We know that such a representation exists. 
However, it might not embed all $m$-dimensional hyper networks.
So we make it does, in a step by step manner by scheduling all hypernetworks.

We build a chain of hypergraphs $M_t:t<\omega$ their limit (defined in a precise
sense) will as required.  Each $M_t$ will have edges labelled by
atoms of $\A$,  and hyperedges will be labelled, as well.

Let $H$ be the set of all $m$ hypernetworks; this is countable.
Let $M_0=M$.

We require inductively that $M_t$  satisfies:

Any $m$ tuple of  $M_t$ is contained in $\rng(v)$ for some $N\in H$ and some embedding $v:N\to  M_t$.
such that this embedding satisfies the following two conditions:

(a) if $i,j<m$  an  edge of $M_t$, then
$M_t(v(i) ,v(j))=N(i,j),$

(b) Whenever $a\in {}^{\leq ^m} m$ with $|a|\neq 2$, then $v(a)$ is a
hyperedge of $M_t$ and is labelled by $N(\bar{a})$.

Let $M$ be limit of the $M_t$s as defined in \ref{step}.
Let $L(A)$ be the signature obtained by adding an $n$ ary relation symbol for each element
of $\A$. Then define $M\models r(x,y)$ iff $M(x,y)\leq r.$
Then by construction $M$ is as desired.

Now define a labelled hypergraph as follows:
$\nodes(M)=\bigcup \nodes(M_t)$,
for any $\bar{x}$ that is an atom-hyperedge; then it is a one in some $M_t$
and its label is defined in $M$ by $M_t(\bar{x})$

The hyperedges are $n$ tuples $(x_0,\ldots x_{m-1})$.
For each such tuple, we let $t<\omega$,
such that $\{x_0\ldots x_{m-1}\} \subseteq  M_t$,
and we set $M(x_0,\ldots x_{m-1})$   to be the unique
$N\in H$ such that there is an embedding
$v:N\to M$ with  $\bar{x}\subseteq \rng(v).$
This can be easily checked to be well defined.

Now we we define the representation $M$ of $\C$.
Let $L(C)$ be the signature obtained by adding an $n$ ary relation symbol for each element
of $\C$.
Now define for $r\in L(C)$
$$M\models r(\bar{x})\text { iff  } M(\bar{x})\in  r.$$
(Note that $M(\bar{x})$
is a hypernetwork, while $r$ is as set of hypernetworks).
This is clearly a representation of $\C.$

\item Let $m<n$, let $$x_n=\{f\in F(n,n+k,r): m\leq j<n\to \exists i<m f(i,j)=Id\}.$$
Then $x_n\in \C(n,n+k,r)$ and ${\sf c}_ix_n\cdot {\sf c}_jx_n=x_n$ for distinct $i, j<m$.
Furthermore
\[{I_n:\C}(m,m+k,r)\cong \Rl_{x_n}\Rd_m {\C}(n,n+k, r).\]
via
\[ I_n(S)=\{f\in F(n, n+k, r): f\upharpoonright m\times m\in S,$$
$$\forall j(m\leq j<n\to  \exists i<m\; f(i,j)=Id)\}.\]
\end{enumarab}
Now we use a lifting argument that is a generalization of Monk's argument in
\cite[theorem 3.2.87]{tarski}. It was also used in \cite{t} adressing more 
algebras, namely, diagonal free cylindric-like algebras, like Pinter's substitution 
and quasi-polyadic algebras. 

However, the result obtained by 
our lifting argument is stronger by far; the ultraproducts will be representable (not just neatly embedding in algebras have 
$\alpha+k+1$ spare dimensions). This will follow from the fact that the lifted finite dimensional 
algebras have representable ultraproducts, too.

Let $k\in \omega$. Let $\alpha$ be an infinite ordinal.
We claim that 
$S\Nr_{\alpha}\CA_{\alpha+k+1}\subset S\Nr_{\alpha}\CA_{\alpha+k};$ furthermore we construct infinitely many algebras 
that witness the strictness
of the inclusion, one for each $r\in \omega$. Their ultraproduct relative to any non principal ultrafilter on $\omega$
will be representable. 

Fix such $r$.
Let $I=\{\Gamma: \Gamma\subseteq \alpha,  |\Gamma|<\omega\}$.
For each $\Gamma\in I$, let $M_{\Gamma}=\{\Delta\in I: \Gamma\subseteq \Delta\}$,
and let $F$ be an ultrafilter on $I$ such that $\forall\Gamma\in I,\; M_{\Gamma}\in F$.
For each $\Gamma\in I$, let $\rho_{\Gamma}$
be a one to one function from $|\Gamma|$ onto $\Gamma.$

Let ${\C}_{\Gamma}^r$ be an algebra similar to $\CA_{\alpha}$ such that
\[\Rd^{\rho_\Gamma}{\C}_{\Gamma}^r={\C}(|\Gamma|, |\Gamma|+k,r).\]
Let
\[\B^r=\prod_{\Gamma/F\in I}\C_{\Gamma}^r.\]
We will prove that
\begin{enumerate}
\item\label{en:1} $\B^r\in S\Nr_\alpha\CA_{\alpha+k}$ and
\item\label{en:2} $\B^r\not\in S\Nr_\alpha\CA_{\alpha+k+1}$.  
\end{enumerate}

For the first part, for each $\Gamma\in I$ we know that $\C(|\Gamma|+k, |\Gamma|+k, r) \in\K_{|\Gamma|+k}$ and
$\Nr_{|\Gamma|}\C(|\Gamma|+k, |\Gamma|+k, r)\cong\C(|\Gamma|, |\Gamma|+k, r)$.
Let $\sigma_{\Gamma}$ be a one to one function
 $(|\Gamma|+k)\rightarrow(\alpha+k)$ such that $\rho_{\Gamma}\subseteq \sigma_{\Gamma}$
and $\sigma_{\Gamma}(|\Gamma|+i)=\alpha+i$ for every $i<k$. Let $\A_{\Gamma}$ be an algebra similar to a
$\CA_{\alpha+k}$ such that
$\Rd^{\sigma_\Gamma}\A_{\Gamma}=\C(|\Gamma|+k, |\Gamma|+k, r)$.
Then, clearly
 $\Pi_{\Gamma/F}\A_{\Gamma}\in \CA_{\alpha+k}$.

We prove that $\B^r\subseteq \Nr_\alpha\Pi_{\Gamma/F}\A_\Gamma$.  Recall that $\B^r=\Pi_{\Gamma/F}\C^r_\Gamma$ and note
that $\C^r_{\Gamma}\subseteq A_{\Gamma}$
(the base of $\C^r_\Gamma$ is $C(|\Gamma|, |\Gamma|+k, r)$, the base of $A_\Gamma$ is $C(|\Gamma|+k, |\Gamma|+k, r)$).
So, for each $\Gamma\in I$,
\begin{align*}
\Rd^{\rho_{\Gamma}}\C_{\Gamma}^r&=\C((|\Gamma|, |\Gamma|+k, r)\\
&\cong\Nr_{|\Gamma|}\C(|\Gamma|+k, |\Gamma|+k, r)\\
&=\Nr_{|\Gamma|}\Rd^{\sigma_{\Gamma}}\A_{\Gamma}\\
&=\Rd^{\sigma_\Gamma}\Nr_\Gamma\A_\Gamma\\
&=\Rd^{\rho_\Gamma}\Nr_\Gamma\A_\Gamma
\end{align*}
$\Rd^{\rho_\Gamma}\A_\Gamma \in \K_{|\Gamma|}$, for each $\Gamma\in I$  then $\Pi_{\Gamma/F}\A_\Gamma\in \K_\alpha$.
Thus (using a standard Los argument) we have:
$\Pi_{\Gamma/F}\C^r_\Gamma\cong\Pi_{\Gamma/F}\Nr_\Gamma\A_\Gamma=\Nr_\alpha\Pi_{\Gamma/F}\A_\Gamma$,
proving \eqref{en:1}.

Now we prove \eqref{en:2}.
For this assume, seeking a contradiction, that $\B^r\in S\Nr_{\alpha}\CA_{\alpha+k+1}$,
$\B^r\subseteq \Nr_{\alpha}\C$, where  $\C\in \CA_{\alpha+k+1}$.
Let $3\leq m<\omega$ and  $\lambda:m+k+1\rightarrow \alpha +k+1$ be the function defined by $\lambda(i)=i$ for $i<m$
and $\lambda(m+i)=\alpha+i$ for $i<k+1$.
Then $\Rd^\lambda(\C)\in \CA_{m+k+1}$ and $\Rd_m\B^r\subseteq \Nr_m\Rd^\lambda(\C)$.

For each $\Gamma\in I$,\/  let $I_{|\Gamma|}$ be an isomorphism
\[{\C}(m,m+k,r)\cong \Rl_{x_{|\Gamma|}}\Rd_m {\C}(|\Gamma|, |\Gamma+k|,r).\]
Let $x=(x_{|\Gamma|}:\Gamma)/F$ and let $\iota( b)=(I_{|\Gamma|}b: \Gamma)/F$ for  $b\in \C(m,m+k,r)$.
Then $\iota$ is an isomorphism from $\C(m, m+k,r)$ into $\Rl_x\Rd_m\B^r$.
Then $\Rl_x\Rd_{m}\B^r\in S\Nr_m\CA_{m+k+1}$.
It follows that  $\C (m,m+k,r)\in S\Nr_{m}\CA_{m+k+1}$ which is a contradiction and we are done.
Now we prove the third part of the theorem, putting the superscript $r$ to use.

Recall that $\B^r=\Pi_{\Gamma/F}\C^r_\Gamma$, where $\C^r_\Gamma$ has the type of $\CA_{\alpha}$
and $\Rd^{\rho_\Gamma}\C^r_\Gamma=\C(|\Gamma|, |\Gamma|+k, r)$.
We know (this is the main novelty here) 
from item (2) that $\Pi_{r/U}\Rd^{\rho_\Gamma}\C^r_\Gamma=\Pi_{r/U}\C(|\Gamma|, |\Gamma|+k, r) \subseteq \Nr_{|\Gamma|}\A_\Gamma$,
for some $\A_\Gamma\in\CA_{|\Gamma|+\omega}$.

Let $\lambda_\Gamma:|\Gamma|+k+1\rightarrow\alpha+k+1$
extend $\rho_\Gamma:|\Gamma|\rightarrow \Gamma \; (\subseteq\alpha)$ and satisfy
\[\lambda_\Gamma(|\Gamma|+i)=\alpha+i\]
for $i<k+1$.  Let $\F_\Gamma$ be a $\CA_{\alpha+\omega}$ type algebra such that $\Rd^{\lambda_\Gamma}\F_\Gamma=\A_\Gamma$.
As before, $\Pi_{\Gamma/F}\F_\Gamma\in\CA_{\alpha+\omega}$.  And
\begin{align*}
\Pi_{r/U}\B^r&=\Pi_{r/U}\Pi_{\Gamma/F}\C^r_\Gamma\\
&\cong \Pi_{\Gamma/F}\Pi_{r/U}\C^r_\Gamma\\
&\subseteq \Pi_{\Gamma/F}\Nr_{|\Gamma|}\A_\Gamma\\
&=\Pi_{\Gamma/F}\Nr_{|\Gamma|}\Rd^{\lambda_\Gamma}\F_\Gamma\\
&\subseteq\Nr_\alpha\Pi_{\Gamma/F}\F_\Gamma,
\end{align*}
But, by the neat embedding theorem, we have $\Pi_{\Gamma/F}\F_{\Gamma}\in \RCA_{\alpha}$ 
and we are done.
\end{proof}
\begin{theorem}\label{infinitedistance} For $n>m$, the variety ${\sf CB}_{m,n}$ is not finitely axiomatizable over 
$S\Nr_m\CA_{n}$. In particular, the third item in theorem \ref{completerepresentation} holds. 

\end{theorem}\label{nt}
\begin{proof} We have shown that \pe\ has a \ws\ in $G^k_{\omega}(\A(n,r))$ when $r\geq k^2$,
hence, $\A(n,r)\in \RA_k$. But this induces a \ws\ for $\C_r$ in the corresponding basis cylindric 
atomic game $G^k_{\omega}(\C_r)$ (with $k$ nodes and $\omega$ rounds),  so that $\C_r\in \sf CB_{m, k}$.

Hence $\C_r\in {\sf CB}_{n,n+1}\sim S\Nr_n\CA_{n+1}$, 
and $\Pi \C_r/F\in S\Nr_n\CA_{n+1}$, and we done.
\end{proof}

In  \cite{t} a partially more general result is proved meaning that it covers other diagonal-free algebras, 
but it is weaker when restricted to the cases $\CA$ and $\QEA$.
Here we are encountered by a typical situation where we cannot have our cake and eat it.
If we want an analogous result for diagonal-free algebras (baring in mind 
that usually it is very hard to transfer such results when
we do not have diagonal elements, for example Andr\'eka very strong 
splitting techniques do not apply when we do not have diagonal elements, and this indeed can be proved), 
then we may have to sacrifize the {\it representability} of the ultraproduct, though the possibility remains 
that this need not be sacrifized. 
In \cite{t} the second item in the above proof, namely, the representability of the ultraproduct
is weakened. This ultraproduct could make it only to one extra dimension
and  {\it not} to $\omega$ many.

Furthermore, the algebras constructed in \cite{t} 
are Monk- like  {\it finite} polyadic equality algebras, 
while the cylindric and polyadic equality algebras constructed before were {\it infinite} since they were constructed 
from hyperbabes having infinitely
many hyperlabels coming from $\omega$.  This makes a  huge difference.

Nevertheless,  in all cases the 
construction can be  based on {\it the same relation algebra}, analogous to  $\A(n,r)$
as specified above,  
except that now the $m$ dimensional 
polyadic equality algebra formed, consists of all 
basic matrices of dimension $m$. No labels are involved so
the algebras are finite. We know from the above that $\A(n,r)$ has an $n+1$ wide 
$m$  
dimensional hyperbasis; in particular it has 
an $m$ dimensional hyperbasis,  
and so consequently it has  an $n$ dimensional 
symmetric cylindric basis, that is, an $n$  polyadic equality basis $(m<n).$ 

Such a basis is  the atom structures of the new finite 
$m$ dimensional polyadic equality algebras, built on a relation algebra very similar to $\A(n,r)$. 
The following is proved in \cite{t}. 

\begin{theorem}\label{thm:cmnr} Let $3\leq m\leq n$ and $r<\omega$.
\begin{enumerate} 
\renewcommand{\theenumi}{\Roman{enumi}}
\item $\C(m, n, r)\in \Nr_m\QPEA_n$,\label{en:one}
\item $\Rd_{\Sc}\C(m, n, r)\not\in S\Nr_m\Sc_{n+1}$, \label{en:two}
\item $\Pi_{r/U} \C(m, n, r)$ is elementarily equivalent to a countable polyadic equality algebra $\C\in\Nr_m\QPEA_{n+1}$.  \label{en:four}
\end{enumerate} 
\end{theorem}
We define the algebras $\C(m,n,r)$ for $3\leq m\leq n<\omega$ and $r$ 
and then give a sketch of \eqref{en:two}. To prove \eqref{en:four}, $r$ has to be a linear order, 
but we omit this part of the proof.
We start with:
\begin{definition}\label{def:cmnr}
Define a function $\kappa:\omega\times\omega\rightarrow\omega$ by $\kappa(x, 0)=0$ 
(all $x<\omega$) and $\kappa(x, y+1)=1+x\times\kappa(x, y))$ (all $x, y<\omega$).
For $n, r<\omega$ let 
\[\psi(n, r)=
\kappa((n-1)r, (n-1)r)+1.\]
All of this is simply to ensure that $\psi(n, r)$ is sufficiently big compared to $n, r$ for the proof of non-embeddability to work.  

For any  $n<\omega$ and any $r<\omega$, let 
\[Bin(n, r)=\set{Id}\cup\set{a^k(i, j):i< n-1,\;j\in r,\;k<\psi(n, r)}\] 
where $Id, a^k(i, j)$ are distinct objects indexed by $k, i, j$.
Let $3\leq m\leq n<\omega$ and let $r$ be any linear order.
Let $F(m, n, r)$ be the set of all  functions $f:m\times m\to Bin(n, r)$ 
such that $f$ is symmetric ($f(x, y)=f(y, x)$ for all $x, y<m$) 
and for all $x, y, z<m$ we have $f(x, x)=Id,\;f(x, y)=f(y, x)$, and $(f(x, y), f(y, z), f(x, z))\not\in Forb$, 
where $Forb$ (the \emph{forbidden} triples) is the following set of triples
 \[ \begin{array}{c}
 \set{(Id, b, c):b\neq c\in Bin(n, r)}\\
 \cup \\
 \set{(a^k(i, j), a^{k'}(i,j), a^{k^*}(i, j')): k, k', k^*< \psi(n, r), \;i<n-1, \; j'\leq j\in r}.
 \end{array}\]
Here $Bin(n,r)$ is an atom structure of a finite relation relation  
and $Forb$ specifies its operations by specifying forbidden triples. 
This atom structure defines a relation algebra; 
it is {\it similar but not identical to   $\A(n,r)$}.  
However, if we replace $nr^{nr}$ in the previous case  by the new Ramsey function 
$\psi(n,r)$  
then the two algebras are the same.

Furthermore, this change will not alter our previous result. The new ultraproduct 
$\Pi_r/U\A(n,r)$ is still representable; one proves that \pe\ can win the game $G^k_{\omega}$ when $r\geq k^2$
as above,  and so is the infinite 
polyadic equality algebra $\C_r=\Ca(H_m^{n+1}(\A(n,r), \omega)$ 
is also representable. Its representability, can be proved, 
by finding a representation of an elementary countable subalgebra of $\Pi_r/U \A(n,r)$
that embeds all $m$ dimensional hypernetworks, and then as above, using this representation one defines a representation 
of a countable elementary subalgebra of $\C_r$, which immediately implies that $\C_r$ 
is representable (as a polyadic equality algebra of dimension $m$).

So, without loss,  we assume that the relation algebras constructed 
in both cases,  depending on the parameters $n$ and $r$ is the same algebra which we denote 
also by $\A(n,r).$ 

Any any such $f\in F(m,n,r)$ is a basic matrix on the atom structure of $\A(n,r)$ 
in the sense of Maddux, and the whole lot of them will be a 
basis, constituting the atom structure of algebras we want.
So here instead of an infinite set of hypernetworks with hyperlabels 
from $\omega$ our cylindric-like algebras are finite.

Now accessibility relations coresponding to substitutions, cylindrifiers are defined, as expected on matrices,  
as follows.
For any $f, g\in F(m, n,  r)$ and $x, y<m$ we write $f\equiv_{xy}g$ if for all $w, z\in m\setminus\set {x, y}$ we have $f(w, z)=g(w, z)$.  
We may write $f\equiv_x g$ instead of $f\equiv_{xx}g$.  For $\tau:m\to m$ we write $(f\tau)$ for the function defined by 
\begin{equation}\label{eq:ftau}(f\tau)(x, y)=f(\tau(x), \tau(y)).
\end{equation}
Clearly $(f\tau)\in F(m, n, r)$. 
Accordingly, the universe of $\C(m, n, r)$ is the power set of $F(m, n, r)$ and the operators (lifting from the atom structure)
are
\begin{itemize}
\item  the boolean operators $+, -$ are union and set complement, 
\item  the diagonal $\diag xy=\set{f\in F(m, n, r):f(x, y)=Id}$,
\item  the cylindrifier $\cyl x(X)=\set{f\in F(m, n, r): \exists g\in X\; f\equiv_xg }$ and
\item the polyadic $\s_\tau(X)=\set{f\in F(m, n, r): f\tau \in X}$,
\end{itemize}
for $x, y<m,\;  X\subseteq F(m, n, r)$ and  $\tau:m\to m$.
 \end{definition}
\medskip
We give a sketch of proof of \ref{thm:cmnr}(\ref{en:two}), which is the heart and soul of the proof, and it is very similar 
to its $\CA$ analogue 4.69-475 in \cite{HHbook2}. We will also refer to the latter when the proofs overlap.

Assume for contradiction  that 
$\Rd_{\Sc}\C(m, n, r)\subseteq\Nr_m\C$ 
for some $\C\in \Sc_{n+1}$, some finite $m, n, r$. 
Then it can be shown inductively 
that there must be a large  set $S$ of distinct elements of $\C$, 
satisfying certain inductive assumptions, which we outline next.  
For each $s\in S$ and $i, j<n+2$ there is an element $\alpha(s, i, j)\in Bin(n, r)$ obtained from $s$ 
by cylindrifying all dimensions in $(n+1)\setminus\set{i, j}$, then using substitutions to replace $i, j$ by $0, 1$.  
Then one shows that $(\alpha(s, i, j), \alpha(s, j, k), \alpha(s, i, k))\not\in Forb$.

The induction hypothesis say, most importantly, that $\cyl n(s)$ is constant, for $s\in S$, 
and for $l<n$  there are fixed $i<n-1,\; j<r$ such that for all $s\in S$ we have $\alpha(s, l, n)\leq a(i, j)$.  
This defines, like in the proof of theorem 15.8 in \cite{HHbook2} p.471, two functions $I:n\rightarrow (n-1),\; J:n\rightarrow r$ 
such that $\alpha(s, l, n)\leq a(I(l), J(l))$ for all $s\in S$.  
The \emph{rank} ${\sf rk}(I, J)$ of $(I, J)$ - as defined in definition 15.9 in \cite{HHbook2} for the $\CA$ case  - is   
the sum (over $i<n-1$) of the maximum $j$ with $I(l)=i,\; J(l)=j$ (some $l<n$) or $-1$ if there is no such $j$.  

Next it is proved that there is a set $S'$ with index functions $(I', J')$, still relatively large 
in terms of the number of times we need to repeat the induction step
where the same induction hypotheses hold but where ${\sf rk}(I', J')>{\sf rk}(I, J)$.

By iterating  this long enough, namely, more than $\psi(n,r)$ times we obtain a non-empty set $T$ 
with index functions of rank strictly greater than $(n-1)\times(r-1)$, 
an impossibility.\footnote{See \cite{HHbook2}, where for $t<nr$, $S'$ was denoted by $S_t$
and proof of property (6) in the induction hypothesis  on p.474 of \cite{HHbook2}.}

Now we sketch the induction step.  Since $I$ cannot be injective there must be distinct $l_1, l_2<n$ 
such that $I(l_1)=I(l_2)$ and $J(l_1)\leq J(l_2)$.  We may use $l_1$ as a "spare dimension"; 
changing the index functions on $l$ will not reduce the rank.  

Since $\cyl n(s)$ is constant, we may fix $s_0\in S$ 
and choose a new element $s'$ below $\cyl l s_0\cdot \sub n l\cyl  l s$, 
with certain properties.  Let $S^*=\set{s': s\in S\setminus\set{s_0}}$.
We wish to maintain the induction hypotheses for $S^*$. Many but not all 
of these are simple to check.  

Although the required functions $I', J'$ may not exist on the whole of $S$, but $S$ remains
large enough to enable selecting a 
possibly proper subset $S'$ of $S^*$, still large in terms of the number of remaining times the induction step must be implemented.
The required functions $I', J'$ will now exist for all but one value of $l<n$ the values $I'(l), J'(l)$ are determined by $I, J$.
For  one value of $l$ there are at most $(n-1)r$ possible values, hence on a large subset the choices agree.  

Next it can be shown that $J'(l)\geq J(l)$ for all $l<n$.   Since 
$$(\alpha(s, i, j), \alpha(s, j, k), \alpha(s, i, k))\not\in Forb$$ 
and by the definition of $Forb$  
either $\rng(I')$ properly extends $\rng(I)$ or there is $l<n$ such that $J'(l)>J(l)$, hence  ${\sf rk}(I', J')>{\sf rk} (I, J)$.

Although in the previous sketch the construction is based on the same relation algebra $\A(n,r)$ as before,
we do not guarantee that the ultraproduct on $r$
of $\C(m,n,r)$ ($2<m<n<\omega)$, based on $\A(n,r)$, like in the cylindric case, is representable.
It is not all clear that we can lift the 
representability $\Pi_r\A(n,r)$ (this was proved above) 
to $\Pi_r\C(m,n,r)$ for any $m\leq n<\omega$ 
which was the case when we had diagonal elements. 

To prove the first two parts, the algebras we considered had 
an atom structure consisting of basic $m$ dimensional 
matrices, hence all the 
elements are generated by two dimensional elements.

In the third part 
elements are {\it essentially} three dimensional, 
so here we deal with three dimensional matrices, or tensors, if you like.
The third dimension is induced by a linear order on $r$.

Indeed, the second parameter $r<\omega$ may be considered as a finite linear order of length $r$.   
A standard Los argument shows that 
$\Pi_{r/U}\C(m, n, r) \cong\C(m, n, \Pi_{r/U} r)$ and $\Pi_{r/U}r$ 
contains an infinite ascending sequence 
Here we will have to extend the definition of $\psi$ 
by letting $\psi(n, r)=\omega,$ for any infinite linear order $r$.

What is proved in \cite{t} is the weaker item (III) of the above theorem.
This is done by proving that the infinite algebra 
$\C(m,n, J)\in \Nr_n\PEA_{m+1}$
when $J$ is an infinite linear order as above, and clearly $\Pi_{r/U} r$ is such. 
It is not clear at all whether this algebra is representable or not.

Now let $\K$ be a class 
of cylindric-like algebras.

Consider the statement:
\begin{athm}{Non-finite axiomatizability}\label{new} Let $\alpha >2$. Then for any $r\in \omega$, for any
$k\geq 1$, there exists $\B^{r}\in S\Nr_{\alpha}\K_{\alpha+k}\sim S\Nr_{\alpha}\K_{\alpha+k+1}$ such 
$\Pi_{r\in \omega}\B^r\in {\sf RK}_{\alpha}$.
In particular, ${\sf RK}_{\alpha}$ is not axiomatizable 
by a finite schema over $S\Nr_{\alpha}\K_{\alpha+k}$.
\end{athm}
This is true for $\CA$ and $\QEA$, but we do not know whether it is true for their diagonal free reducts $\Sc$ and $\QA$.

If we define the class of 
polyadic equality algebras corresponding to ${\sf CB}_{m,n}$ 
then we get the analogous result in theorem \ref{nt}
for polyadic equality algebras.

This  last statement has to do with {\it completeness} or rather the lack thereof. Our next theorem concerns 
definability \cite{conference}. It says that definability for infinite dimensions can be tricky, 
and indeed anti-intuitive.

We define the 
`neat reduct functor' from a
certain sub-category of $\CA_{\alpha+\omega}$
to $\RCA_{\alpha}$. More precisely, let
$$\L=\{\A\in \CA_{\alpha+\omega}: \A=\Sg^{\A}\Nr_{\alpha}\A\}.$$
Note that $\L\subseteq \RCA_{\alpha+\omega}$. The reason is that any $\A\in \L$ is generated by $\alpha$ -dimensional elements,
so is dimension complemented (that is $\Delta x\neq \alpha$ for all $x$), and such algebras are representable.
Consider $\Nr_{\alpha}$ as a functor from $\bold L$ to $\CA_{\alpha}$, but we restrict morphisms to one to one homomorphisms; that is we take only
embeddings, or injective homomorphisms.

Then the following can be easily destilled from 
\cite{conference}. Here $\sf MGR$ denotes the merry go round identities.

\begin{theorem}
\begin{enumarab}
\item  Let $\alpha$ be an infinite ordinal. Then there exists an $\RCA_{\alpha}$ that generates
two non- isomorphic algebras in $\omega$ extra dimensions, and-dually there exist
two non- isomorphic algebras that generate the same algebra. Using the definition and 
notation in \cite[definition 5.2.1, 5.2.2]{Sayedneat}, 
there are representable infinite dimensional algebras that lack both $NS$ 
nor $UNEP.$
\item $\Nr$ does not have a right adjoint
\item $\CA + \sf MGR$ does not have $AP$.
\end{enumarab}
\end{theorem}
\begin{proof} The second part of the first item follows from the first item in theorem \ref{SL}.
Now we prove the rest. We will be sketchy referring for details 
to \cite{conference}.

Let $\A=\Fr_4\CA_{\alpha}$ with $\{x,y,z,w\}$ its free generators. Let $X_1=\{x,y\}$ and $X_2=\{x,z,w\}$.
Let $r, s$ and $t$ be defined as follows:
$$ r = {\sf c}_0(x\cdot {\sf c}_1y)\cdot {\sf c}_0(x\cdot -{\sf c}_1y),$$
$$ s = {\sf c}_0{\sf c}_1({\sf c}_1z\cdot {\sf s}^0_1{\sf c}_1z\cdot -{\sf d}_{01}) + {\sf c}_0(x\cdot -{\sf c}_1z),$$
$$ t = {\sf c}_0{\sf c}_1({\sf c}_1w\cdot {\sf s}^0_1{\sf c}_1w\cdot -{\sf d}_{01}) + {\sf c}_0(x\cdot -{\sf c}_1w),$$
where $ x, y, z, \text { and } w$ are the first four free generators
of $\A$.
Then $r\leq s\cdot t$.
Let $\D=\Fr_4\RCA_{\alpha}$ with free generators $\{x', y', z', w'\}$.
Let  $\psi: \A\to \D$ be defined by the extension of the map $t\mapsto t'$, for $t\in \{x,y,x,w\}$.
For $i\in \A$, we denote $\psi(i)\in \D$ by $i'$.
Let $I=\Ig^{\D^{(X_1)}}\{r'\}$ and $J=\Ig^{\D^{(X_2)}}\{s'.t'\}$, and let
$$L=I\cap \D^{(X_1\cap X_2)}\text { and }K =J\cap \D^{(X_1\cap X_2)}.$$
Then $L=K$, and $\A_0=\D^{(X_1\cap X_2)}/L$  can be embedded into
$\A_1=\D^{(X_1)}/I$ and $\A_2=\D^{(X_2)}/J$,
but there is no amalgam even in $\CA_{\omega}.$ In particular,
$\CA_{\omega}$ with $\sf MGR$ does not have $AP$.

A piece of helpful terminology. 
If $\A\subseteq \Nr_{\alpha}\B$ and $\A$ generates $\B$, then $\B$ is called a minimal dilation 
of $\B$.

We claim that $\A_0$ genetates non-isomorphic algebras if we restrict isomorphisms to 
those that fixes $\A$ pointwise. (In principal, these minimal dilations can be isomorphic by an isomorphism that moves an
element of  $\A$).

Suppose for contradiction that 
$\A_0$ has the unique neat embeding property, that is, any two minimal $\omega$ 
dilations of $\A$ are isomorphic by an isomorphism that fixes
$\A$ pointwise. Then $\A_0$ lies in the amalgamation 
base of $\RCA_{\alpha}$ \cite{Sayedneat}.
We conclude that $\A_0$ does not have the unique neat embedding property, and so it fullfils the first part of the first 
item.

More generally, if $\D_{\beta}$ is taken as the free 
$\RCA_{\alpha}$ on $\beta$ generators, so that our algebra in the previous theorem is just $\D_4$, 
where $\beta\geq 4$, then the algebra constructed from $\D_{\beta}$ as above, 
will not have the unique neat embedding property, that is it generates non isomorphic algebras in extra 
dimensions.  
This immediately gives that the neat reduct functor does
not have a right adjoint.
In fact, the unique neat embedding property is equivalent 
to the left adjointness of the neat reduct 
functor.
\end{proof}

\subsection{Yet some more on neat games and atom structures}

In the following definition for a class of algebras $\K$ having a Boolean reduct $\K\cap \At$ denotes the atomic
algebras in $\K$. We also are loose about  what we mean by an atomic game, but in all cases it is a game played on finite atomic  networks,
and is a variant of the usual Lyndon game. This is still quite ambiguous, but is sharp enough for our purposes.
The game $F^m$ encountered in theorem \ref{neat11}, is an example of such games; we will encounter stronger games below,
like the game $J$, where \pa\ is allowed more moves, other than the normal cylindrifier
move. 

But we will also allow not necesarily atomic games, 
like games characterizing the class $\RCA_n$, here \pa\ can play elements in the algebra in question that are not necessarily atoms.

We do not require that \ws\ s are codable in first all logic 
(this is hard to do when we have amalgamation moves like the hyperbasis games played in our previous 
theorem \cite{HHbook2}),
but, in all cases, all our games will be deterministic. Only one player wins, and there are 
no
draws. 

Our definition is motivated by the following phenomena. 

There are classes of algebras $\K$ 
such that there are two algebras having the same atom structure, one in $\K$ and the other is not, 
example the clas $\RCA_n$ for finite $n>2$ and the class $\Nr_n\CA_m$ for all $m>n$.

Conversely, there are classes of algebras, like the class of completely representable
algebras in any finite dimension, and also $S_c\Nr_n\CA_{m}$ for any $m>n>2$ $n$ finite,  that do not
have this property. This is {\it not} marked by first
order definability, for the class of neat reducts and last two  classes not elementary, while
$\RCA_n$ is a variety for any $n$.

\begin{definition}\label{grip}
\begin{enumarab} 

\item A class $\K$ is gripped by its atom structures, if whenever $\A\in \K\cap \At$, and $\B$ is atomic such that
$\At\B=\At\A$, then $\B\in \K$.

\item A class $\K$ is strongly gripped by its atom structures, if whenever $\A\in \K\cap \At$, and $\B$ is atomic such that
$\At\B\equiv \At\A$, then $\B\in \K$.

\item A class $\K$ of atom structures is infinitary gripped
by its atom structures if whenever $\A\in \K\cap \At$ and $\B$ is atomic, such that $\At\B\equiv_{\infty,\omega}\At\B$,
then $\B\in \K$.
\item An atomic  game is strongly gripping for $\K$ if whenever $\A$ is atomic, and
\pe\ has a \ws\ for all finite rounded games on $\At\A$, then  $\A\in \K$

\item An atomic  game is gripping if whenever \pe\ has a \ws\ in the $\omega$ rounded game on $\At\A$, then $\A\in \K$.
\end{enumarab}
\end{definition}

Notice that infinitary gripped implies strongly gripped implies gripped (by its atom structures).
For the sake of brevity, we write only (strongly) gripped, without referring to atom structures.

In the next theorem, all items except the first applies to all algebras considered.
The first applies to any class $\K$ between $\sf Sc$ and
$\PEA$ (where the notion of neat reducts is not trivial).
The $n$th  Lyndon condition is a first order sentences that codes a \ws\ for \pe\ in $n$ rounds.

The elementary class satisfying all such sentences is denoted by ${\sf LCA_m}$.
It is not hard to show that ${\sf UpUr}\sf CRA_m={\sf LCA_m}$
for any $n>2$, to be proved below in a while. This basically follows from the simple observation that if \pe\ has a \ws\ in all finite rounded
atomic games on an atom structure of a $\CA_m$,
then this algebra is necessarily elementary equivalent to a countable completely representable
algebra \cite[lemma 4.3]{HH}.

\begin{theorem}\label{SL}
\begin{enumarab}
\item The class of neat reducts for any dimension $>1$ is not gripped,
hence is neither strongly gripped nor infinitary gripped.
\item The class of completely representable algebras is gripped but not strongly gripped.
\item The class of algebras satisfying the Lyndon conditions is strongly gripped.
\item The class of representable algebras is not gripped.
\item The Lyndon usual atomic game is gripping but not strongly gripping
for completely representable algebras, it is strongly gripping for ${\sf LCA_m}$, when $m>2$.
\item There is a game, that is {\it not} atomic,  that is strongly gripping for $\Nr_n\CA_{\omega}$, call it $H$. 
In particular, if there exists an algebra $\A$ and $k\geq 1$ such \pa\ has a \ws\ in $F^{n+k}$ and \pe\ has 
a \ws\ in $H_m$, the game $H$ truncated to $m$ rounds, for every finite $m$, 
then for any class $\K$ such that  $\Nr_n\CA_{\omega}\subseteq \K\subseteq  S_c\Nr_n\CA_{n+k},$ 
$\K$ will not be elementary.
\end{enumarab}
\end{theorem}

\begin{proof}
\begin{enumarab}

\item This example is an adaptation of an example used in \cite{SL} to show that that for any pair of ordinals
$1<\alpha<\beta$, the class of $\Nr_{\alpha}\CA_{\beta}$ is not closed under
forming subalgebras, and it was also used in other contexts proving negative results on various amalgamation properties for
cylindric-like algebras \cite{STUD}, \cite{Sayedneat}. It proves to be a nut cracker, so it
will be used several times below.

Here we slightly generalize the example by allowing an arbitrary field to rather than $\mathbb{Q}$
to show that there  is an atom structure that carries simultaneously an algebra in $\Nr_{\alpha}\CA_{\alpha+\omega}$
and an algebra not in $\Nr_{\alpha}\CA_{\alpha+1}$. This
works for all $\alpha>1$ (infinite included) and other cylindric-like algebras as will be  clear from
the proof. Indeed, the proof works for any class of algebras whose signature is between $\Sc$ and $\sf QEA$.
(Here we are using the notation $\QEA$ instead of $\PEA$ because we are
allowing infinite dimensions).

Let $\alpha$ be an ordinal $>1$; could be infinite. Let $\F$ is field of characteristic $0$.
$$V=\{s\in {}^{\alpha}\F: |\{i\in \alpha: s_i\neq 0\}|<\omega\},$$
$${\C}=(\wp(V),
\cup,\cap,\sim, \emptyset , V, {\sf c}_{i},{\sf d}_{i,j}, {\sf s}_{\tau})_{i,j\in \alpha, \tau\in FT_{\alpha}}.$$
Then clearly $\wp(V)\in \Nr_{\alpha}\sf QPEA_{\alpha+\omega}$.
Indeed let $W={}^{\alpha+\omega}\F^{(0)}$. Then
$\psi: \wp(V)\to \Nr_{\alpha}\wp(W)$ defined via
$$X\mapsto \{s\in W: s\upharpoonright \alpha\in X\}$$
is an isomorphism from $\wp(V)$ to $\Nr_{\alpha}\wp(W)$.
We shall construct an algebra $\A$, $\A\notin \Nr_{\alpha}{\sf QPEA}_{\alpha+1}$.
Let $y$ denote the following $\alpha$-ary relation:
$$y=\{s\in V: s_0+1=\sum_{i>0} s_i\}.$$
Let $y_s$ be the singleton containing $s$, i.e. $y_s=\{s\}.$
Define as before
${\A}\in {\sf QPEA}_{\alpha}$
as follows:
$${\A}=\Sg^{\C}\{y,y_s:s\in y\}.$$

Now clearly $\A$ and $\wp(V)$ share the same atom structure, namely, the singletons.
Then we claim that
$\A\notin \Nr_{\alpha}{\sf QPEA}_{\beta}$ for any $\beta>\alpha$.
The first order sentence that codes the idea of the proof says
that $\A$ is neither an elementary nor complete subalgebra of $\wp(V)$.
Let $\At(x)$ be the first order formula asserting that $x$ is an atom.
Let $$\tau(x,y) ={\sf c}_1({\sf c}_0x\cdot {\sf s}_1^0{\sf c}_1y)\cdot {\sf c}_1x\cdot {\sf c}_0y.$$
Let $${\sf Rc}(x):=c_0x\cap c_1x=x,$$
$$\phi:=\forall x(x\neq 0\to \exists y(\At(y)\land y\leq x))\land
\forall x(\At(x) \to {\sf Rc}(x)),$$
$$\alpha(x,y):=\At(x)\land x\leq y,$$
and  $\psi (y_0,y_1)$ be the following first order formula
$$\forall z(\forall x(\alpha(x,y_0)\to x\leq z)\to y_0\leq z)\land
\forall x(\At(x)\to \At(\sf c_0x\cap y_0)\land \At(\sf c_1x\cap y_0))$$
$$\to [\forall x_1\forall x_2(\alpha(x_1,y_0)\land \alpha(x_2,y_0)\to \tau(x_1,x_2)\leq y_1)$$
$$\land \forall z(\forall x_1 \forall x_2(\alpha(x_1,y_0)\land \alpha(x_2,y_0)\to
\tau(x_1,x_2)\leq z)\to y_1\leq z)].$$
Then
$$\Nr_{\alpha}{\sf QPEA}_{\beta}\models \phi\to \forall y_0 \exists y_1 \psi(y_0,y_1).$$
But this formula does not hold in $\A$.
We have $\A\models \phi\text {  and not }
\A\models \forall y_0\exists y_1\psi (y_0,y_1).$
In words: we have a set $X=\{y_s: s\in V\}$ of atoms such that $\sum^{\A}X=y,$ and $\A$
models $\phi$ in the sense that below any non zero element there is a
{\it rectangular} atom, namely a singleton.

Let $Y=\{\tau(y_r,y_s), r,s\in V\}$, then
$Y\subseteq \A$, but it has {\it no supremum} in $\A$, but {\it it does have one} in any full neat reduct $\B$ containing $\A$,
and this is $\tau_{\alpha}^{\B}(y,y)$, where
$\tau_{\alpha}(x,y) = {\sf c}_{\alpha}({\sf s}_{\alpha}^1{\sf c}_{\alpha}x\cdot {\sf s}_{\alpha}^0{\sf c}_{\alpha}y).$

In $\wp(V)$ this last is $w=\{s\in {}^{\alpha}\F^{(\bold 0)}: s_0+2=s_1+2\sum_{i>1}s_i\},$
and $w\notin \A$. The proof of this can be easily distilled from \cite[main theorem]{SL}.
For $y_0=y$, there is no $y_1\in \A$ satisfying $\psi(y_0,y_1)$.
Actually the above proof proves more. It proves that there is a
$\C\in \Nr_{\alpha}{\sf QEA}_{\beta}$ for every $\beta>\alpha$ (equivalently $\C\in \Nr_n\QEA_{\omega}$), and $\A\subseteq \C$, such that
$\Rd_{\Sc}\A\notin \Nr_{\alpha}\Sc_{\alpha+1}$.
See \cite[theorems 5.1.4-5.1.5]{Sayedneat} for an entirely different example.

\item The algebra $\PEA_{\Z,\N}$, and its various reducts down to $\Sc$s,
shows that the class of completely representable algebras is not strongly gripped.
Indeed, it can be shown that \pe\ can win all finite rounded atomic games but \pa\ can win the
$\omega$ rounded game. It is known that this class is gripped. An atom structure is completely representable iff one,
equivalently, all atomic algebras, sharing
this atom structure are completely representable.

\item This is straightforward from the definition of Lyndon conditions \cite{HHbook}.

\item Any weakly representable atom structure that is not strongly representable detects this, see e.g.
the main results in \cite[theorem 1.1, corolary 1.2, corolary 1.3]{Hodkinson},
\cite{weak}, \cite[theorems 1.1, 1.2]{ANT} and theorems \ref{can}, and theorem \ref{hodkinson}.
For a potential stronger result, see theorem \ref{blurs} above.
\item From item (2) and the rest follows directly from the definition.

\item  We allow more moves for \pa\ in the neat game. 

We define a game that is strictly stronger than the game $J$, played on hypergraphs, 
which are coloured graphs endowed with arbitrary long labelled hyperedges  
in which \pa\ has even more moves and this allows is to remove $\At$ 
from both sides of the above equation formulated in (*), obtaining a much stronger result
The main play of the stronger game $H(\A)$ is a play of the game $J(\A).$ 

Recall that $J(A)$ was played on $\lambda$ neat hypernetworks. 
The base of the main board at a certain point will be the the neat $\lambda$ hypernework, call its 
network part $X$ and we write 
$X(\bar{x})$ for the atom that labels the edge $\bar{x}$ on the main board. 
But now \pa\ can make other moves too, 
which makes it harder for \pe\ to win and so a \ws\ for \pe\ in this new $\omega$ rounded game 
will give a stronger result. 

A definition. 
An $n$  network is a finite complete graph with nodes including $0, \ldots, n-1$ 
with all edges labelled by {\it arbitrary elements} of $\A$. No consistency properties are assumed. 

\pa\ can play an arbitrary $n$ network $N$, \pe\ must replace $N(0, \ldots, n-1),$  by 
some element $a\in A$. The idea, is that the constraints represented by $N$ correspond to an element of the $\RCA_\omega$ being constructed on $X$, 
generated by $A$.

The final move is that \pa\ can pick a previously played $n$ network $N$ and pick any  tuple $\bar{x}$ 
on the main board whose atomic label is below $N(0,1\ldots, n-1)$. 

\pe\ must respond by extending the main board from $X$ to $X'$ such that there is an embedding $\theta$ of $N$ into $X'$
 such that $\theta(0)=x_0\ldots , \theta(n-1)=x_{n-1}$ and for all $i_0, \ldots i_{n-1} \in N,$ we have 
$X(\theta(i_0)\ldots, \theta(i_{n-1}))\leq N(i_0,\ldots, i_{n-1})$. 

This ensures that in the limit, the constraints in 
$N$ really define the element $a$. If \pe\ has a \ws\ in $H(A)$ then the extra moves mean that every $n$ dimensional element generated by 
$\A$ in the $\B\in \RCA_\omega$ 
constructed in the play is an element of $\A$, so that $\A$ exhausts all $n$ dimensional elements 
of $\B$, so we actually have $\A\cong \Nr_n\B$.

The rest will follow from the standard reasoning. The algebra $\A$ will be elementary equivalent to a countable $\B\in \Nr_n\CA_{\omega}$,
but it will not be in $S_c\Nr_n\CA_{n+3}$

\end{enumarab}
\end{proof}

Concerning the last item, we can say more. Indeed, more generally the following.
Assume that If  ${\sf L}$ is a class of algebras, $G$ be a game and  $\A\in {\sf L}$, then \pe\ has a \ws\ in $G$. 
Asume also that $\Nr_n\CA_{\omega}\subseteq \sf L$. 
If there is an atomic algebra $\A$ such that \pe\ can win all finite rounded games of $J$ 
and \pa\ has a \ws\ in $G$ then any class between 
$\Nr_n\CA_{\omega}$ and $\sf L$ is not elementary.

\subsection{Neat embeddings and various notions of representability}

We shall repeatedly use the example of item (1) in \ref{SL}. In what follows
$\K\in \{\Sc, \PA, \PEA, \CA\}$,  ${\sf RK}_n$ denotes
the class of representable algebras of dimension $n$ while $\sf CRK_n$ denotes the class of completely representable
$\sf RK_n$s. 
The class of algebras satisfying the Lyndon conditions is denoted by 
${\sf LCK}_n$. 
Recall the definition of $S_c\K$ introduced in definition \ref{s}; $\A\in S_c\K$ if there is an algebra in $\K$ such that $\A$ embeds completely into 
this algebra.

For the well known definitions of pseudo universal and pseudo elementary classes,
the reader is referred  to \cite[definition 9.5, definition 9,6]{HHbook2}.
In fact all our results below hold for any class whose signature is between
$\Sc$ and $\PEA$. (Here $\Df$ is not counted in
because the notion of neat reducts for this class is trivial \cite[theorem 5.1.31]{tarski}).
A small minority of results in the next theorem partially intersect 
with results proved in \cite{recent}.

\begin{theorem}\label{maintheorem}
\begin{enumarab}

\item For $n>2$, the inclusions $\Nr_n\K_{\omega}\subseteq S_c\Nr_n\K_{\omega}\subseteq S\Nr_n\K_{\omega}$ are proper.
The first strict inclusion can be witnessed by a finite algebra for $n=3$, while the second cannot by witnessed by a finite
algebra for any $n$.
In fact, $m>n>1$, the inclusion $\Nr_n\K_m\subseteq S_c\Nr_n\K_m$ is proper
and for $n>2$ and  $m\geq n+3$ the inclusion $S_c\Nr_n\K_m\subseteq S\Nr_n\K_m$ is also proper.

\item For any pair of ordinals $1<\alpha<\beta$ the class $\Nr_{\alpha}\K_{\beta}$ is not elementary.
In fact, there exists an uncountable atomic algebra $\A\in \Nr_{\alpha}\QEA_{\alpha+\omega}$, hence $\A\in \Nr_{\alpha}\QEA_{\beta}$
for every $\beta>\alpha$, and $\B\subseteq_c \A$,
such that $\B$ is completely representable, $\A\equiv \B$,  so that $\B$ is also atomic,
but $\Rd_{sc}\B\notin \Nr_{\alpha}\Sc_{\alpha+1}$.
For finite dimensions, we have $\At\A\equiv_{\infty} \At\B$.

\item For finite $n$, the elementary theory of $\Nr_n\K_{\omega}$ is recursively enumerable.

\item For $n>1$, the class ${\sf S}_c\Nr_n\K_{\omega}$ is not elementary (hence not pseudo-universal)
but it is pseudo-elementary, and the  elementary theory of $S_c\Nr_n\K_{\omega}$ is  recursively enumerable.

\item For $n>1$, the class ${\sf S}_c\Nr_n\K_{\omega}$ is closed under forming
strong subalgebras but is
not closed under forming subalgebras

\item For $n>2$, the class ${\sf UpUr}{\sf S}_c\Nr_n\K_{\omega}$ is not finitely axiomatizable. For any $m\geq n+2,$
$S\Nr_n\K_m$
is not finitely axiomatizable and for $p\geq 2$, $\alpha>2$ (infinite included) and $\K\in \{\CA, \sf QEA\}$
the variety $S\Nr_{\alpha}\K_{\alpha+2}$ cannot be finitely axiomatized by a universal set of
formulas containing only finitely many variables.

\item For $n>1$, the class consisting of atomic algebras in  ${\sf UpUr}{\sf S}_c\Nr_n\K_{\omega}$
coincides with the class  of algebras satisfying the  Lyndon conditions
so that: $${\sf UpUr}{\sf S}_c\Nr_n\K_{\omega}\cap \At={\sf UpUr}[{\sf S}_c\Nr_n\K_{\omega}\cap \At]={\sf LCK_n}={\sf UpUr CRK_n}.$$

\item For $n>2$, both ${\sf UpUr}\Nr_n\K_{\omega}$ and ${\sf UpUr}S_c\Nr_n\K_{\omega}$ are properly contained in $\sf RK_n$;
by closing under forming subalgebras both
resulting classes coincide with $\sf RK_n.$
Furthermore, the former is properly contained in the latter.

\item For $n\geq 1$, the class of completely representable ${\sf CRK_n}$ algebras is
closed under ${\sf S}_c$, but the class $\Nr_n\K_{\omega}$, for $n>1$, is not.
If $\A\in {\sf CRK_n}$, and $\At\B=\At\A$, then $\B\in {\sf CRK_n}$.
(That is ${\sf CRA_n}$ is gripped by its atom structures). This is not the case with $\Nr_n\K_{\omega}.$

\item \label{robin} ${\sf CRK}_n\subseteq S_c\Nr_n\K_{n+k}$ for any finite $k$. For $n>2$, any class $\K$ that contains
the class $S_c\Nr_n\K_{\omega}$ and is contained in $S_c\Nr_n\K_{n+3}$ is not elementary.
Furthermore, if there is an atom structure such that \pe\ has a \ws\ in the game $H_n$ defined  in last item of theorem \ref{main}, 
for every $n\in \omega$, and \pa\ can win $F^3$, then we can replace the class $S_c\Nr_n\K_{\omega}$ by the smaller $\Nr_n\K_{\omega}$. 
(By the first item of \cite{SL}, and second item above, 
we know that it is {\it strictly} smaller).  

\item Any countable atomic algebra in $\Nr_n\K_{\omega}$
is completely representable, but there are uncountable
algebras in $\Nr_n\CA_{\omega}$  that are not completely representable. Such algebras exist also in infinite dimension.
In particular, for finite $n>2$, the classes  ${\sf UpUr}\Nr_n\K_{\omega}$ and ${\sf CRAS_n}$ are not related both ways.

\item In contrast, every algebra in $\Nr_n\CA_{\omega}$ 
has an $\omega$ complete relativized representation, 
and is strongly representable, that is its \d\ completion 
is representable.
\end{enumarab}
\end{theorem}

\begin{proof}
\begin{enumarab}

\item The  inclusions are obvious.

The strictness of the inclusion follows from item (1)  of theorem \ref{SL}, since
$\A$ is dense in $\wp(V)$, hence it is in $S_c\Nr_n\K_{\omega}$, but it is not in $\Nr_n\K_{n+1}$, {\it a fortiori} it is not
in $\Nr_n\K_{\omega}$ for the latter is clearly contained in the former.

For the strictness of the second  inclusion, let $\A$ be the rainbow algebra $\PEA_{\Z, \N}$. Then $\PEA_{\Z, \N}$ is representable,
in fact, it satisfies the Lyndon conditions, hence is also strongly representable, but it is  {\it not} completely
representable, in fact its $\Df$ reduct is not completely representable, and its $\Sc$ reduct  is not
in $S_c\Nr_n\Sc_{\omega}$, for had it been in this class, then it would be completely representable, inducing a complete
representation of $\A$ by the following standard `omitting types' argument. 

With abuse of notation we write $\A$ for $\Rd_{sc}\A$.
Assume that $\A\subseteq_c\Nr_n\B$, with $\B\in \Sc_{\omega}$. Then  we can assume that $\A$ is countable, since $\Tm\At\A$
is countable and is dense
in $\Cm\At\A$; so a complete representation of $\Tm\At\A$ 
will give a complete representation of $\A$ itself, because the class of completely representable algebras is
gripped.
Furthermore, we can assume that $A$ generates $\B$ so that $\B$ is also countable
$\A$ is a complete  subalgebra of $\B$ and $\B\in {\sf LSc}_{\omega}$ (a locally finite Pinter's algebra)

Let $X$ be the set of co-atoms of $\A$ (a co-atom is the complement of an atom).
Then $\prod^{\A}X=0$, because $\A$ is atomic;  by completeness it follows
that $\prod ^{\Nr_nB}X=0$, hence $\prod ^{\B}X=0$,
so by the usual omitting types theorem for first order logic without equality
(which does not apply to the uncountable case, but here our algebra is countable)
there is a generalized set countable algebra $\C$ and
an injective homomorphism
$f: \B\to \C$ such that $\bigcap_{x\in X} f(x)=\emptyset$.
The restriction of $f$ to $\A$ in  the obvious way, gives an atomic,
hence, a complete representation of $\A$.
If $\A\in \sf RK_n$ is finite, then any representation is complete, hence it is in $S_c\Nr_n\K_{\omega}$.

We have proved that there is a  representable algebra, namely, $\PEA_{\Z, \N}$
such that  $\Rd_{Sc}\PEA_{\Z, \N}\notin S_c\Nr_n\Sc_{\omega}$. The required now follows
by using the neat embedding theorem that says that ${\sf RK}_n=S\Nr_n\K_{\omega}$
for any $\K$ as specified above \cite{tarski}, indeed we have ${\Rd_K}{\sf PEA}_{\Z, \N}$ is not in $S_c\Nr_n\K_{\omega}$ but it is
(strongly) representable, that is, it is  in $S\Nr_n\K_{\omega}$.

Note that the \d\ completion of $\PEA_{\Z, \N}$ is not representable, for if it were,
then this induces a complete representation
of $\PEA_{\Z, \N}$. In fact, we can go further stipulating that
its $\Df$ reduct is not  representable, because by the same token this induces a complete representation of $\Rd_{df}\PEA_{\Z, \N}$ which
in turn induces a complete representation of $\PEA_n$ itself, since the latter is generated by elements whose dimension sets $<n$.
On the other hand, its canonical extension
is completely representable a result of Monk \cite[corollary 2.7.24]{tarski}.

For the second part concerning finite $3$ dimensional algebras witnessing the strictness of inclusions.
Take the finite polyadic equality 
algebra $\A$ consisting of three dimensional matrices (as defined by Monk)
over any integral non-permutational relation algebra $\R$.
Such relation algebras exist  \cite[theorem 36]{r}. The algebra $\A$ is finite, hence completely representable, 
hence  $\A\in S_c\Nr_3\QEA_{\omega}.$ Suppose for contradiction that $\Rd_{Sc}\A\in \Nr_n\Sc_{\omega}$, 
so that $\At\A\in \At\Nr_n\Sc_{\omega}.$
Then we claim that $\A$ has a $3$-homogeneous complete representation, which is impossible, because $\R$
does not have a homogeneous representation.

It can be shown, using arguments similar to \cite[theorem 33]{r},
that \pe\ has a \ws\ in an $\omega$ rounded game $K$ but played on atomic networks (not $\lambda$ neat hypernetworks)
with moves similar (but not identical) to $J$, so \pa\ s moves are like above, except that in amalgamation moves on networks there is an additional
restriction. The networks he chooses can overlap only on
at most $3$ nodes.
\pe\ uses her \ws\  to define a sequence
of networks $N_0\subseteq \ldots N_r$ such  that this sequence respects the cylindrifier
move in the sense that if $N_r(\bar{x})\leq {\sf c}_ia$ for $\bar{x}\in \nodes(N_r)$, then there
exists $N_s\supseteq N_r$ and a node $k\in \omega\sim N_r$ such that $N_s(\bar{y})=a$;
and also respects the partial isomorphism move, in the sense that if
if $\bar{x}, \bar{y}\in \nodes(N_r)$ such that $N_r(\bar{x})=N_r(\bar{y})$,
then there is a finite surjective map extending $\{(x_i, y_i): i<n\}$ mapping onto $\nodes(N)$
such that $\dom(\theta)\cap \nodes(N_r)=\bar{y}$,
and we seek an extension $N_s\supseteq N_r$, $N_r\theta$ (some
$s\geq r$).
Then if $\tau$ is a partial isomorphism
of $N_a$ and $X$ is any finite
subset of $\nodes(N_a)$ then there is a
partial isomorphism $\theta\supset \tau, \rng(\theta)\supset X$.

Define the representation  $\cal N$ of $\A$ with domain $\bigcup_{a\in A}\nodes(N_a)$, by
$$S^{\cal N}=\{\bar{x}: \exists a\in A, \exists s\in S, N_a(\bar{x})=s\},$$
for any subset $S$ of $\At\A$. (This is similar to the proof of item (3) in theorem \ref{main}).
Then this representation of $\A$,
is obviously complete, and by the definition of the game is $n$ homogeneous.

The strictness of the last inclusion cannot be witnessed by a finite representable
algebra for any finite dimension $n>2$, because any such algebra, is atomic (of course)
and completely representable, hence it will be necessarily in $S\Nr_n\K_{\omega}$.

Concerning the second strictness of inclusions (concerning neat embeddability in finitely many extra dimensions),
the first is witnessed by $\A$ constructed in \ref{SL} or the $\B$ constructed in the next item,
while
the second follows by noting  $S_c\Nr_n\K_m$ is not elementary by theorem \ref{nofinite}
while $S\Nr_n\K_m$ is a
variety.

\item  \cite{IGPL, quasi, MLQ, Fm, Note, Sayedneat}.
In \cite{IGPL} it is proved that for any $1<\alpha<\beta$, there exists $\A\in \Nr_{\alpha}\CA_{\beta}$ and
$\B\notin \Nr_{\alpha}\CA_{\alpha+1}$ such that both are atomic, $\A\equiv \B$. 
Furthermore, $\A$ does not depend on $\beta$ and $\B$ is completely representable.

We now show that $\At\A\equiv_{\infty}\At\B$. We work only with $\CA$s.
We need to give more detailed information about the constrction of $\A$ and $\B$.
We consider the case when $\A$ and $\B$ are cylindric algebras as constructed in \cite[theorem 5.1.3]{Sayedneat}.

A finite  atom structure, 
namely, an $n$ dimensional cartesian square,  with accessibility relations corresponding to the concrete interpretations of 
cylindrfiers and diagonal elements, is fixed in advance.

Then its atoms are
split twice.  Once, each atom is split into uncountably many, and once each into uncountably many except for one atom which
is only split  into {\it countably} many atoms. These atoms are called big atoms, which mean that they are cylindrically equivalent to their
original. This is a general theme in splitting arguments.
The first splitting gives an algebra $\A$ that is a full neat reduct of an algebra in arbitrary extra dimensions;
the second gives an algebra $\B$ that is not a full neat reduct
of an algebra in just one extra dimensions, hence in any higher
dimensions. Both algebras are representable, and elementary equivalent
because first order logic cannot se this cardinality twist. 
We will show in a minute that their atom structures are $L_{\infty, \omega}$ equivalent. 

We, henceforth, work with dimension $3$. The proof for higher finite dimensions 
is the same. However, we make a slight perturbation 
to the construction in {\it op.cit}; we require that the interpretation of the uncountably tenary 
relation symbols in the signature of $\sf M$ on which the set algebras $\A$ and $\B$ are based
are {\it disjoint}, not only distinct \cite[Theorems 5.3.1. 5.3.2]{Sayedneat}.

The Boolean reduct of $\A$ can be viewed as a finite direct product of disjoint Boolean relativizations of $\A$,  
denoted in \cite[theorem 5.3.2]{Sayedneat} by $\A_u$; $\A_u$ is the finite-cofinite algebra on a set having the same cardinality as the signature;
it is relativized to $1_u$ as defined in {\it opcit}, $u\in {}^33$. 

Each component will be atomic by our further restriction on $\sf M$,
so that $\A$ itself, a product of atomic algebras is also atomic. The language of Boolean algebras can now be expanded 
so that $\A$ is interpretable in an expanded structure $\P$, 
based on the  same atomic Boolean product. Now $\B$ can be viewed as obtained from $\P$, 
by replacing one of the components of the product with an elementary
{\it countable} Boolean subalgebra, and then giving it the same interpretation.
By the Fefferman Vaught theorem (which says that replacing in a product one of its components by an elementary 
equivalent one, the resulting product remains elementary equivalent to the original product) we have $\B\equiv \A$.
In particular, $\B$ is also atomic.

First order logic will not see this cardinality twist, but a suitably chosen term
not term definable in the language of 
$\CA_3$, namely, the substitution operator, $_3{\sf s}(0,1)$  will 
witnessing that the twisted algebra $\B$ is not a neat reduct.

The Boolean structure of both algebras are in fact very simple. For a set $X$, let  ${\sf Cof} (X)$ denote 
the finite co-finite Boolean algebra on $X$, that is ${\sf Cof}(X)$ has universe $\{a\in \wp(X): |a|<\omega, \text { or } |X\sim a|<\omega\}$.
Let $J$ be any set having the same cardinality as the signature of $\sf M$  
so that $J$ can simply be ${\sf M}$. 
Then $\A\cong \prod_{u\in {}^33}\A_u$, where $\A_u={\sf Cof}(J)$ 
and $\B=\prod_{u\in {}^33}\B_u$, where $\B_{Id}={\sf Cof(\sf N)}$ ($\sf N$ can in fact 
be any set such that  $|\sf N|<|J|$ but for definiteness let it be the least infinite cardinal) 
and otherwise $\B_u=\A_u$.

To prove that $\At\A\equiv_{\infty}\At\B$, we devise a pebble game similar to the rainbow pebble game
but each player has the option to choose an element from {\it both} structures, 
and not just stick to one so that it is {\it a back and forth game} not just a forth game.  

Pairs of pebbles are outside the board.
\pa\ as usual starts the game by placing a pebble on an element of one of the structures. \pe\
responds by placing the other pebble on the an element on the other structure.
Between them they choose an atom $a_i$ of $\At\A$
and an atom  $b_i$ of $\At\B$, under the restriction that player \pe\
must choose from the other structure from player \pa\ at each step.
A win for \pe\ if the binary relation resulting from the choices of the two players $R=\{(a,b): a\in \At(\A), b\in \At(\B)\}$ is a partial isomorphism.

At each step, if the play so far $(\bar{a}, \bar{b})$ and \pa\ chooses an atom $a$
in one of the structures, we have one of two case.
Either $a.1_u=a$ for some $u\neq Id$
in which case
\pe\ chooses the same atom in the other structure.
Else $a\leq 1_{Id}$
Then \pe\ chooses a new atom below $1_{Id}$
(distinct from $a$ and all atoms played so far.)
This is possible since there finitely many atoms in
play and there are infinitely many atoms below
$1_{u}$.
This strategy makes \pe\ win, since atoms below $1_u$ are cylindrically equivalent to $1_u$.
Let $J$ be a back and forth system which exists.
Order $J$ by reverse inclusion, that is $f\leq g$
if $f$ extends $g$. $\leq$ is a partial order on $J$.
For $g\in J$, let $[g]=\{f\in J: f\leq g\}$. Then $\{[g]: g\in J\}$ is the base of a
topology on
$J.$ 

Let $\C$ be the complete
Boolean algebra of regular open subsets of $J$ with respect to the topology
defined on $J.$
Form the Boolean extension $\M^{\C}.$
We want to define an isomorphism in $\M^{\C}$ of $\breve{\A}$ to
$\breve{\B}.$
Define $G$ by
$||G(\breve{a},\breve{b})||=\{f\in {J}: f(a)=b\}$
for $c\in \A$ and $d\in \B$.
If the right-hand side,  is not empty, that is it contains a function $f$, then let
$f_0$ be the restriction of $f$ to the substructure of $\A$ generated by $\{a\}$.
Then $f_0\in J.$ Also $\{f\in J:  f(c)=d\}=[f_0]\in \C.$
$G$ is therefore a $\C$-valued relation. Now let $u,v\in \M$.
Then
$||\breve{u}=\breve{v}||=1\text { iff }u=v,$
and
$||\breve{u}=\breve{v}||=0\text { iff } u\neq v$
Therefore
$||G(\breve{a},\breve{b})\land G(\breve{a},\breve{c})||\subseteq ||\breve{b}=\breve{c}||.$
for $a\in \A$ and $b,c\in \B.$
So ``$G$ is a function." is valid.
It is one to one because its converse is also a function.
(This can be proved the same way).
Also $G$ is surjective.

One can alternatively show that $\A\equiv_{\infty\omega}\B$ using "soft model theory" as follows:
Form a Boolean extension $\M^*$ of the universe $\M$
in which the cardinalities of $\A$ and $\B$ collapse to
$\omega$.  Then $\A$ and $\B$ are still back and forth equivalent in $\M^*.$
Then $\A\equiv_{\infty\omega}\B$ in $\M^*$, and hence also in $\M$
by absoluteness of $\models$.
\footnote{Consider the two relations $\cong$ (isomorphism) and $\equiv$ (elementary equivalence)
between
structures. In one sense isomorphism is a more intrinsic property of structures,
because it is defined directly in terms of
structural properties. $\equiv$, on the other hand, involves a (first order) language.
But in another sense elementary equivalence is more
intrinsic because the existence of an isomorphism can depend on some
subtle questions about the surrounding universe of sets.

For example if $\M$ is a transitive model  of set theory containing vector
spaces $V$ and $W$ of dimensions $\omega$ and $\omega_1$ over
the same countable field, then $V$ and $W$ are not isomorphic in $\M$,
but they are isomorphic in an extension of $\M$
obtained by collapsing the cardinal $\omega_1$ to $\omega.$
By contrast, the question whether structures $\A$ and $\B$ are
elementary equivalent depends only on
$\A$ and $\B$, and not on sets around them.
Then $\equiv_{\infty}$ relation which coincides with so called Boolean-isomorphisms
as ilustrated above.  This notion hovers between these two notions, and is
purely structural. It can be characterized by back and forth systems.}

\item We first show that $\Nr_n\CA_{\omega}$ is pseudo-elementay. This is similar to the proof of \cite[theorem 21]{r} using a three sorted first order theory.
from which we can infer the elementary theory 
$\Nr_n\CA_{\omega}$ is recursively  enumerable for any finite $n$.
 
To show that $\Nr_n\K_{\omega}$  is pseudo-elementary, we use a three sorted defining theory, with one sort for a cylindric algebra of dimension $n$
$(c)$, the second sort for the Boolean reduct of a cylindric algebra $(b)$
and the third sort for a set of dimensions $(\delta)$; the argument is analogous to that of Hirsch used for relation algebra reducts \cite[theorem 21]{r}.
We use superscripts $n,b,\delta$ for variables
and functions to indicate that the variable, or the returned value of the function,
is of the sort of the cylindric algebra of dimension $n$, the Boolean part of the cylindric algebra or the dimension set, respectively.
We do it for $\CA$s.  The other cases can be dealt with in exactly the sme way.

The signature includes dimension sort constants $i^{\delta}$ for each $i<\omega$ to represent the dimensions.
The defining theory for $\Nr_n{\sf CA}_{\omega}$ includes sentences stipulating
that the constants $i^{\delta}$ for $i<\omega$
are distinct and that the last two sorts define
a cylindric algebra of dimension $\omega$. For example the sentence
$$\forall x^{\delta}, y^{\delta}, z^{\delta}(d^b(x^{\delta}, y^{\delta})=c^b(z^{\delta}, d^b(x^{\delta}, z^{\delta}). d^{b}(z^{\delta}, y^{\delta})))$$
represents the cylindric algebra axiom ${\sf d}_{ij}={\sf c}_k({\sf d}_{ik}.{\sf d}_{kj})$ for all $i,j,k<\omega$.
We have have a function $I^b$ from sort $c$ to sort $b$ and sentences requiring that $I^b$ be injective and to respect the $n$ dimensional
cylindric operations as follows: for all $x^r$
$$I^b({\sf d}_{ij})=d^b(i^{\delta}, j^{\delta})$$
$$I^b({\sf c}_i x^r)= {\sf c}_i^b(I^b(x)).$$
Finally we require that $I^b$ maps onto the set of $n$ dimensional elements
$$\forall y^b((\forall z^{\delta}(z^{\delta}\neq 0^{\delta},\ldots (n-1)^{\delta}\rightarrow c^b(z^{\delta}, y^b)=y^b))\leftrightarrow \exists x^r(y^b=I^b(x^r))).$$

In all cases, it is clear that any algebra of the right type is the first sort of a model of this theory.
Conversely, a model for this theory will consist of an $n$ dimensional cylindric algebra type (sort c),
and a cylindric algebra whose dimension is the cardinality of
the $\delta$-sorted elements, which is at least $|m|$.
Thus the three sorted theory defines the class of neat reduct, furthermore, it is clearly recursive.

Finally, if $\K$ be a pseudo elementary class, that is
$\K=\{M^a|L: M\models U\}$ of $L$ structures, and $L, L^s, U$ are recursive.
Then there a set of first order recursive theory  $T$ in $L$, 
so that for any $\A$ an $L$ structure, we have
$\A\models T$ iff there is a $\B\in \K$ with $\A\equiv \B$. In other words, 
$T$ axiomatizes the closure of $\K$ under elementary equivalence, see 
\cite[theorem 9.37]{HHbook} for unexplained notation and proof.
\item Now we prove that $S_c\Nr_n\K_{\omega}$ is not elementary.
One possible proof goes as follows. Consider the rainbow algebra $\PEA_{\bf K_{\omega}, \bf K}$, where
$\K$ is the disjoint union of the graphs $\K_n$ $(1\leq n<\omega$), and (as defined before)
$\bf K_{\omega}$ is the complete irreflexive graph with $\omega$ nodes.
Then \pe\ has a \ws\ for all finite rounded games, but \pa\ can win
the $\omega$ rounded game.
Indeed, \pe\ has a \ws\ in  the private \ef\ game ${\sf EF}_n^{\omega}(\bf K_{\omega}, \bf K)$ for each $n\in \omega$.
She can always respond to a move by \pa\ by placing a pebble in $\bf \K_n\subseteq \bf K$. When there are only $n$ rounds this
is possible. But \pa\ can win the $\omega$ rounded game by placing successive pebbles on distinct elements of $\bf K_{\omega}$.

At the outset \pe\ must
choose which $\bf K_n$ in $\bf K$ to respond in, and from there on
she must stick to it.  After $n$ rounds she will have pebbled every element
of this copy and in the next round she loses. Now let $\A=\PEA_{\bf K_{\omega}, \bf K}$; as above
\ws\ s for either player can be lifted to the rainbow algebra.
Then we claim that $\Rd_{sc}\A\notin S_c\Nr_n\Sc_{\omega}$, for if it were then it would have a complete representation by the above
argument which impossible. Now let $\B$ be an ultrapower of $\A$,
using an elementary chain argument, let $\D$ be an elementary countable subalgebra
of $\B$ in which \pe\ has a \ws\ in $\omega$ rounds, hence $\D$ is completely representable, and so  also by the above argument
it is in $S_c\Nr_n\PEA_{\omega}$. We have shown that $\Rd_{sc}\A\notin S_c\Nr_n\Sc_{\omega}$, $\D\in S_c\Nr_n\PEA_{\omega}$
and $\A\equiv \D$. It thus follow that $S_c\Nr_n\K_{\omega}$ is not elementary.
Here the rainbow algebras $\PEA_{\Z,\N}$ or $\PEA_{\omega, \omega}$ can be also used.

It is pseudo-elementary because of the following reasoning.
For brevity, let ${\sf L}= S_c\Nr_n\K_{\omega}\cap \At$, and let $\sf  CRK_n$ denote the class of completely representable algebras.
Let $T$ be the first order theory that axiomatizes ${\sf Up Ur CRK}_n={\sf LCK_n}$.  
It suffices to show, since $\sf CRK_n$ is pseudo-elementary \cite{HHbook}, that
$T$ axiomatizes ${\sf UpUr  L}$ as well.
First, note that $\sf CRA_n\subseteq \sf L$.
Next assume that $\A\in {\sf UpUr L}$, then $\A$ has a countable elementary (necessarily atomic)
subalgebra in $\sf L$  which is completely representable by the above argument, and we
are done.

\item We show that $S_c\Nr_n\K_{\omega}$ is not closed under forming subalgebras, hence it is not pseudo-universal.
That it is closed under $S_c$ follows directly from the definition.
Consider the  $\Sc$ reduct of either ${\sf PEA}_{\Z, \N}$ or $\PEA_{\omega, \omega}$ or $\PEA_{\K,_{\omega},\K}$ of the previous item.
Fix one of them, call it $\A$. Because $\A$ has countably many atoms, $\Tm\A\subseteq \A\subseteq \Cm\At\A$,
and all three are completely representable or all three not completely representable sharing the same atom structure $\At\A$,
we can assume without loss that $\A$ is countable. 
Now $\A$ is not  completely representable, hence by the above 'omitting types argument' in item (1),
its $\Sc$ reduct is not in $S_c\Nr_n\Sc_{\omega}$.

On the other hand, $\A$ is strongly
representable, so its canonical extension
is representable, indeed completely representable, hence as claimed
this class is not closed under
forming subalgebras, because the $\Sc$  reduct of the canonical extension of $\A$
is in $S_c\Nr_n\sf Sc_{\omega}$, $\Rd_{sc}\A$ is not in $S_c\Nr_n\Sc_{\omega}$ and $\A$ embeds into its canonical 
extension. The first of these statements follow from the fact that if $\D\subseteq \Nr_n\B$, 
then $\D^+\subseteq \Nr_n\B^+$.

\item To prove non finite axiomatizability we give two entirely different proofs.

(a) First we use the flexible construction in \cite{ANT}.
There are several parameters used to define the relation algebra denoted by $\M$ defined on p. 84 in {\it op.ct.}, 
and denoted by $\R$, in theorem \ref{blurs}.
Both algebras were not representable, 
because they are based on graphs having finite colouring, expressed otherwise, the number of blurs (non principal ultrafilters)
were finite.

Now let $l\in \omega$, $l\geq 2$, and let $\mu$ be a non-zero cardinal. Let $I$ be a finite set,
$|I|\geq 3l.$ Let
$J=\{(X,n): X\subseteq I, |X|=l,n<\mu\}.$
Here  $I$ is the atoms of $\M$. $J$ is the set of blurs, consult  \cite[definition 3.1]{ANT} for the definition of blurs.
Pending on $l$ and $\mu$, let us call these atom structures ${\cal F}(I, l,\mu).$
If $\mu\geq \omega$, then $J$ would be infinite,
and $\Uf$, the set of non principal ultrafilters corresponding to the blurs, will be a proper subset of the ultrafilters.
It is not difficult to show that if $l\geq \omega$
(and we relax the condition that $I$ be finite), then
$\Cm{\cal F}(I, l,\mu)$ is completely representable,
and if $l<\omega$, then $\Cm{\cal F}(I, l,\mu)$ is not representable.

Let ${\D}$ be a non-trivial ultraproduct of the atom structures ${\cal F}(I, i,1)$, $i\in \omega$. Then $\Cm{\D}$
is completely representable.
Thus $\Tm{\cal F}(I, i,1)$ are ${\sf RRA}$'s
without a complete representation while their ultraproduct has a complete representation.

Also the sequence of complex algebras $\Cm{\cal F}(I, i,1)$, $i\in \omega$
consists of algebras that are non-representable with a completely representable ultraproduct.
This implies using the above that this ultraproduct (though uncountable) is in $S_c\Nr_n\CA_{\omega}$.
Then because our algebras posses
$n$ dimensional cylindric basis, the result lifts easily to cylindric algebras.

(b) Alternatively, one can prove the cylindric case directly as follows.Take $\G_i$ to be the disjoint union of cliques of size $(n(n-1)/2)+i$, or
let $\G_i$ be the graph with nodes $\N$ and edge relation $(i,j)\in E$ if $0<|i-j|<n(n-1)/2+i$.
Let $\alpha_i$ be the corresponding atom structure, as defined in \cite{weak} and $\A_i$ be the term algebra based on $\alpha_i$.
Then, as shown in \cite{weak}, $\Cm \A_i$ is not representable because, as proved in \cite{weak},
$\alpha$ is weakly but not strongly representable, see theorem \ref{hodkinson}.

But $\prod_{i\in \omega}\Cm\A_i=\Cm(\prod_{i\in \omega}\A_i)$, so
the one  graph is based on the disjoint union of the cliques which is arbitrarily large, and the second on graphs
which have arbitrary large chromatic
number, hence both are completely representable.

The rest follows from theorems \ref{nofinite}, \cite[theorem 2]{Andreka} and  \cite{c}.

\item Let $\A\in S_c\Nr_n\K_{\omega}$. Then \pe\ can win $G_n$ for all $n$ hence it satisfies the Lyndon conditions; this will be proved
in the last item.
Conversely, let $\A$ be an algebra satisfying the Lyndon conditions.
We can assume that it is countable and atomic, using the Tarski \ls\ downward theorem, 
 by taking an elementary countable subalgebra.
Recall that the class of algebras satisfying the Lyndon conditions is elementary by definition.

Then \pe\ has a \ws\ for all finite rounded games on $\A$. Take an ultrapower of $\A$, then \pe\ has an \ws\ in this
ultrapower; by an elementary chain argument again, one can extract an elementary countable
subalgebra of the ultrapower, $\D$ say,  where \pe\ still has a \ws\ in the $\omega$ rounded game on $\D$.
$\D$ is then countable and completely representable, hence
$\D\in S_c\Nr_n\K_{\omega}$, and $\A\in {\sf UpUr}\{\D\}$, hence $\A\in {\sf Up Ur}S_c\Nr_n\K_{\omega}$.

That ${\sf Up Ur}{\sf CRK}_n={\sf LCK}_n$ is almost the same argument.
One side is obvious. For the other side,
let $\A\in \sf LCK_n$. Then \pe\ has a \ws\ in the $k$ rounded game for all $k$, so as usual by ultrapowers and
an elementary chain argument the algebra is elementary equivalent to $\B$, where
$\B$ is countable atomic and \pe\ has a \ws\ in the $\omega$ rounded game on $\At\B$, hence
$\B$ is completely representable by \cite[theorem 3.3.3]{HHbook2},
and we are done.

Note that we can prove the previous item, in view of the hitherto established equality,
by showing that the rainbow $\CA_{m+3, m+2}$ witnesses that ${\sf LCA_n}$ is not finitely
axiomatizable.

\item The algebra $\A$ in item (1) in theorem \ref{SL}, witnesses the strictness of the stated  last inclusion.
The strictness of second inclusion from the fact that the class in question coincides 
with the class of algebras satisfying the Lyndon conditions, and we have proved in
that this class this is properly contained in even the class of  (strongly) representable algebras.

Indeed,  let $\Gamma$ be any graph with infinite chromatic number,
and large enough finite girth. Let $\rho_k$ be the $k$ the Lyndon condition for $\Df$s.
Let $m$ be also large enough so that any $3$ colouring of the edges of a complete graph
of size $m$ must contain a monochromatic triangle; this $m$ exists by Ramseys's theorem.
Then $\M(\Gamma)$, the complex algebra constructed on $\Gamma$, as defined in \cite[definition sec 6.3,  p.78]{HHbook}
will be representable as a polyadic equality algebra but it will fail $\rho_k$ for all $k\geq m$. The idea is
that \pa\ can win in the $m$ rounded atomic game coded by $\sigma_m$, by forcing a forbidden monochromatic triangle.

The last required from from the fact that the resulting classes coincide with $S\Nr_n\K_{\omega}$
which, in turn,  coincides with the class of representable algebras by the
neat embedding theorem of Henkin.

\item $\Nr_n\K_{\omega}$ is not closed under $S_c$ by item (1) theorem \ref{SL} and second item.
But ${\sf CRK_n}$ is closed under $S_c$ because
of the following straightforward argument.
If $\B$ is completely representable, and $\A\subseteq_c \B$, then $\At\A=\At\B$,
and $\sum ^{\B}\At\A=1$, hence if $M$ is a complete representation of $\B$, then $\bigcup_{x\in At\A} f(x)=1^M$,
so $f\upharpoonright \A$ is an atomic, hence, a complete representation of
$\A$.  Let $\At$ be an atom structure that carries a completely representable algebra $\B$, and assume
that $g:\At\to \wp(V)$, where $V$ is a generalized set algebra  is the complete representation.  Assume that $\A$ is atomic and $\At\A=\At\B$.
Then we define  a complete representation
of $\A$ as follows. Let $a\in A$, then $a=\sum\{x\in \At : x\leq a\}$. Let $J=\{x\in \At: x\leq a\}\subseteq \At$.
Now define $f$ with universe $\A$ by $f(a)=\bigcup\{g(x): x\in J\}$. This can be easily checked to be a complete representation.
The rest says that ${\sf CRA_n}$ is gripped while $\Nr_n\K_{\omega}$ is not.
Indeed, we  have yet again by item(1) theorem \ref{SL}, that
$\At\A=\At\C$, $\C\in \Nr_n\QEA_{\omega}$
but $\Rd_{sc}\A\notin \Nr_n\Sc_{n+1}.$
(Recall that the latter is completely representable).

\item \label{neat} First  we have ${\sf CRK}_n\subseteq S_c\Nr_n\K_{\omega}\subseteq S_c\Nr_n\K_{n+3}$.
The first inclusion was proved above, the second is obvious.

Second part immediately from theorem \ref{neat11}, 
by noting that the class of completely representable algebras coincide with the class $S_c\Nr_n\K_{\omega}$ on atomic
countable algebras. It readily follows that $\K$ is not
elementary. The last part follows from the argument in item (4) of theorem \ref{main}. 

\item For the  statement that ${\sf CRA_n}$ and ${\sf UpUr}\Nr_n\CA_{\omega}$ are not related both ways, see \cite[remark 31]{r} and \cite{BSL}.
In \cite{r} a sketch  of constructing  an uncountable relation algebra 
$\R\in \Ra\CA_{\omega}$ (having an $\omega$ dimensional cylindric basis)
with no complete representation is given. In \cite{BSL} it is the special case of this example when $\kappa=\omega$ 
but the idea in all three proofs is essentialy the same,
using a variant of the rainbow  relation algebra 
$\R_{\omega_1, \omega}$. 

Although \pe\ has a \ws\ in the 
$\omega$ usual rounded atomic game played on networks, this is not enough to build a complete representation, her \ws\ can only
be used to build an $\omega$ relattivized complete representatio, as indicated below.
Because the algebra is uncountable these two last notions
are distinct, the latter strictly stronger than the former, witness the proof of theorem 
\ref{longer}.

Assume that $\R=\Ra\B$ and $\B\in \PEA_{\omega}$, then 
$\Rd_{K}\Nr_n\B$ is as required, for a complete representation of it, induces easily a complete
representation of $\Ra\CA_{\omega}$.
The latter example shows that ${\sf UpUr} \Nr_n\K_{\omega}\nsubseteq {\sf CRA}_n$,

On the other hand,  the example in item (1) of theorem \ref{SL}
shows that the other inclusion does not
hold as well.
So neither class is contained in the other proving the
required.

Now give the details of the construction in \cite[remark 31]{r}.
We pove more, namely,  there exists an {\it uncountable} neat reduct, that is, an algebra in $\Nr_n\sf QEA_{\omega}$,
such that its $\Df$ reduct is not completely representable; however, it has an $\omega$ relativized representation, as a $\PEA_n$.

This example also shows that the condition of countability in characterizing complete representations via neat embeddings
as well as a the maximality condition in \cite[theorem 3.2.9]{Sayed}
(imposed on $< {}2^{\omega}$ non-principal types by taking $\kappa$ to be $\omega$) 
in an  omitting types theorem proved by Shelah, addressing first order logic, 
restricted to $L_n$, formulated in theorem \ref{Shelah1}.

Using the terminology of rainbow constructions, 
we allow the greens to be of cardinality $2^{\kappa}$ for any
infinite cardinal $\kappa$, and the reds to be of cardinality $\kappa$.
Here a \ws\ for \pa\ witnesses that  the algebra has {\it no} complete
representation. But this is not enough because we want our algebra to be
in $\Ra\CA_{\omega}$; we will show that it will be.

As usual we specify the atoms and forbidden triples.
The atoms are $\Id, \; \g_0^i:i<2^{\kappa}$ and $\r_j:1\leq j<
\kappa$, all symmetric.  The forbidden triples of atoms are all
permutations of $(Id, x, y)$ for $x \neq y$, \/$(\r_j, \r_j, \r_j)$ for
$1\leq j<\kappa$ and $(\g_0^i, \g_0^{i'}, \g_0^{i^*})$ for $i, i',
i^*<2^{\kappa}.$  In other words, we forbid all the monochromatic
triangles.

Write $\g_0$ for $\set{\g_0^i:i<2^{\kappa}}$ and $\r_+$ for
$\set{\r_j:1\leq j<\kappa}$. Call this atom
structure $\alpha$.

Let $\A$ be the term algebra on this atom
structure; the subalgebra of $\Cm\alpha$ generated by the atoms.  $\A$ is a dense subalgebra of the complex algebra
$\Cm\alpha$. We claim that $\A$, as a relation algebra,  has no complete representation.

Indeed, suppose $\A$ has a complete representation $M$.  Let $x, y$ be points in the
representation with $M \models \r_1(x, y)$.  For each $i< 2^{\kappa}$, there is a
point $z_i \in M$ such that $M \models \g_0^i(x, z_i) \wedge \r_1(z_i, y)$.

Let $Z = \set{z_i:i<2^{\kappa}}$.  Within $Z$ there can be no edges labeled by
$\r_0$ so each edge is labelled by one of the $\kappa$ atoms in
$\r_+$.  The Erdos-Rado theorem forces the existence of three points
$z^1, z^2, z^3 \in Z$ such that $M \models \r_j(z^1, z^2) \wedge \r_j(z^2, z^3)
\wedge \r_j(z^3, z_1)$, for some single $j<\kappa$.  This contradicts the
definition of composition in $\A$ (since we avoided monochromatic triangles).

Let $S$ be the set of all atomic $\A$-networks $N$ with nodes
 $\omega$ such that $\{\r_i: 1\leq i<\kappa: \r_i \text{ is the label
of an edge in N}\}$ is finite.
Then it is straightforward to show $S$ is an amalgamation class, that is for all $M, N
\in S$ if $M \equiv_{ij} N$ then there is $L \in S$ with
$M \equiv_i L \equiv_j N.$
Hence the complex cylindric algebra $\Ca(S)\in \CA_\omega$.

Now let $X$ be the set of finite $\A$-networks $N$ with nodes
$\subseteq\omega$ such that
\begin{enumerate}
\item each edge of $N$ is either (a) an atom of
$\c A$ or (b) a cofinite subset of $\r_+=\set{\r_j:1\leq j<\kappa}$ or (c)
a cofinite subset of $\g_0=\set{\g_0^i:i<2^{\kappa}}$ and
\item $N$ is `triangle-closed', i.e. for all $l, m, n \in \nodes(N)$ we
have $N(l, n) \leq N(l,m);N(m,n)$.  That means if an edge $(l,m)$ is
labeled by $1'$ then $N(l,n)= N(mn)$ and if $N(l,m), N(m,n) \leq
\g_0$ then $N(l,n).\g_0 = 0$ and if $N(l,m)=N(m,n) =
\r_j$ (some $1\leq j<\omega$) then $N(l,n).\r_j = 0$.
\end{enumerate}
For $N\in X$ let $N'\in\Ca(S)$ be defined by
\[\set{L\in S: L(m,n)\leq
N(m,n) \mbox{ for } m,n\in nodes(N)}\]
For $i,\omega$, let $N\restr{-i}$ be the subgraph of $N$ obtained by deleting the node $i$.
Then if $N\in X, \; i<\omega$ then $\cyl i N' =
(N\restr{-i})'$.
The inclusion $\cyl i N' \subseteq (N\restr{-i})'$ is clear.

Conversely, let $L \in (N\restr{-i})'$.  We seek $M \equiv_i L$ with
$M\in N'$.  This will prove that $L \in \cyl i N'$, as required.
Since $L\in S$ the set $X = \set{\r_i \notin L}$ is infinite.  Let $X$
be the disjoint union of two infinite sets $Y \cup Y'$, say.  To
define the $\omega$-network $M$ we must define the labels of all edges
involving the node $i$ (other labels are given by $M\equiv_i L$).  We
define these labels by enumerating the edges and labeling them one at
a time.  So let $j \neq i < \omega$.  Suppose $j\in \nodes(N)$.  We
must choose $M(i,j) \leq N(i,j)$.  If $N(i,j)$ is an atom then of
course $M(i,j)=N(i,j)$.  Since $N$ is finite, this defines only
finitely many labels of $M$.  If $N(i,j)$ is a cofinite subset of
$a_0$ then we let $M(i,j)$ be an arbitrary atom in $N(i,j)$.  And if
$N(i,j)$ is a cofinite subset of $\r_+$ then let $M(i,j)$ be an element
of $N(i,j)\cap Y$ which has not been used as the label of any edge of
$M$ which has already been chosen (possible, since at each stage only
finitely many have been chosen so far).  If $j\notin \nodes(N)$ then we
can let $M(i,j)= \r_k \in Y$ some $1\leq k < \kappa$ such that no edge of $M$
has already been labeled by $\r_k$.  It is not hard to check that each
triangle of $M$ is consistent (we have avoided all monochromatic
triangles) and clearly $M\in N'$ and $M\equiv_i L$.  The labeling avoided all
but finitely many elements of $Y'$, so $M\in S$. So
$(N\restr{-i})' \subseteq \cyl i N'$.

Now let $X' = \set{N':N\in X} \subseteq \Ca(S)$.
Then the subalgebra of $\Ca(S)$ generated by $X'$ is obtained from
$X'$ by closing under finite unions.
Clearly all these finite unions are generated by $X'$.  We must show
that the set of finite unions of $X'$ is closed under all cylindric
operations.  Closure under unions is given.  For $N'\in X$ we have
$-N' = \bigcup_{m,n\in \nodes(N)}N_{mn}'$ where $N_{mn}$ is a network
with nodes $\set{m,n}$ and labeling $N_{mn}(m,n) = -N(m,n)$. $N_{mn}$
may not belong to $X$ but it is equivalent to a union of at most finitely many
members of $X$.  The diagonal $\diag ij \in\Ca(S)$ is equal to $N'$
where $N$ is a network with nodes $\set{i,j}$ and labeling
$N(i,j)=1'$.  Closure under cylindrification is given.
Let $\C$ be the subalgebra of $\Ca(S)$ generated by $X'$.
Then $\A = \Ra(\C)$.
Each element of $\A$ is a union of a finite number of atoms and
possibly a co-finite subset of $a_0$ and possibly a co-finite subset
of $a_+$.  Clearly $\A\subseteq\Ra(\C)$.  Conversely, each element
$z \in \Ra(\C)$ is a finite union $\bigcup_{N\in F}N'$, for some
finite subset $F$ of $X$, satisfying $\cyl i z = z$, for $i > 1$. Let $i_0,
\ldots, i_k$ be an enumeration of all the nodes, other than $0$ and
$1$, that occur as nodes of networks in $F$.  Then, $\cyl
{i_0} \ldots
\cyl {i_k}z = \bigcup_{N\in F} \cyl {i_0} \ldots
\cyl {i_k}N' = \bigcup_{N\in F} (N\restr{\set{0,1}})' \in \A$.  So $\Ra(\C)
\subseteq \A$.
$\A$ is relation algebra reduct of $\C\in\CA_\omega$ but has no
complete representation; in fact $\C\in \QEA_{\omega}$, because the basis is obviosly symmetric.

Let $n>2$. Let $\B=\Nr_n \C$. Then
$\B\in \Nr_n\QEA_{\omega}$, is atomic, but has no complete representation; in fact because it is binary generated its 
$\Df$ reduct is not completelt representable (witness the argument in the last part of the proof of theorem \ref{hodkinson}.)

However, this algebra is strongly representable, that is $\Cm\At\A$ is representable, 
because it is full neat reduct (see the last part of the proof).

Now we show that it is $\omega$ completely representable, namely, it has  a representation $M$ that respects all finite cliques.
Recall that clique, as in usual graph theory, 
is a sequence $s$ of length $m\geq 2$, such that for any $i,j\in \rng(s)$, we have $M(i, j)$.

For $\A$ we play the usual $\omega$ rounded atomic game on atomic networks.
It suffices to show that \pe\ has a \ws\ in the $\omega$ rounded game.
(Because the algebra is uncountable this does not guarantee that the algebra itself is completely representable; in fact, we
know that  it is not).

Assume that \pe\ survives till the $r$ round $r<\omega$, and that the current play is in round $r+1$. \pa\ chooses a previously played
network $N_s$, the edge $x,y\in N_t$ and atoms $a,b\in A$ such that $N_t(x,y)\leq a;b.$
Then \pa\ has to choose a witness for this composition and enlarge the network such that $N_{t+1}(x,z)=a$ and
$N_t+1)(x,y)=b$, and we can assume that
$N_{t+1}$ has just one extra node. If there is a witness in $N_t$ there
is nothing to prove so assume not.
Let \pa\ play the triangle move $(N_s, i, j, k,a, b)$
in round $r+1$. \pe\ has to choose labels for the edges $\{(x, k), (k,x)\}$,
$x\in \nodes(N_s)\sim \{i,j\}$.
She chooses the labels for  the edges $(x,k)$ one at a time
and then determines the labels of the reverse edge $(k, x)$
uniquely. We give the uncountably many atoms the green colour $\g$,
and the countably many the colour $\r$.

We have several cases:
If it is not the case that $N_s(x, i)$ and $a$ are both green,
and it is not the case that $N_s(x, j)$ and $b$ are both green, \pe\ lets $N_{s+1}(x,k)$ a new green $\g$.

If $N_s(i,j)=\r$, $N_s(x, i)=\g,$ $N_s(x, j)=\g$ and $a=b$, then \pe\ lets $N_{s+1}(x,k)$ a new $\r$.
If neither, then $N(x, i)=\g$, $a=\g$ and $N_s(x, j)=b$ or $N_s(x, i)=a$ and $N_{s}(x,j)=\g$
and $b=\g$, she lets
$N_{s+1}(x, k)$ a new $\r$. This implies that \pa\ can win $G_{\omega}$ and we are done.
Note that the $\Df$ reduct of this algebra is not completely representable because it is binary generated.

For the second part, we lift our finite dimensional examples to the transfinite like we did in theorem \ref{2.12}.
Let $\alpha$ be an infinite ordinal.
Then we claim that there exists $\B\in \Nr_{\alpha}\QEA_{\alpha+\omega}$ that is atomic, uncountable and its $\Df$ reduct 
is not completely representable.
For $k\geq 3$, let $\C(k)\in \Nr_k\QEA_{\omega}$ be an atomic cuncountble $\sf PEA_k$ 
whose $\Df$ reduct is not completely representable; exists by the above.
Let $I=\{\Gamma: \Gamma\subseteq \alpha,  |\Gamma|<\omega\}$.
For each $\Gamma\in I$, let $M_{\Gamma}=\{\Delta\in I: \Gamma\subseteq \Delta\}$,
and let $F$ be an ultrafilter on $I$ such that $\forall\Gamma\in I,\; M_{\Gamma}\in F$.
For each $\Gamma\in I$, let $\rho_{\Gamma}$
be a one to one function from $|\Gamma|$ onto $\Gamma.$
Let ${\C}_{\Gamma}$ be an algebra similar to $\CA_{\alpha}$ such that
$\Rd^{\rho_\Gamma}{\C}_{\Gamma}={\C}(|\Gamma|)$. In particular, ${\C}_{\Gamma}$ has an atomic Boolean reduct.
Let $\B=\prod_{\Gamma/F\in I}\C_{\Gamma}$
We will prove that
$\B\in \Nr_\alpha\QEA_{\alpha+\omega},$  $\B$ is atomic and $\B$ is not completely representable. The last two requirements are
easy. $\B$ is atomic, because it is an ultraproduct of atomic algebras.
$\B$ is not completely representable, even on weak units, because any such  representation induces
a complete (square) representation of its $k$ neat reducts, $k\geq 3$, which we know do not have even a $\Df$ complete
representation.

For each $\Gamma\in I$, we know that $\C(|\Gamma|+k) \in\PEA_{|\Gamma|+k}$ and
$\Nr_{|\Gamma|}\C(|\Gamma|+k)\cong\C(|\Gamma|)$.
Let $\sigma_{\Gamma}$ be an injective map
 $(|\Gamma|+\omega)\rightarrow(\alpha+\omega)$ such that $\rho_{\Gamma}\subseteq \sigma_{\Gamma}$
and $\sigma_{\Gamma}(|\Gamma|+i)=\alpha+i$ for every $i<\omega$. Let
$\A_{\Gamma}$ be an algebra similar to a
$\QEA_{\alpha+\omega}$ such that
$\Rd^{\sigma_\Gamma}\A_{\Gamma}=\C(|\Gamma|+k)$.  Then $\Pi_{\Gamma/F}\A_{\Gamma}\in \QEA_{\alpha+\omega}$.
We now prove that $\B= \Nr_\alpha\Pi_{\Gamma/F}\A_\Gamma$.
Using the fact that neat reducts commute with forming ultraproducts, for each $\Gamma\in I$, we have
\begin{align*}
\Rd^{\rho_{\Gamma}}\C_{\Gamma}&=\C(|\Gamma|)\\
&\cong\Nr_{|\Gamma|}\C(|\Gamma|+k)\\
&=\Nr_{|\Gamma|}\Rd^{\sigma_{\Gamma}}\A_{\Gamma}\\
&=\Rd^{\sigma_\Gamma}\Nr_\Gamma\A_\Gamma\\
&=\Rd^{\rho_\Gamma}\Nr_\Gamma\A_\Gamma
\end{align*}
We deduce that

$$\B=\Pi_{\Gamma/F}\C_\Gamma\cong\Pi_{\Gamma/F}\Nr_\Gamma\A_\Gamma=\Nr_\alpha\Pi_{\Gamma/F}\A_\Gamma\in \Nr_{\alpha}\QEA_{\alpha+\omega}.$$

\item Now for the last part.
We write $Id_{-i}$ for the function $\{(k,k): k\in n\sim\{i\}\}.$
For a network $N$ and a partial map $\theta$ from $n$ to $n$, that is $\dom\theta\subseteq n$, recall that 
$N\theta$ is the network whose labelling is defined by $N\theta(\bar{x})=N(\tau(\bar{x}))$
where for $i\in n$, $\tau(i)=\theta(i)$ for $i\in \dom\theta$
and $\tau(i)=i$ otherwise.
Recall too that $F^m$ is the usual atomic game
on networks, except that the nodes are $m$ and \pa\ can re use nodes.

Let $\A\in S_c\Nr_n\K_{\omega}$ (possibly uncountable) be atomic. 
We know that if $\A$ is countable then it
is completely representable, hence strongly representable, but as just shown, if $\A$ is uncountable, and is in the strictly 
smaller class  $\Nr_n\CA_{\omega}$ it might not be completely representable.
Howecer in all cases it has an $\omega$ relativized representation, and it is also {\it  strongly} representable.

Suppose that $\A\subseteq_c \Nr_n\C$, $\C\in \K_{\omega}$.
Define (see definition \ref{subs}, for unexplained notation here, namely, the substitution operator) 
for an  $\A$-network $N$ with $\nodes(N)\subseteq \N$,
$\widehat N\in\C$ by
\[\widehat N =
 \prod_{i_0,\ldots, i_{n-1}\in\nodes(N)}{\sf s}_{i_0, \ldots, i_{n-1}}N(i_0\ldots, i_{n-1})\]

The following is not hard to check \cite[lemma 26]{r}:
\begin{enumerate}
\item For any $x\in\C\setminus\set0$ and any
finite set $I\subseteq m$ there is a network $N$ such that
$\nodes(N)=I$ and $x\;.\;\widehat N\neq 0$.
\item
For any networks $M, N$ if
$\widehat M\;.\;\widehat N\neq 0$ then $M\equiv^{\nodes(M)\cap\nodes(N)}N$.
\item\label{it:-i}
If $i\not\in\nodes(N)$ then ${\sf c}_i\widehat N=\widehat N$.

\item \label{it:-j} $\widehat{N Id_{-j}}\geq \widehat N$.

\item\label{it:ij} If $i\not\in\nodes(N)$ and $j\in\nodes(N)$ then
$\widehat N\neq 0 \rightarrow \widehat{N[i/j]}\neq 0$
where $N[i/j]=N\circ [i|j]$

\item\label{it:theta} If $\theta$ is any partial, finite map $n\to n$
and if $\nodes(N)$ is a proper subset of $n$,
then $\widehat N\neq 0\rightarrow \widehat{N\theta}\neq 0$.
\end{enumerate}

We prove only the first two items. 
The proof of the first part is based on repeated use of
lemma ~\ref{lem:atoms2}. We define the edge labelling of $N$ one edge
at a time. Initially no hyperedges are labelled.  Suppose
$E\subseteq\nodes(N)\times\nodes(N)\ldots  \times\nodes(N)$ is the set of labelled hyper
edges of $N$ (initially $E=\emptyset$) and
$x\;.\;\prod_{\bar c \in E}{\sf s}_{\bar c}N(\bar c)\neq 0$.  Pick $\bar d$ such that $\bar d\not\in E$.
By lemma~\ref{lem:atoms2} there is $a\in\At(\c A)$ such that
$x\;.\;\prod_{\bar c\in E}{\sf s}_{\bar c}N(\bar c)\;.\;{\sf s}_{\bar d}a\neq 0$.
Include the edge $\bar d$ in $E$.  Eventually, all edges will be
labelled, so we obtain a completely labelled graph $N$ with $\widehat
N\neq 0$.
it is easily checked that $N$ is a network.
For the second part, if it is not true that
$M\equiv^{\nodes(M)\cap\nodes(N)}N$ then there are is
$\bar c \in^{n-1}\nodes(M)\cap\nodes(N)$ such that $M(\bar c )\neq N(\bar c)$.
Since edges are labelled by atoms we have $M(\bar c)\cdot N(\bar c)=0,$
so $0={\sf s}_{\bar c}0={\sf s}_{\bar c}M(\bar c)\;.\; {\sf s}_{\bar c}N(\bar c)\geq \widehat M\;.\;\widehat N$.
We leave the easy proof of rest of the items to the reader.

We show that \pe\ has a \ws\ in $G_{\omega}.$
In the game \pe\ always
plays networks $N$ with $\nodes(N)\subseteq n$ such that
$\widehat N\neq 0$. In more detail, in the initial round, let \pa\ play $a\in \At\A$.
\pe\ plays a network $N$ with $N(0, \ldots, n-1)=a$. Then $\widehat N=a\neq 0$.
At a later stage suppose \pa\ plays the cylindrifier move
$(N, \langle f_0, \ldots, f_{n-2}\rangle, k, b, l)$
by picking a
previously played network $N$ and $f_i\in \nodes(N), \;l<n,  k\notin \{f_i: i<n-2\}$,
and $b\leq {\sf c}_kN(f_0,\ldots,  f_{i-1}, x, f_{i+1}, \ldots f_{n-2})$.
Let $\bar a=\langle f_0.\ldots, f_{i-1}, k, f_{i+1}, \ldots f_{n-2}\rangle.$
Then ${\sf c}_k\widehat N\cdot {\sf s}_{\bar a}b\neq 0$.
Then by the above and lemma \ref{lem:atoms2}, 
there is a network  $M$ such that
$\widehat{M}\cdot \widehat{{\sf c}_kN}\cdot {\sf s}_{\bar a}b\neq 0$. Then $M$ is the required
response.

It is now clear that \pe\ can win $F^{\omega}$, hence $G_{\omega},$ hence $G_k$ for all finite $k\geq n$.
Thus  $\A$ satisfies the Lyndon conditions and so
it is strongly representable, that is, its \d\ completion, namely, $\Cm\At\A$
is representable. This also proves the missing part in theorem \ref{neat11}; the proof there refers to this last item.

\end{enumarab}
\end{proof}
Henceforth  we carry out our discussions and  formulate our theorems for cylindric algebras.
Everything said about cylindric algebras carries over to the other cylindric-like algebras 
approached in our previous investigations like $\Sc, \PA$ and $\PEA$.

However, $\Df$s does not count here, for the notion of neat reducts
for such algebras is trivial.
Lifting from atom structures, we define several classes of atomic representable algebras.
${\sf CRA_n}$ denotes the class of completely representable algebras of dimension $n$.
Let ${\sf SRCA_n}$ be the class of strongly representable atomic algebras of dimension $n$,
$\A\in {\sf SRCA_n}$ iff $\Cm\At\A\in \RCA_n$.  ${\sf WRCA_n}$ denotes the class of weakly representable algebras of dimension $n$,
and this is just $\RCA_n\cap \At$.  We  have the following strict inclusions lifting them up from atom structures \cite{HHbook2} (*):
$${\sf CRA}_n\subset {\sf LCA_n}\subset {\sf SRCA_n}\subset {\sf WCRA}_n$$  

The second  and fourth classes are elementary but not finitely axiomatizable, bad Monk  algebras converging to a good one (Monk's original algebras
are like that),
can witness this, while ${\sf SRCA_n}$ is not closed under both ultraroots and ultraproducts, 
good Monks algebras converging to a bad one witnesses 
this. Rainbow algebras, in fact all three given in item (5) of the last theorem witness that ${\sf CRA_n}$ is not elementary,
since they are not completely represebntable bur are elementary equivalent to (countable) algebras 
that are. From this we readily conclude that ${\sf CRCA_n}$ is properly contained in ${\sf LCA}_n$.

For a cylindric algebra atom structure $\F$ the first order algebra over $\F$ is the subalgebra
of $\Cm\F$ consisting of all sets of atoms that are first order
definable with parameters from $S$. ${\sf FOCA_n}$ denotes the class of atomic such algebras of dimension $n$.

This class is strictly bigger than ${\sf SRCA_n}$.
Indeed, let $\A$ be any of the two Monk algebra constructed in theorem \ref{hodkinson} or the algebra based
or the rainbow term construction obtained by blowing up and blurring a finite rainbow algebra, proving that
$S\Nr_{n}\CA_{n+4}$ is not atom canonical.
Recall that such algebras were defined using first order formulas, the first in a Monk's signature, the second in the rainbow signature
(the latter is first order since we had only
finitely many greens).  Though the usual semantics was perturbed, first order logic did not see the relativization, only 
infinitary formulas saw it,
and thats why the complex algebras could not be represented. 
These examples all show that ${\sf SRCA_n}$ is properly contained in ${\sf FOCA_n}.$
Another way to view this is to notice that ${\sf FOCA_n}$ is elementary, that ${\sf SRA_n}\subseteq {\sf FOCA_n}$, but
${\sf SRCA_n}$ is not elementary.

Let $\K$ be the class of atomic representable algebras having $NS$, and ${\sf L}$ be the class of atomic
representable algebras having  $NS$ the unique neat embedding property; these are defined 
in \cite[definitions, 5.2.1, 5.1.2]{Sayedneat}.
Obviously, the latter is contained in the former, and both are contained in $\Nr_n\CA_{\omega}$ by definition.

There is a  plathora of very interesting classes between the
atomic algebras in the amalgamation base of ${\sf RCA_n}$ and atomic algebras in
${\sf RCA_n}.$ Some are elementary, some are not.
Some can be explicitly defined in terms of the (strength of) neat embeddings, some are not, at least its not obvious
how they can.
Recall that for a class $\K$ with a Boolean reduct, $\K\cap \At$ denotes the class of atomic algebras in $\K$; the former is elementary iff the
latter is. 
 
Let $n>2$ be finite. Then we have the following inclusions (note that $\At$ commutes with ${\sf UpUr})$:
$${\sf L}\cap \At\subset \K\cap \At\subset \Nr_n\CA_{\omega}\cap \At\subset {\sf UpUr}\Nr_n\CA_{\omega}\cap \At$$
$$\subset {\sf Up Ur}S_c\Nr_n\CA_{\omega}\cap \At={\sf UpUr}{\sf CRA_n}={\sf LCA}_n\subset {\sf SRCA_n}\subset
{\sf UpUr}{\sf SRCA}_n\subseteq {\sf FOCA_n}$$
$$\subset S\Nr_n\CA_{\omega}\cap \At={\sf WRCA_n}= \RCA_n\cap \At.$$

The majority of  inclusions, and indeed their strictness, 
can be destilled without much difficulty from our previous work.
However, we feel 
that  a brief discussion and some comment are in order.

\begin{itemize}
\item The first two strict inclusions are witnessed in \cite{recent}

\item   The third inclusion is witnessed by a slight modification 
of the algebra $\B$ used in the proof of \cite[theorem 5.1.4 ]{Sayedneat}, 
showing that for any pair of ordinals $1<n<m\cap \omega$, the class $\Nr_n\CA_m$ is not elementary. 
This was done in item (2) in the previous theorem. 
In the constructed model $\sf M$ \cite[lemma 5.1.3]{Sayedneat} on which (using the notation in {\it op.cit}), 
the two algebras $\A$ and $\B$  are based, 
one requires (the stronger) that the interpretation of the $3$ ary relations symbols in the signature 
in $\sf M$ are {\it disjoint} not only distinct as above.
This can be indeed done without much difficulty, and indeed such a model $\sf M$ was obtained 
in \cite{IGPL} using a more basic step-by-step construction.
Atomicity of $\B$ follows immediately,  since its Boolean reduct  is now a product of atomic algebras. 
These are denoted by $\A_u$ except for one countable component $\B_{Id}$, $u\in {}^33\sim \{Id\}$, cf. \cite{Sayedneat} 
p.113-114. Fourth inclusion from item (1) in theorem \ref{SL}. Fifth from item (8) in previous theorem.

Last inclusion follows from the following  $\RA$ to $\CA$ adaptation of an example of Hirsch and Hodkinson
which we use to show that 
${\sf FOCA_n}\subset  {\sf WCA_n}$. This is not at all obvious because they are both elementary.
Take an $\omega$ copy of the  $3$ element graph with nodes $\{1,2,3\}$ and edges
$1\to 2\to\ 3$. Then of course $\chi(\Gamma)<\infty$. Now  $\Gamma$  has a three first order definable colouring.
Since $\M(\Gamma)$ as defined in \cite{HHbook2} is not representable, then the algebra of first order 
definable sets is not representable because $\Gamma$ is interpretable in
$\rho(\Gamma)$, the atom structure constructed from $\Gamma$ as defined in \cite{HHbook2}.
However, it can be shown that the term algebra is representable. (This is not so easy to prove).

\end{itemize}

Call an atomic representable algebra strongly Lyndon if it atomic and \pe\ can win 
the neat game introduced in item (6) in \ref{main}, namely, $H$ restricted to $k$ rounds.
Call this game $H_k$ and let ${\sf SLCA_n}$ denote the elementary closure of this  class.

Then $H_k$ coded by $\rho_k$ is strictly harder than that coded by
$\sigma_k$ (the $k$ th Lynodon condition)  as far as \pa\ is concerned.
A \ws\ for \pe\ in $H_k$ for every $k\geq n$, forces the algebra to be in
${\sf UpUr}\Nr_n\CA_{\omega}$ while a \ws\ for $\sigma_k$ for every $k\geq n$, forces the algebra to
be in ${\sf Up Ur }S_c\Nr_n\CA_{\omega}$ and these classes are distinct.

On the other hand if $J_k$ denotes the $k$ rounded atomic game introduced in item (4) of theorem \ref{main},
Then a \ws\ for \pe\ in $H_k\implies $ a \ws\ for \pe\ in $J_k$ $\implies$ a \ws\ for \pe\ in $G_k.$
We  know that one of the two arrows cannot be reversed; it is possible that they both cannot be reversed.

We do not know whether the elementary closure of ${\sf SRCA_n}$ coincides with ${\sf FOCA_n}.$ We know only 
the ${\sf UpUr SRA_n}$
is contained in ${\sf FOCA_n}$ since ${\sf SRCA_n}\subseteq {\sf FOCA_n}$ 
(this  follows immediately from the definitions)
and that  ${\sf FOCA_n}$ is elementary.

Now that we have three classes in (*) connected in quite a satisfactory way to neat embedding properties, 
it is only natural to get a possibly similar characterization for ${\sf SRCA}_n.$
Here by satisfactory we mean that the classes in question can be obtained by applying `natural' operations on 
$\Nr_n\CA_{\omega}\cap \At$, like complete subalgebras and 
elementary closure, possibly restricting to the countable algebras 
(as is the case with complete representability). We are happy to keep `natural' at this level of ambiguity; for 
in our next theorem we introduce another 
new `natural' operator.

We characterize the class of strongly representable atom structures via neat embeddings, modulo an 
{\it inverse} of Erdos' theorem.

For an atomic algebra $\A$, by an atomic subalgebra we mean a subalgebra of $\A$ containing all its atoms, equivalently a superalgebra of 
$\Tm\At\A$. 
We write  $S_{at}$ to denote this operation applied to an algebra or to a 
class of algebras. $\M(\Gamma)$ denotes the Monk algebra based on the graph 
$\Gamma$ as in \cite{HHbook2}.
Notice that although we have ${\sf UpUr}{\sf SRCA}_n\subseteq {\sf FOCA_n}$, the latter is not closed under forming atomic subalgebras, since
forming subalgebras does not preserve first order sentence.
Notice too that ${\sf Up}{\sf SRCA_n}={\sf Ur}{\sf SRCA_n}$, and hence both are elementary.
The next characterization therefore seems to be plausible.

\begin{theorem} Assume that for every atomic representable algebra that is not strongly representable, there exists a graph $\Gamma$ with finite
chromatic number such that $\A\subseteq \M(\Gamma)$ and $\At\A=\rho(\M(\Gamma))$. 
Assume also that for every graph $\Gamma$ with $\chi(\Gamma)<\infty$, there exists
$\Gamma_i$ with  $i\in \omega$, such that $\prod_{i\in F}\Gamma_i=\Gamma$, for some non principal ultrafilter $F$. 
Then $S_{at}{\sf Up}{\sf SRCA}_n={\sf WRCA_n}=S\Nr_n\CA_{\omega}\cap \At.$ 
\end{theorem}
\begin{proof}
Assume that $\A$ is atomic, representable but not strongly representable. Let $\Gamma$ be a graph with $\chi(\Gamma)<\infty$ such that
$\A\subseteq \M(\Gamma)$ and $\At\A=\rho(\M(\Gamma))$. 
Let $\Gamma_i$ be a sequence of graphs each with infinite chromatic number converging to $\Gamma$, that is, their ultraproduct
is $\Gamma$
Let $\A_i=\M(\Gamma_i)$. Then $\A_i\in {\sf SRSA_n}$, and we have: 
$$\prod_{i\in \omega}\M(\Gamma_i)=\M(\prod_{i\in \omega} \Gamma_i)=\M(\Gamma).$$
And so $\A\subseteq_{at} \prod_{i\in \omega}\A_i$, and we are done.
\end{proof}

\end{document}